\documentclass[aos, 11pt, noinfoline]{imsart}

\RequirePackage[OT1]{fontenc}
\RequirePackage{amsthm,amsmath} 
\RequirePackage[colorlinks,citecolor=blue,urlcolor=blue]{hyperref}
  
\usepackage[normalem]{ ulem }
\usepackage{soul}  
  
\usepackage{amsmath}
\usepackage{amsthm}
\usepackage{amsfonts}
\usepackage{amssymb}
\usepackage{graphicx}
\usepackage{subfigure}
\usepackage{multirow}
\usepackage{color}
\usepackage{rotating}
\usepackage{mathtools}
\usepackage{accents}
\usepackage{amsmath}

\usepackage[all,cmtip]{xy}

\newcommand{\iid}{\stackrel{iid}{\sim}}

\def\bb{\bm{b}}
\newcommand{\Lmax}{L_{max}}

\newcommand{\bm}[1]{\boldsymbol{#1}}
\newcommand{\E}{\mathbb{E\,}}

\renewcommand{\P}{\mathbb{P}}
\newcommand{\mX}{\mathcal{X}}
\newcommand{\mE}{\mathcal{E}}
\newcommand{\mL}{\mathcal{L}}

\newcommand{\M}{\mathcal{M}}

\newcommand{\R}{\mathbb{R}}

\newcommand{\x}{\bm{x}}

\newcommand{\lmax}{L_{max}}
\newcommand{\vk}{\bm{k}}
\newcommand{\y}{\bm{y}}

\renewcommand{\mX}{\mathcal{X}}

\newcommand{\wh}[1]{\widehat{#1}}

\def\X{\bm{X}}
\usepackage[all,cmtip]{xy}
 
\newcommand{\leqa}{\lesssim}
\newcommand{\geqa}{\gtrsim}

\newtheorem{definition}{Definition}%[section]
\newtheorem{lemma}{Lemma}%[section]
\newtheorem{theorem}{Theorem}%[section]
\newtheorem{example}{Example}%[section]
\newtheorem{proposition}{Proposition}%[section]
\newtheorem{remark}{Remark}%[section]
\newtheorem{extension}{Extension}%[section]
\newtheorem{corollary}{Corollary}%[section]

\usepackage{stmaryrd}

\def\C {\,|\:}

\def\e{\mathrm{e}}
\def\d{\mathrm{d\,}}

\def\bT{\mathbb{T}}
\def\mT{\mathcal{T}}
\def\b{\bm{\beta}}
\def\wt{\widetilde}

\usepackage{tikz}
\usepackage{tikz-qtree}
\usetikzlibrary{positioning,shadows,arrows}

\newcommand{\EM}{\ensuremath}

\newcommand{\al}{\alpha}
\newcommand{\be}{\beta}

\newcommand{\ga}{\gamma}
\newcommand{\Ga}{\Gamma}
\newcommand{\la}{\lambda}
\newcommand{\La}{\Lambda}
\renewcommand{\1}{\mathbb{I}}

\newcommand{\ta}{\tau}
\newcommand{\veps}{\varepsilon}
\newcommand{\vphi}{\varphi}

\newcommand{\cA}{\EM{\mathcal{A}}}
\newcommand{\cB}{\EM{\mathcal{B}}}
\newcommand{\cC}{\EM{\mathcal{X}}}
\newcommand{\cc}{\EM{\mathcal{C}}}
\newcommand{\cD}{\EM{\mathcal{D}}}
\newcommand{\cE}{\EM{\mathcal{E}}}

\newcommand{\cH}{\EM{\mathcal{H}}}
\newcommand{\cI}{\EM{\mathcal{I}}}
\newcommand{\cJ}{\EM{\mathcal{J}}}

\newcommand{\cL}{\EM{\mathcal{L}}}
\newcommand{\cM}{\EM{\mathcal{M}}}
\newcommand{\cN}{\EM{\mathcal{N}}}

\newcommand{\cS}{\EM{\mathcal{S}}}
\newcommand{\cT}{\EM{\mathcal{T}}}

\newcommand{\cZ}{\EM{\mathcal{Z}}}

\newcommand{\psg}{{\langle}}
\newcommand{\psd}{{\rangle}}

\DeclareMathAlphabet{\mathpzc}{OT1}{pzc}{m}{it}

\newcommand{\RR}{\mathbb{R}}

\newcommand{\mb}{\mathbb{B}}
\newcommand{\bt}{\mathbb{T}}

\newcommand{\N}{\mathbb{N}}

\newcommand{\ix}{\mathbb{X}}

\newcommand{\given}{\,|\,}

\newcommand{\re}[1]{{\color{red}{#1}}}

\newcommand{\ma}[1]{{\color{magenta}{#1}}}

\definecolor{blendedblue}{rgb}{0.2,0.2,0.7}
\newcommand{\sbl}[1]{{\color{blendedblue}{#1}}}

\newcommand{\rn}{\sqrt{n}}

\newcommand{\ei}{\end{enumerate}}
\newcommand{\ba}{\begin{array}{rcl}}
\newcommand{\ea}{\end{array}}

\newcommand{\mockalph}[1]{}

\usepackage{amssymb,tikz}

\usepackage{xr-hyper}
\begin{document}

\begin{frontmatter}
\title{\small Uncertainty Quantification for Bayesian CART}
% OR: Rates and confidence bands for Bayesian CART
\runtitle{Bayesian CART: rates and bands}
%\title{Multiscale Analysis of Bayesian CART}
%\runtitle{Bayesian CART is Multiscale}
%
%\thankstext{T1}{Footnote to the title with the ``thankstext'' command.}
\begin{aug}

\author{\fnms{Isma\"{e}l} \snm{Castillo}\thanksref{t1,m1}\ead[label=e1]{ismael.castillo@upmc.fr}}
\author{\fnms{and Veronika} \snm{Ro\v{c}kov\'{a} }\thanksref{t2,m2}\ead[label=e2]{veronika.rockova@chicagobooth.edu}}
%\and
%\author{\fnms{Third} \snm{Author}\thanksref{t1,m2}
%\ead[label=e3]{third@somewhere.com}
%\ead[label=u1,url]{http://www.foo.com}}

\thankstext{t1}{The author gratefully acknowledges support from the Institut Universitaire de France and from the ANR grant ANR-17-CE40-0001 (BASICS).}
%St\'{e}phanie van der Pas is assistant professor in Statistics at the LUMC and the Mathematical Institute of Leiden University.: 
%Ismael.Castillo@upmc.fr}
%\thankstext{t3}{Second supporter of the project}
\thankstext{t2}{
%Veronika Ro\v{c}kov\'{a} is assistant professor in Econometrics and Statistics at the Booth School of Business of the University of Chicago:
The author gratefully acknowledges  support from the James S. Kemper Foundation Faculty Research Fund at  the University of Chicago Booth School of Business and the National Science Foundation (grant DMS-1944740).}

%\runauthor{V. Ro\v{c}kov\'{a}}

\affiliation{Sorbonne Universit\'e \thanksmark{m1}}
\affiliation{University of Chicago \thanksmark{m2}}

\address{Sorbonne Universit\'e \& Institut Universitaire de France\\
  Laboratoire de Probabilit\'es, 
Statistique et Mod\'elisation\\ 4, Place Jussieu, 75005 Paris, France\\
  \printead{e1}}

\address{University of Chicago\\ Booth School of Business\\ 5807 S. Woodlawn Avenue\\ Chicago, IL, 60637, USA \\
 \printead{e2}}

\end{aug}

\begin{abstract}

This work  affords new insights into  Bayesian CART  in the context of structured wavelet shrinkage. The main thrust  is to develop a formal inferential framework for Bayesian tree-based  regression.
 We reframe Bayesian CART as  a {\sl $g$-type} prior which departs from the typical wavelet product priors by harnessing correlation induced by the tree topology.
The practically used Bayesian CART priors are shown to attain adaptive near rate-minimax posterior concentration in the {\em supremum norm} in regression models. 
%Gaussian white noise and non-parametric regression. %For inference with certain functionals, we derive an adaptive non-parametric Bernstein-von Mises theorem implying that quantile credible sets are optimal confidence sets. 
 {For the fundamental goal of uncertainty quantification, we construct {\em adaptive} confidence bands for the regression function with uniform coverage  under self-similarity.} In addition, we show that tree-posteriors enable optimal inference in the form of  efficient confidence sets for smooth functionals of the regression function. 

\iffalse
This paper affords new insights about Bayesian CART in the context of  {\em structured} wavelet shrinkage. 
 We show that practically used Bayesian CART priors lead to adaptive rate-minimax posterior concentration in the supremum norm in Gaussian white noise, performing optimally up to a logarithmic factor. To further explore the benefits of structured shrinkage, we  propose the  {\sl $g$-prior} for trees, which departs from the typical wavelet product priors by harnessing correlation induced by the tree topology.
  Building on supremum norm adaptation,  an adaptive non-parametric Bernstein--von Mises theorem for Bayesian CART is derived using multiscale techniques.  {For the fundamental goal of uncertainty quantification, we construct {\em adaptive} confidence bands  with uniform coverage for the regression function under self-similarity.
\fi

 \end{abstract}



\end{frontmatter}

%\tableofcontents

%\title{\bf Multi-scale Analysis of Bayesian CART}
%\author{Ism\"{a}el Castillo\footnote{Professor of Statistics at {\em Sorbonne Universit\'{e}}; \texttt{Ismael.Castillo@upmc.fr} }  \, and Veronika Ro\v{c}kov\'{a}\footnote{Assistant Professor of Econometrics and Statistics at the {\em Booth School of Business of the University of Chicago}; \texttt{VeronikaRockova@ChicagoBooth.edu}; The authors gratefully acknowledge the support by the James S. Kemper Research Fund at the Booth School of Business.}}
%\special{header=special.pro}

\section{Introduction}\label{sec:intro}
The widespread popularity of Bayesian tree-based regression  has raised considerable interest in theoretical understanding
of their empirical success.   However, theoretical literature on  methods such as Bayesian CART and BART is still in  its infancy.
In particular, statistical {\em inferential} theory for regression trees and forests (both frequentist and Bayesian) has been severely under-developed.
%Considerable effort has been recently devoted to frequentist justifications of Bayesian methods in problems of recovering infinite-dimensional objects.
%This work sheds light on aspects of the BART method of \cite{bart}, which has been widely celebrated by practitioners  but largely overlooked by theoreticians.

{This work sheds light on Bayesian CART \cite{cart1, cart2} which is a popular learning tool based on ideas of recursive  partitioning  and which forms an integral constituent of BART \cite{bart}. 
Bayesian Additive Regression Trees (also known as BART) have emerged as one of today's most effective general approaches to predictive modeling under minimal assumptions.} 
Their empirical success has been amply illustrated  in the context of non-parametric regression \cite{bart}, classification \cite{murray}, variable selection \cite{kapelner2015,liu18,linero2016}, shape constrained inference  \cite{mbart}, causal inference \cite{hill, hahn_causal}, to name a few. 
The BART model deploys an additive aggregate of individual trees using Bayesian CART as its building block. 
%, as opposed to a single tree (Bayesian CART). 
While theory for random forests, the frequentist counterpart, has seen numerous recent developments \cite{wager2015,biau,scornet, hooker,wager_athey}, theory for Bayesian CART and BART has not kept pace with its application. With the first theoretical results (Hellinger convergence rates) emerging very recently  \cite{rockova_vdp,linero2018, rockova_saha}, many fundamental questions pertaining to, e.g.,  {convergence in stronger losses such as the supremum norm}, as well as {\em uncertainty quantification} (UQ), have remained to be addressed. 
{This work takes a leap forward in this important direction by developing a formal frequentist statistical framework for uncertainty quantification with confidence bands for Bayesian CART. 

{We first show that Bayesian CART reaches a (near-)optimal  posterior convergence rate under the {\em supremum-norm} loss, a natural  loss for UQ of  regression functions. Many methods that are adaptive for the $L^2$--loss actually fail to be adaptive in an $L^\infty$--sense, as we illustrate below.  We are actually not aware of any sharp supremum-norm convergence rate result for related machine learning methods in the literature, including CART, random forests and deep learning. 
%The sup-norm loss as it intuitively suggests the existence of a band $C_n$ around $f_0$ that shrinks at the sup-norm rate and contains $f_0$ with probability close to one. 
%{\color{blue}  
Regarding inference, we provide a construction of an {\em adaptive} credible band  for the unknown regression function with (nearly, up a to logarithmic term) optimal uniform coverage under self-similarity.}  In addition, we provide efficient confidence sets and bands for a family of smooth functionals. Uncertainty quantification for  related random forests or deep learning has been an open problem, with distributional results available only for point-wise prediction using bootstrap techniques \cite{hooker}.
%{\color{blue}
Our results  make a needed contribution to the literature on the widely sought-after UQ for (tree-based) machine learning methods.
%\sbl{Maybe also briefly mention basis selection?}

% and  (b) asymptotic normality via a full-fledged non-parametric and adaptive Bernstein-von Mises theorem, {\color{blue}(c) valid credible bands for uncertainty quantification of $f_0$ under self-similarity}. 

Regarding supremum-norm (and its {associated} discrete $\ell_\infty$ version) posterior contraction rates,  their derivation is {typically} more delicate compared to the more familiar testing distances (e.g. $L^2$ or Hellinger) for which  general theory has been available since the seminal work \cite{ghosal_etal}. Despite the lack of unifying theory, however, 
advances have been made in the last few years   \cite{ginenickl11, castillo_sup, hoffmann} including specific models \cite{scricciolo14, yoo_ghosal_16, nickl_ray_diffusions19, nickl_soehl_ejs_19}.  However, Bayesian {\em adaptation} for the supremum loss has been obtained, to the best of our knowledge, {\em only} through spike-and-slab priors (the work \cite{yoo_vdv_lepski} uses Gaussian process priors, but adaptation is obtained via Lepski's method).
%a rich and complex body of work has emerged with nearly or exactly rate minimax posteriors under specific priors  \cite{hoffmann,castillo_sup,castillo_polya}. 
In particular, \cite{hoffmann} show that  spike-and-slab priors on wavelet coefficients yield the {\em exact}  adaptive minimax rate  in the white noise model and \cite{yoo_et_al_anisotropic} considers the anisotropic case in a regression framework. 
%{\color{blue} {\em I wonder whether we need all these references since they are for density estimation. Also, we may need to cut down the reference list.}
For density estimation,  \cite{castillo_polya, castillo_mismer2} derive optimal $\|\cdot\|_\infty$--rates for P\'olya tree priors, while \cite{naulet} considers adaptation for log-density spike and slab priors. 
%In this work, we establish  results for the  widely used Bayesian CART priors which form an attractive  alternative to the spike-and-slab approach. 
%{\color{blue}
In this work, we consider  Gaussian white noise and non-parametric regression with Bayesian CART which is widely used in practice.
%\ma{While theoretically sound and interesting, these results have not been directly translated to practice. [remove?]}
% and they  {\em do not} incorporate basis selection which is integral to wavelet shrinkage in higher dimensions. 
%In this work, we establish precise connections to the more pragmatic Bayesian CART priors which, unlike plain wavelet shrinkage methods, enable basis selection  {\sl and}  are computationally tractable in higher dimensions.

%we show that Bayesian CART attains  nearly ideal performance (up to a log term), both in the classical univariate setting but also in higher dimensions.
% a paradigmatic example is
\iffalse
\sbl{[optional] This work uses multiscale techniques, in the sense that one reconstructs a complicated object, a function, from the simpler individual wavelet coefficients thereof; those typically have values globally decreasing along the scales (that is, at higher resolutions), see e.g. \cite{jsh13} and \cite{castillo_nickl1}. In \cite{castillo_nickl1}, a notion of multiscale space was introduced, with norms  strongly related to the supremum norm, see Section \ref{sec-bvm}. We will say slightly informally that a method is multiscale if it achieves adaptive supremum norm rates (up to log factors) and contracts at rate $1/\sqrt{n}$ in multiscale spaces.}
\fi

Bayesian CART  is a method of  function estimation based on ideas of recursive partitioning of the predictor space. 
%In Section \ref{sec:unbal}, we discuss unbalanced Haar wavelet relaxations (Fryzlewicz (2007)) which will be one of the key tools deployed in our theoretical development.
The work \cite{donoho} highlighted  the link between dyadic CART and best ortho-basis selection using Haar wavelets in two dimensions; \cite{piotr}
 furthered this connection by considering unbalanced Haar wavelets of \cite{girardi}. 
 {CART methods have been also studied in the machine learning literature, see e.g. \cite{blanchard04, scott_nowak, willett_nowak}}  {and references therein}. 
Unlike plain wavelet shrinkage methods and standard spike-and-slab priors, general Bayesian CART priors have extra flexibility by allowing for (some) {\em basis selection}.  First results in this direction are derived in Section \ref{sec:non_dyadic}. This aspect is particularly useful in higher-dimensional data, where  CART methods have been regarded as an attractive alternative to other  methods \cite{engel}.

By taking the Bayesian point of view, we relate Bayesian CART to structured wavelet shrinkage  %under the white noise model 
using libraries of {\em weakly} balanced Haar bases. Each tree provides an underlying  skeleton or a `sparsity structure' which supervises 
 the sparsity pattern (see e.g. \cite{baraniuk}). 
We show that  Bayesian CART  borrows strength between coefficients in the tree ancestry by giving rise to a variant of the {\em $g$-prior} \cite{zellner}. Similarly as independent product priors, we show that these dependent priors {\em also} lead to adaptive supremum norm concentration rates (up to a logarithmic factor). To illustrate that local (internal) sparsity is a key driver of adaptivity, we  show that dense trees are incapable of adaptation.

 To convey the main ideas,
the mathematical development will be performed through the lense of a Gaussian white noise model. Our techniques, however,  also apply in non-parametric regression. Results in this setting are briefly presented in Section \ref{sec-reg} with details  postponed until the Supplement (Section \ref{sec:npreg}). 
%{We focus primarily  on the Gaussian white noise model with  Supplement}. 
The white noise model is defined through the following stochastic differential equation, for an integer $n\ge 1$,
\begin{equation}\label{model2}
dX(t)=f_0(t)dt+\frac{1}{\sqrt n}dW(t),\quad t\in[0,1],%\quad n\ge 1,
\end{equation}
where $X(t)$ is an observation process,  $W(t)$ is the standard Wiener process on $[0,1]$ and $f_0$ is unknown and belongs to $L^2[0,1]$, set of squared--integrable functions on $[0,1]$.   
The model \eqref{model2} % is asymptotically equivalent (under some assumptions) to non-parametric regression \cite{brown_low}  and 
is observationally equivalent to a Gaussian sequence space model after projecting the observation process onto a wavelet basis $\{\psi_{lk}:l\geq0,0\leq k\leq 2^{l}-1\}$ of $L^2[0,1]$. This sequence model writes as 
\begin{equation}\label{eq:model2}
X_{lk}=\beta_{lk}^0+\frac{\varepsilon_{lk}}{\sqrt n},\qquad\varepsilon_{lk}\iid\mathcal{N}(0,1),%\quad l\geq 0,\quad k=0,\dots, 2^l-1,
\end{equation}
 where  the wavelet coefficients $\beta_{lk}^0=\langle f_0,\psi_{lk}\rangle= \int_0^1f_0(t)\psi_{lk}(t)dt$ of $f_0$ are indexed by a scale index $l\ge -1$ and a location index $k\in\{0,\ldots,(2^{l}-1)_+\}$. %Note that $X_{lk}$ is an efficient frequentist estimator of each $\beta_{lk}^0$.
A paradigmatic example is the standard Haar wavelet basis 
\begin{equation}\label{haar}
\psi_{-10}(x)=\1_{[0,1]}(x)\quad\text{and}\quad\psi_{lk}(x)=2^{l/2}\psi(2^lx-k)\quad (l\ge 0),
\end{equation}
obtained with orthonormal dilation-translations of  $\psi=\1_{(0,1/2]}-\1_{(1/2,1]}$, where $\1_A$ denotes the indicator of a set $A$. 
Later in the text, we also consider weakly balanced Haar wavelet relaxations (Section \ref{sec:non_dyadic}), as well as smooth wavelet bases (Section \ref{sec:thm-wav}).
% which will be one of the key tools deployed in our theoretical development.
%Any prior on a sequence space $\ell_2$ can be viewed as a prior on the functional parameter $f\in L^2[0,1]$ (under the wavelet isometry $L^2\approx\ell_2$).

\iffalse
Our paper makes contributions on both methodological and theoretical fronts. On the methodological side, we propose tree-shaped sparsity priors which exert local and global sparsity for wavelet shrinkage. In order to borrow strength between coefficients in the tree ancestry, we then propose a variant of the {\em $g$-prior} \cite{zellner} for structured wavelet shrinkage. Similarly as independent product priors, we show that these dependent priors {\em also} lead to adaptive supremum norm concentration rates (up to a logarithmic factor). To illustrate that local (internal) sparsity is a key driver of adaptivity, we will show that dense trees are incapable of adaptation.

\fi

%{\color{blue} Ismael, perhaps you want to reword this paragraph  based on what will be in Section 3.}
One of the key motivations behind the Bayesian approach is  the mere fact that the posterior is an actual distribution, whose limiting shape  can be analyzed towards  obtaining {\em uncertainty quantification} and inference. 
Our results in this direction can be grouped in two subsets. First,  for uncertainty quantification for $f_0$ itself, we  construct adaptive and honest  confidence bands under self-similarity (with coverage converging to one).  
  Exact asymptotic coverage is achieved through intersections with a multiscale credible band (along the lines of \cite{ray}).   Confidence bands construction for regression surfaces is a fundamental task in nonparametric regression and can indicate whether there is empirical evidence to support conjectured features such as multi-modality or exceedance of a level.  Results of this type are, to date, unavailable for classical CART, random forests and/or deep learning.  
  Second, we consider inference for smooth functionals of $f_0$, including linear ones and the primitive functional $\int_0^\cdot f_0$, for which exact optimal confidence sets are derived from posterior quantiles. While these results for functionals are stated in the main paper (Theorem \ref{thm-fun} below), their derivation is most naturally obtained through a general limiting shape result, stated and proved in the Supplement (Theorem \ref{thm_bvm}).   Such an adaptive Bernstein-von Mises theorem for Bayesian CART is obtained following the approach of \cite{castillo_nickl2, castillo_nickl1}; it is only the second result of this kind (providing {\em adaptation}) after the recent result of Ray \cite{ray}.
The paper is structured as follows. Section \ref{sec:tree_prior} introduces regression tree-priors, 
as well as the notion of tree-shaped sparsity and the $g$-prior for trees.  In Section \ref{sec:dyadic}, we state supremum-norm inference properties of Bayesian dyadic  CART (estimation and confidence bands). Section \ref{sec:non_dyadic} considers flexible partitionings allowing for basis choice. A brief discussion can be found in Section \ref{sec:disc}. The proof of our {master   Theorem \ref{thm-one}} can be found in Section  \ref{sec:proof_thm-one}.
The supplementary file \cite{CR19supp} gathers the proofs of the remaining results. The sections and equations of this supplement are referred to with an additional symbol ``S-'' in the numbering.

% while the proofs of further results and technical lemmata can be found in the appendix Section \ref{sec:app}.

%has not been developed to fruition.
%\subsection{Notation}
{{\em Notation.}} %Let $L^2[0,1]$ denote the set of square integrable functions on $[0,1]$ and by 
Let $\mathcal{C}([0,1])$ denote the set of continuous functions on $[0,1]$ and  let $\phi_\sigma$ denote the normal density with zero mean and variance $\sigma^2$. Let $\N=\{0,1,2,\ldots\}$ be the set of natural integers and $\N^*=\N\setminus\{0\}$. We denote by $I_K$ the $K\times K$ identity matrix, 
 Also, $B^c$ denotes the complement of a set $B$. For an interval $I=(a,b]\subset (0,1]$, let $|I|=b-a$ be its diameter  and  $a\vee b=\max(a,b)$. 
The notation $x\leqa y$ means $x\le Cy$ for $C$ a large enough universal constant, and $:=$ (or $=:$) means ``the left-hand side is defined as''.

\section{Trees and Wavelets}\label{sec:tree_prior}
In this section, we discuss multiscale prior assignments on functions $f\in L^2[0,1]$ (i.e. priors on the sequence of wavelet coefficients  $\beta_{lk}=\psg f,\psi_{lk}\psd)$  inspired by (and including) Bayesian CART. %The multiscale setup \eqref{eq:model2} enables one to assign a prior on $f\in L^2[0,1]$   directly through a prior on the sequence of its wavelet coefficients  $(\beta_{lk}=\psg f,\psi_{lk}\psd)$. This section discusses multiscale prior assignments inspired by %Any prior on a sequence space $\ell_2$ can be viewed as a prior on the functional parameter $f\in L^2[0,1]$ (under the wavelet isometry $L^2\approx\ell_2$). 
%As a jumping-off point, we will describe Bayesian CART \cite{cart1,cart2} prior distributions on $(\beta_{lk})$ 
%which yield  piecewise constant reconstructions of the regression surface where  the pieces correspond to terminal nodes of recursive partitioning. 
%The partitioning process can be captured with tree structures which form a skeleton of the prior on  $(\beta_{lk})$ and which are also endowed with a prior.
Such methods  recursively subdivide the predictor space into cells where  $f$ can be estimated locally. The partitioning process can be captured with a tree object (a hierarchical collection of nodes) and a set of splitting rules attached to each node.
 Section \ref{sec:prior_trees}  discusses priors on the tree object.  The splitting rules are ultimately tied to a chosen  basis, where the traditional Haar wavelet basis yields deterministic dyadic splits (as we explain in Section \ref{sec:rpart}).  Later in Section \ref{sec:non_dyadic}, we extend our framework  to random unbalanced Haar bases which allow for more flexible  splits. Beyond random partitioning, an integral component of CART methods are histogram heights assigned to each partitioning cell. 
We flesh out connections between Bayesian histograms and wavelets  in Section  \ref{tree_wavelet_prior}. Finally, we discuss Bayesian CART priors over histogram heights in Section \ref{sec:prior_g}.

\subsection{Priors on Trees  $\Pi_\bT(\cdot)$}\label{sec:prior_trees}

First, we need to make precise our definition of a tree object which will form a skeleton of our prior on $(\beta_{lk})$ for each given basis $\{\psi_{lk}\}$. Throughout this paper, we will largely work with the Haar basis. 
\iffalse
 The first construction is the one typically used in practice when deploying Bayesian CART (also BART, CART or other forest-based algorithms) and considers locally constant functions (histograms). The final output function can be seen as encoded through a triplet $(\cT, \cS, \cH)$, where $\cT$ is a tree, $\cS$ a set of splitting rules and $\cH$ a set of values (the histogram heights). Jointly $(\cT, \cS)$ determine a certain partition of the space of $x$ (the predictor space), while $\cH$ sets the value of the function on each piece of the partition. Often, as is done in the original CART construction (e.g. \cite{breiman_book}), $(\cT,\cS)$ is defined via jointly sampling the tree $\cT$ and labelling its nodes by `decision rules' which recursively cut the parameter space (here $[0,1]$, say). Our techniques will enable quite a few different choices of priors on $\cT,  \cH$, while for technical simplicity we will impose some restrictions on the splits $\cS$, first restricting to given dyadic splits, and later extending the result to allow for (some) more freedom in the randomness of the split choices. The second construction we consider is related but somewhat simpler to analyse mathematically and consists in defining a function via a pair $(\cT, \cH)$ as well as a given wavelet basis $(\psi_{lk})$ (e.g. Haar if one wishes to work with histograms, but we consider also smoother choices below), leading to a function whose multiscale coefficients are prescribed by the tree $\cT$.
\fi

\begin{definition}[Tree terminology]
We define a  {\sl binary tree} $\mT$ as a collection of  nodes $(l,k)$, where $l\in\N,\, k\in\{0,\dots,2^l-1\}$, that satisfies
\[ (l,k)\in\cT,\, l\ge 1 \ \Rightarrow\ (l-1,\lfloor k/2\rfloor) \in \cT.  \]
In the last display, the node $(l,k)$ is a {\sl child} of its {\sl parent} node $(l-1,\lfloor k/2 \rfloor)$.  A {\sl full binary tree} consists of nodes with exactly $0$ or $2$ children. 
For a node $(l,k)$, we refer to $l$ as the {\sl layer index} (or also {\sl depth}) and $k$ as the {\em position} in the $l^{th}$ layer (from left to right). The  cardinality $|\mT|$ of a tree $\mT$ is its total number of nodes and  the {\sl depth}  is  defined as $d(\cT)=\max\limits_{(l,k)\in\mT}l$. 
%We define a full {\sl binary tree} $\mT$ as a collection of  nodes $(l,k)$, where $l\in\N$ is the layer index  (also called {\em depth}) and   $k\in\{0,\dots,2^l-1\}$ is the position of the node $(l,k)$ in the $l^{th}$ layer of the tree (from left to right). {Nodes in a tree are  hierarchically organized in the sense that $(l,k)\in\mT \Rightarrow (l-1,\lfloor k/2\rfloor)\in\mT$}. The {\sl depth} of a tree $\mT$ is defined as $d(\cT)=\max\limits_{(l,k)\in\mT}l$.
\end{definition}

A node $(l,k)\in\cT$ belongs to {the} set $\cT_{ext}$ of {\em external} nodes (also called {\em leaves})  of $\cT$ if it has no children and to {the} set $\cT_{int}$  of {\em internal} nodes, otherwise. 
By definition $|\mT|=|\mT_{int}|+|\mT_{ext}|$, where, for full binary trees,  we have $|\mT|=2|\mT_{int}|+1$.
An example of a full binary tree is depicted in Figure \ref{figure1a}. 
In the sequel,   $\bT$ denotes the set of full binary trees of depth no larger than $L=L_{max}=\lfloor \log_2 n\rfloor$,
%The choice of full binary trees is traditional and mostly for simplicity of presentation, and 
 a typical cut-off in wavelet analysis. Indeed,  trees can be associated with certain wavelet decompositions, as will be seen in Section \ref{BCART_priors}.

{Before defining tree-structured priors over the entire functions $f$'s, we first discuss} various ways  of assigning a prior distribution over $\bT$, that is over trees themselves. We focus on the Bayesian CART prior \cite{cart1}, which  became an integral component of many Bayesian tree regression methods  including BART \cite{bart}.

\subsubsection{\sl  Bayesian CART Priors}\label{sec:bc1}
The Bayesian CART construction of \cite{cart1}  assigns a prior over  $\bT$ via the heterogeneous Galton-Watson (GW) process.  %We now provide an algorithmic description of this process. 
The prior description utilizes the following top-down left-to-right exploration metaphor (see also \cite{rockova_saha}). Denote with $Q$ a queue of nodes waiting to be explored. 
Each node $(l,k)$ is assigned a {random} binary indicator $\gamma_{lk}\in\{0,1\}$ for whether or not it is split. 
%\ma{[It seems to me that it is only  $\mT_{int}$ since the bottom leaves have $\gamma_{lk}=0$]} \sbl{[Ah, I forgot about the indicators indeed! But I think $\cT_{int}$ had similarly an issue because with the process below (old version) the nodes at level $L$ could still be split. Also, $\cT_{int}=\emptyset$ is a bit ambiguous (e.g. what is $\cT$ then?). New suggestion below.]}
Starting with  ${{\mT}}=\emptyset$, one initializes the exploration process  by putting the root node $(0,0)$ tentatively in the queue, i.e. $Q=\{(0,0)\}$.
One then repeats the following three steps until $Q=\emptyset$: 
\begin{itemize}
\item[(a)] Pick a node $(l,k)\in Q$ with the highest priority (i.e. the smallest index $2^l+k$) and {if $l<L_{max}$}, split it with probability  
\begin{equation}\label{eq:split_prob}
p_{lk}=\mathbb{P}(\gamma_{lk}=1).
\end{equation}
{If $l=L_{max}$, set $\ga_{lk}=0$.}
%\sbl{where $\ga_{lk}$ is a Bernoulli$(p_{lk})$ variable. 
%If $l\geq L_{max}$, set $\ga_{lk}=0$.\\
%{\em (note: added this because $\ga_{lk}$ was not defined yet?)}}
\item[(b)] If $\gamma_{lk}=0$, remove $(l,k)$ from $Q$.
\item[(c)]   If $\gamma_{lk}=1$,  then
\begin{itemize}
%\item[(i)] split $I_{lk}$, e.g. at the dyadic midpoint of this interval (dyadic trees),
\item[(i)] add $(l,k)$ to the tree, i.e. 
$
%{\mT}\ma{_{int}}{\,\leftarrow\,}{\mT\ma{_{int}}}\cup\{(l,k)\},
{{\mT_{int}}}{\,\leftarrow\,}{{\mT_{int}}}\cup\{(l,k)\},
$
\item[(ii)] remove $(l,k)$ from $Q$ and  if $l< L_{max}$ add its children to $Q$, i.e. 
$$
Q{\ \leftarrow\ }Q\backslash\{(l,k)\}\cup\{(l+1,2k),(l+1,2k+1)\}.
$$
%\item Sample $\beta_{lk}\sim \mathcal{N}(0,1)$.
\end{itemize}
\end{itemize}
%Note that the step  (i) can be modified by incorporating other than dyadic splits, as we explain later in Section \ref{sec:unbal}.
The tree skeleton is probabilistically underpinned by the cut probabilities $(p_{lk})$ which are typically assumed to decay with the depth $l$ as a way to penalise too complex trees. While \cite{cart1} suggest $p_{lk}=\alpha/(1+l)^\gamma$ for some $\alpha\in(0,1)$ and $\gamma>0$, \cite{rockova_saha} point out that this decay may not be fast enough and suggest instead  $p_{lk}=\Gamma^{-l}$ for some $2<\Gamma<n$, which leads to a (near) optimal empirical $L^2$--convergence rate.
We use a similar assumption in our analysis, and also assume that the split probability depends only on $l$, and simply denote $p_l=p_{lk}$. 
%Each dyadic tree $\mT\in\bT$  can be uniquely identified by a collection of binary indicators $\Gamma(\mT)=\{\gamma_{00},\gamma_{10},\dots,\gamma_{d(\mT),2^{d(\mT)}-1}\}$ for whether or not each node $(l,k)$ was split. %We relate this representation to spike-and-slab wavelet shrinkage  in Section \ref{sec:ssl_connection}.
%Unlike in the homogeneous case (where all $\gamma_{lk}$'s are iid), \eqref{eq:split_prob} defines a {\sl heterogeneous} GW process where the offspring distribution is allowed to vary from generation to generation, i.e. the variables $\gamma_{lk}$ are independent but {\sl non-identical}. 

%\end{minipage}\qquad \hfill
%\begin{minipage}[b]{.4\textwidth}
%\subsubsection{\sl  The Bayesian CART Prior (\`{a} la \cite{cart2})}\label{sec:bc2}
Independently of \cite{cart1}, \cite{cart2} proposed another variant of Bayesian CART, which first draws  the number of leaves (i.e. external nodes) $K=|\mT_{ext}|$ at random from a certain prior %$\pi(K)$ 
on integers, e.g.  a    Poisson distribution (say, conditioned to be non-zero). Then, a tree $\mT$ is sampled uniformly at random from all full binary trees with $K$ leaves. Noting that there are $\mathbb{C}_{K-1}$ such trees, with $\mathbb{C}_K$ the $K$--th Catalan number (see Lemma \ref{catalan}), this leads to $\Pi(\cT)= (\la^K/[K! (e^{\la}-1)])\cdot \mathbb{C}_{K-1}^{-1}$. As we restrict to trees in $\bT$, i.e. with depth at most $L=L_{max}$, we slightly update the previous prior choice by setting, for some $\la>0$, with $K=|\mT_{ext}|$,
\begin{equation}\label{prior:K}
\Pi_\bT(\cT) \propto \frac{\lambda^{K}}{(\e^\lambda-1)K!}\frac{1}{\mathbb{C}_{K-1}}\,\1_{\cT\in\bT}, 
\end{equation}
where $\propto$ means `proportional to'. {We call the resulting prior $\Pi_\bT$ the `conditionally uniform prior' with a parameter $\la$.}

\begin{figure}[!t]
    \subfigure[A full binary tree]{
     \label{figure1a}
    \begin{minipage}[t]{0.45\textwidth}

\tikzset{every node/.style={rectangle,  fill=gray, text=black}
  }
%\begin{figure}
\begin{center}
\scalebox{0.8}{\begin{tikzpicture}[%every tree node/.style={draw,circle},
   level distance=1.4cm,sibling distance=1cm, 
   edge from parent path={(\tikzparentnode) -- (\tikzchildnode)}]
\Tree [.\node[draw,circle,label={(0,0)},color=blue] {}; 
    \edge ;%node[auto=right] {}; 
    [.\node[draw,rectangle,label={(1,0)},color=red]{};  
          ]
     \edge;% node[auto=left] {};      
    [.\node[draw,circle,label={(1,1)},color=blue]{};
    \edge ;%node[auto=right] {};  
      [.\node[draw,circle,label={(2,2)},color=blue]{}; 
       \edge ;%node[auto=right] {};  
      [.\node[draw,rectangle,label={(3,4)},color=red]{}; 
            ]  
      \edge ;%node[auto=left] {}; 
      [.\node[draw,rectangle,label={(3,5)},color=red]{};     
           ] 
] 
\edge ;%node[auto=left] {}; 
    [.\node[draw,rectangle,label={(2,3)},color=red]{};
         ] 
    ] ]
\end{tikzpicture}}
\end{center}
%\end{minipage}
    \end{minipage}}
     \subfigure[Binary tree of prior cut probabilities]{
     \label{figure1b}
     \begin{minipage}[t]{0.45\textwidth}
     \vspace{-4cm}
\begin{center}
\scalebox{0.7}{
\begin{tikzpicture}[%every tree node/.style={draw,circle},
   level distance=1.4cm,sibling distance=1cm, 
   edge from parent path={(\tikzparentnode) -- (\tikzchildnode)}]
\Tree [.\node[draw,circle] {$p_{00}$}; 
    \edge node[auto=right] {}; 
    [.\node[draw,circle]{$p_{10}$};  
      \edge node[auto=right] {};  
      [.\node[draw,circle]{$p_{20}$}; 
      \edge[dashed]; {}
      \edge[dashed]; {}
      ]  
      \edge node[auto=left] {}; [.\node[draw,circle]{$p_{21}$}; 
      \edge[dashed]; {} 
      \edge[dashed]; {}
      ] 
          ]
     \edge node[auto=left] {};      
    [.\node[draw,circle]{$p_{11}$};
    \edge node[auto=right] {};  
      [.\node[draw,circle]{$p_{22}$}; 
       \edge[dashed]; {}
      \edge[dashed]; {}
      ] \edge node[auto=left] {}; 
    [.\node[draw,circle]{$p_{23}$};
     \edge[dashed]; {}
      \edge[dashed]; {}
     ] 
    ] ]
\end{tikzpicture}}
\end{center}
%\end{figure}
    \end{minipage}}
    
    \caption{(Left) A full binary tree $\mT=\mT_{int}\cup\mT_{ext}$. Red nodes are external nodes $\mT_{ext}$ and blue nodes are internal nodes $\mT_{int}$. (Right) A binary tree of cut probabilities $p_{lk}$ in \eqref{eq:split_prob}.}    \label{figure1}

\end{figure}
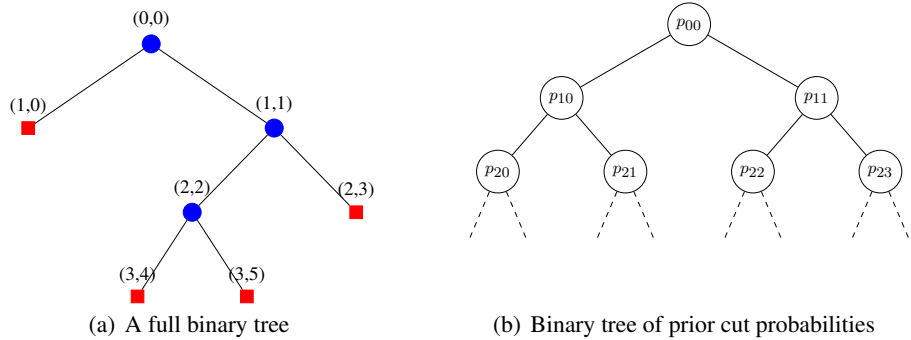
%\subsection{Prior Distributions on Trees} \label{sec:tree_prior}

\subsubsection{Trees and Random Partitions} \label{sec:rpart}
Trees provide a  structured framework for generating random partitions of the predictor space (here we choose $(0,1]$ for simplicity of exposition).
In CART methodology, each node $(l,k)\in\mT$ is associated with a partitioning interval $I_{lk}\subseteq (0,1]$.
Starting from the trivial partition $I_{00}=(0,1]$, the simplest way to obtain a  partition {is} by  successively dividing each $I_{lk}$ into $I_{lk}=I_{l+1\, 2k}\cup I_{l+1\,2k+1}$.
One central example is {\em dyadic}  intervals  $I_{lk}$ which correspond to the domain of the balanced Haar wavelets  $\psi_{lk}$ in \eqref{haar}, i.e.
\begin{equation}\label{eq:domain}
 I_{00}=(0,1],\quad I_{lk}=(k2^{-l},(k+1)2^{-l}]\quad\text{for $l\ge 0$ and $0\le k<2^l$}.
 \end{equation}
For any fixed depth $l\in\N$,  the  intervals $\cup_{0\leq k<2^l}I_{lk}$ form a deterministic regular (equispaced) partition of $(0,1]$. 
Trees, however, generate {\em more flexible} partitions $\cup_{(l,k)\in\cT_{ext}}I_{lk}$  by keeping  only those intervals $I_{lk}$   attached to the leaves of the tree. 
Since $\cT$ is treated as random with a prior $\Pi_\bT$ (as defined in Section \ref{sec:prior_trees}), the resulting partition  will  also be random.
  
%    $\mathbb{T}_d$  denotes  the set of full binary trees of height (we shall also say depth) equal to $d$. 
\begin{example} \label{exam1}
Figure \ref{figure1a} shows  a full binary tree $\mT=\mT_{int}\cup\mT_{ext}$, where $\mT_{int}=\{(0,0), (1,1),(2,2)\}$ and $\mT_{ext}=\{(1,0),(2,3),(3,4),(3,5)\},$
resulting in the partition of $(0,1]$ given by
\begin{equation}\label{eq:dyadic_interval}
(I_{lk})_{(l,k)\in\mT_{ext}}=\{(0,1/2],(1/2,5/8],(5/8,3/4],(3/4,1]\}.
\end{equation}
\end{example}

%Using the two-index enumeration, we implicitly assume that nodes  inside $\mT$ are ordered  in a top-down left-to-right fashion according to their priority  $2^l+k$. 
%Each node $(l,k)$ can be associated with  a partitioning interval $I_{lk}\subseteq (0,1]$. For example,  in dyadic trees these intervals correspond to the domain of the balanced Haar wavelets  $\psi_{lk}$ in \eqref{haar}, i.e.
%\begin{equation}\label{eq:domain}
% I_{00}=(0,1],\quad I_{lk}=(k2^{-l},(k+1)2^{-l}]\quad\text{for $l\ge 0$ and $0\le k<2^l$}.
% \end{equation}
% For any fixed $l\in\N$,  the  intervals $I_{l\,0},\ldots,I_{l\,2^l-1}$ form a regular (equispaced) partition of $(0,1]$. 
%

%Trees provide a framework for generating more flexible (irregular) partitions via recursive binary splitting. Starting from the trivial partition I00, one can generate new partitions by successively applying a splitting rule on a chosen internal node, say (l, k).  In dyadic trees, this amounts to splitting at the dyadic rational midpoint of Ilk,
The set of possible {\em split points} obtained with  \eqref{eq:domain} is confined to dyadic rationals.  
One can interpret the resulting partition as the result of recursive splitting where, at each level $l$, 
  intervals $I_{lk}$ for each internal node $(l,k)\in\mT_{int}$ are cut in half and  intervals $I_{lk}$ for each external node $(l,k)\in\mT_{ext}$ are left alone.  
 We will refer to such a recursive splitting process as {\em dyadic  CART}.  There are several ways to generalize this construction, for instance by considering arbitrary splitting rules that iteratively dissect the intervals at values other than {the midpoint}. %dyadics. 
 We explore {such extensions} %these non-dyadic extensions 
 in Section \ref{sec:non_dyadic}.

%\begin{figure*}[h]
%\centering
%\begin{minipage}[b]{.4\textwidth}
%%\begin{figure}

\subsection{Tree-shaped Priors on $f$}\label{tree_wavelet_prior}
This section outlines two strategies for assigning a tree-shaped prior distribution on $f$ underpinned by  a tree skeleton $\mT\in\bT$.
%We now introduce tree-based wavelet shrinkage priors as a more flexible alternative to sieve priors. 
 %In Section \ref{sec:non_dyadic}, we will consider even more general (non-dyadic) trees where the splits are not necessarily dyadically ordered.
Each tree $\mT=\mT_{int}\cup\mT_{ext}$ can be  associated with two sets of coefficients: (a)  {\sl internal} coefficients $\beta_{lk}$ attached to wavelets $\psi_{lk}$ for $(l,k)\in\mT_{int}$ and (b) 
 {\sl external} coefficients $\wt\beta_{lk}$ attached to partitioning  intervals $I_{lk}$ for $(l,k)\in\mT_{ext}$ (see  Section \ref{sec:rpart}).
 %where $I_{lk}$ were defined earlier in \eqref{eq:dyadic_interval}.
While wavelet priors (Section \ref{sec:wavelet_prior}) assign the prior distribution internally on  $\beta_{lk}$,  Bayesian CART priors \cite{cart1,cart2} (Section \ref{BCART_priors}) assign the prior  externally on $\wt\beta_{lk}$. We discuss and relate these two strategies in more detail below.

\subsubsection{Tree-shaped Wavelet Priors}\label{sec:wavelet_prior}
Traditional (linear) Haar wavelet reconstructions for $f$ deploy {\em all} wavelet coefficients $\beta_{lk}$ with resolutions $l$ smaller than  some $d>0$. This strategy amounts to fitting a {\em flat tree} with $d$ layers  (i.e. a tree that contains all nodes up to  a level $d$, see Figure \ref{bayestrees2}) or, equivalently, a regular dyadic regression histogram with $2^d$ bins. 
%Below, we expand on the connection between {\em dyadic} trees and balanced Haar wavelets \eqref{haar} by considering more general tree structures than just flat trees. 
This construction can be made more flexible by selecting coefficients prescribed by trees that are not necessarily flat.
Given a full binary tree $\mT\in\bT$, 
{%\sbl{[thanks for noting this;  I needed the notation at some point in the BvM proof; this made me also realise that we were later using both $\wt\b$ and $\wt\b_\cT$ for the same object. I updated that, please see next page] }
% of reals $\sbl{\{\beta_{-10}, (\beta_{lk})_{0\le l\le L-1, 0\le k<2^l} \}}$\ma{[is this the collection referred to as $\b$ below?]},
 one can build the following wavelet reconstruction of $f$ using {\em only} active wavelet coefficients that are {\em inside} a tree $\mT$
\begin{align}
f_{\mT,\b}(x)=\beta_{-10}\psi_{-10}(x)+\sum_{(l,k)\in\mT_{int}}\beta_{lk}\psi_{lk}(x) = \sum_{(l,k)\in\mT_{int}'}\beta_{lk}\psi_{lk}(x)
,\label{tree_expand22}
\end{align}
where   $\b=(\beta_{-10}, (\beta_{lk})_{0\le l\le L-1, 0\le k<2^l})'$} is a vector of wavelet coefficients  and where $\cT_{int}'=\cT_{int}\cup \{(-1,0)\}$ is the `rooted' tree with the index $(-1,0)$ added to $\cT_{int}$. Note that $|\cT_{int}'|=|\cT_{ext}|$.

Define a   {\em tree-shaped wavelet prior} on $f_{\mT,\b}$   as the prior induced by the hierarchical model
\begin{align}
\cT \qquad & \sim \qquad\ \Pi_{\bT}  \nonumber \\
(\beta_{lk})_{lk}\C\mT \ & \sim \, \bigotimes_{(l,k)\in\mT_{int}'} \pi(\beta_{lk}) \ \otimes \ \bigotimes_{(l,k)\notin \mT_{int}'}\delta_0(\beta_{lk}), \label{eq:prior_beta}
\end{align}
where $\Pi_{\bT}$ is a prior on trees as described in Section \ref{sec:bc1} and where  the active wavelet coefficients $\beta_{lk}$ for $(l,k)\in\mT_{int}$ follow a distribution with a bounded and positive density $\pi(\beta_{lk})$  on $\RR$.   The prior \eqref{eq:prior_beta} %specifies the coefficients of $\b$ 
%$(\be_{lk})$ for $l\le L-1$, and the prior on $\b$ 
 is seen as a distribution on $\RR^{2^L}$, where  all remaining coefficients, i.e. ${\be_{lk}}$'s for $l\ge L$, are set to $0$.%, so that a distribution on the collection of all wavelet coefficients is now specified.

The prior \eqref{eq:prior_beta} contains the so-called {\em sieve priors} \cite{castillo_nickl1} (i.e. flat trees) as a special case, where the sieve is with respect to the approximating spaces $\text{Vect}\{\psi_{lk}, l<d\}$ for some $d\ge 0$. For nonparametric estimation of $f_0$,  it is well-known that sieve priors can achieve (nearly) adaptive rates in the $L^2$--sense (see e.g. \cite{gvbook}). In turns out, however,  that sieve priors (and therefore flat tree  priors) are too rigid to enable adaptive results for stronger losses such as the supremum norm, as we demonstrate in Theorem \ref{thm_lb} in Section \ref{sec:flatneg} (Supplement). This theorem illustrates that supremum norm adaptation using Bayesian (or other likelihood-based) methods is a delicate phenomenon that is not attainable by many typical priors. %thm_lb  
% In \cite{castillo_nickl1}, a notion of a multiscale space was introduced, with norms  strongly related to the supremum norm, see Section \ref{sec-bvm}.
%we  formalize a negative result showing that  flat trees are  incapable of $\ell_\infty$ adaptation.

By definition, the prior \eqref{eq:prior_beta} weeds out all wavelet coefficients $\beta_{lk}$ that are not supported by the tree skeleton (i.e. are not {\em internal} nodes in $\mT$). This has two shrinkage implications: global and local.
First, the global level of truncation (i.e. the depth of the tree) in \eqref{eq:prior_beta} is not fixed but random. Second,  unlike in sieve priors, only some low resolution   coefficients are active  depending on whether or not the tree  splits the node $(l,k)$.  These two shrinkage aspects create hope that  tree-shaped wavelet priors \eqref{eq:prior_beta} attain adaptive  supremum norm rates (up to log factors) and enable construction of adaptive confidence bands. We see later in Section \ref{sec:dyadic} %demonstrate in Section \ref{sec:dyadic} 
that this optimism is indeed warranted.
%`multiscale behavior', i.e. they achieve adaptive supremum norm rates (up to log factors) and contract at the rate $1/\sqrt{n}$ in multiscale spaces  \cite{castillo_nickl1}.
%We will demonstrate in Section \ref{sec:dyadic} that this optimism is indeed warranted. %Before doing so, however, we make a few connections to existing literature.
%Throughout the paper we will say (slightly informally) that a method is multiscale if it has the multiscale behavior described above.

%\vspace{-0.2cm}
%\paragraph{Connection to Spike-and-Slab Priors}\label{sec:ssl_connection}
\iffalse
The tree-priors introduced above induce {\em irregular} dyadic partitions, where the partitioning intervals $I_{lk}$ are {\em not} necessarily of equal length. Standard wavelet thresholding  \cite{donoho_johnstone95} reconstructs   regression surfaces with an inverse  wavelet transform of thresholded coefficients, which {\em also} yields a piece-wise constant reconstruction where the pieces are not necessarily of equal size.   This brings us to the following interesting connection to spike-and-slab wavelet priors.
\fi
For adaptive   wavelet shrinkage, \cite{chipman_wavelet} propose a Gaussian mixture spike-and-slab prior on the wavelet coefficients.
The point mass spike-and-slab incarnation of this prior was studied by \cite{hoffmann}
%, who  show  adaptive minimax posterior concentration over H\"{o}lder balls for the sup-norm loss, and by %(without any additional $\log n$ term). 
  and \cite{ray}. % who subsequently quantified an adaptive BvM  property in the multiscale setting using a similar prior.  
  Independently for each  wavelet coefficient $\beta_{lk}$ at resolutions larger than some $l_0(n)$ (strictly increasing sequence), the prior  in \cite{ray} can be written in the standard spike-and-slab form 
\begin{align}\label{eq:beta_ss}
\pi(\beta_{lk}\C\gamma_{lk})&=\gamma_{lk}\pi(\beta_{lk})+(1-\gamma_{lk})\delta_0(\beta_{lk}),
\end{align}
where $\gamma_{lk}\in\{0,1\}$ for whether or not the coefficient is active with $\P(\gamma_{lk}=1\C\theta_l)=\theta_l$. Moreover, the prior on all coefficients at resolutions no larger than $l_0(n)$ is dense, i.e. $\theta_l=1$ for $l\leq l_0(n)$. The value $\theta_l$ can be viewed as the probability that a given wavelet coefficient $\beta_{lk}$ at resolution $l$ will contain `signal'. 
%Donoho et al (1995) suggested the universal threshold value as a probabilistic upper bound on the size of the noise over empirical coefficients.
\iffalse\cite{chipman_wavelet} suggest setting $\theta_l$ equal to the proportion of signal coefficients (at resolution $l$) as determined by the universal threshold value, whereas
\cite{ray} specifies $n^{-a}\leq \theta_{l}\leq 2^{-l(1+b)}$ for some $a>0$ and $b>1/2$ for $l_0(n)<l\leq L_n$, where $L_n=\lfloor\log n/\log 2\rfloor$ and $l_0(n)\asymp(\log n)^{1/(2\nu+1)}$ for some $\nu>0$. \fi

There are undeniable similarities between \eqref{eq:prior_beta} and \eqref{eq:beta_ss}, in the sense that the binary inclusion indicator $\gamma_{lk}$ in  \eqref{eq:beta_ss}  can be regarded as the node  splitting indicator $\gamma_{lk}$ in \eqref{eq:split_prob}. While the indicators $\gamma_{lk}$ in  \eqref{eq:beta_ss} are {\sl independent} under the spike-and-slab prior, they are hierarchically constrained under the CART prior, where the pattern of non-zeroes encodes the tree oligarchy.
%While the spike-and-slab prior is widely regarded as the methodological ideal \cite{hoffmann}, it is not very practical in higher dimensions and is confined to one given basis. The Bayesian CART prior  is widely used in practice, allows for basis selection (as seen in Section \ref{sec:non_dyadic}) but  is not yet theoretically well understood. 
The  seeming resemblance of the CART-type prior  \eqref{eq:prior_beta} to the spike-and-slab prior \eqref{eq:beta_ss} makes one naturally wonder whether, unlike sieve-type priors,   CART posteriors attain adaptive supremum-norm inference. %We provide affirmative answers to this question in Section \ref{sec:dyadic}.

\subsubsection{Bayesian CART Priors}\label{BCART_priors}
\iffalse
Each tree $\mT=\mT_{int}\cup\mT_{ext}$ can be  associated with two sets of coefficients: (a)  {\sl internal} coefficients $\beta_{lk}$ attached to wavelets $\psi_{lk}$ for $(l,k)\in\mT_{int}'$ and (b) 
 {\sl external} coefficients $\wt\beta_{lk}$ attached to intervals $I_{lk}$ for $(l,k)\in\mT_{ext}$, {as defined in the next paragraph}.
 %where $I_{lk}$ were defined earlier in \eqref{eq:dyadic_interval}.
 Bayesian CART priors \cite{cart1,cart2}, as opposed to wavelet priors, assign the prior distribution externally on $\wt\beta_{lk}$. 
 %\ma{Going forward, we will abbreviate these histogram priors as BCART priors.}
\fi
A perhaps more transparent approach to assigning a tree-shaped prior on $f$ is through histograms (as opposed to wavelet reconstructions from Section \ref{sec:wavelet_prior}). Each tree $\cT\in\bT$ generates a random partition via intervals $I_{lk}$  (see Section \ref{sec:rpart}) and gives rise to the following histogram representation  %indexed by {\sl ordered external} node coefficients $\wt\beta_{lk}$ (ascending ordering according to $2^l+k$) of $\mT$, 
\begin{align}
\wt f_{\mT,\wt\b}(x)=\sum_{(l,k)\in\mT_{ext}}\wt \beta_{lk}\mathbb{I}_{I_{lk}}(x),\label{tree_expand}
\end{align}
 where ${\wt\b}=(\wt\beta_{lk}:(l,k)\in\mT_{ext})'$  is a vector of reals interpreted as step heights
and where  $I_{lk}$'s are {obtained from the tree $\cT$} as in Section \ref{sec:rpart} (and as illustrated in Example \ref{exam1}). 
%\eqref{eq:dyadic_interval}). 
%\sbl{\em [Note: $f_{\mT,\wt\b}$ is a slight abuse of notation as we do not define $\b_{\cT}$ but this OK I think]}
We now define the  {\em (Dyadic) Bayesian CART prior} on $f$ using the following hierarchical model on the {\em external} coefficients rather than {\em internal} coefficients (compare with   \eqref{eq:prior_beta})
\begin{align}
\cT \qquad & \sim \qquad \Pi_{\bT}  \nonumber \\
(\wt \beta_{lk})_{(l,k)\in\mT_{ext}}\,\C\,\mT \ & \sim \, \bigotimes_{(l,k)\in\mT_{ext}} {\wt\pi}(\wt \beta_{lk}),  \label{eq:prior_betati}
\end{align}
where $\Pi_{\bT}$ is as in Section \ref{sec:prior_trees}, and where the height $\wt \beta_{lk}$ at a specific $(l,k)\in\mT_{ext}$ has a bounded and positive density $\wt\pi(\wt \beta_{lk})$ on $\RR$.  This model {\em coincides} with the widely used Bayesian CART priors using {a midpoint} dyadic splitting rule (as we explained in Section \ref{sec:rpart}). In practice, the density $\wt\pi$ is often chosen as  centered Gaussian with some variance $\sigma^2>0$ \cite{cart1,cart2}. 
%
%The basic observation is that each dyadic tree $\mT=\mT_{int}\cup\mT_{ext}$ can be  associated with two sets of coefficients: (a)  {\sl internal} coefficients $\beta_{lk}$ attached to wavelets $\psi_{lk}$ for $(l,k)\in\mT_{int}$ and (b) 
% {\sl external} coefficients $\wt\beta_{lk}$ attached to intervals $I_{lk}$ for $(l,k)\in\mT_{ext}$, where $I_{lk}$ were defined earlier in \eqref{eq:dyadic_interval}.
%%each interval $I_{lk}$ forms a basis of a histogram bin and is equipped with a step height $\wt \beta_{lk}$ sampled from $\wt\pi(\cdot)$. 
%We denote with $\wt\b=(\wt\beta_{lk}:(l,k)\in\mT_{int})'$ the vector of {\sl ordered external} node coefficients $\wt\beta_{lk}$ (ascending ordering according to $2^l+k$) for a given tree $\mT$.  Given the binary tree $\mT$ and step heights $\wt\b$, one can approximate $f_0$ with a histogram reconstruction
%\begin{align}
%\wt f_{\mT,\wt\b}(x)=\sum_{(l,k)\in\mT_{ext}}\wt \beta_{lk}\mathbb{I}_{I_{lk}}(x).\label{tree_expand}
%\end{align}
%%where the histogram bins  $I_{lk}$ are the dyadic intervals of the terminal nodes $(l,k)\in\mT_{ext}$.  
%%While  the external coefficients correspond to the histogram reconstruction, the internal coefficients ones correspond to the Haar reconstruction.

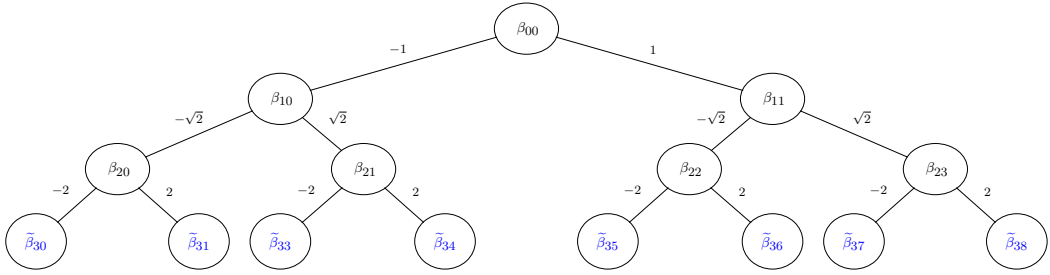
\begin{figure}
\begin{center}
\scalebox{0.6}{
\xymatrix@-1pc{
%& & & &  & &  *++[o][F-]\txt{$\beta_{-10}$}  \ar@{-}[d]^1  &&& & & &\\
& & & &  & &  *++[o][F-]\txt{$\ \beta_{00}$}  \ar@{-}[dlll]_{-1} \ar@{-}[drrr]^{1} &&& & & &\\
 & & &     *++[o][F-]{\txt{$\ \beta_{10}$}}\ar@{-}[dll]_{-\sqrt{2}} \ar@{-}[dr]^{\sqrt{2}}& & && &&  *++[o][F]\txt{$\ \beta_{11}$}  \ar@{-}[dl]_{-\sqrt{2}}\ar@{-}[drr]^{\sqrt{2}}& &&\\
&      *++[o][F]\txt{$\ \beta_{20}$} \ar@{-}[dl]_{-2} \ar@{-}[dr]^2   && & *++[o][F-]{\txt{$\ \beta_{21}$}}\ar@{-}[dl]_{-2} \ar@{-}[dr]^2 &   & &  &  *++[o][F-]{\txt{$\ \beta_{22}$}}\ar@{-}[dl]_{-2} \ar@{-}[dr]^2 & & & *++[o][F-]\txt{{$\ \beta_{23}$}}\ar@{-}[dl]_{-2} \ar@{-}[dr]^2& \\
*++[o][F]\txt{\color{blue}$\wt\beta_{30}$} & & *++[o][F]\txt{\color{blue}$\wt\beta_{31}$} & *++[o][F]\txt{\color{blue}$\wt\beta_{33}$}& &*++[o][F]\txt{\color{blue}$\wt\beta_{34}$}  &&*++[o][F]\txt{\color{blue}$\wt\beta_{35}$}  &&  *++[o][F]\txt{\color{blue}$\wt\beta_{36}$}  &   *++[o][F]\txt{\color{blue}$\wt\beta_{37}$}  & & *++[o][F]\txt{\color{blue}$\wt\beta_{38}$}    \\
}}
\end{center}
\caption{Flat tree with edges  weighted by the amplitude of the Haar wavelets. }
\label{bayestrees2}
\end{figure}

The histogram prior \eqref{tree_expand} can be rephrased in terms of wavelets.
Indeed, the histogram representation  \eqref{tree_expand}  can  be rewritten in terms of the {\em internal} coefficients, i.e.  $\wt f_{\mT,\wt\b}(x)=f_{\mT,\b}(x)$ as in \eqref{tree_expand22}, with $\beta_{lk}$'s and $\wt\beta_{lk}$'s linked via
 %, where
%\begin{align}
%f_{\mT,\b}(x)\equiv\beta_{-10}\psi_{-10}(x)+\sum_{(l,k)\in\mT_{int}}\beta_{lk}\psi_{lk}(x).\label{tree_expand22}
%\end{align}
%For each external coefficient $\wt\beta_{lk}$, we define its {\sl ancestry} as all the  internal node coefficients on the path from the root $(0,0)$ to $(l,k)$, i.e.
%$$
%Anc(\wt\beta_{lk})=\{\beta_{lk}:(l,k)\in [(0,0)\leftrightarrow(l,k)]_{\mT}\equiv\{(0,0),(1, \lfloor k/2^{l-1}\rfloor),\dots,(l-1,\lfloor k/2\rfloor)\}\}.
%$$ 
%From \eqref{tree_expand} and \eqref{tree_expand22}, one can conclude  the following link between $\beta_{lk}$'s and $\wt\beta_{lk}$'s:
\begin{equation}
\wt \beta_{lk}
=\beta_{-10}+\sum_{j=0}^{l-1}s_{\lfloor k/2^{l-j-1}\rfloor}2^{j/2}\beta_{j\lfloor k/2^{l-j}\rfloor}
,\label{pinball}
\end{equation}
where $s_{k}=(-1)^{k+1}$. The identity \eqref{pinball} follows the fact that for $x\in I_{lk}$ we obtain  $\wt\beta_{lk}= \sum_{(l',k')\in P_{lk}}\beta_{l'k'}\psi_{l'k'}$
from   \eqref{tree_expand}, where $P_{lk}\equiv\{(j,\lfloor k/2^{l-j}\rfloor): j=0,\dots, l-1\}$ are the ancestors of the bottom node $(l,k)$. 
Note that $\psi_{j\lfloor k/2^{l-j}\rfloor}=2^{j/2}s_{\lfloor k/2^{l-j-1}\rfloor}$ where $s=(-1)^{k+1}$ for whether $x$ belongs to the left (positive sign) or right (negative sign) of the wavelet piece. There is a pinball game metaphor behind  \eqref{pinball}. A ball is dropped through a series of dyadically arranged pins of which the ball can bounce off to the right (when $s_k=+1$) or to the left (when $s_k=-1$). The ball ultimately lands in one of the  histogram bins $I_{lk}$ whose coefficient  $\wt\beta_{lk}$ is obtained by 
aggregating $\beta_{lk}$'s of those pins  $(l,k)$ that the ball encountered on its way down. The pinball aggregation process can be understood from Figure  \ref{fig:setting}. The duality between the equivalent representations \eqref{tree_expand} and \eqref{tree_expand22} through \eqref{pinball} provides various  avenues for constructing prior distributions, and enables an interesting interpretation of  Bayesian CART  \cite{cart1,cart2} as a correlated wavelet prior, as we now see.

\subsection{The $g$-prior for Trees}\label{sec:prior_g}
%%Thus far, we have been primarily concerned with priors $\Pi_\mathbb{T}$ on tree topologies. 
%The cornerstone of our methodological contribution is the tree-shaped wavelet prior \eqref{eq:prior_beta}. %, whose structured sparsity is promising for adaptive inference.
%% is the cornerstone of our development and we have already highlighted that its  tree-shaped sparsity makes it a good candidate for adaptive inference. 
%This independent product prior, however, deviates from the Bayesian CART priors used in practice \cite{cart1,cart2}. The difference is perhaps subtle but nevertheless interesting and pertains to the correlation structure among the wavelet coefficients $\beta_{lk}$.

We now discuss various ways of assigning a prior distribution on the bottom node histogram heights $\wt\beta_{lk}$ and, equivalently,  the internal Haar wavelet coefficients $\beta_{lk}$. This section also describes an interesting connection between  the widely used Bayesian CART prior \cite{cart1,cart2} and a $g$-prior \cite{zellner} on wavelet coefficients.
For a given tree $\mT$, let   $\b_\mT=(\beta_{lk}:(l,k)\in\mT_{int}')'$ denote the vector of {\sl ordered internal} node coefficients $\beta_{lk}$ including the extra root node $(-1,0)$ (and with ascending ordering according to $2^l+k$).  Similarly, $\wt\b_\mT=(\beta_{lk}:(l,k)\in\mT_{ext})'$ is the vector of {\sl ordered external} node coefficients $\wt\beta_{lk}$.
The duality between $\b_\mT$ and $\wt\b_\mT$  is apparent from the pinball equation  \eqref{pinball} written in  matrix form
\begin{equation}\label{onetoone}
\wt\b_\mT= A_\mT\b_\mT,
\end{equation}
where $ A_\mT$ is a square  $|\mT_{ext}|\times |\mT_{int}'|$ matrix (noting $|\mT_{ext}|=|\mT_{int}'|$), further referred to as the {\sl pinball matrix}.  
Each row of $ A_\mT$ encodes the ancestors of the external node, where the nonzero entries correspond to the internal nodes in the family pedigree. The entries are rescaled, where  younger ancestors are assigned more weight. For example, the tree $\mT$ in Figure \ref{bayestrees3} induces a pinball matrix $ A_\mT$ in  Figure \ref{pinball_matrix}.
The pinball matrix $A_\cT$ can be easily expressed in terms of a diagonal matrix and an orthogonal matrix as 
\begin{equation}\label{eq:lemma1}
 A_\mT A_\mT'=\bm D_\mT,\quad\text{where}\quad \bm D_\mT=\mathrm{diag}\{\wt d_{lk,lk}\}_{(l,k)\in\mT_{ext}},\ \ \wt d_{lk,lk}=2^l.
\end{equation}
This results from the fact that the collection $(2^{l/2}\mathbb{I}_{lk},\, (l,k)\in\cT_{ext})$ is an orthonormal system spanning the same space as 
$(\psi_{jk},\, (j,k)\in\cT_{int}')$, so $\bm D_{\mT}^{-1/2}A_\cT$ is an orthonormal change--of--basis matrix.
%The pinball matrix is  sparse and its entries reflect  the amplitude of the Haar wavelets.
%It is interesting to note that when $\mT$ is the {\sl flat tree} (as in Figure \ref{bayestrees2}), one obtains $ A_\mT' A_\mT= A_\mT A_\mT'=|\mT_{ext}|\times \mathrm{I}_{|\mT_{ext}|}$. 
%While this is not true in general (i.e. for trees $\mT$ that are asymmetric), $ A_\mT$ is always partially orthogonal as characterized below.
%\begin{lemma}\label{lemma1}
%Let $\mT$ be a binary tree and 
%Let $ A_\mT$ be the pinball matrix defined in \eqref{onetoone}. Then
%\end{lemma}
%\begin{proof}
%Denote with  $\wt d_{lk,mn}$ the off-diagonal element of $ A_\mT A_\mT'$ which corresponds to  two external nodes $(l,k)$ and $(m,n)$. 
%If $(l,k)$ and $(m,n)$ are siblings, then $\wt d_{lk,mn}=1+\sum_{i=0}^{l-1}2^i-2^l=0$ because the siblings have the same ancestors and their weights differ only by the sign of their   parent. In general, $\wt d_{lk,mn}=1+\sum_{i=0}^{s-1}2^i-2^s=0$, where $s$ is the depth of the most recent (deepest) ancestor that $(l,k)$ and $(m,n)$ have in common.
%The diagonal entries satisfy $\wt d_{lk,lk}=1+\sum_{i=0}^{l-1}2^i=2^l$.
%\end{proof}
We now exhibit precise connections   between the theoretical wavelet prior \eqref{eq:prior_beta} which  draws $\beta_{lk}\sim \pi$  and the practical Bayesian CART histogram prior which draws $\wt\beta_{lk}\sim {\wt\pi}$. 
%Recall that CART assigns a prior  $\wt\pi(\cdot)$ for each $\wt\beta_{lk}$'s,  while the Haar wavelet prior assigns $\pi(\cdot)$ on $\beta_{lk}$'s.
%Because the mapping \eqref{onetoone} is one-to-one, one can reverse-engineer the prior $\wt\pi(\cdot)$ from $\pi(\cdot)$, and vice versa.

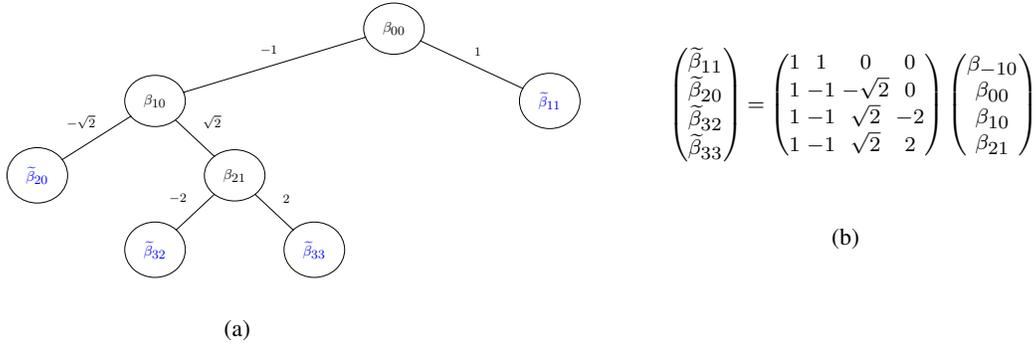
\begin{figure}[!t]
   \subfigure[ ]{
    \label{bayestrees3}
     \begin{minipage}[b]{0.45\textwidth}
 \scalebox{0.6}{
\xymatrix@-1pc{
%& & & &  & &  *++[o][F-]\txt{$V_{-10}$}  \ar@{-}[d]^1  &&& & & &\\
& & &  & &  *++[o][F-]\txt{$\beta_{00}$}  \ar@{-}[dlll]_{-1} \ar@{-}[drrr]^{1} &&&   \\
  & &     *++[o][F-]{\txt{$\beta_{10}$}}\ar@{-}[dll]_{-\sqrt{2}} \ar@{-}[dr]^{\sqrt{2}}& & && &&  *++[o][F]\txt{\color{blue}$\wt \beta_{11}$}  \\
     *++[o][F]\txt{\color{blue}$\wt \beta_{20}$}   && & *++[o][F-]{\txt{$\beta_{21}$}}\ar@{-}[dl]_{-2} \ar@{-}[dr]^2 &   & &  &   &   \\
  &   & *++[o][F]\txt{\color{blue}$\wt \beta_{32}$}& &*++[o][F]\txt{\color{blue}$\wt \beta_{33}$}  &&  &&             \\
    &   &  & &  &&  &&           \\
}}
    \end{minipage}}
    \hspace{2cm}
    \subfigure[ ]{
    \label{pinball_matrix}
    \begin{minipage}[t]{0.37\textwidth}
    \hspace{-5cm}
  {\scriptsize
  $$
\left(\begin{matrix}
\wt \beta_{11}\\
\wt \beta_{20}\\
\wt \beta_{32}\\
\wt \beta_{33}
\end{matrix}\right)=
\left(
\begin{matrix}
1&1&0&0\\
1&-1&-\sqrt{2}&0\\
1&-1&\sqrt{2}&-2\\
1&-1&\sqrt{2}&2\\
\end{matrix}
\right) 
\left(\begin{matrix}
\beta_{-10}\\
 \beta_{00}\\
 \beta_{10}\\
 \beta_{21}
\end{matrix}\right)
$$}
\vspace{0.6cm}
  \end{minipage}}
    \caption{(a) Example of a full binary tree, edges weighted by the amplitude of the Haar wavelets. (b) Pinball matrix of the tree in (a). }
    \label{fig:setting}
\end{figure}

Recall that the wavelet prior \eqref{eq:prior_beta} assumes independent wavelet coefficients, e.g. through the standard Gaussian prior  $\b_\mT\sim\mathcal{N}(0,I_{|\mT_{ext}|})$. 
 Starting {\em from within} the tree, this translates into the following {independent product prior} on the bottom coefficients $\wt\beta_{lk}$ through  \eqref{onetoone}  
 \begin{equation}
\wt\b_\mT\sim\mathcal{N}(0,\bm{D}_\mT),\label{eq:product_prior}\quad\text{where}\quad \bm D_\mT \quad \text{was defined in \eqref{eq:lemma1}},
\end{equation}
i.e. $\mathrm{var}\, \wt\beta_{lk}=2^l$ where   the  variances  increase with  the resolution $l$.  
 
The Bayesian CART prior \cite{cart1,cart2}, on the other hand, starts {\em from outside} the tree by assigning
 $\wt\b_\cT\sim\mathcal{N}(0,g_nI_{|\mT_{ext}|})$  for some $g_n>0$, ultimately setting the bottom node variances equal. This translates into the following `$g$-prior'  on the {\em internal} wavelet coefficients
through  the duality \eqref{onetoone}.
 \begin{definition}
Let $\mT\in\mathbb{T}$ with a pinball matrix $A_\mT$ and denote with $\b_\mT$  the internal wavelet coefficients. We define the {\em g-prior} for trees as 
 \begin{align}
\bm \b_\mT&\sim\mathcal{N}\left(0,g_n\,( A_\mT' A_\mT)^{-1}\right)\quad\text{for some $g_n>0$}.\label{eq33}
\end{align}
\end{definition}
Note that, except for very special cases (e.g. flat trees) $A_\mT' A_\mT$ is in general not diagonal, unlike $A_\mT A_\mT'$. This means that the  correlation structure induced by the Bayesian CART prior on internal wavelet coefficients is non-trivial, although $A_\mT' A_\mT$ admits some partial sparsity. {We characterize basic properties of the pinball matrix in Section \ref{ap:sec_pinball} in the Supplement.}
 %$A_\mT$ {\sl is not necessarily orthogonal} (for trees other than flat trees), where the prior seizes  correlation between hierarchically linked wavelet coefficients. 
For example,  Proposition \ref{prop:eigenspectrum} shows that matrices $A_\mT'A_\mT$ and $A_\mT A_\mT'$ have the same eigenspectrum consisting of values $2^l$ where $l$ corresponds to the depth of the bottom nodes.  This means that the $g$-prior variances (diagonal elements of $g_n(A_\mT'A_\mT)^{-1}$) are lower-bounded by the minimal eigenvalue of $g_n(A_\mT'A_\mT)^{-1}$ which equals $g_n2^{-l}$ (where $l$ is the depth of the deepest external node) which is lower-bounded by $g_n/n$. Since the traditional wavelet prior assumes variance   $1$, the choice $g_n=n$ matches the lower bound $1$ by undersmoothing all possible variance combinations. While other choices could be potentially used (see \cite{fernandez,kass_wasserman,foster_george} in the context of linear regression), we will  consider $g_n=n$ in our results below. 
 
We regard \eqref{eq33} as the  `$g$-prior  for trees' due to its apparent similarity to $g$-priors for linear regression coefficients \cite{zellner}. The $g$-prior has been shown to have many favorable properties in terms of invariance or predictive matching   \cite{bayarri, barbieri}.   Here, we explore the benefits of the $g$-type correlation structure in the context of structured wavelet shrinkage where each `model' is defined by a tree topology. The correlation structure \eqref{eq33} makes this prior  very different from any other prior studied in the context of wavelet shrinkage.

\section{Inference with (Dyadic) Bayesian CART}\label{sec:dyadic}

In this section we investigate the inference properties of tree-based posteriors, showing that (a) they attain the minimax rate of posterior concentration in the supremum-norm sense (up to a log factor), and (b) %exhibit an {\em adaptive} non-parametric BvM behavior in typical multiscale spaces. We complement these results by addressing the question of 
enable uncertainty quantification: for $f$ in the form of adaptive confidence bands, and for smooth functionals thereof, in terms of Bernstein-von Mises type results. {For clarity of exposition, we focus now on the one-dimensional case, but the results readily extend to the multi-dimensional setting with $\RR^d$, $d\ge 1$ fixed, as predictor space; see   Section \ref{sec:multid} for more details. }
%
%$\ $
%\ma{[To be put later below the statement??] We provide  affirmative answers to the question whether Bayesian (Dyadic) CART is theoretically close to the ``ideal" spike-and-slab prior \cite{hoffmann,ray}. Occupying the middle ground  between flat trees (with only random cutoffs) and spike-and-slab priors (with general sparsity patterns), our tree-shaped priors {\em are}  ``multiscale" in the sense that (a) they attain the minimax rate of posterior concentration in the sup-norm sense (up to a log factor), and (b) exhibit an {\em adaptive} non-parametric BvM behavior. Bayesian dyadic CART thus  matches the excellent performance of spike-and-slab priors, up to only a log term. }

\subsection{Posterior supremum-norm convergence} Let us recall the standard inequality (see e.g. \eqref{sup_bound} below), for $f_0$ a continuous function and $f$ a  Haar histogram \eqref{tree_expand22}, with  coefficients  $\be^0_{lk}$ and $\be_{lk}$,
\begin{equation}\label{stine}
 \|f-f_0\|_\infty \le |\be_{-10}- \be^0_{-10}|+\sum_{l\ge -1} 2^{l/2} \max_{0\le k<2^l} |\be_{lk}- \be^0_{lk}|=:\ell_\infty(f,f_0).
\end{equation} 
%The white noise facilitates clean  mathematical derivation of the posterior distribution under the variety of priors we discussed in previous sections. Recall that $L=L_{max}=\lfloor \log_2 n\rfloor$ is the maximal depth of considered trees.  
As $\ell_\infty$ dominates $\|\cdot\|_\infty$, it is enough to derive results for the $\ell_\infty$--loss. 

Given a tree $\mT\in\bT$, and recalling that trees in $\bT$ have depth at most $L:=L_{max}=\lfloor \log_2{n}\rfloor $, %denoting as before $\b_\mT$  the vector of its ordered internal node coefficients, %chosen from one of the priors $\Pi_\bT$ on $\bT$ discussed in Section \ref{sec:tree_prior},
 we consider a generalized tree-shaped prior $\Pi$ on the {\sl internal wavelet} coefficients, recalling the notation $\cT_{int}'$ from Section \ref{tree_wavelet_prior},
\begin{align}
\cT \qquad & \sim \qquad \Pi_{\bT}  \nonumber \\
(\beta_{lk})_{l\le L,k<2^l} \given \cT\ & \sim\  \pi(\b_\mT) \, \otimes 
\bigotimes_{(l,k)\notin\cT_{int}'} \delta_0(\beta_{lk}),\label{eq:prior_beta2} 
%\\
%(\wt \beta_{lk})_{(l,k)\in\mT_{ext}}\,\C\,\mT \ & \sim \, \bigotimes_{(l,k)\in\mT_{ext}} \ga(\beta_{lk}),  \label{eq:prior_betati}
\end{align}
where $\pi(\b_\mT)$ is a law  to be chosen on $\RR^{|\mT_{int}'|}$, not necessarily of a product form.  
 This is a generalization of \eqref{eq:prior_beta}, which allows for {\em correlated} wavelet coefficients (e.g. the $g$-prior).
Let $\X_\mT$ denote the vector of ordered responses $X_{lk}$ in \eqref{eq:model2} for $(l,k)\in\mT_{int}'$.
%We define the support of a multi-index sequence $\{\beta_{lk}\}$ as  the set of indexes $(l,k)$ such that $\beta_{lk}\neq 0$ and denote it by $S((g_{lk}))$.  
From the white noise model, we have 
$$
\X_\mT=\b_\mT+\frac{1}{\sqrt{n}}\bm\varepsilon_{\mT},\quad\text{ with} \quad\bm\varepsilon_{\mT}\sim\mathcal{N}(0,I_{|\mT_{ext}|}) \ \ (\text{given }\cT).
$$
By Bayes' formula, the posterior distribution $\Pi[\cdot\given X]$ of the variables $(\beta_{lk})_{l\le L, k}$ has density
\begin{align}\label{eq:posterior}
%\pi\left(\{\beta_{lk}\}_{l k}\C \{X_{lk}\}_{l\le L, k}\right)
%&\propto \sum_{\cT\in \bT} 
%\Pi_{\bt}(\cT) \left[\prod_{(l,k)\in\cT_{int}} \phi_{\frac{1}{\sqrt{n}}}(X_{lk}-\beta_{lk})
% \pi(\b_\mT) \right]
%\left[\prod_{(l,k)\notin\cT} \phi_{\frac{1}{\sqrt{n}}}(X_{lk}-\beta_{lk})\delta_0(\beta_{lk})\right]\\
%& = \sum_{\cT\in \bT}
%\Pi_{\bt}(\cT)
%\e^{-\frac{n}{2}\|\b_\mT\|^2_2+n\X_\mT'\b_\mT}\pi(\b_\mT)
%\left[ \prod_{l\le L,k} \phi_{\frac{1}{\sqrt{n}}}(X_{lk})\right]\left[\prod_{(l,k)\notin\cT} \delta_0(\beta_{lk})\right]\\
 \sum_{\cT\in \bT} \Pi[\mT\C X]\cdot\pi(\b_\mT\C X)\cdot\prod_{(l,k)\notin\cT_{int}'} \mathbb{I}_0(\beta_{lk}),
 \end{align}
 where, denoting as shorthand $N_X(\cT)=\int \e^{-\frac{n}{2}\|\b_\mT\|^2_2+n\X_\mT'\b_\mT} \pi(\b_\mT)d\b_\mT$, 
% \vspace{-0.5cm}
  \begin{align}
 \pi(\b_\mT\C X)&= \frac{\e^{-\frac{n}{2}\|\b_\mT\|^2_2+n\X_\mT'\b_\mT} 
\pi(\b_\mT)}{N_X(\cT)},\label{eq:posterior_beta} \\
\Pi[\cT\given X] &= \frac{W_X(\cT)}{\displaystyle \sum_{\cT\in \bT} W_X(\cT)},\quad \text{with}\quad W_X(\cT) = 
\Pi_{\bT}(\cT) N_X(\cT).\label{eq:W}
 \end{align}
Let us note that the sum in the last display is finite, as we restrict to trees of depth at most $L=L_{max}$.  Note that the classes of priors $\Pi_\bT$ from Section \ref{sec:tree_prior} are non-conjugate, i.e.  the posterior on trees is given by the somewhat intricate expression \eqref{eq:W} and does not belong to one of the classes of $\Pi_\bT$ priors. While the posterior expression \eqref{eq:posterior_beta}  allows for general priors $\pi(\b_\mT)$, we will focus on  conditionally conjugate Gaussian priors  for simplicity. This assumption is not essential and can be relaxed. For instance, in case $\pi(\b_\cT)$ is of a product form, one could use a product of e.g. Laplace distributions, using similar ideas as in \cite{castillo_nickl2}, Theorem 5.

Our first result 
 exemplifies the potential of tree-shaped priors by showing that  Dyadic  Bayesian CART achieves the minimax rate of posterior concentration over H\"{o}lder balls in the sup-norm sense, i.e. $\veps_n=(n/\log n)^{-\alpha/(2\alpha+1)}$, up to a logarithmic term.
Define a H\"{o}lder-type ball of  functions on $[0,1]$ as
\begin{equation}\label{eq:haar}
 \cH(\al,M):= \left\{f\in {\mathcal{C}[0,1]}:\ \max_{l\geq 0,\ 0\le k<2^l} 2^{l(\frac12+\al)}|\langle f,\psi_{lk}\rangle| \vee |\langle f,\psi_{-10}\rangle| 
 \leq M \right\}. 
 \end{equation}
For balanced Haar wavelets $\psi_{lk}$ as in \eqref{haar},   $\cH(\al,M)$ contains the a standard  $\alpha$-H\"{o}lder (resp. Lipschitz when $\al=1$) ball of functions  for any $\al\in(0,1]$, defined as 
\begin{equation}\label{eq:haar2}
\cH^{\al}_M:= \left\{f:\, \|f\|_\infty\le M,\ \frac{|f(x)-f(y)|}{|x-y|^\al}\le M \quad\forall x,y\in [0,1] \right\}.
 \end{equation}
Our master rate-theorem, whose proof can be found in Section \ref{sec:proof_thm-one}, 
 is stated below. It will be extended in various directions in the sequel.

%From this definition by classical results in wavelet theory it follows that for the balanced Haar wavelets $\psi_{lk}$ in \eqref{haar} we have
%\[ \cH(\al,M)\subset \left\{f:\ \max_{0\le k<2^l} |\langle f,\psi_{lk}\rangle| \le M 2^{-l(\frac12+\al)} \right\}. \]

\begin{theorem} \label{thm-one}
Let $\Pi_{\bt}$ be  the  Galton-Watson process prior  from Section \ref{sec:prior_trees} with $p_{lk}=\Gamma^{-l}$ and $\Ga>2e^3$.
Consider the tree-shaped wavelet prior \eqref{eq:prior_beta2} with $\pi(\b_\mT)\sim\mathcal{N}(0,\Sigma_\mT)$, where  $\Sigma_{\mT}$  is either $I_{|\mT_{int}'|}$ or $g_n(A_\mT'A_\mT)^{-1}$ with $g_n=n$.
Define
%\vspace{-0.5cm}
\begin{equation}\label{rate_final}
\veps_n %= (\log{n})^{\frac{\al}{2\al+1}} \left(\frac{\log{n}}{n}\right)^{\frac{\al}{2\al+1}}
= \left(\frac{\log^2{n}}{n}\right)^{\frac{\al}{2\al+1}}\quad \text{for}\quad \al>0.
\end{equation}
Then for any $\al\in(0,1]$, $M>0$, any sequence $M_n\to\infty$  we have for $n\rightarrow\infty$
\begin{equation}\label{statement}
 \sup_{f_0\in \cH(\al,M)} E_{f_0}\Pi\left[f_{\mT,\b}:\, \ell_\infty(f_{\mT,\b},f_0) > M_n\veps_n\given X \right] \to 0.
\end{equation}
By \eqref{stine}, the statement \eqref{statement} also holds for the supremum loss $\|\cdot\|_\infty$.
% $\ell_\infty$ norm \cite{hoffmann} defined as
%$\ell_\infty(f_{\mT,\b},f_0)=\sum_{j\in\N}2^{j/2}\max_k|\beta_{jk}-\beta_{jk}^0|$.
\end{theorem}

 \iffalse  
\begin{theorem} \label{thm-one}
Let $\Pi_{\bt}$ be one of the Bayesian CART priors  from Section \ref{sec:prior_trees}, with parameters $\Ga>2e^3$ and {$c>7/4$}, i.e
\begin{enumerate}
\item[(i)]the Galton-Watson process prior  with $p_{lk}=\Gamma^{-l}$, or %for some large enough $2<\Gamma<n$, or
\item[(ii)]the conditionally uniform prior   with  $\lambda=1/n^c$  in \eqref{prior:K}, or %for some $c>0$, or
 \item[(iii)] the exponential prior \eqref{eq:random_height} with a parameter $c$.
 %for some  $c>0$ large enough.
\end{enumerate}
Consider the tree-shaped wavelet prior \eqref{eq:prior_beta2} with $\Pi_\bt$ as above and $\pi(\b_\mT)\sim\mathcal{N}(0,\Sigma_\mT)$, where  $\Sigma_{\mT}$  is either $I_{|\mT_{ext}|}$ or $g_n(A_\mT'A_\mT)^{-1}$ with $g_n=n$.
Define
\begin{equation}\label{rate_final}
\veps_n %= (\log{n})^{\frac{\al}{2\al+1}} \left(\frac{\log{n}}{n}\right)^{\frac{\al}{2\al+1}}
= \left(\frac{\log^2{n}}{n}\right)^{\frac{\al}{2\al+1}}\quad \text{for}\quad \al>0.
\end{equation}
Then for any $\al\in(0,1]$, $M>0$, any sequence $M_n\to\infty$  we have for $n\rightarrow\infty$
\begin{equation}\label{statement}
 \sup_{f_0\in \cH(\al,M)} E_{f_0}\Pi\left[f_{\mT,\b}:\, \ell_\infty(f_{\mT,\b},f_0) > M_n\veps_n\given X \right] \to 0.
\end{equation}
By \eqref{stine}, the statement \eqref{statement} also holds for the supremum loss $\|\cdot\|_\infty$.
% $\ell_\infty$ norm \cite{hoffmann} defined as
%$\ell_\infty(f_{\mT,\b},f_0)=\sum_{j\in\N}2^{j/2}\max_k|\beta_{jk}-\beta_{jk}^0|$.
\end{theorem}

\fi

\begin{extension}\label{ext:one}
While Theorem \ref{thm-one} is formulated  for Bayesian CART obtained with Haar wavelets, the concept of tree-shaped sparsity extends to general wavelets that give rise to smoother objects than just step functions.
With $\{\psi_{lk}\}$ an $S$--regular wavelet basis on $[0,1]$, e.g. the  boundary-corrected wavelet basis of \cite{cdv} (see \cite{gine_nickl}, Chapter 4, with adaptation of the range of indices $l$), and with $f_0\in\cH(\al,M)$ defined in \eqref{eq:haar} for some $M>0$ and {\em arbitrary} $0<\al\le S$, 
% with $S\ge 1$ vanishing moments (e.g. CDV-type boundary corrected and $S\ge 1$ an integer).  
one indeed obtains the  statement \eqref{statement} by choosing $\Gamma\ge \Gamma_0(S)>0$ or $c\ge c_0>0$ large enough, see Section \ref{sec:thm-wav}.
%\subsection{Smooth wavelets} \label{sec:wave}
%Instead of using the Haar basis, which produces `hard' cuts (i.e. histograms) one can in fact use any other basis, for instance one of smooth wavelets, say the CDV basis from \cite{cdv93} in the construction of the prior. It is not difficult to check that the results of Theorems \ref{} to \ref{} carry over to this wavelet prior, except that the regularity condition $\al\le 1$ is not needed anymore, and becomes just $\al\le \rho$, where $\rho$ is the smoothness index of the corresponding CDV basis (and $\rho$ can be chosen larger than a given arbitrary large constant). On the other hand, the interpretations as  random partitions are lost.
\end{extension}

Theorem \ref{thm-one} encompasses both original Bayesian CART proposals for priors on bottom coefficients $\wt\b_\mT\sim\mathcal{N}(0,I_{|\mT_{ext}|})$ (the case $\Sigma_\mT=g_n(A_\mT A_\mT')^{-1}$ discussed in Section \ref{sec:prior_g}) as well as the mathematically slightly simpler wavelet priors $\Sigma_\mT=I_{|\mT_{ext}|}$ (discussed in Section \ref{sec:wavelet_prior}). We did not fully optimize the  constants in the statement; for instance, one can  check that $\Ga>2$ for the $g$-prior works. 
The rate $\veps_n$ in \eqref{rate_final} coincides with the minimax rate for the supremum norm in the white noise model up to a logarithmic factor   $(\log{n})^{\frac{\al}{2\al+1}}$.  We next show that this logarithmic factor is in fact real, i.e. {\em not} an artifact of the upper-bound proof. We state the results for smooth-wavelet priors, which enable to cover arbitrarily large regularities, but a similar result could also be formulated for the Haar basis.

\begin{theorem} \label{sharp_lb}
Let $\Pi_{\bt}$ be one of the Bayesian CART priors from Theorem \ref{thm-one}. Consider the tree-shaped wavelet prior \eqref{eq:prior_beta2} with $\pi(\b_\mT)\sim\mathcal{N}(0,\Sigma_\mT)$, where  $\Sigma_{\mT}$  is $I_{|\mT_{ext}|}$ and $\{\psi_{lk}\}$ an $S$--regular wavelet basis, $S\ge 1$. 
Let $\veps_n$ be the rate defined in \eqref{rate_final} for a given $0<\al\leq S$.  Let the parameters of $\Pi_\bT$ verify either $\Ga\ge \Ga_0(S)$ a large enough constant, or  {$c\ge c_0>0$ large enough}. For any $M>0$, there exists $m>0$ such that, as $n\rightarrow\infty$, 
\begin{equation}\label{lbrate}
\inf_{f_0\in \cH(\al,M)} E_{f_0}\Pi\left[\ell_\infty(f_{\mT,\b},f_0) \le m\,\veps_n\given X \right] \to 0. 
\end{equation}
\end{theorem}

In other words, there exists a sequence of elements of $\cH(\al,M)$ along which the posterior convergence rate is {\em slower} than $m\,\veps_n$  in terms of the $\ell_\infty$--metric. In particular, the upper-bound rate of Theorem \ref{thm-one} {\em cannot} hold uniformly over $\cH(\al,M)$ with a rate faster than $\veps_n$, which shows that the obtained rate is sharp (note the reversed inequality in \eqref{lbrate} with respect to \eqref{statement}; we refer to \cite{ic08} for more details on the notion of posterior rate lower bound). The proof of Theorem  \ref{sharp_lb} can be found in Section \ref{proof:sharp_lb}.

%As we noted in Section \ref{}, the priors $\pi(\b_\mT)\sim\mathcal{N}(0,\Sigma_\mT)$ in Theorem \ref{thm-one}  correspond to the practical histogram priors. Indeed, the $g$-prior  obtained with   $\Sigma_{\mT}=g_n(A_\mT'A_\mT)^{-1}$
%is equivalent to $\wt\b_\mT\sim\mathcal{N}(0,I_{|\mT_{ext}|})$ and choosing $\Sigma_{\mT}=I_{|\mT_{ext}|}$ is equivalent to  $\wt\b_\mT\sim\mathcal{N}(0,\bm D_\mT)$ where $\bm D_\mT$ is defined in \eqref{eq:lemma1}.
%\end{remark}

%\re{[Recheck $c\ge 1$ or $c$ large enough in thms 2/3]}\ma{In Theorem 1, it seems that   $c>3/4$ would be large enough} \sbl{[With the non-injectivity factor $2^d$ it seems we get $c>7/4$? Below there were some references pointing to \eqref{eq:rate} instead of \eqref{statement}; I changed that]}

\begin{extension}\label{remark:general}
%It can be shown (Section \ref{sec:thm-one} in the Supplement) that 
Theorem \ref{thm-one} holds for a variety of other tree priors. 
This  includes the  conditionally uniform prior mentioned in Section \ref{sec:bc1}  with  $\lambda=1/n^c$  in \eqref{prior:K}, or an exponential-type prior 
$\Pi_{\bT}(\mT)\, \propto \,\e^{-c |\mT_{ext}|\log n} \1_{\mT\in\bT}$ for some $c>0$. One can also assume a general Gaussian prior on active wavelet coefficients with an unstructured covariance matrix $\Sigma_\mT$ which satisfies $\lambda_{min}(\Sigma_{\mT})\gtrsim 1/\sqrt{\log n}$ and  $\lambda_{max}(\Sigma_{\mT})\lesssim n^a$ for some $a>0$. Detailed proofs can be found in the Supplement (Section \ref{sec:thm-one}).
\end{extension}

Only very few priors (actually  {\em only} point mass spike-and-slab based priors, as discussed in the Introduction) were shown to attain  adaptive posterior sup-norm concentration rates.   Theorem \ref{thm-one} now  certifies  Dyadic Bayesian  CART as one of them.
\iffalse The rate $\veps_n$ in \eqref{rate_final} coincides with the minimax rate for the supremum norm in the white noise model up to a logarithmic factor, which equals  $(\log{n})^{\frac{\al}{2\al+1}}$. This means that there may be a slight logarithmic term to pay for tree priors (no more than $\sqrt{\log{n}}$ in the limit $\al\to\infty$ and no price in the limit $\al\to 0$). This log-factor is in fact real, i.e. it is {\em not} an artifact from the upper-bound proof. This is proved in Section \ref{proof:sharp_lb}, where a lower-bound result, implying that the rate $\veps_n$ is sharp, is established. 
\fi  The  logarithmic penalty in the rate \eqref{rate_final} reflects that  Bayesian CART priors occupy the middle ground  between flat trees (with only a depth cutoff) and spike-and-slab priors (with general sparsity patterns).  
%Depending on the tree topology, some tree priors may be more or less suited to achieve adaptation in the supremum sense.
%Recall that the spike-and-slab prior achieves the {\sl actual} $\ell_\infty$ minimax rate {\em without} any additional factor. Interestingly, the very same prior misses the $\ell_2$ minimax rate by a log factor \cite{hoffmann}.  This illustrates that $\ell_2$ and $\ell_\infty$ adaptations  require different desiderata when constructing priors. 
As mentioned earlier, 
%product priors that correspond to separable rules {\em do not} yield adaptation with exact rates  in the  $\ell_2$ sense \cite{cai08}. Mixture priors that are adaptive in $\ell_2$, on the other hand, may not yield $\ell_\infty$ adaptation.  
 flat trees are incapable of supremum-norm adaptation, as we formally prove in Section \ref{sec:flatneg}. The fact that the more flexible Bayesian CART priors still achieves supremum-norm adaptation in a near-optimal way is  a rather notable feature. From a more general perspective, we note that while general tools are available to derive adaptive $L^2$-- or Hellinger--rate results in broad settings (e.g. model selection techniques, or  the theory of posterior rates in \cite{ghosal_etal}), deriving adaptive $L^\infty$--results is often  obtained in a case-by-case basis; two possible techniques are wavelet thresholding (when empirical estimates of wavelet coefficients are available) and Lepski's method (which requires some `ordered' set of estimators, typically in terms of variance; for tree-estimators for instance it would not readily be applicable).  The fact that tree methods enable for supremum-norm adaptation in nonparametric settings is one of the main take-away messages of this work.

\subsection{Adaptive Honest Confidence Bands for $f_0$}\label{sec:uq}
We now turn to the ultimate landing point of this paper, uncertainty quantification for $f_0$ and its functionals.
{The existence of adaptive confidence sets in general is an interesting and delicate question (see Chapter 8 of \cite{gine_nickl}). In the present context of regression function estimation under the supremum norm loss, it is in fact impossible to build adaptive confidence bands without further restricting the parameter space. We do so by imposing some classical self-similarity conditions (see \cite{gine_nickl}, \cite{ray} for more details).}

%In order for uncertainty quantification to be as informative as possible, it is desirable that the confidence sets shrink as fast as possible. When the degree of smoothness $\alpha$ is a priori known, one can intersect  \eqref{crediblem} with qualitative restrictions on $f_0$ to obtain ``optimal"  frequentist confidence intervals (whose diameters shrink at the sup-norm rate). For the more practical case when $\alpha$ is unknown, \cite{ray} obtained multiscale credible balls  under the spike-and-slab prior that are adaptive and have uniform coverage over self-similar functions \cite{picard,bull, gine_nickl2,nickl_szabo}.  

\iffalse
It is known \cite{low} that confidence bands that are simultaneously adaptive and honest do not exist in full generality. 
%The frequentist properties of adaptive Bayesian credible sets have begun to be understood recently in a series of works \cite{botond,botond2,ray}.
Gine and Nickl \cite{gine_nickl} point out, however, that adaptive and honest confidence sets exist for certain generic subsets of $\mathcal{F}$ consisting
  of  the so-called self-similar functions  (see also). 
  \fi
\begin{definition}(Self-similarity) \label{def-ssi}
Given an integer $j_0>0$, we say that  $f\in\cH(\alpha,M)$ is {\em self-similar} if, for some constant $\veps>0$, 
\begin{equation} \label{sscond}
\| K_j(f) - f \|_\infty \ge \veps 2^{-j\alpha}\quad \text{for all } j \ge j_0,
\end{equation}
where $K_j(f)=\sum_{l\le j-1}\sum_k \, \langle\psi_{lk},f\rangle \, \psi_{lk}$ is the wavelet projection  at level $j$.
The class of all such self-similar functions will be denoted by $\cH_{SS}(\alpha,M,\veps)$.
\end{definition}
Section 8.3.3 in \cite{gine_nickl} describes self-similar functions  as typical representatives of the H\"{o}lder class.
As shown in Proposition 8.3.21 of \cite{gine_nickl}, self-dissimilar functions are nowhere dense in the sense that they cannot approximate any open set in $\mathcal H(\alpha, M)$. In addition, Bayesian non-parametric priors for H\"{o}lder functions charge self-similar functions with probability $1$. Finally, self-similarity does not affect the difficulty of the statistical estimation problem, where the $(\ell_\infty)$ minimax rate is not changed after adding this assumption. A variant of the self-similarity condition was shown to be {\em necessary} for adaptive inference, in that such condition cannot essentially be weakened for uniform coverage with an optimal rate to hold \cite{bull}.
% implying that one must exclude some functions from consideration in order to construct adaptive bands \cite{bull}.

Following \cite{ray}, we construct adaptive honest credible sets by  first defining a pivot centering estimator, and then  determining a data-driven radius.

\begin{definition}(The Median Tree) \label{def:mtree}
Given a posterior  $\Pi_\bT[\cdot\given X]$ over trees, we define the {\em median tree}  $\cT^*_X=\cT^*(\Pi_\bT[\cdot\given X])$ as the set of nodes  %$(l,k)$ such that 
%\sbl{[here one should be a bit more formal about the posterior on trees. The posterior on $f$ is actually induced by $f\given \cT,X$ and $\cT\given X$ and it is the later distribution we refer to $\to$ introduce notation(?) ]}
\begin{equation} \label{bulktree}
\cT^*_X=\left\{ (l,k),\ l\le L_{max},\ \ \Pi[(l,k)\in \cT_{int} \given X] \ge 1/2 \right\}.
\end{equation}
\end{definition}
Similarly as in the median probability model \cite{barbieri,barbieri2},
a node belongs to $\cT^*_X$ if its (marginal) posterior probability to be selected by a tree estimator  exceeds $1/2$. 
Interestingly, as the terminology suggests,  $\cT^*_X$ {\em is an actual tree}, i.e. the nodes follow hereditary constraints (see Lemma \ref{lembulk} in the Supplement). 
%We call the set $\cT^*_X$, which is a binary tree by Lemma \ref{lembulk}, the bulk-median tree. 
We  define the resulting median tree estimator as 
\begin{equation} \label{bulkest}
\wh f_T (x) = \sum_{(l,k)\in \cT^*_X} X_{lk} \psi_{lk}(x).
\end{equation}
Moreover,  we define a {\em radius},  for some $v_n\rightarrow\infty$ to be chosen, 
as
\begin{equation}  \label{radiusprox}
\sigma_n = \sigma_n(X) = \sup_{x\in[0,1]} \sum_{l=0}^{\lmax} v_n \sqrt{\frac{\log{n}}{n}}
\sum_{k=0}^{2^l-1} \1_{(l,k)\in \cT_X^*} |\psi_{lk}(x)|.
\end{equation}
A credible band with a radius $\sigma_n(X)$ as in \eqref{radiusprox} and a center $\wh f_T$ as in \eqref{bulkest} is
\begin{equation} \label{credible}
\cc_n=\left\{f:\ %\|f-\ix\|_{\cM(w)} \le R_n/\rn,\
 \|f-\wh f_T\|_\infty \le \sigma_n(X) \right\}.
\end{equation}

Theorem \ref{csthm}, proved in Section \ref{proof:csthm}, shows that valid frequentist uncertainty quantification with Bayesian CART is attainable (up to log factors). 
Indeed, the confidence band \eqref{credible} has a near-optimal diameter and a  uniform frequentist coverage under self-similarity.

\begin{theorem} \label{csthm}
Let $0<\al_1\le \al_2 \le 1,\, M\ge 1$ and $\veps>0$. Let $\Pi$ be any prior as in the statement of Theorem \ref{thm-one}. Let $\sigma_n$ be as in \eqref{radiusprox} with $v_n$ such that $(\log{n})^{1/2}=o(v_n)$ and  let $\wh f_T$ denote the median tree estimator \eqref{bulkest}.  Then for $\cc_n$ defined in \eqref{credible}, uniformly over $\alpha\in[\alpha_1,\alpha_2],$ as $n\to\infty$,
\[  \inf_{f_0\in \cH_{SS}(\alpha,M,\veps)} P_{f_0}(f_0\in \cc_n)  \to 1.\]
For every $\al\in[\al_1,\al_2]$ and uniformly over $f_0\in \cH_{SS}(\al,M,\veps)$, the  diameter $|\cc_n|_\infty=\sup_{f,g\in \cc_n}\|f-g\|_\infty$ and the credibility of the band  verify, as $n\to \infty$,
\begin{align}
 |\cc_n|_\infty & = O_{P_{f_0}}((n/\log n)^{-\al/(2\al+1)} v_n),\label{cred_diam}\\
 \Pi[\cc_n\given X] & = 1 + o_{P_{f_0}}(1). 
\end{align} 
\end{theorem}
Similarly as for Theorem \ref{thm-one}, the results of Theorem \ref{csthm} carry over to wavelet priors over a smooth wavelet basis, leading to the construction of confidence sets with arbitrary regularities $0<\al_1\le \al_2<\infty$. The undersmoothing factor $v_n$ is commonplace in the context of confidence bands, with the condition $v_n\gg (\log{n})^{1/2}$ reflecting the slight logarithmic price to pay for trees noted earlier in terms of $L^\infty$--estimation accuracy. In the previous statement both confidence and credibility of $\cc_n$ tend to $1$. It is possible to achieve exact coverage by intersecting $\cc_n$ further with another ball. A natural way to do so (from the `estimating many functionals' perspective, see \cite{castillo_nickl1}) is to intersect with a multiscale ball (we refer to Section \ref{sec-bvm} and \ref{sec:simul} in the Supplement for details and demonstrations). For stability reasons, this intersection-band seems also preferable in practice and we present in Figure \ref{figure:uq} on the right an illustration of coverage of such a band in nonparametric regression. 
Apart from the intersection band, another
natural choice is an $L^\infty$--credible band. Namely, given a centering
estimator $\hat{f}$ (such as the median-tree estimator), one can
consider an $L^\infty$--ball around $\hat{f}$ that captures $0.95\%$ of the
posterior mass (see Figure \ref{figure:uq} on the left). We are not aware of
any frequentist validation results for such bands in the adaptive
$L^\infty$--setting.  Results for such type of credible sets have been
obtained in the $L^2$--setting, for instance, in \cite{szabo_etal}.  To guarantee coverage, the authors need
to incorporate a `blow-up' factor (diverging to infinity) to the radius
of the set (see \cite{ray} for more discussion). Finally, another possibility would be to `paste together'
marginal pointwise credible intervals (see Figure \ref{figure:uq} on the left). It is not clear how much
`blow-up' would be needed to guarantee frequentist coverage under
self-similarity and, again, we are not aware of any theoretical results  for such sets.

%The confidence set \eqref{credible} is different from constructions considered in \cite{ray}, which intersect $\mathcal C_n$ with a multiscale credible ball to obtain an exact asymptotic coverage, say $1-\gamma$ for some $\gamma\in(0,1)$. Our construction  has the added benefit of coverage that converges to one, where asymptotic exactness  could be obtained using a similar intersection (as a consequence of an  adaptive non-parametric Bernstein-von Mises theorem). %{\color{blue} Ismael, maybe some of your text from the reply letter would be useful here. You know this stuff better than me, so I would prefer that you make the remarks.}
%\re{Update text here}

\begin{figure}[!] \label{fig:band}
\includegraphics[width=6cm]{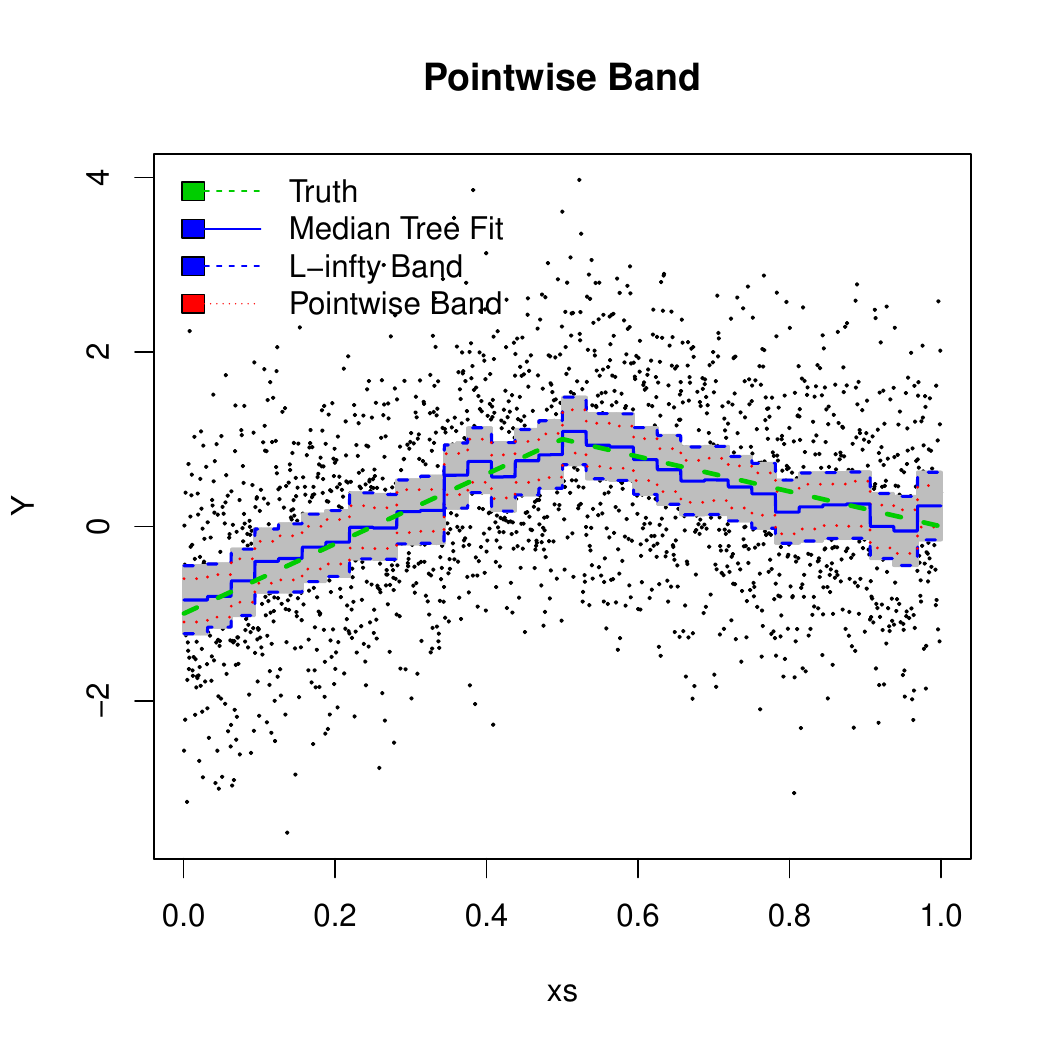}
\includegraphics[width=6cm]{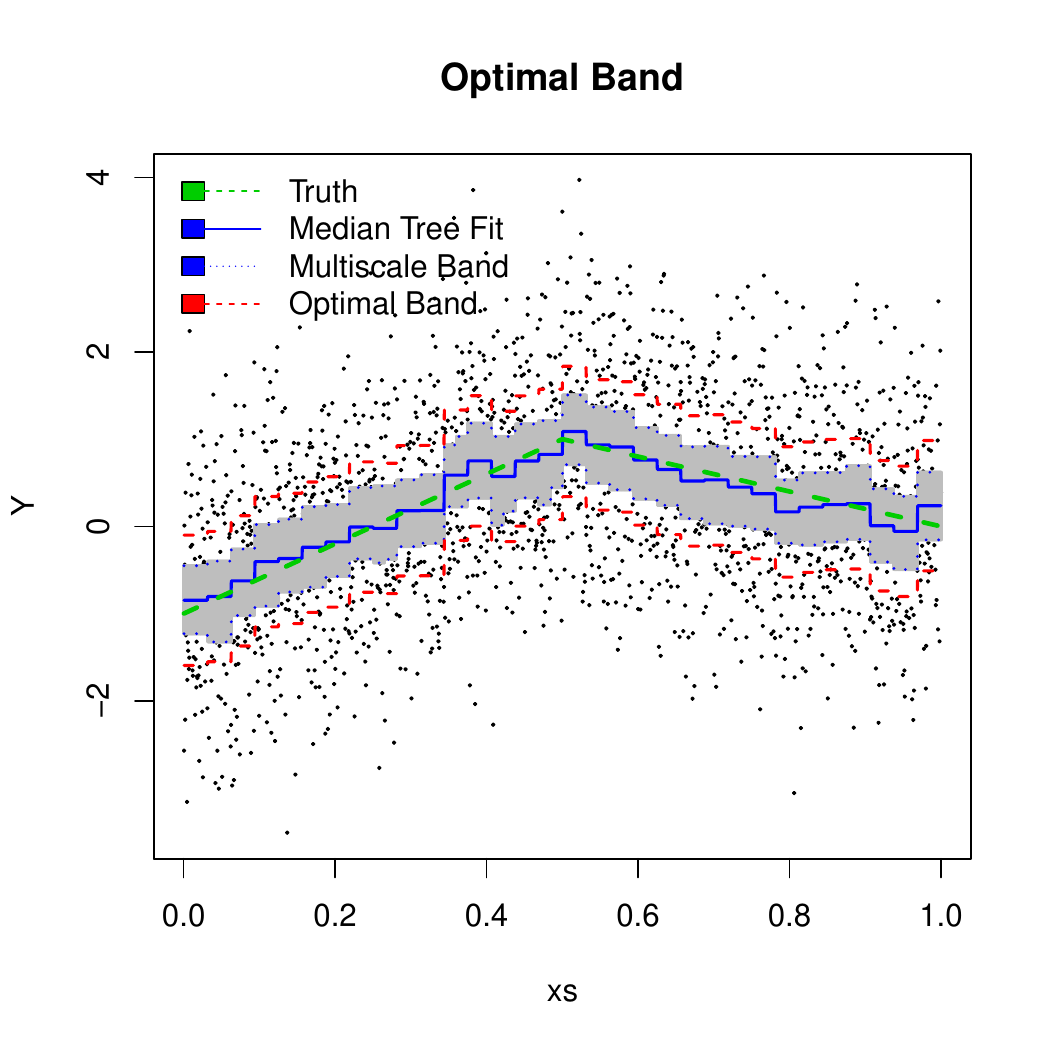}
%\begin{subfigure}{0.5\textwidth}\includegraphics{pointwise.pdf}\caption{Pointwise Bands}\label{fig:optimal}\end{subfigure}
\caption{\small (Left) Pointwise $0.95\%$ credible intervals together with a $95\%$ $L^\infty$--credible band (gray area). (Right) Not-intersected multiscale $0.95\%$ credible band \eqref{multballemp} (gray area)   using  $w_l=l^{1/2+0.01}$  (see Supplement, Section \ref{sec:multiscale_band}) together with the `optimal' set \eqref{credible} obtained with $v_n=1$. The true function is $f_0(x)=(4x-1)\mathbb{I}(x\leq 1/2)+(-2x+2)\mathbb{I}(x>1/2)$.}\label{figure:uq}
\end{figure}

\subsection{Inference for Functionals of $f_0$: Bernstein-von Mises Theorems}\label{sec:bvm}
By slightly modifying the Bayesian CART prior on the coarsest scales, it is possible to obtain asymptotic normality results, in the form of Bernstein-von Mises theorems, that imply that posterior quantile-credible sets are optimal-size confidence sets. In the next result, $\be_S$ denotes the bounded-Lipschitz metric on the metric space $S$ (see also the Supplement Section \ref{sec-bvm}).
\begin{theorem} \label{thm-fun}
Assume the Bayesian CART priors $\Pi_\bT$ from Theorem \ref{thm-one} constrained to trees that fit $j_0(n)$ layers, i.e. $\gamma_{lk}=1$ for $l\leq j_0(n)$, for $j_0(n)\asymp \sqrt{\log{n}}$. 
\begin{enumerate}
\item {\em BvM for smooth functionals $\psi_b(f):=\psg f,b\psd$.} Let $b\in L^\infty[0,1]$ with coefficients $(b_{lk}=\psg b,\psi_{lk}\psd)$. Assume $\sum_{k}|b_{lk}| \le c_l$ for all $l\ge 1$ with $\sum_{l} l^2 c_l<\infty$. Then, in $P_{f_0}$-probability,
\[ \beta_{\RR}\left(\mathcal{L}(\sqrt{n}(\psi_b(f)-\hat\psi_b)\given X),\mathcal{L}(
\cN(0,\|b\|_2^2)
%\int b^2(u)du)
)\right)\rightarrow0.\] 
\item {\em Functional BvM for the primitive $F(\cdot)=\int_0^\cdot f$}. 
Let  $(G(t):t\in[0,1])$ be a Brownian motion. Then, in $P_{f_0}$-probability,
\[ \beta_{C([0,1])}\left(\mathcal{L}\left(\sqrt{n}(F(\cdot)-\int_0^\cdot dX^{(n)}\C X)\right),\mathcal{L}(G)\right)\rightarrow0\] 
\end{enumerate}
\end{theorem}
As a consequence of this result, quantile credible sets for the considered functionals are optimal confidence sets. For $\al\in(0,1)$, let $q_{\al/2}^{\psi_b}(X)$ and $q_{1-\al/2}^{\psi_b}(X)$ be the $\al/2$ and $1-\al/2$ quantiles of the induced posterior distribution on the functional $\psi_b=\int_0^1 f(u) b(u) du$ and set $I_b(X):=[q_{\al/2}^{\psi_b}(X),q_{1-\al/2}^{\psi_b}(X)]$. Theorem \ref{thm-fun} (part 1) then implies  (see \cite{castillo_nickl1} for a proof) that 
\[ P_{f_0}\left[ \psi_b(f_0)\in I_b(X)\right] \to 1-\al.\]
%We illustrate numerically in the regression model in the supplement Section.
Similarly, let $R_n(X)$ be the data-dependent radius chosen from the induced posterior distribution on $F(\cdot)=\int_0^\cdot f$ as follows, for $\hat{F}(\cdot)=\int_0^\cdot dX^{(n)}$, 
\begin{equation}\label{bandF}
 \Pi[ \| F  -  \hat{F} \|_\infty \le R_n(X)\given X] = 1-\al.
\end{equation} 
Consider the band $\cc^F(X):=\{F:\, \|F-\hat{F} \|_\infty \le R_n(X)\}$.
Then Theorem \ref{thm-fun} (part 2) implies (see \cite{castillo_nickl1}, Corollary 2 for a related statement and proof), for $F_0(\cdot) =\int_0^\cdot f_0$,
\[ P_{f_0}\left[  F_0 \in \cc^F(X)  \right] \to 1-\al. \]
In other words, the band \eqref{bandF} has exact asymptotic coverage. It can also be checked that it is  optimal efficient in semiparametric terms (that is, its width is optimal asymptotically).  
We derive Theorem \ref{thm-fun} as a consequence of an adaptive nonparametric BvM (Theorem \ref{thm_bvm} in the Supplement; see Section \ref{sec:proof_thm-fun} for a proof, where other possible choices for $j_0(n)$ are discussed), only obtained 
so far for {\em adaptive} priors in the work of Ray \cite{ray}, which considered (conjugate) spike and slab priors. Derivation of the band \eqref{bandF} in practice is easily obtained once posterior samples are available.
Theorem \ref{thm-fun} is illustrated, in the regression framework studied in Section \ref{sec:npreg}, on a numerical  example with a piece-wise linear regression function (details on the implementation are in Section \ref{sec:simul}) in Figure \ref{figure:uqmain}.
The left panel presents a  histogram of posterior samples (together with $2.5\%$ and $97.5\%$ quantiles) of the rescaled primitive functional $\wt F(x)=n F(x)= \sum_{t_i\leq x} f(t_i)$ for $x=0.8$ with true value is marked with a red solid line. The right panel portrays the confidence band \eqref{bandF} which uniformly captures  the true functional (dotted line).

\begin{figure}[!] \label{fig:band}
\includegraphics[width=6cm]{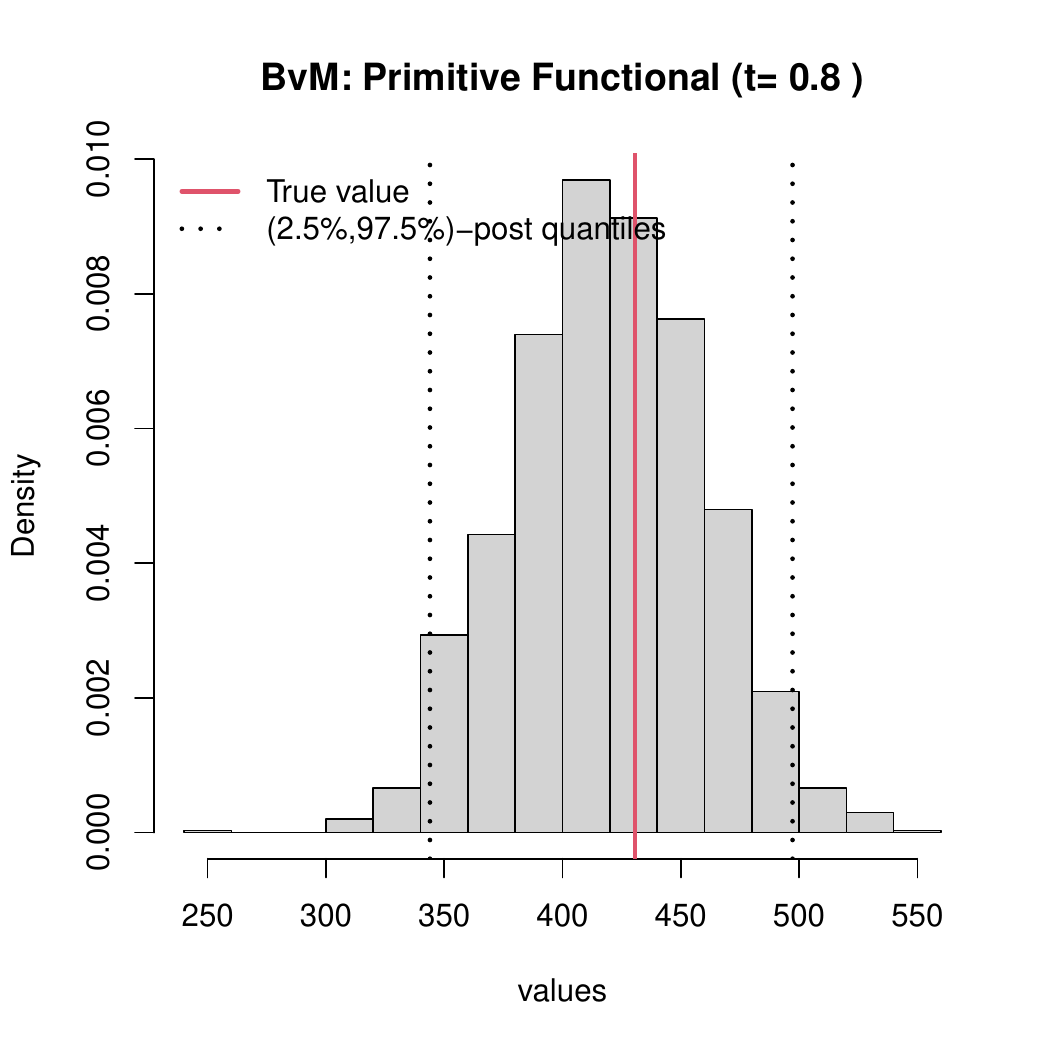}
\includegraphics[width=6cm]{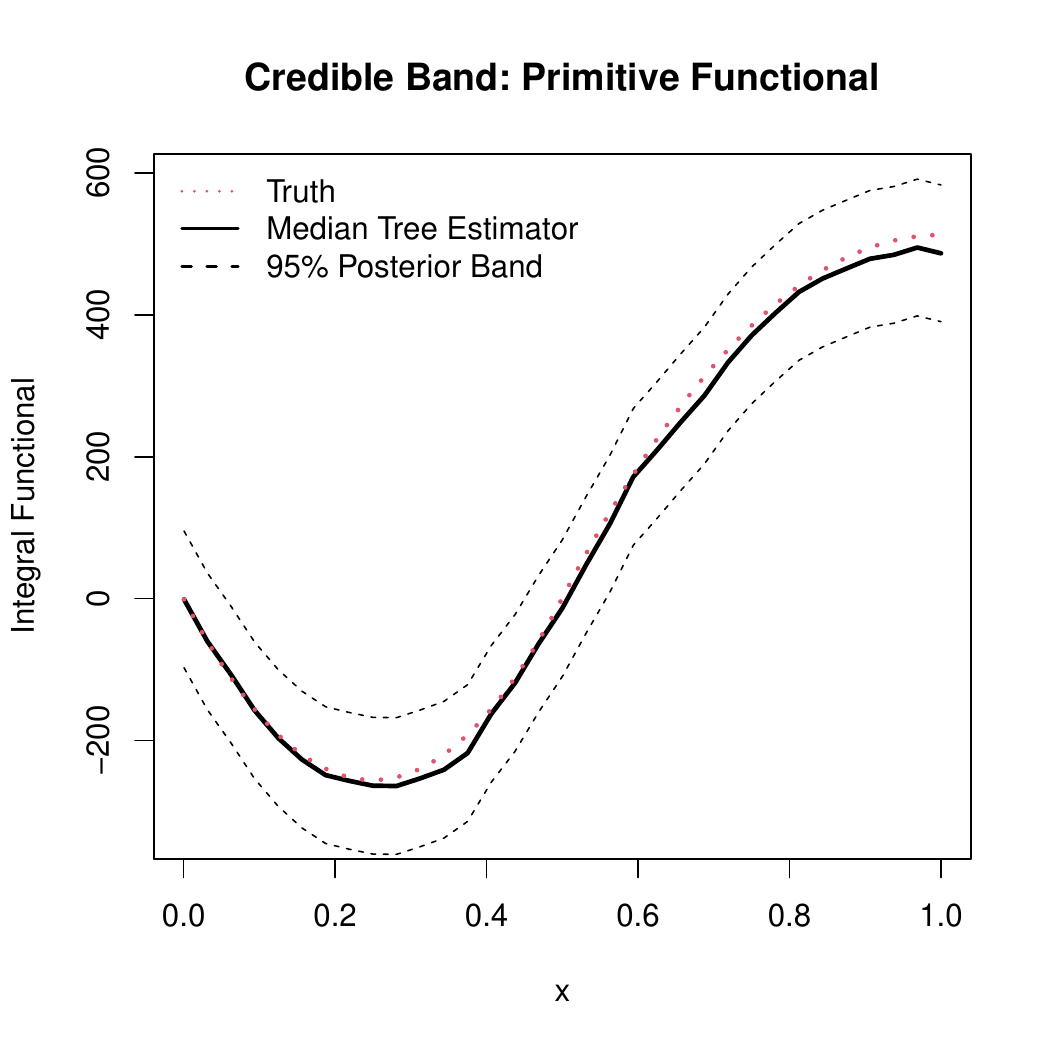}
%\begin{subfigure}{0.5\textwidth}\includegraphics{pointwise.pdf}\caption{Pointwise Bands}\label{fig:optimal}\end{subfigure}
\caption{\small   (Left) $0.95\%$ credible interval for the (rescaled) primitive functional $\wt F(x)$ with $x=0.8$; (Right) the confidence band \eqref{bandF}  obtained for  $f_0(x)=(4x-1)\mathbb{I}(x\leq 1/2)+(-2x+2)\mathbb{I}(x>1/2)$.}\label{figure:uqmain}
\end{figure}

\subsection{Lower bound: flat trees are (grossly) suboptimal for the $\|\cdot\|_\infty$--loss} \label{sec:flatneg}
Recall that the spike-and-slab prior achieves the {\sl actual} $\ell_\infty$--minimax rate {\em without} any additional factor. Interestingly, the very same prior misses the $\ell_2$--minimax rate by a log factor \cite{hoffmann}.  This illustrates that $\ell_2$ and $\ell_\infty$ adaptations  require different desiderata when constructing priors. Product priors that correspond to separable rules {\em do not} yield adaptation with exact rates  in the  $\ell_2$ sense \cite{cai08}. Mixture priors that are adaptive in $\ell_2$, on the other hand, may not yield $\ell_\infty$ adaptation.  
We now provide one example of this phenomenon in the context of flat (complete binary) trees.

The {\em flat} tree of depth $d=d(\mT)$ is the binary tree which contains all possible nodes until level $d$, 
%{\sl all} internal nodes  at levels smaller than $d$ have $2$ children and {\sl none} of the nodes at level $d$ has any children, 
i.e. $\gamma_{lk}=\mathbb{I}_{l<d}$. An example of a flat tree with $d=3$ layers is in Figure \ref{bayestrees2}.  %Such trees can be obtained by drawing the depth $d$ from  some prior $\pi(d)$.  This is arguably 
The simplest possible prior on tree topologies (confined to symmetric trees) is just the Dirac mass at a given flat tree of fixed depth $d=D$; an adaptive version thereof puts a prior $D$ and samples from the set of all flat trees. {Such priors coincide with  so-called {\em sieve} priors, where the sieve spans the expansion basis (e.g. Haar) up to level $D$.}  
%The ability of Bayesian Dyadic CART to adapt to the unknown smoothness index $\alpha$ in the $\ell_\infty$ sense is rather curious. 
%While adaptation in the $\ell_2$ sense can be expected for far less elaborate priors (such as flat trees),  as we illustrated below the $\ell_\infty$ adaptation is more intricate.
%In this section, we formally show that flat trees do not attain $\ell_\infty$ adaptivity and, by doing so, we provide further motivation for the tree-shaped sparsity priors in \eqref{eq:prior_beta2}.
Flat dyadic trees only keep Haar wavelet  coefficients at resolutions smaller than some  $d>0$ (i.e. $\gamma_{lk}=0$ for $l\geq d$). The implied prior on $(\beta_{lk})_{lk}$ can be written as, with $\pi(\beta_{lk})\propto \sigma_l^{-1} \phi\left(\beta_{lk}/\sigma_l\right)$,
%These trees correspond to regular regression histograms, as seen from \eqref{eq:dyadic2}. 
\begin{align}
(\beta_{lk})\C d &\, \sim\, \bigotimes_{l<d,k}\, \pi(\beta_{lk})\ \otimes\  \bigotimes_{l\geq d,k}\,\delta_0(\beta_{lk}),\label{prior:betalk}
\end{align}
where $\phi(\cdot)$ is some bounded density that is strictly positive on $\R$  and $\sigma_l$ are fixed positive scalars. 
The sequence $(\sigma_l)$  is customarily chosen so as it decays with the resolution index $l$, e.g. $\sigma_l=2^{-l(\beta+1/2)}$ for some $0< \beta\leq \alpha.$
 %Choosing $\sigma_l=1/\sqrt{2^L}$ and $\pi(\cdot)$ as the standard normal density, \eqref{prior:betalk} yields the standard normal prior $\wt\beta_k\iid\mathcal{N}(0,1)$ on jump sizes in regression histograms (van der Pas and Rockova (2017)). While this prior is enough to obtain near-optimal posterior concentration in the $\ell_2$ sense in non-parametric regression.
%,  here we consider a modification\footnote{The Gaussian iid prior does not seem to have the BvM property, correct?} that yields the BvM phenomenon. 
This   ``undersmoothing" prior requires the knowledge of (a lower bound on) $\alpha$ and yields a {\sl non-adaptive} non-parametric  BvM behavior \cite{castillo_nickl1}.

A tempting strategy to manufacture adaptation is  to treat the threshold $d$  as random {through} a prior {$\pi(d)$ on integers} {(and take constant $\sigma_l$)}, which corresponds to the hierarchical prior on regular regression histograms  \cite{rockova_vdp,vdp_rockova}.
%The prior \eqref{prior:betalk} can also be written as
%\begin{equation}\label{prior:spikeslab}
%$
%\pi\left(\{\beta_{lk}\}_{lk}\C d\right)=\prod_{l,k}\wt \pi(\beta_{lk}\C d),
%$
%\end{equation}
%where
%$
%\wt\pi(\beta_{lk}\C d)=\pi(\beta_{lk})\mathbb{I}(l<d)+\delta_0(\beta_{lk})\mathbb{I}(l\geq d).
%$
%Assuming  $d\sim\pi(d)$, this prior has {\sl spike-and-slab marginals} in the sense that
%$
%\pi(\beta_{lk})=\sum_{d}\pi(d)\wt\pi(\beta_{lk}\C d)=\pi(\beta_{lk})\P(l<d)+\delta_0(\beta_{lk})\P(l\geq d).
%$ 
%Curiously, similar marginals are  obtained with the spike-and-slab prior  \eqref{eq:beta_ss} after integrating out $\gamma_{lk}$.}
{It is not hard to check that the flat-tree prior \eqref{prior:betalk} with random $d$ has a marginal mixture distribution  similar to the one of the spike-and-slab prior  on  each coordinate $(l,k)$. }
Despite marginally similar, the  probabilistic structure of these two priors is very different. 
Zeroing out signals internally, the spike-and-slab prior \eqref{eq:beta_ss}  {\em is 
$\ell_\infty$--adaptive} \cite{hoffmann}. The flat tree prior \eqref{prior:betalk}, on the other hand, fits a few dense layers {\em without} internal sparsity and {\em is $\ell_2$--adaptive} (up to a log term) \cite{vdp_rockova}.  However, as shown in the following Theorem, flat trees fall short of $\ell_\infty$--adaptation. %for specific truths $f_0$.

\begin{theorem} \label{thm_lb}
Assume  the flat tree prior \eqref{prior:betalk} with random $d$, where $\pi(d)$ is non-increasing and where the active wavelet coefficients $\beta_{lk}$ are Gaussian iid $\cN(0,1)$. Moreover, assume $\{\psi_{lk}\}$ is an $S$--regular wavelet basis for some $S\geq 1$. 
For any $0<\al\le S$ and $M>0$, there exists $f_0\in \mathcal{H}(\al,M)$ such that 
\[ E_{f_0}\Pi\left[\ell_\infty(f_{\mT,\b},f_0) < \zeta_n \given X \right] \to 0, \]
where the lower-bound rate $\zeta_n$ is given by
$ \zeta_n = \left(\frac{\log{n}}{n}\right)^{\frac{\al}{2\al+2}}. $
\end{theorem}

%\begin{remark}
Theorem \ref{thm_lb}, proved in Section \ref{sec:proof_thm_lb}, can be applied to standard priors $\pi(d)$ with exponential decrease, proportional to $\e^{-d}$ or $\e^{-d\log{d}}$, or to a uniform prior over $\{1,\ldots,L_{max}\}$.
%By construction the ``difficult" function $f_0$ belongs to the space defined through the wavelet coefficients $\mathcal{H}(\al,M)$, which does not exactly coincide with the original H\"{o}lder space \eqref{eq:haar2}. Nevertheless, all our upper-bound results can be stated over the larger space $\mathcal{H}(\al,M)$, as can be checked from the proofs. 
%\end{remark}
\iffalse
\subsection{Suggestions}
Writing
\begin{enumerate}
\item [Supplement.] lemma for $\Pi(\cT)/\Pi(\cT_-)$
\item [Supplement.] lemma for number of (full) binary trees with $q$ nodes
\end{enumerate}
\fi
 In \cite{arbel}, a negative result is also derived for sieve-type priors, but only for the posterior mean and for Sobolev classes instead of the, here arguably more natural, H\"older classes for supremum losses  (which leads to different rates for estimating the functional--at--a--point). Here, we show that when the target is the $\ell_\infty$--loss for H\"older classes the sieve-prior is severely sub-optimal.

\subsection{Nonparametric Regression: Overview of Results}  \label{sec-reg}
Our results obtained under the white noise model can be transported to the more practical nonparametric regression model. While these two models are asymptotically equivalent \cite{brown_low} (under uniform smoothness assumptions satisfied, e.g.,  by $\alpha$-H\"{o}lderian functions with  $\alpha>1/2$), it is not automatic that the knowledge of a (wavelet shrinkage/non-linear) minimax procedure in  one model  implies the optimality in the other.
It turns out, however, that our results {\em can} be carried over to fixed-design regression without necessarily assuming $\alpha>1/2$. We assume outcomes $Y=(Y_1,\dots,Y_n)'$ arising from 
\begin{equation}\label{model}
Y_i=f_0(t_i)+\varepsilon_i,\quad \varepsilon_i\iid\mathcal N(0,1),\quad i=1,\dots, n=2^{L_{max}+1}
\end{equation} 
where $f_0$ is an unknown regression function and $\{t_i\in[0,1]:1\leq i\leq n\}$ are fixed design points. For simplicity, we consider a regular grid, i.e. $t_i=i/n$ for $1\leq i\leq n$ and assume $n$ is a power of $2$.  In Section  \ref{sec:npreg}, we show that most results for Bayesian CART  obtained earlier in   white noise carry over to the model \eqref{model} with a few minor changes. One minor modification concerns the loss function. We mainly consider the  `canonical' supremum-norm loss  for the fixed design setting, that is, 
 the `max-norm' defined for given functions $f,g$ by 
\[  \|f-g\|_{\infty,n} = \max_{1\le i\le n} |f(t_i)-g(t_i)|,   \]
 but it is also possible to consider the whole supremum-norm loss $\|\cdot\|_\infty$. We postpone statements and proofs to the Supplement, Sections \ref{sec:npreg} and \ref{sec:proofsreg}.  In a numerical study (Section \ref{sec:simul}), we illustrate that the implementation of Bayesian CART \cite{cart1,cart2} and the construction of our confidence bands is rather straightforward. For example, Figure \ref{figure:uq} shows how inference can be carried out with Bayesian CART posteriors in non-parametric regression with a piece-wise linear regression function using the intersecting band construction (detailed in Section \ref{sec:multiscale_band}). Contrary to  point-wise credible intervals (on the left) that are easy to produce but do not cover, our multiscale confidence band (on the right)  uniformly captures the true regression function. More details on this example are presented in Section \ref{sec:simul}.

\section{Non-dyadic  Bayesian CART}\label{sec:non_dyadic}
A limitation of {midpoint splits} in dyadic trees is that they treat the basis as fixed, allowing the jumps to occur {\em only} at {pre-specified} dyadic locations even when not justified by  data.
%Considering nested dyadic partitions in two dimensions, Donoho (1997) showed the equivalence between basis selection for a library of nested balanced Haar wavelets and dyadic CART.   
General CART regression methodology \cite{breiman_book, gey_nedelec} avoids this restriction by treating the basis as {\em unknown}, where the partitioning cells   shrink and stretch with data. In this section, we leave behind {`static'} dyadic trees to focus on the analysis of  Bayesian (non-dyadic)  CART \cite{cart1,cart2} and its connection to Unbalanced Haar  (UH) wavelet basis selection.
\iffalse\re{[In both sections 4.1 and 4.2 and the proofs of corresponding theorems, we need to verify some `density'-type condition of the constructed system. This will give a control of the bias term $\|f_0-\text{Proj}_{L^2}^{\text{span of the system}}f_0\|_\infty$. For the UH case, it should follow from the balancing condition that condition (B1) can be reverted and {\em all} Haar wavelets from any given level $L\le L_{max}/\log^3(n)$ (or so) can be reconstructed using the basis. For the multiD case, it is not clear to me this is granted through the current construction.]}\fi
%\sbl{[I tried to make clear that the restriction is not that much through assuming "dyadic trees" (which is a synonym of binary trees, so could be confusing) than through the pre-determined splits]}

\vspace{-0.2cm}
\subsection{Unbalanced Haar Wavelets}\label{sec:unbal}
{UH wavelet basis functions  \cite{girardi} 
%can be used to construct design-adapted wavelet basis  \cite{delouille}
are {\em not} necessarily translates/dilates of any mother wavelet function and, as such, allow for different support lengths and design-adapted  split locations.
Here, we particularize the constructive definition of UH wavelets given by \cite{piotr}. %, who  developed a multiscale procedure for non-parametric regression based on the UH wavelet transform.
Assume that  possible values for splits  are chosen from  a set of   $n=2^{L_{max}}$  breakpoints  $\mX= \{x_i:x_i=i/n,1\leq i\leq n\}$. %\footnote{In  non-parametric regression, $\mX$ could be regarded as the set of observed covariate values.}.
 Using the  scale/location index enumeration,  
%each pair $(l,k)$ is now equipped with 
{pairs $(l,k)$ in the tree are now equipped with} 
(a) a {\em breakpoint} $b_{lk}\in\mX$ and  (b) {\em left and right brackets} $(l_{lk},r_{lk})\in\mX\cup\{0,1\}$. Unlike  balanced Haar wavelets \eqref{haar}, where $b_{lk}=(2k+1)/2^{l+1}$, the breakpoints $b_{lk}$ are {\em not required} to be {regularly dyadically constrained} and are chosen from $\mX$ in a hierarchical fashion as follows. One starts by setting $l_{00}=0, r_{00}=1$. Then
\begin{itemize}
\item[(a)] The first breakpoint $b_{00}$ is selected  from $\cC\cap(0,1)$.
%\item[(b)]  If  $\cC\cap (l_{00},b_{00}]\neq\emptyset$, we say that the node $(1,0)$ is {\em admissible}\footnote{we need to account for the fact that some of the intervals will run out of observations and cannot be split further?}. If admissible, choose $b_{10}$ from $\cC\cap(0,b_{00}]$ and set
%\[ l_{10}=0,\ l_{11}=b_{00}. \]
%Similarly, if $(1,1)$ is admissible, i.e. $\cC\cap (b_{00},r_{00}]\neq\emptyset$, we choose $b_{11}$ from $\cC\cap (b_{00},1]$ and set $\ l_{11}=b_{00},\ r_{11}=1$. 
\item[(b)] For each $1\leq l\leq L_{max} $ and  $0\leq k<2^{l}$, set
  \begin{align}
 l_{lk}=l_{(l-1) \lfloor k/2\rfloor}, \quad & \quad  r_{lk}=b_{(l-1) \lfloor k/2\rfloor},\quad \text{if $k$ is even}, \label{stepsab}\\
 l_{lk}=b_{(l-1) \lfloor k/2\rfloor}, \quad & \quad  r_{lk}=r_{(l-1) \lfloor k/2\rfloor},\quad \text{if $k$ is odd}. \notag
 \end{align}
 If  
$
\cC\cap (l_{lk},r_{lk}]\neq\emptyset,
$ 
 choose $b_{lk}$ from $\cC\cap (l_{lk},r_{lk}]$.
\end{itemize}
Let $A$ denote the set of {\em admissible} nodes $(l,k)$, in that $(l,k)$ is such that $\cC\cap (l_{lk},r_{lk}]\neq\emptyset$, obtained through an instance of the sampling process described above and  let
\[B=(b_{lk})_{(l,k)\in A}\] be the corresponding set of breakpoints. 
Each collection of split locations $B$ gives rise to 
nested intervals
$$
L_{lk}=(l_{lk},b_{lk}]\,\quad \text{and}\quad R_{lk}=(b_{lk},r_{lk}].
$$ 
Starting with the mother wavelet  $\psi_{-10}^B=\psi_{-10}=\mathbb{I}_{(0,1)}$,
one then recursively constructs wavelet functions $\psi_{lk}^B$ with a support $I_{lk}^B=L_{lk}\cup R_{lk}$ as
\begin{equation}\label{eq:unbalanced_haar}
\psi_{lk}^B(x)=\frac{1}{\sqrt{|L_{lk}|^{-1}+|R_{lk}|^{-1}}}\left(\frac{\mathbb{I}_{L_{lk}}(x)}{|L_{lk}|}-\frac{\mathbb{I}_{R_{lk}}(x)}{|R_{lk}|}\right).
\end{equation}
%where $|I|$ denotes the length of the interval $I$.
%As mentioned in Proposition 1 of Fryzlewicz (2007), due to the hierarchical ordering of the split points $B$, the functions $\psi_{lk}^B$ form an orthonormal basis.
By construction, the  system $\Psi_A^B=\{\psi_{-10}^B,\psi_{lk}^B: (l,k)\in A\}$ is orthonormal in $L^2[0,1]$.%[0,1]$. %Indeed, by construction, the functions $\psi_{lk}^B$ have a unit $L^2$--norm and, for a given depth $l$, they have disjoint supports. 
%Furthermore,  $\psi_{lk}^B$ integrate to $0$ on their support and thereby each  $\psi_{lk}^B$ is orthogonal to $\psi_{l'k}^B$ for any $l'<l$. 
}
%We further set $\vphi(x)=\ind_{(0,1]}(x)$ and, for admissible $l,k$,
%\begin{equation}\label{eq:unbalanced_haar}
%\psi_{lk}^B(x)=\frac{1}{\sqrt{|L_{lk}|^{-1}+|R_{lk}|^{-1}}}
%\left(\frac{\ind_{L_{lk}}(x)}{|L_{lk}|}-\frac{\ind_{R_{lk}}(x)}{|R_{lk}|}\right),
%\end{equation}
%Similarly as in  \eqref{eq:domain}, where $I_{lk}$ denotes the  support of the balanced Haar wavelets $\psi_{lk}$, we denote with $I_{lk}^B$ the support of $\psi_{lk}^B$. 
%\footnote{For $l>L$, one can extend the system $\cS_L$ above as follows. For levels deeper than $l$, one follows the previous construction but instead of sampling $b_{ij}$ over $(l_{ij},r_{ij}]$, one defines $b_{ij}$ as the mid-point of the interval $(l_{ij},r_{ij}]$ (or possibly the closest element of $\cC$).}
With UH wavelets, the decay of wavelet coefficients $\beta_{lk}=\psg f,\psi_{lk}^B\psd$ for a $\al$--H\"older function $f$ verifies $|\beta_{lk}^B|\leqa\max\{|L_{lk}|,|R_{lk}|\}^{\alpha+1/2}$, see Lemma \ref{lemma:haar}. 
%\proof Section \ref{sec:proof_lemma:haar}.
%For the classical Haar basis \eqref{haar}, one obtains \eqref{eq:haar} from \eqref{eq:haar3} by noting $\max\{|L_{lk}|,|R_{lk}|\}=2^{-(l+1)}$. 
\cite{piotr} points out that the computational complexity of the discrete UH transform could be unnecessarily large and imposes the balancing requirement 
$\max\left\{|L_{lk}|,|R_{lk}|\right\}\leq E (|L_{lk}|+|R_{lk}|)\ \forall (l,k)\in A
$, for some $1/2\leq E<1$.    %Under this assumption, it turns out that $ |\beta_{lk}^B|\leq M\,2^{\alpha-1/2} C^{l(\alpha+1/2)}$
%and one can show that the posterior probability of trees with more than $\mathcal{L}^\star_C$  layers goes to zero, where 
%$(1/C)^{\mathcal{\mathcal{L}_C^\star}}\doteq M^{1/(2\alpha+1)}\left(\frac{n}{\log n}\right)^{1/(2\alpha+1)}.$   
Similarly, in  order to control the combinatorial complexity of the basis system, we require that the UH wavelets are {\em weakly balanced} in the following sense.

%\sbl{[Somewhat of a detail, but I just noticed that $C$ is not very practical, as we cannot write anymore $C$ to refer to a large enough universal constant... We use this quite a bit in the proofs, so I suggest that we call rather `$E$' the '$C$' of the balancing condition]}

\begin{definition} \label{defbalance}
A system $\Psi_A^B=\{\psi_{-10}^B,\psi_{lk}^B: (l,k)\in A\}$ of UH wavelets is {\em weakly balanced} with balancing constants $E,D\in\N^*$ %such that $D\geq 1$ and $C\geq 2^{D-1}$ 
if, for any $(l,k)\in A$,
\begin{equation}\label{eq:balance2}
\max\limits\left(|L_{lk}|,|R_{lk}|\right) = \frac{M_{lk}}{2^{l+D}}\quad\text{for some $M_{lk}\in\{1,\dots, E+l\}$}.
\end{equation}
\end{definition}
%Note that balanced Haar wavelets are also weakly balanced  for any $D\geq 1$ and $2^{D-1}\leq C$ since 
%$|L_{lk}|=|R_{lk}|=1/2^{l+1}$. 
%To illuminate the balancing condition \eqref{eq:balance2}, we consider the following ``counterexample".
%
%Our   theoretical development relies in part on combinatorial properties of weakly balanced UH  systems and on the speed of decay of the multiscale functionals $\beta_{lk}^B=\langle f,\psi_{lk}^B\rangle$  as the layer index $l\in\N$ increases.
%These two fundamental properties are encapsulated in Lemma \ref{lemma:complex} which is vital to the proof of the upcoming Theorem \ref{thm-three}. 
%
%The first one pertains to the combinatorial complexity expressed in terms of the number of balanced Haar wavelet functions $\psi_{jk}$ for $j\le l+D$ needed to express each $\psi_{lk}^B$ as their linear combination. The second one pertains to the speed of decay of $|\beta_{lk}^B|$'s as a function of the layer index $l\in\N$. These two properties are summarized in the Lemma below.
%Conditions (B1)--(B2) in Lemma \ref{lemma:complex} are sufficient for achieving  near-optimal sup-norm posterior concentration,  uniformly  over {\em all} weakly balanced bases. 
 %While \cite{piotr} showed uniform consistency   over UH bases satisfying a certain balancing requirement, here we show the optimal rate of posterior convergence in the $\ell_\infty$ sense.

Note that in the actual BART implementation, the splits are chosen from  sample quantiles  to ensure balancedness (similar to our condition \eqref{eq:balance2}). Quantile splits (Example \ref{exquantiles} below) are a natural way to generate many weakly balanced systems, providing a much increased flexibility compared to dyadic splits, which correspond to uniform quantiles. Other examples together with a graphical depiction of the unbalanced Haar wavelets for certain non-dyadic choices of split points $b_{lk}$ are in the Supplement (Figure \ref{fig:diagram} in Section \ref{sec-suppnd}).   

%Lemma \ref{lemma:complex} states that weakly unbalanced wavelets  allow for a bit slower decay of wavelet coefficients, relative to classical Haar wavelets.
{\begin{example}[Quantile Splits] \label{exquantiles} 
 Denote with $G$ a c.d.f with a density $g$ on $[0,1]$ that satisfies $\|g\|_\infty \leq2^{D-1}/(2E)$ {for $E,D>0$ chosen below} and $\|1/g\|_\infty\leq C_q$ for some $C_q>0$. Let us  define a dyadic projection of $G$ as
\[ G_{l}^{-1}(x):=2^{-l}\lfloor 2^l G^{-1}(x)\rfloor, \] 
 and next define the breakpoints, for ${l\le L_{max}}$ and $0\le k<2^l$, as
\begin{equation} 
b_{lk}=G_{\lmax+D}^{-1}[(2k+1)/2^{l+1}].\label{balanced_splits}
\end{equation}
    %This is a natural generalization of  dyadic points $b_{lk}=(2k+1)/2^{l+1}$.  
The system $\Psi_{A}^B$ obtained from steps (a) and (b)  with splits  \eqref{balanced_splits} is weakly balanced {for $E=2+3C_q2^{D-1}$}. This is verified in Lemma \ref{lem-quantiles} in the Appendix (Section \ref{ap:sec_quantiles}). Moreover, Figure \ref{fig1} in  illustrates the implementation of the quantile system, where splits are placed more densely in areas where $G(x)$ changes more rapidly.
\end{example}

\begin{figure}[!t]
\vspace{-0.5cm}
   \subfigure[$g(x)\sim \mathcal{B}(1,1)$]{
    \label{fig1}
     \begin{minipage}[b]{0.45\textwidth}
     \centering {\includegraphics[height=4cm,width=6cm]{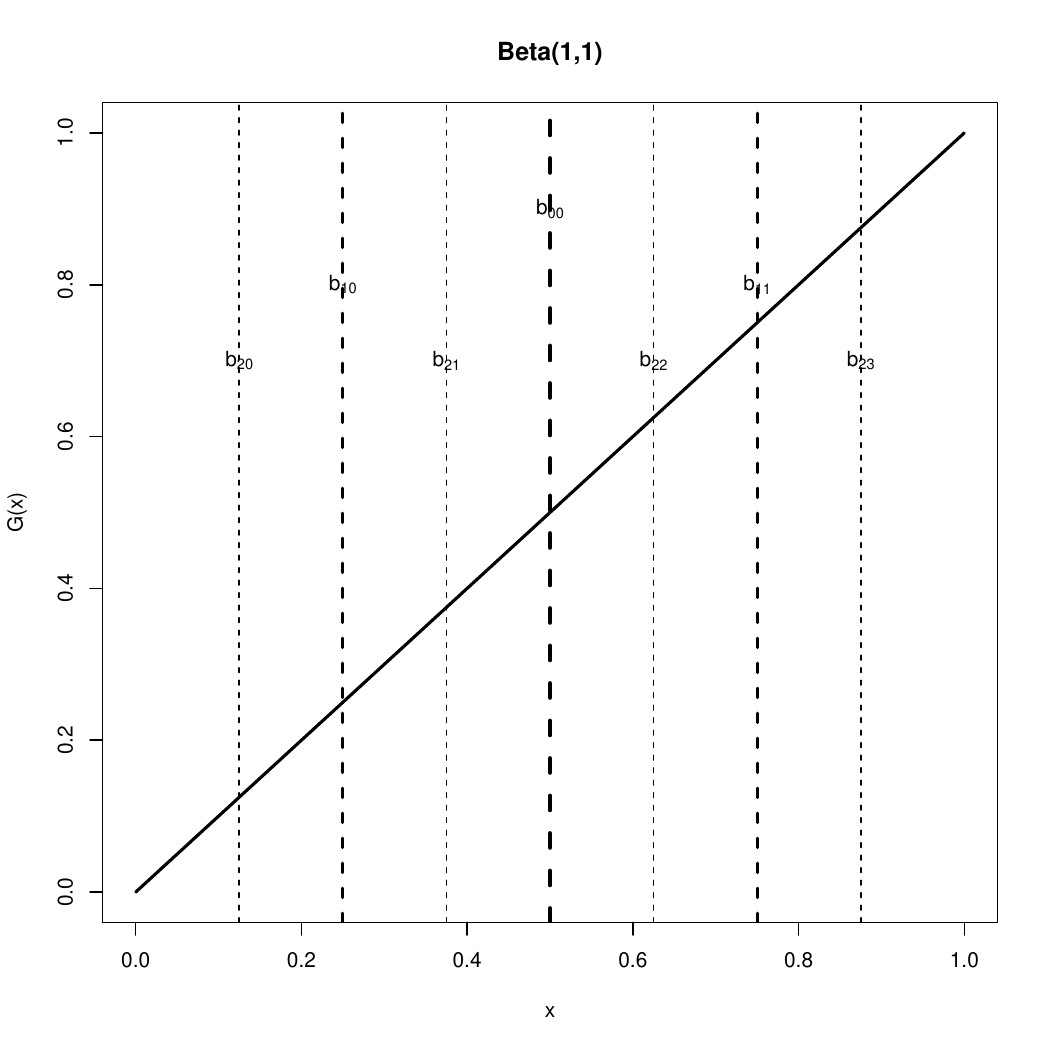}}
    \end{minipage}}
    \subfigure[$g(x)\sim \mathcal{B}(2,5)$]{
    \label{fig:estimates2}
    \begin{minipage}[b]{0.45\textwidth}
       \centering {\includegraphics[height=4cm,width=6cm]{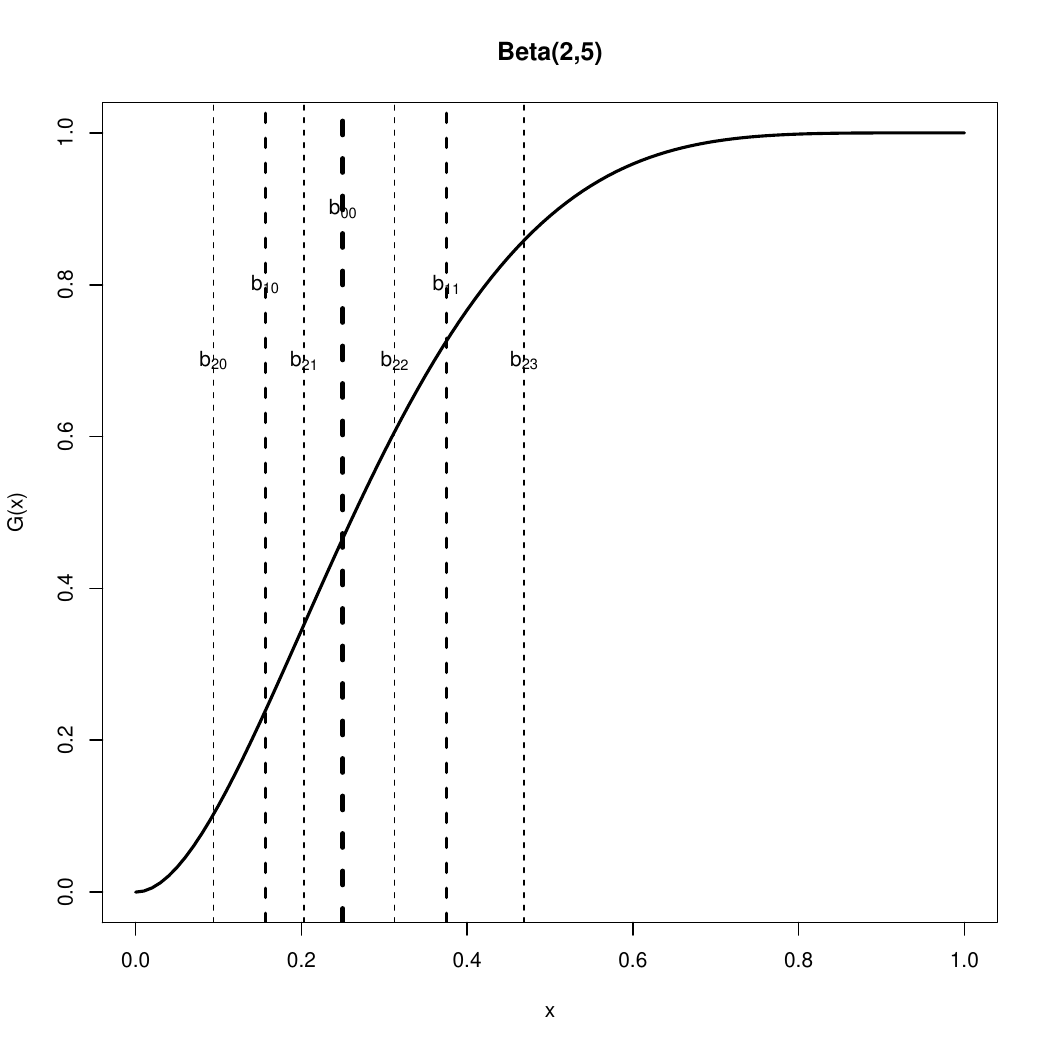}}
    \end{minipage}}
    \caption{  \small  Example of quantile splits for a uniform density $g(x)$ and a non-uniform beta density $g(x)$ using $L_{max}=6$. }
    \label{fig1}
\end{figure}

%Lemma \ref{lemma:complex} states that weakly unbalanced wavelets  allow for a bit slower decay of wavelet coefficients, relative to classical Haar wavelets.

%\subsection{Non-dyadic CART prior and multiscale properties}  
%\subsection{Sup-norm Convergence}\label{sec:non_dyadic}
%The recursive construction of the  weakly balanced Haar basis pertains closely to the Bayesian CART prior of \cite{cart1}.  
% Instead of  confining $b_{lk}$ to dyadic midpoints (step (i) in Section \ref{sec:bc1}), such a prior draws $b_{lk}$  from available observations.
%We consider a related, and more general, strategy which  separates the prior on the basis $\Pi_\mb$  from  the prior on the trees $\Pi_\bT$. 
The {\em non-dyadic Bayesian CART} prior  is then defined as follows:
% in the following three steps:
\begin{itemize}
\item {\em Step 1. (Basis Generation)} 
Sample $B=(b_{lk})_{0\le k<2^l-1, l\le L}$ from $\Pi_\mb$ by following the  steps a)--b) around \eqref{stepsab} subject to satisfying the {\em balancing condition} \eqref{eq:balance2}. %(or assumptions (B1) and (B2)). 
 %\sbl{[for Step 3 it may actually be convenient to take the full infinite set of breakpoints i.e. considering the extension to $l>L$ mentioned above. It should not matter much anyways if we take $L$ close to $\log_2{n}$ because larger levels will not carry any signal.]}\\
\smallskip
\item {\em Step 2. (Tree Generation)} 
Independently of $B$, sample a  binary tree $\cT$  from one of the priors $\Pi_\bt$ described in Section \ref{sec:prior_trees}. 
\smallskip
\item {\em Step 3. (Step Heights Generation)} Given $\mT$, we obtain the coefficients $(\beta_{lk}^B)$ from the tree-shaped prior \eqref{eq:prior_beta2}. 
Using the UH wavelets, the prior on the internal coefficients $\beta_{lk}^B$ can be translated into a model on the histogram heights $\wt\beta_{lk}^B$ through \eqref{tree_expand22}.
\end{itemize} 
\iffalse
Similarly as in Section \ref{sec:prior_g}, the prior on the internal coefficients $\beta_{lk}^B$ can be translated into a model on the histogram heights $\wt\beta_{lk}^B$ through
\begin{equation}\label{tree_expand2}
f_{\mT,\wt\b}^{B}(x)= \sum_{(l,k)\in\mT_{ext}}\wt\beta_{lk}^BI_{lk}^B(x)=\beta_{-10}^B\psi^B_{-10}(x)+\sum_{(l,k)\in\mT_{int}}\beta_{lk}^B\psi_{lk}^B(x).
\end{equation}
Again, the {\sl internal unbalanced} Haar wavelet coefficients  $\b_\mT=(\beta^B_{lk}:(l,k)\in\mT_{int})$ are linked to  the {\sl external}  histogram coefficients $\wt\b_\mT=(\wt\beta^B_{lk}:(l,k)\in\mT_{ext})$
through 
%\begin{equation}\label{onetoone2}
$\wt\b_{\mT}=A_\mT \b_\mT
$
%\end{equation} 
for some sparse matrix $A_\mT\in \R^{|\mT_{ext}|\times|\mT_{ext}|}$ (a generalization of \eqref{onetoone}). 
Having this representation, we can   deploy the $g$-prior as before. %The Gaussian white noise model projects onto the unbalanced Haar system as follows: by taking the bracket with $\psi_{lk}^B$, we have
%\[ X_{lk}^B = \beta_{lk}^{0B} + \frac{1}{\rn}\veps_{lk}^B, \]
%where $X_{lk}^B=\langle X,\psi_{lk}^B\rangle$, $\beta_{lk}^{0\,B}=\langle f_0,\psi_{lk}^B\rangle$ and  $\veps_{lk}^B=\langle W, \psi_{lk}^B\rangle$. As the functions $\psi_{lk}^
%B$ form an orthonormal system, the variables $\veps_{lk}$ are iid standard Gaussian. For the given unbalanced Haar system $\Psi_{A}^B=\{\psi_{lk}^B:(l,k)\in A)\}$, where $B$ satisfy the balancing requirement, we can formulate a variant of Theorem \ref{thm-one}.  
\fi
{An example of such a prior is obtained by first randomly drawing quantiles (e.g. by drawing a density at random verifying conditions as in Example \ref{exquantiles}) to generate the breakpoints for Step 1 and then following the construction from Section \ref{sec:tree_prior} for Steps 2--3.}  
%The general spirit of the proof of Theorem \ref{thm-one} can still be used. The main difference in the next result, proved in Section  \ref{sec:proof_thm-three}, is the additional hierarchical layer for sampling the breakpoints from $\Pi_\mb$.  
The following theorem is proved in Section  \ref{sec:proof_thm-three}.%, is positioned for the (non-smooth) weakly balanced UH wavelets.

\begin{theorem} \label{thm-three}
Let $\Pi_{\mb}$ be any prior on breakpoint collections that  satisfy weak balancedness according to Definition \ref{defbalance}.
Let $\Pi_{\bt}$ be  the  Galton-Watson process prior  from Section \ref{sec:prior_trees} with $p_{lk}=\Gamma^{-l^4}$.
Consider the tree-shaped wavelet prior \eqref{eq:prior_beta2} with $\pi(\b_\mT)\sim\mathcal{N}(0,I_{|\mT_{ext}|})$. Let $f_0\in\cH^{\al}_M$ as in \eqref{eq:haar2} for some $M>0$ and  $0<\al\le 1$ and define
\begin{equation}\label{eq:rate}
\veps_n =(\log{n})^{1+\frac{3}{2}} \left(\frac{\log{n}}{n}\right)^{\frac{\al}{2\al+1}}.
\end{equation}
{Then, there exist $\Gamma_0,c_0>0$ depending only on the constants $E,D$ in the weak balancedness condition such that, for any $\Gamma\ge \Gamma_0$ and $c\ge c_0$,} for any $M_n\to\infty$, we have, for $n\rightarrow\infty$
\begin{equation}\label{eq:convergence}
 E_{f_0}\Pi\left[\,\ell_\infty(f_{\mT,\b},f_0)\ge \|f_{\mT,\b}-f_0\|_\infty > M_n\veps_n\given X \right] \to 0.
 \end{equation}
\end{theorem}

{In the context of piecewise constant priors, Theorem \ref{thm-three} allows further flexibility in the choice of the prior as compared to Theorem \ref{thm-one} in that the location of the breakpoints, on the top of their structure given by the tree prior, can vary in their location according to its own specific prior. 
%\begin{remark}
%The log-factor  in Theorem \ref{thm-three} depends on the amount of unbalancedness. 
%A more general statement with a log factor $(\log{n})^{1+\frac{\delta}{2}}$ for some $\delta>0$ can be obtained  under Conditions $(B1)$ and $(B2)$ in Lemma \ref{lemma:complex} and under the assumption $\|\psi_{lk}^B\|\lesssim 2^{l}$.  According to Lemma \ref{lemma:complex}, weakly balanced systems verify $(B1)$ and $(B2)$ with $\delta=3$.
Whether one can further weaken the balancing condition to still get optimal multiscale results is an interesting open question that goes beyond the scope of this paper. % \ma{[I put this remark here because it felt like it was betraying a bit the good message in Theorem 3. We argue that the assumption is not as restrictive with our quantile example.]} 
In addition, the log-factor in \eqref{eq:rate} could be further optimized, similarly as in Theorem \ref{thm-one}.

\section{Discussion} \label{sec:disc}
In this paper we explored  connections between Bayesian tree-based regression methods and structured wavelet shrinkage. We demonstrated that Bayesian tree-based methods  attain (almost) optimal convergence rates in the supremum norm and obtain limiting results for functionals, that follow from a non-parametric and adaptive Bernstein--von Mises theorem. The developed framework  also allows us to construct  adaptive credible bands around $f_0$ under self-similarity. 
To allow for non-dyadically organized splits, we introduced weakly balanced Haar wavelets (an elaboration on unbalanced Haar wavelets of \cite{girardi}) and showed that Bayesian CART performs basis selection from this library and attains a near-minimax rate of posterior concentration under the sup-norm loss. 

Although for clarity of exposition we focused on the white noise model, our results can be extended to the more practical  regression model for fixed regular designs (Section \ref{sec:npreg} in the Supplement) or possibly more general designs under some conditions. We note that the techniques of proof are non-conjugate in their key tree aspect, which opens the door to applications in many other statistical settings.  {A version of Bayesian CART for density estimation following the ideas of the present work is currently investigated by  T. Randrianarisoa as part of his PhD thesis. More precisely, using the present techniques, it is possible to develop multiscale rate results for P\'olya trees with `optional stopping' along a tree, in the spirit of \cite{wongma10}. Our confidence set construction can be also shown to have local adaptation properties. The ability of Bayesian CART to spatially adapt in this way will be investigated in a followup work.
Further natural extensions include high-dimensional versions of the model, extending the multi-dimensional version briefly presented here, as well as forest priors. These will be considered elsewhere. }
%\sbl{[to smooth/complete; mention other high-dim extensions?]} \ma{[Add comments on PTs]}

%In addition, we formalized a fully non-parametric and adaptive Bernstein-von Mises theorem for dyadic trees and constructed an adaptive credible band around $f_0$ under self-similarity. These results  constitute an important contribution to the literature on uncertainty quantification for machine learning methods. 

\iffalse
Our results can be extended in various ways. Since our analysis has been performed under the white noise model, it will be of interest to transport our sup-norm convergence results  to the actual non-parametric regression model \cite{yoo}. Another interesting extension will be density estimation (Thibaudet's reference here), high-dimensional settings  with subset selection uncertainty and sum-of-trees models as indicated in Example 4.2.
\fi

%it follows  from Low (1997) that confidence bands that are simultaneously adaptive and honest do not exist in generality. Gine and Nickl (2010) point out, however, that they can be obtained for generic subsets of functions which are dense and statistically well-behaved (self-similar functions).   The results in this paper can be leveraged for the construction of the adaptive multiscale credible balls under self-similarity.

\section{Proof of Theorem \ref{thm-one}} \label{sec:proof_thm-one}
%\label{sec:proofs}

The proof proceeds in three steps. In Section \ref{sec:dim} we first show that the posterior concentrates on not too deep trees.
In Section \ref{subsec:signal}, we then show that the posterior probability of missing signal vanishes and, finally, in Section \ref{subsec:concentration} we show that the posterior distribution concentrates around signals.
To better convey main  ideas, we present the proof for the independent prior $\b_\mT\sim\mathcal N(0,\Sigma_{\mT})$ with $\Sigma_{\mT}=I_{K}$ for $K=|\mT_{ext}|$ and the Galton-Watson (GW) tree prior from Section \ref{sec:bc1} with a split probability $p_l$. 
The proof for the $g$-prior $\Sigma_{\mT}=g_n(A_\mT'A_\mT)^{-1}$ is more technically involved and is presented in Section \ref{sec:thm-one} in the Supplement.

%Recall the following notation $L_{max}=\lfloor \log_2 n\rfloor$ for the maximal tree depth that can be obtained with $n$ observations. 
%Throughout the proof, we  work with  the following event 
We will be working conditionally on the event 
\begin{equation} \label{event}
\cA = \left\{ \max_{-1\le l\le L, \, 0\le k <2^l} \veps_{lk}^2 \le 2\log\left( 2^{L+1}\right) \right\},
\end{equation}
where  $L=L_{max}=\lfloor\log_2n\rfloor$. Since $\veps_{lk}\sim\mathcal{N}(0,1)$, this event has a large probability in the sense that  $P(\cA^c)\lesssim (\log n)^{-1}$, which  follows from $P\left[\max\limits_{1\leq i\leq N}|Z_i|>\sqrt{2\log N}\right]\leq c_0/\sqrt{\log N}$ for some $c_0>0$ when $Z_i\sim \mathcal{N}(0,1)$ for $1\leq i\leq N$.

\subsection{Posterior Probability of Deep Trees}\label{sec:dim}
The first step is to show that, on the event $\cA$,   the posterior concentrates on reasonably small trees, i.e. trees whose depth $d(\mT)$ is no larger than an `optimal' depth which depends on the unknown smoothness $\alpha$. 
Let us define such a depth $\cL_c=\cL_c(\al,M)$ as
%$M>0$ is the H\"{o}lder constant in \eqref{eq:haar}.
\begin{equation} \label{cutoff}
{\cL_c}=\left\lceil \log_2\left((8M)^{\frac1{\al+1/2}}\left( \frac{n}{\log{n}} \right)^{\frac1{2\al+1}}\right)\right\rceil.
\end{equation}
\begin{lemma}\label{lemma:dim}
Under the assumptions of Theorem \ref{thm-one},
%where $\doteq$ means that one takes the closest integer $\cL_c$ to the real solution of the equation.
 on the event $\cA$, 
\begin{equation}\label{eq:dim}
\Pi[d(\mT)>\cL_c\C X]\rightarrow 0 \quad (n\rightarrow\infty).
\end{equation}
\end{lemma}
\begin{proof}
 %At the end of this section, we adapt the proof to the other two tree priors from Section \ref{sec:bc2}.
%Recall that $\bT$ denotes all binary trees with depth no larger than $L_{max}=\lfloor\log_2 n\rfloor$. 
Consider one  tree $\mT\in\bT$ such that $d(\mT)\geq1$ and denote with $\mT^{-}$ a pruned subtree obtained from $\mT$ by turning its deepest rightmost internal node, say $(l_1,k_1)$, into a terminal node. Then $\mT^{-}=\mT_{int}^-\cup\mT_{ext}^-$, where
$$
\mT_{int}^-=\mT_{int}\backslash\{(l_1,k_1)\},\quad \mT_{ext}^-=\mT_{ext}\backslash\{(l_1+1,2k_1),(l_1+1,2k_1+1)\}\cup \{(l_1,k_1)\}.
$$
Note that $\mT^-$ is  a full binary tree and that the mapping $\mT\rightarrow\mT^-$ is not necessarily injective. Indeed, there are up to $2^{d(\mT^-)}$ trees $\mT$ that give rise to the same pruned tree $\mT^-$. 
%(one attaches a pair of leaves to a leave of $\mT^{-}$ \ma{or a right internal sibling node to an internal node of $\mT^-$ at depth $d(\mT^-)$ (??) [remove? I do not understand]}).  
Let $\bt_d=\{\mT\in \bT:d(\mT)=d\}$ denote the set of all full binary trees of depth {\em exactly} $d\ge 1$. Then, using the notation \eqref{eq:W},
% where we just cut out the left-most (for definiteness) leave $\nu_{LM}$ of depth exactly $d(\cT)$ as well as its neighbor leave (which has the same parent node as $\nu_{LM}$). Note that cutting a couple of neighboring leaves in a full binary tree, one is left with a full binary tree, so $\cT^-$ is still a full binary tree.
% Given a tree $\cT$ of depth at least $\cL_c$, 
%  we denote by $\cT_c$ the subtree of $\cT$ from the root to the node of depth $\cL_c$. 
\begin{align}\label{Wratio}
\Pi[\bt_d \given X] & = 
%\sum_{d=\cL_c+1}^{L_{max}} \sum_{d(\cT)=d} \Pi[\cT\given X] \\
%& = \sum_{d=\cL_c+1}^{L_{max}} \sum_{d(\cT)=d} \frac{\Pi[\cT\given X]}{\Pi[\cT_c\given X]}\Pi[\cT_c\given X] \\
\frac{ \sum_{\cT\in \bt_d} W_X(\cT)}{ \sum_{\cT\in \bT} W_X(\cT)}=
\frac{ \sum_{\cT\in \bt_d} \frac{W_X(\cT)}{W_X(\cT^-)}W_X(\cT^-)}{ \sum_{\cT\in \bT} W_X(\cT)},
% \le a(n,d)\frac{ \sum_{\cT\in \bt_d} W_X(\cT^-)}{ \sum_{\cT\in \bT} W_X(\cT)}
\end{align}
$$
\text{where} \qquad \frac{W_X(\mT)}{W_X(\mT^-)}=\frac{\Pi_{\bT}(\mT)}{\Pi_{\bT}(\mT^-)}\frac{\int \prod_{(l,k)\in \mT_{int}'} \e^{nX_{lk}\beta_{lk}-n\beta_{lk}^2/2}d\pi(\b_{\mT})}{\int \prod_{(l,k)\in {\mT_{int}^-}'}\e^{nX_{lk}\beta_{lk}-n\beta_{lk}^2/2}d\pi(\b_{\mT^-})}.
$$
Let $\X_\mT=(X_{lk}:(l,k)\in\mT_{int}')'$ and $\b_{\mT}=(\beta_{lk}:(l,k)\in\mT_{int}')'$ be the top-down left-to-right ordered sequences (recall that we  order nodes according to the index $2^l+k$). Assuming 
 $\b_\mT\sim\mathcal{N}(0,\Sigma_\mT)$, and denoting  $K=|\mT_{ext}|=|\mT_{int}|+1$, 
\begin{align}
\frac{W_X(\mT)}{W_X(\mT^-)}&=\frac{\Pi_{\bT}(\mT)}{\Pi_{\bT}(\mT^-)}\sqrt{\frac{{|\Sigma_{\mT^{-}}|}}{2\pi{|\Sigma_\mT|}}}\frac{\int \e^{n\X_{\mT}'\b_{\mT}-\b_{\mT}'[nI_K+\Sigma_{\mT}^{-1}]\b_{\mT}/2}
d\b_{\mT}}{\int \e^{n\X_{\mT^-}'\b_{\mT^{-}}-\b_{\mT^{-}}'[nI_{K-1}+\Sigma_{\mT^{-}}^{-1}]\b_{\mT^{-}}/2}d\b_{\mT^{-}}}
\nonumber
%\label{eq:W_ratio}
\\
&=\frac{\Pi_{\bT}(\mT)}{\Pi_{\bT}(\mT^-)}\sqrt{\frac{{|\Sigma_{\mT^{-}}|}}{{|\Sigma_\mT|}}}\sqrt{\frac{{| nI_{K-1}+\Sigma_{\mT^-}^{-1}|}}{{| nI_K+\Sigma_{\mT}^{-1}|}}}\frac{\e^{n^2\X_{\mT}'(n I_K+\Sigma_{\mT}^{-1})^{-1}\X_{\mT}/2}}{\e^{n^2\X_{\mT^-}'(n I_{K-1}+\Sigma_{\mT^-}^{-1})^{-1}\X_{\mT^-}/2}}.\label{eq:ratio_W}
\end{align}
%Denote with $S_{\mT,\mT^-}$ the Schur complement of $(I+n\Sigma_{\mT^-)$ in $(I+n\Sigma_\mT)$, then 
Since $X_{l_1k_1}$ corresponds to the node $(l,k)$ with the highest index $2^l+k$, one can write $\X_{\mT}=(\X_{\mT^-},X_{l_1k_1})'$. 

We focus on  the GW prior from Section \ref{sec:bc1} and on the independent prior $\Sigma_{\mT}=I_K$ and present proofs for the remaining priors in Section \ref{sec:thm-one}. Using the expression \eqref{eq:ratio_W} and  since
$(l_1,k_1)$ is the deepest rightmost internal node  in $\mT$, and $\mT$ is of  depth $d=d(\mT)=l_1+1$, using the definition of the GW prior,%we can write
\[ 
\frac{W_X(\cT)}{W_X(\cT^-)} = \frac{\Pi_\bT(\cT)}{\Pi_\bT(\cT^-)} 
\prod_{(l,k)\in \cT_{int}'\setminus \cT_{int}'^-}\frac{\e^{\frac{n^2}{2(n+1)}X_{lk}^2}}{\sqrt{n+1}}
= \frac{p_{d-1}(1-p_d)^2}{1-p_{d-1}}\frac{\e^{\frac{n^2}{2(n+1)}X_{l_1 k_1}^2}}{\sqrt{n+1}}.
\]
Suppose $\mT$ has depth $d(\mT)>\cL_c$. Then $l_1\ge \cL_c$ and from the H\"{o}lder continuity \eqref{eq:haar},  one gets $8|\beta_{l_1k_1}|\le \sqrt{\log{n}/n}$, where $\cL_c$ is as in \eqref{cutoff}. Then, 
conditionally on the event \eqref{event}, 
\begin{equation}\label{eq:Xlk}
 |X_{l_1k_1}|\le \frac{1}{\sqrt{n}}\left[\frac18\sqrt{\log{n}} + \sqrt{2\log{n} + \log{4}} \right] 
 \end{equation}
and thereby  $2X_{l_1k_1}^2\le 5\log{n}/n$. Recall that, under the  GW-prior, the split probability is $p_d=\Gamma^{-d}$.
As $\Gamma>2$,  one has
$p_d < 1/2$  and so, for any $d> \cL_c$,
\begin{equation*}\label{W_ratio}
 \frac{W_X(\cT)}{W_X(\cT^-)} \le  2\,p_{d-1}\exp\left(\frac{5\,n\log n}{4(n+1)}-\frac{1}{2}\log(1+n)\right)<2n^{3/4}p_{d-1}.
 \end{equation*}
 Going back to the ratio \eqref{Wratio}, we now  bound, with $a(n,d)\eqqcolon 2n^{3/4}p_{d-1}$,
 $$
 \frac{\Pi[\bt_d \given X]}{a(n,d)}\leq  \frac{ \sum_{\cT\in \bt_d} W_X(\cT^-)}{ \sum_{\cT\in \bT} W_X(\cT)}\leq 
  \frac{ \sum_{\cT\in \bt_d^-} 2^{d(\mT^-)}W_X(\cT)}{ \sum_{\cT\in \bT} W_X(\cT)}
%  a(n,d) 2^{d(\mT^-)}\frac{ \sum_{\cT\in \bt_d^-} 2^{d(\mT^-)}W_X(\cT)}{ \sum_{\cT\in \bt_d^-} W_X(\cT)}\leq 
 \le 2^d,
 $$
where $\bT_d^-$ is the image of $\bT_d$ under the map $\cT\to\cT^-$, and using that   at most $2^{d(\cT^-)}$ trees are mapped to the same $\cT^-$.  
%set of all possible trees $\mT^-$ that correspond to some $\mT\in\bT_d$.
Using this bound one deduces that, on the event $\cA$, with $L=L_{max}=\log_2{n}$, 
\begin{align*} \label{toodeep}
\Pi[ d(\cT) > \cL_c \given X]&= \sum_{d=\cL_c+1}^{L} \Pi[\bt_d \given X] \le 4\,n^{3/4}\sum_{d=\cL_c+1}^{L} 2^{d-1} p_{d-1}\\
&<4\,n^{3/4}\,L\exp\left[-\cL_c\log(\Gamma/2)\right].
\end{align*}
As $\cL_c\asymp(\log{n})/(1+2\al),$ the right hand side goes to zero as soon as, e.g. $\log(\Gamma/2)>7(1+2\al)/8$ that is, for $\al\le 1$, $\Gamma>2e^3$. \qedhere

\end{proof}

\subsection{Posterior Probability of Missing Signal}\label{subsec:signal}
The next step is showing that the posterior probability of missing a node with large enough signal vanishes.
\begin{lemma}\label{lemma:sig}
Let us denote, for $A>0$ to be chosen suitably large,
\begin{equation}\label{signalset}
S(f_0;A) = \left\{ (l,k): \  |\beta_{lk}^0| \ge A\frac{\log{n}}{\sqrt{n}} \right\}.
\end{equation}
%Let $(l_S,k_S)$ be a  signal node in the sense that
%\begin{equation}\label{condlog}
%l_S\le \cL_c\quad\text{and}\quad
 %|\beta^0_{l_S\,k_S}|\ge A\log{n}/\sqrt{n}.
%\end{equation}
%for some $A>0$. 
Under the assumptions of Theorem \ref{thm-one}, on the event $\cA$ from 
\eqref{event},
\begin{equation}\label{eq:sig}
%\Pi\left[(l_S,k_S)\notin\mT_{int}\C X\right]\rightarrow 0 \qquad (n\to\infty).
%\Pi\left[S(f_0;A)^c \C X\right]\rightarrow 0 \qquad (n\to\infty).
{\Pi\left[\left\{\cT:\, S(f_0;A)\nsubseteq \cT\right\} \C X\right]\rightarrow 0 \qquad (n\to\infty).}
\end{equation}
\end{lemma}
\begin{proof}
As before, we present the proof with the GW prior from  Section \ref{sec:bc1} and for the independent prior with $\Sigma_\mT=I_K$, referring to Section \ref{sec:thm-one} for the $g$-prior.
%Let us evaluate the posterior probability that the tree has no cut at $(l_S,k_S)$. 
Let us first  consider a given node $(l_S,k_S)\in S(f_0;A)$, for $A$ to be specified below, and note that the H\"older condition on $f_0$ implies $l_S\le \cL_c$ (for $n$ large enough). 
Let $\bt_{\setminus (l_S,k_S)}=\{\mT\in\bT:(l_S,k_S)\notin \mT_{int}\}$  denote the set of trees that miss the signal node in the sense that they {\em do not  have a cut} at $(l_S,k_S)$. 
For any such  tree $\cT\in \bt_{\setminus (l_S,k_S)}$ we then denote by  $\cT^+$ the smallest full binary tree (in terms of the number of nodes) that contains $\cT$ and that splits on $(l_S,k_S)$.
 Such a tree can be constructed from $\mT\in\bt_{\setminus (l_S,k_S)}$  as  follows.
 Denote by $(l_0,k_0)\in\mT_{ext}\cap[(0,0)\leftrightarrow (l_S,k_S)]$ the external node of $\cT$ which is {\em closest}  to $(l_S,k_S)$ on the route 
  from the root  to $(l_S,k_S)$ in a flat tree (denoted by $[(0,0)\leftrightarrow (l_S,k_S)]$). 
 %This route will be denoted with $[(l_0,k_0)\leftrightarrow (l_S,k_S)]$.
Next,  denote by $\mT^+$  the extended tree obtained from $\mT$ by sequentially splitting all $(l,k)\in[(l_0,k_0)\leftrightarrow (l_S,k_S)]$. Similarly as for $\mT\to \mT^-$ above,
%in the proof of Lemma \ref{lemma:dim}, 
the map $\mT\rightarrow\mT^+$ is not injective and we denote by $\bT_{(l_S,k_S)}$ the set of all extended trees $\mT^+$ obtained from some $\mT\in\bT_{\backslash (l_S,k_S)}$.
Now, the posterior probability  $\Pi\left[ \bt_{\setminus (l_S,k_S)} \given X\right]$ of missing the signal node $(l_S,k_S)$ equals
\begin{align}\label{ratio_Wplus}
& \frac{ \sum_{\cT\in \bt_{\setminus (l_S,k_S)}} W_X(\cT)}{ \sum_{\cT\in \bT} W_X(\cT)}
 \le \frac{\sum_{\cT\in \bt_{\setminus (l_S,k_S)}} \frac{W_X(\cT)}{W_X(\cT^+)}W_X(\cT^+)}{\sum_{\cT\in \bt_{(l_S,k_S)}}  W_X(\cT)}.
\end{align} 
Let us denote by $\mT^{(j)}$ for $j=-1,\dots,s$ the sequence of nested trees obtained by extending one branch of $\mT$ towards $(l_S,k_S)$ by splitting the nodes 
$[(l_0,k_0)\leftrightarrow (l_S,k_S)]$, where $\mT^+=\mT^{(s)}$ and $\mT=\mT^{(-1)}$.  Then
%\footnote{It is possible to remove all the nodes all at once. One would split the nodes $(l,k)\in[(l_0,k_0)\leftrightarrow (l_S,k_S)]$ {\sl as well as} their siblings so that the prior covariance matrix is block diagonal.} 
\begin{equation}\label{eq:ratio_W2}
\frac{W_X(\mT)}{W_X(\mT^+)}=\frac{\Pi_{\bT}(\mT)}{\Pi_{\bT}(\mT^+)}
\prod_{j=0}^s\frac{ N_X(\mT^{(j-1)}) }{ N_X(\mT^{(j)}) }.
%=\frac{\Pi_{\bT}(\mT)}{\Pi_{\bT}(\mT^+)}\frac{N_X(\mT^{cut}_{\wt l})}{N_X(\mT^+)}
%\prod_{l=\wt l}^{d-1} \frac{N_X(\mT^{cut}_{l+1})}{N_X(\mT^{cut}_{l})}
\end{equation}
Under the GW process prior with  $p_l=\Gamma^{-l}$ for some $\Gamma>2$, the ratio of prior tree probabilities in the last expression satisfies 
%bounded by the worst-case scenario, where $(l_0,k_0)$ is the root node, to obtain
%\begin{equation} \label{eq:prior_ratio_plus}
%\frac{\Pi_\bT(\cT)}{\Pi_\bT(\cT^+)}=\frac{1-p_{l_0}}{1-p_{l_S+1}}\left( \prod_{l=l_0}^{l_S}\frac{1}{p_l}\right)\left(\prod_{l=l_0+1}^{l_S+1}\frac{1}{1-p_l}\right),
 \begin{equation} \label{eq:prior_ratio_plus}
\frac{\Pi_\bT(\mT)}{\Pi_\bT(\mT^+)}=\frac{1-p_{l_0}}{p_{l_0}}\times \left(\prod_{l=l_{0}+1}^{l_S}\frac{1}{p_l(1-p_l)}\right)\times \frac{1}{(1-p_{l_S+1})^2}.
 \end{equation}
The first term is due to the fact that $\mT^+$ splits the node $(l_0,k_0)$ while $\mT$ does not. The second term in the denominator is the extra   prior probability  of $\mT^+$  over $\mT$ that is due to the branch reaching out to  $(l_S,k_S)$. Along this branch (note that this is the smallest possible branch), one splits {\em only} one daughter node for each layer $l$ (thereby the term $p_l$) and not the other (thereby the term $1-p_l$). The third term above is due to the fact that the two daughters of $(l_S,k_S)$ are not split.
The quantity \eqref{eq:prior_ratio_plus}  is bounded by $2^{l_S-l_0+2}\Gamma^{(l_0+l_S)(l_S-l_0+1)/2}<4\Gamma^{2l_S^2}$.

%Assuming  $p_l \asymp \Gamma^{-l}$ for a large enough constant $\Gamma>2$, this can be upper bounded with a constant multiple of $\Gamma^{l_S+1}$.

Assuming  $\Sigma_\mT=I_{K}$,  we can write for any $\cT$ in $\bt_{\setminus (l_S,k_S)}$
\begin{equation}\label{eq:bound_ratio}
\frac{W_X(\cT)}{W_X(\cT^+)} = \frac{\Pi_\bT(\cT)}{\Pi_\bT(\cT^+)} 
\prod_{(l,k)\in \cT^+\setminus \cT}\frac{\sqrt{n+1}}{\e^{\frac{n^2}{2(n+1)}X_{lk}^2}}.
%= \frac{p_{d-1}(1-p_d)}{1-p_{d-1}}\frac{e^{\frac{n+1}{2}(Y_{l_1k_1}^2+Y_{l_1(k_1+1)}^2)}}{n+1}
\end{equation}
Using the definition of the model and the inequality $2ab\ge -a^2/2-2b^2$ for $a,b\in\RR$, we obtain  $X_{l_Sk_S}^2\ge (\beta_{l_Sk_S}^0)^2/2-\veps_{l_Sk_S}^2/n$. On the event $\cA$, one gets 
$$
\exp\left\{-\frac{n^2}{2(n+1)}X_{l_Sk_S}^2\right\}\le \exp\left\{-\frac{n^2(\beta_{l_Sk_S}^0)^2}{4(n+1)} + \frac{n(\log 2)(\log_2n +1)}{n+1}\right\}.
$$ 
The term in \eqref{eq:bound_ratio} can be thus bounded, for any $\cT\in \bt_{\setminus (l_S,k_S)}$, by
\[  \frac{W_X(\cT)}{W_X(\cT^+)} 
\le C \Gamma^{2l_S^2}\, \exp\left\{\frac{3(l_S-l_0+1)(\log_2n+1)}{2}-\frac{n A^2\log^2{n} }{4(n+1)}\right\}\eqqcolon b(n,l_S).
 \]
% where we bounded the exponential terms in \eqref{eq:bound_ratio}
% from below by $1$. 
We now continue to bound the ratio \eqref{ratio_Wplus}. For each given $\mT^+$, there are {\em at most} $l_S$ trees $\wt\mT\in\bT_{\backslash(l_S,k_S)}$ which have the same extended tree $\wt\mT^+=\mT^+$. This is because $\mT^+$ is obtained by extending one given branch by adding no more than $l_S$ nodes.
Using this fact, \eqref{ratio_Wplus}, and the definition of $b(n,l_S)$ on the last display, %we can write 
$$
\frac{\Pi\left[ \bt_{\setminus (l_S,k_S)} \given X\right]}{b(n,l_S)} \leq  \frac{\sum_{\cT\in \bt_{\setminus (l_S,k_S)}}W_X(\cT^+)}{
\sum_{\cT\in   \bt_{(l_S,k_S)}}  W_X(\cT)}
\leq \,l_S\frac{\sum_{\cT\in  \bt_{(l_S,k_S)}}W_X(\cT)}{\sum_{\cT\in \bt_{(l_S,k_S)}}  W_X(\cT)}.
$$
By choosing $A=A(\Gamma)>0$ large enough, this leads to 
\[  
\Pi\left[\bt_{\setminus (l_S,k_S)} \C X\right]
%=\Pi\left[ \bt_{\setminus (l_S,k_S)} \given X\right] 
\leqa \e^{(3/2+3\log\Gamma)(\log_2n+1)^2-\frac{A^2}{8}\log^2{n}}\leqa \e^{-\frac{A^2}{16}\log^2 n}. \]
%From these considerations, it also follows   the posterior probability of missing {\em any} signal node $(l_S,k_S)$ goes to zero.
%Denote with $S(f_0,A)=\{(l_S,k_S): (l_S,k_S)\,\,\text{satisfies \eqref{condlog}}\}$ the set of all signal nodes. Then, on the event $\cA$ and for $A>0$ large enough, we have
Then the result follows as, on the event $\cA$,
\[ 
\sum_{ (l_S,k_S)\in S(f_0,A)} \Pi\left[ \bt_{\setminus (l_S,k_S)} \given X\right]
\leqa 2^{\cL_c+1}\e^{-\frac{A^2}{16}\log^2{n}}\leqa \e^{-\frac{A^2}{32}\log^2{n}}\rightarrow0.\qedhere
\]

\end{proof}

\subsection{Posterior Concentration Around Signals}\label{subsec:concentration}
Let us now show that the posterior does not distort large signals too much.
\begin{lemma}\label{lemma:signal}
Let us denote, for $\cL_c$ as in \eqref{cutoff} and $S(f_0;A)$ as in \eqref{signalset},
%\begin{equation}\label{signalset}
%S(f_0;A) = \left\{ (l,k): l\leq \mL_c \quad\text{and}\quad  |\beta_{lk}^0| \ge A\frac{\log{n}}{\sqrt{n}} \right\}
%\end{equation}
%and
%define 
\begin{equation} \label{Tdr} 
\mathsf{T}=\{\mT: d(\mT)\leq \mathcal{L}_c, \ S(f_0;A)\subset \mT\}.
\end{equation}
 Then, on the event $\mathcal{A}$,  for some $C'>0$, uniformly over $\mT\in\mathsf{T}$, %the conditional posterior satisfies 
\begin{equation}\label{post_signal}
\int\max_{(l,k)\in\cT_{int}'}  |\beta_{lk}-\beta_{lk}^0| d\Pi[\b_\mT\given \X_\mT]<C'\sqrt{\frac{\log n}{n}},
\end{equation}
 with $\X_\mT=(X_{lk}:(l,k)\in\mT_{int}')'$  the ordered vector of active responses.
\end{lemma}
\begin{proof}
For a given  tree $\mT$ with $K=|\mT_{ext}|$ leaves,  we denote by $\b_{\mT}=(\beta_{lk}:(l,k)\in\mT_{int}')'$ the vector of wavelet (internal node) coefficients, with $\X_\mT$ the corresponding responses and with $\bm\varepsilon_\mT$ the white noise disturbances. 
It follows from \eqref{eq:posterior_beta} that, given $\X_\mT$ (so for fixed $\veps_{lk}$) and $\cT$, the  vector $\b_{\mT}$ has a Gaussian distribution
$
\b_{\mT}\C\X_{\mT}\sim\mathcal{N}(\bm{\mu}_{\mT},\wt\Sigma_{\mT}),
$
where  $\wt\Sigma_{\mT}=(nI_{K}+\Sigma_{\mT}^{-1})^{-1}$ and 
$
\bm{\mu}_{\mT}=n\wt\Sigma_{\mT}\left(\b^0_\mT+\frac{1}{\sqrt{n}}\bm\varepsilon_{\mT}\right).
$
Next, using Lemma \ref{lemma:exp_max}, we have
\begin{align}\label{eq:exp}
\E \left[ \|\b_\mT-\b^0_\mT\|_\infty \given \X_\mT\right]&\leq \|\bm{\mu}_{\mT}-\b^0_\mT\|_\infty+
 \sqrt{2\,\bar{\sigma}^2\log K}+2\sqrt{2\pi\bar{\sigma}^2},
% &\leq \|\b^0_\mT\|_\infty+\sqrt{|\mT_{ext}|}\lambda_{max}(\wt\Sigma_{\mT})\|\X_\mT\|_\infty+ \sqrt{2\,\bar{\sigma}^2\log |\mT_{ext}|}+2\sqrt{2\pi\bar{\sigma}^2},
 \end{align}
 where $\bar\sigma^2=\max\mathrm{diag}(\wt\Sigma_{\mT})$. Focusing on the first term, we can write
 \begin{align} \label{tech2}
 \|\bm{\mu}_{\mT}-\b^0_\mT\|_\infty&\leq \sqrt{n}\|\wt\Sigma_{\mT}\bm\veps_\mT\|_\infty+\|(n\wt\Sigma_{\mT}-I_{K})\b_\mT^0\|_\infty.
\end{align}
Using the fact $(I+B)^{-1}=I-(I+B^{-1})^{-1}$, we obtain $n\wt\Sigma_{\mT}-I_{K}=-(I_{K}+n\Sigma_{\mT})^{-1}$. 
From now on, we focus on the simpler case $\Sigma_\mT=I_K$ and refer to Section \ref{suppl:subsec:signal} (Supplement) for the proof for the $g$-prior.
With $\Sigma_\mT=I_K$  we can  write  
$
\|(n\wt\Sigma_{\mT}-I_{K})\b_\mT^0\|_\infty= \frac{\|\b_\mT^0\|_\infty}{1+n}
<C/n.
$ Using the fact that   $\|\bm\varepsilon_\mT\|_\infty\lesssim \sqrt{\log n}$ on the event $\mathcal{A}$, we obtain
$\sqrt{n}\|\wt\Sigma_{\mT}\bm\veps_\mT\|_\infty\lesssim \sqrt{\frac{\log n}{n}}.$ 
%\begin{align}
 % \|\bm{\mu}_{\mT}-\b^0_\mT\|_\infty&\leq \sqrt{n|\mT_{ext}|}\lambda_{max}(\wt\Sigma_{\mT})\|\bm\veps_\mT\|_\infty+\frac{\|\b_\mT^0\|_\infty\sqrt{K}}{(1+n\lambda_{min}(\Sigma_{\mT}))}\\
 %&\leq \sqrt{\frac{C\,|\mT_{ext}|\log n}{n}}+\frac{M\sqrt{|\mT_{ext}|} \lambda_{max}( A_\mT' A_\mT)}{n\,g_n}\lesssim \sqrt{\frac{|\mT_{ext}|\log n}{n}},
 %\end{align}
%where we used the fact that
%  $\|\bm\varepsilon_\mT\|_\infty\lesssim \sqrt{\log n}$ on the event $\mathcal{A}$. 
The sum of the remaining two terms in \eqref{eq:exp}  can be bounded by a multiple of $\sqrt{\log n/n}$ by noting that 
$
\bar\sigma^2=1/(n+1).
%\leq \sqrt{|\mT_{ext}|}\lambda_{max}(\wt\Sigma_{\mT})\leq \sqrt{|\mT_{ext}|}/n.
$
The statement \eqref{post_signal} then follows from \eqref{eq:exp}. \qedhere

\end{proof}

\subsection{Supremum-norm Convergence Rate}\label{sec:proof_rate}
%Now we have all the pieces needed to complete the proof of Theorem \ref{thm-one}.
Let us write $f_0=f_0^{\cL_c} + f_0^{\setminus \cL_c}$, where $f_0^{\cL_c}$ is the $L^2$--projection of $f_0$ onto the first $\cL_c$ layers of wavelet coefficients. Under the H\"older condition the equality holds  pointwise and $\|f_0^{\setminus \cL_c}\|_\infty\leq \sum_{l>\cL_c} 2^{l/2} 2^{-l(1/2+\al)}\lesssim  (\log n/n)^{\alpha/(2\alpha+1)}$.

The following inequality bounds the supremum norm by the $\ell_\infty$--norm,
\begin{align} 
 \|f -f_0\|_\infty &\le  \sum_{l\ge -1}  \max_{0\le k<2^l} |\beta_{lk}-\beta_{lk}^0|\cdot\Big\|\sum_{0\le k<2^{-l}} |\psi_{lk}| \Big\|_\infty \notag \\
 & \le |\psg f-f_0 , \vphi\psd| + \sum_{l\ge 0} 2^{l/2}  \max_{0\le k<2^l} |\beta_{lk}-\beta_{lk}^0|=\ell_\infty(f,f_0).\label{sup_bound}
% &  =  \sum_{l\ge -1} 2^{0\vee l/2}  \max_{0\le k<2^l} |\beta_{lk}-\beta_{lk}^0| =: \ell_\infty(f,f_0).\label{sup_bound}
\end{align} 
\iffalse
The following bounds the supremum norm by a maximum-type norm on wavelet coefficients
\begin{align}\label{sup_bound}
 \|f_{\mT,\b} -f_0\|_\infty 
 &\le  \sum_{l\ge -1}  \max_{0\le k<2^l} |f_{lk}-\beta_{lk}^0|\sum_{0\le k<2^{-l}} \|\psi_{lk}\|_\infty  \\
& \le |\psg f_{\mT,\b}-f_0,\vphi\psd| + \sum_{l\ge 0} 2^{l/2}  \max_{0\le k<2^l} |\beta_{lk}-\beta_{lk}^0| =: \ell_\infty(f_{\mT,\b},f_0).
% &  =  \sum_{l\ge -1} 2^{0\vee l/2}  \max_{0\le k<2^l} |f_{lk}-\beta_{lk}^0| =: \ell_\infty(f,f_0).
\end{align} 
\fi
We use the notation $S(f_0;A), \mathsf{T}$ as in \eqref{signalset} and \eqref{Tdr} 
and  
\begin{equation}\label{eq:tree_events}
%\mathsf{T}=\{\mT: d(\mT)\leq \mathcal{L}_c, S(f_0;A)\subset \mT\}
%\quad\text{ and}\quad  
\mE=\{f_{\mT,\b}:\mT\in \mathsf{T}\}.
\end{equation}
Using the definition of the event $\cA$ from \eqref{event}, one can write
%One can bound
%\begin{align*}
%E_{f_0}\Pi[f_{\mT,\b}:\ \|f_{\mT,\b}-f_0\|_\infty > \veps_n \given X]  & \le \P_{f_0}[\cA^c] 
% + E_{f_0}\left\{\Pi[f_{\mT,\b}:\ \|f_{\mT,\b}-f_0\|_\infty > \veps_n \given X]\1_{\cA}\right\}.
 %\end{align*}
% Using  Lemma \ref{lemma:dim} and Lemma \ref{lemma:sig}, we then obtain
\begin{align}
& \quad E_{f_0}\Pi[f_{\mT,\b}:\ \|f_{\mT,\b}-f_0\|_\infty > \veps_n \given X] \le P_{f_0}[\cA^c] + E_{f_0}\Pi[\cE^c \given X]\notag \\
 & \qquad \qquad
 + E_{f_0}\left\{\Pi[f_{\mT,\b}\in\cE:\ \|f_{\mT,\b}-f_0\|_\infty > \veps_n \given X] \1_{\cA}\right\}.\label{eq:conditioning}
%&\qquad\quad \le o(1) + 
%\veps_n^{-1}E_{f_0}\left\{\int_{\cE} \|f_{\mT,\b}-f_0\|_\infty d\Pi[f_{\mT,\b}\given X] \1_{\cA}\right\}.
%& \le E_{f_0}\Pi[\cE^c\given X] + \veps_n^{-1}
%\sum_{l\le \cL_c} 2^{l/2} E_{f_0}\int_{\cE}\max_{0\le k<2^l} |\beta_{lk}-\beta_{lk}^0|d\Pi[f_{\mT,\b}\given X] + \veps_n^{-1}\|f_0^{\setminus \cL_c}\|_\infty.
\end{align}
By Markov's inequality  and the previous bound \eqref{sup_bound}, % the last term can be further bounded by
\begin{align*} 
\lefteqn{\Pi[f_{\mT,\b}\in\cE:\ \|f_{\mT,\b}-f_0\|_\infty > \veps_n \given X]\1_{\cA} \le  
\veps_n^{-1}
\int_{\cE} \|f_{\mT,\b}-f_0\|_\infty d\Pi[f_{\mT,\b}\given X] \1_{\cA}}&&\\
& \le \veps_n^{-1}
\sum_{l\le \cL_c} 2^{l/2} \left\{\int_{\cE}\max_{0\le k<2^l} |\beta_{lk}-\beta^0_{lk}|d\Pi[f_{\mT,\b}\given X] \1_{\cA}\right\} 
+ \veps_n^{-1}\|f_0^{\setminus \cL_c}\|_\infty.
\end{align*}
With $\mathsf{T}$ as in \eqref{Tdr},  the integral in the last display can be written, for $l\le \cL_c$, 
\begin{align*}
\lefteqn{\int_{\cE}\max_{0\le k<2^l} |\beta_{lk}-\beta_{lk}^0|d\Pi[f_{\mT,\b}\given X]  = 
\sum_{\cT\in \mathsf{T}} \pi[\cT\given X] 
\int\max_{0\le k<2^l} |\beta_{lk}-\beta_{lk}^0|d\Pi[\b_\mT\C\X_\mT]  } && \\
 & =\sum_{\cT\in \mathsf{T}} \pi[\cT\given X] \int\max\left(\max_{0\le k<2^l,\, (l,k)\notin\cT_{int}'} |\beta_{lk}^0| , \max_{0\le k<2^l,\, (l,k)\in\cT_{int}'}  |\beta_{lk}-\beta_{lk}^0| \right)d\Pi[\b_\mT\given \X_\mT]  \\
 & \le \min\left(\max_{0\leq k<2^l}|\beta^0_{lk}|, A\frac{\log{n}}{\sqrt{n}}\right)
 + \sum_{\cT\in \mathsf{T}} \pi[\cT\given X] \int \max_{0\le k<2^l,\, (l,k)\in\cT_{int}'}  |\beta_{lk}-\beta_{lk}^0| d\Pi[\b_\mT\given\X_\mT],
\end{align*}
where we have used that on the set $\cE$, selected trees cannot miss any true signal larger than $A\log{n}/\sqrt{n}$. This means that any node $(l,k)$ that is {\em not} in a selected tree must satisfy $|\beta_{lk}^0|\le A\log{n}/\sqrt{n}$.

\iffalse
\begin{align*}
\lefteqn{
 \|f_0^{\setminus \cL_c}\|_\infty+
\sum\limits_{\substack {\mT\in\mathsf{T} }} \pi[\cT\given X] \sum_{l\leq \cL_c}2^{l/2}
\int\max_{\substack {0\leq k< 2^l}} |\beta_{lk}-\beta_{lk}^0|d\Pi[\{\beta_{lk}\}\given X_\mT,\mT] }  && \\
 & =\|f_0^{\setminus \cL_c}\|_\infty+\sum_{\cT\in \mathsf{T}} \pi[\cT\given X]\sum_{l\leq \cL_c}2^{l/2} \int\max\left(\max_{\substack {0\leq k< 2^l\\(l,k)\notin\cT_{int}}} |\beta_{lk}^0| , \max_{(l,k)\in\cT}  |\beta_{lk}-\beta_{lk}^0| \right)d\Pi[\b_\mT\given X_\mT,\cT] \\
 & \le \|f_0^{\setminus \cL_c}\|_\infty+ A\frac{\log{n}}{\sqrt{n}}
 + \sum_{\cT\in \mathsf{T}} \pi[\cT\given X] \sum_{l\leq \cL_c}2^{l/2}\int\max_{(l,k)\in\cT}  |\beta_{lk}-\beta_{lk}^0| d\Pi[\b_\mT\given X_\mT,\cT],
\end{align*}
where we have used that on the event $\cE$, selected trees cannot miss any true signal larger than $A\log{n}/\sqrt{n}$, so any node $(l,k)$ that is not in a selected tree must satisfy $|\beta_{lk}^0|\le A\log{n}/\sqrt{n}$.
\fi

Let $L^*=L^*(\al)$ be the integer closest to the solution of the equation in $L$ given by  $M2^{-L(\al+1/2)}=A\log{n}/\sqrt{n}$. Then, using that $f_0\in \cH(\al,M)$,
\begin{align}
\lefteqn{ \sum_{l\le \cL_c} 2^{\frac{l}{2}}{\min\left(\max_{0\le k<2^l} |\beta_{lk}^0|, A\frac{\log{n}}{\sqrt{n}} \right)}
 \le   \sum_{l\le L^*} 2^{\frac{l}{2}} A\frac{\log{n}}{\sqrt{n}} +  
 \sum_{L^*<l\le \cL_c} 2^{\frac{l}{2}} M2^{-l(\frac12+\al)}}\qquad\qquad
\notag \\
&& \le C 2^{L^*/2} A\frac{\log{n}}{\sqrt{n}}+C 2^{-L^*\al} \le \wt C 2^{-L^*\al}\le c \left(n^{-1}\log^2 n \right)^{\frac{\alpha}{2\alpha+1}}.\label{Lstar_ineq}
\end{align}
Using  $P_{f_0}[\cA^c]+E_{f_0}\Pi[\mE^c\C X]=o(1)$ and
 Lemma \ref{lemma:signal},
one obtains 
\begin{align*}
\lefteqn{E_{f_0} \Pi[f_{\mT,\b}:\ \|f_{\mT,\b}-f_0\|_\infty > \veps_n \given X] \le o(1) +}&&
 \\
 & \qquad \veps_n^{-1}
\sum_{l\le \cL_c} 2^{l/2} \left[ \min\left(\max_{0\le k<2^l} |\beta_{lk}^0|, A\frac{\log{n}}{\sqrt{n}} \right) +  C'\sqrt{\frac{\log{n}}{n}} \right] + \veps_n^{-1}\|f_0^{\setminus \cL_c}\|_\infty &\\ 
& \le o(1) + \veps_n^{-1}\left[c \left(\frac{\log^2 n}{n}\right)^{\frac{\alpha}{2\alpha+1}} + 2\,C' \sqrt{ \frac{2^{\cL_c}\log{n}}{n} } \right] + 
\veps_n^{-1}\|f_0^{\setminus \cL_c}\|_\infty&\\
& \le o(1) + \veps_n^{-1}\left[c(\log{n})^{\alpha/(2\alpha+1)}+2\,C'\right] \left(\frac{\log{n}}{n}\right)^{\frac{\al}{2\al+1}} + \veps_n^{-1}\|f_0^{\setminus \cL_c}\|_\infty& 
\end{align*} 
 for some $C'>0$.
Choosing
$\veps_n = M_n  \left((\log^2{n})/n\right)^{\frac{\al}{2\al+1}}$, the right hand side goes to zero for any arbitrarily slowly increasing sequence $M_n\rightarrow\infty$. \qed

%\sbl{[Bibliography: arxiv ref or update for items 18, 39, 46, 48, 56, 65; spelling/accents for 43, 52 (maybe I missed some; it seems some were not retrieved from mathscinet perhaps)  ]}

%\bibliographystyle{imsart-number}
%\bibliography{References}
\clearpage

\begin{frontmatter}

\title{Supplement to ``Uncertainty Quantification for Bayesian CART"}
\runtitle{Supplement to ``Uncertainty Quantification for Bayesian CART"}

\begin{aug}
\author{\fnms{Isma\"el} \snm{Castillo}\thanksref{m1}\ead[label=e1]{ismael.castillo@upmc.fr}}
\and
\author{\fnms{Veronika} \snm{Ro\v{c}kov\'{a} }\thanksref{m2}\ead[label=e2]{veronika.rockova@chicagobooth.edu}}

\affiliation{Sorbonne Universit\'e \thanksmark{m1}}
\affiliation{University of Chicago \thanksmark{m2}}

\address{ 
  Sorbonne Universit\'e \& 
Institut 
Universitaire de France\\
 Laboratoire de Probabilit\'es, Statistique 
 et Mod\'elisation, LPSM,\\ 4, Place Jussieu, 75005 Paris cedex 05, France\\ 
  \printead{e1}
 }
 
 \address{University of Chicago, Booth School of Business\\ 5807 S. Woodlawn Avenue\\ Chicago, IL, 60637, USA \\
 \printead{e2}}

\end{aug}

\begin{abstract}
This supplementary file contains additional material, including results for nonparametric regression, a simulation study, an adaptive nonparametric Bernstein--von Mises theorem,  and  details on tensor--multivariate versions of the considered prior distributions. It also contains all remaining proofs for the results stated in the main paper. 
\end{abstract}

\begin{keyword}[class=AMS]
\kwd[Primary ]{62G20, 62G15}
\end{keyword}

\begin{keyword}
\kwd{Bayesian CART}
\kwd{Posterior Concentration} 
\kwd{Nonparametric Bernstein--von Mises theorem}
\kwd{Recursive Partitioning}
\kwd{Regression Trees}
  \end{keyword}
 
\end{frontmatter}

\setcounter{tocdepth}{2}
\tableofcontents

\section{Additional results}

%\section{Appendix} \label{sec:app}

\subsection{Nonparametric regression: $\|\cdot\|_\infty$--rate and bands}\label{sec:npreg}
%Our results for the white noise model do not automatically carry over to non-parametric regression. The  asymptotic equivalence between these two models (Brown and Low reference) holds under a uniform smoothness assumption which is satisfied by, e.g., $\alpha$--H\"olderian smooth functions with $\alpha > 0.5$. From a sequence of optimal procedures in one problem, it should be possible to prescribe a construction of an asymptotically equivalent sequence in the other. This recipe is particularly convenient for linear estimators. For Bayesian methods, however, it is not obvious that the knowledge of a (wavelet shrinkage/non-linear) minimax procedure in one problem automatically implies the optimality in the other. We now formulate a variant of Theorem \ref{thm-one} for a fixed design (equidistant) non-parametric regression assuming $\alpha>1/2$.

Assume outcomes $Y=(Y_1,\dots,Y_n)'$ arising from 
\begin{equation}\label{model}
Y_i=f_0(t_i)+\varepsilon_i,\quad \varepsilon_i\iid\mathcal N(0,1),\quad i=1,\dots, n=2^{L_{max}+1}
\end{equation}
where $f_0$ is an unknown regression function and $\{t_i\in[0,1]:1\leq i\leq n\}$ are fixed design points. For simplicity we consider a regular grid, i.e. $t_i=i/n$ for $1\leq i\leq n$ and assume $n$ is a power of $2$.  Irregularly spaced design points could also be considered with more technical proofs.  Below we show that many results for Bayesian CART posteriors obtained in the white noise model in the main paper carry over to the regression model \eqref{model}. Although the white noise and regression models can be shown to be `asymptotically equivalent' in the Le Cam sense under smoothness assumptions, this does not enable to transfer results for posterior distributions from one model to the other.

%We consider two approaches to derive supremum norm contraction rates in this setting. The approaches differ by the object they model through the prior distribution. In Approach 1, the prior is put on empirical wavelet coefficients. In Approach 2, the prior sits on the original (continuous) wavelet coefficients. 
To derive a supremum norm contraction rate in this setting, we follow an approach close in spirit to the practically used Haar wavelet transform. Below, we put a prior on the {\em empirical} wavelet coefficients of the regression function $f$. Indeed, those coefficients can be directly related to the values $f(t_i)$.

Let $X=(x_{ij})$ denote the $(n\times p)$ regression matrix of $p=2^{L_{max}}=n/2$ regressors constructed from Haar wavelets $\psi_{lk}$ up to the maximal resolution $\Lmax$, i.e., for $1\le i\le n$,
$$
x_{ij}=\begin{cases}
\psi_{-10}(t_i)=1& \quad\text{for}\quad j=1\\
\psi_{lk}(t_i)&\quad \text{for}\quad j=2^l+k+1.
\end{cases}
$$  
The columns have been ordered according to the index $2^l+k$ (from smallest to largest). Note that the matrix $X$ is orthogonal and $X'X=nI_{n}$. In the sequel we denote $F_0=(f_0(t_1),\dots, f_0(t_n))'$ the vector of realized values of the true regression function at the design points. Also, we set, for two functions $f,g$ defined on $[0,1]$,
\[  \|f-g\|_{\infty,n} = \max_{1\le i\le n} |f(t_i)-g(t_i)|.   \]

%{\em Approach 1 (prior on empirical wavelet coefficients).} 
For given indexes $l,k$, the empirical wavelet coefficient $b_{lk}^0$ of $f_0$ is
\begin{equation} \label{empwave}
b_{lk}^0=n^{-1}\sum_{i=1}^n f_0(t_i)\psi_{lk}(t_i).
\end{equation}
Let $\bb^0=(b_{lk}^0)$ denote the vector of ordered empirical coefficients. Since $X^{-1}=X'/n$, we see that $\bb^0=X^{-1}F_0$ or $F_0=X\bb^0$, so  model \ref{model} can be rewritten,  with $\veps=(\veps_1,\ldots,\veps_n)'$,
\begin{equation}   \label{mod:emp}
 Y = X\bb^0 + \veps.
\end{equation} 
Setting $Z=X^{-1}Y$ and $\eta=\sqrt{n}X^{-1}\veps$, we have
\begin{equation} \label{mod:wne}
 Z=\bb^0+\eta/\sqrt{n}.
\end{equation}
 Since $\eta/\rn\sim\cN(0,X'X/n)$ is a vector of iid $\cN(0,1)$ variables (as $X'X=nI$), the vector $Z$ follows a Gaussian sequence model truncated at the $n$th observation. On the other hand, note that the likelihood in model \eqref{mod:emp} is proportional to $\exp(-\|Y-X\bb^0\|^2/2)=\exp(-n\|Z-\bb^0\|^2/2)$. Therefore, the posterior distribution induced by putting a prior distribution on $\bb^0$ in model \eqref{mod:emp} is the same as the one obtained by choosing the same prior on $\bb^0$ in model \eqref{mod:wne}. Let us denote by $\Pi_b[\cdot\given Y]$ this posterior distribution on vectors $\bb$. Theorem \ref{thm_np_emp} below states that its convergence rate in terms of the maximum norm $\|f-f_0\|_{\infty,n}$ is {\em the same} as the rate obtained in the main paper. If one rather wishes to control $\|f-f_0\|_\infty$, it is possible to produce a procedure that yields a rate for the whole function $f$ by interpolation as follows. 
 
Let $\cI$ be the map that takes a vector of values $\vphi:=(f(t_i))$ of size $n$ and maps it to the piecewise-linear function $\cI(\vphi)$ on $[0,1]$ that linearly interpolates between the values $f(t_i)$ (for definiteness, assume  $\cI(\vphi)$ takes the constant value $f(t_1)$ on $[0,t_1]$). Further define, for $\chi$ the map $\chi(\bb)=\cI(X\bb)$,
 \begin{equation} \label{interpost}
\bar\Pi_Y = \Pi_b[\cdot\given Y]\circ \chi^{-1}.
 \end{equation}
In words, $\bar\Pi$ is simply the distribution on functions on $[0,1]$ induced as follows: sampling from the posterior  $\Pi_b[\cdot\given Y]$ induces a posterior on vectors $(f(t_i))$, from which one obtains a distribution on functions $f$'s by the linear interpolation $\cI$.

\begin{theorem} \label{thm_np_emp}
Let $\Pi_b$ denote the prior distribution on empirical wavelet coefficients in model \eqref{model}/\eqref{mod:emp} defined in the same way as the prior in Theorem \ref{thm-one}, with $\veps_n$ the rate as in that statement. Then, for $\Pi_b[\cdot\given Y]$ the corresponding posterior distribution, for any $\al\in(0,1]$, $M>0$ and any sequence $M_n\to\infty$, 
\[ \sup_{f_0\in \cH_M^\al} E_{f_0}\Pi_b[f:\ \|f-f_0\|_{\infty,n}>M_n\veps_n \given Y]\to 0.\]
and, for $\bar\Pi_Y$ the distribution defined in \eqref{interpost},
\[ \sup_{f_0\in \cH_M^\al} E_{f_0}\bar\Pi_Y[f:\ \|f-f_0\|_{\infty}>M_n\veps_n]\to 0.  \]
\end{theorem}

The proposed approach uses the relationship between function values and empirical wavelet coefficients given by the Haar transform. It is worth noticing that it naturally  yields a result on the canonical empirical max-loss $\|\cdot\|_{\infty,n}$ and avoids regularity conditions on the true $f_0$ (such as a minimal smoothness level $\al>1/2$). This approach could be extended to handle non-equally spaced designs and/or a smoother wavelet basis, under appropriate conditions on the matrix $X'X$. This is beyond the scope of this paper, but we refer to the work \cite{yoo} for results in this vein for spike-and-slab priors.

\begin{remark} \label{rem-inter}
Instead of a continuous linear interpolation $\cI$ as above, one may use instead a piecewise constant interpolation: defining $\cJ((f(t_i))$ to be the histogram that takes the value $f(t_i)$ on $[t_i,t_{i+1})$, Theorem \ref{thm_np_emp} holds for $\tilde\Pi_Y$ defined in a similar way as $\bar\Pi_Y$ but with $\cJ$ in place of $\cI$, as can easily been seen from the proof of Theorem \ref{thm_np_emp}.
\end{remark} 

\begin{remark} \label{appr2}
A seemingly different approach to estimating $f$ consists in putting a prior distribution directly on the function $f$ via putting a prior distribution on its (Haar--) wavelet coefficients. Note however that the induced prior on $f(t_i)$ is the same as the one above, since $f(t_i)=\sum_{l,k} \beta_{lk} \psi_{lk}(t_i)$ can be rewritten $F=X\be$ since by definition the prior does not put mass on $\beta_{lk}$ for $l> L_{max}$. One may also note that for a piecewise constant function over intervals $[t_i,t_{i+1})$, empirical Haar--wavelet coefficients (the $b_{lk}$'s) and Haar--wavelet coefficients (the $\be_{lk}$'s) coincide. This implies that the posterior distributions induced on the vector $(f(t_i))$ through both approaches coincide. Since posterior samples are piecewise constant on dyadic intervals, the posterior distribution on $f$'s obtained from using the prior on $\beta_{lk}$'s as above coincides with the posterior $\tilde{\Pi}_Y$ from Remark \ref{rem-inter}.  
\end{remark}

We now turn to the problem of construction of confidence bands for $f$ in the regression model \eqref{model}. We follow the  approach above and model the empirical wavelet coefficients $\bb^0$ in \eqref{mod:emp}--\eqref{mod:wne} via the tree prior $\Pi_b$ as in Theorem \ref{thm_np_emp}. Recall Definition \ref{def:mtree} of the median tree associated with a posterior over trees. Let $\cT_Y^{\sim}$ denote the median tree associated with the posterior distribution $\Pi_b[\cdot\given Y]$. 
Given the observed noisy empirical wavelet coefficients $Z_{lk}$ as in \eqref{mod:wne}, let us define an {\em empirical median tree} estimator $\tilde f_{T}$ over gridpoints $(t_i)$ as
\begin{equation} \label{bulkemp}
\tilde f_{T}(t_i) =  \sum_{(l,k)\in \cT_Y^{\sim}} Z_{lk}\psi_{lk}(t_i).
\end{equation}
Define, for some $v_n\rightarrow\infty$ to be chosen, 
\begin{equation}  \label{radiusproxre}
\tilde\sigma_n = \tilde\sigma_n(Y) = \max_{1\le i\le n} \sum_{l=0}^{\lmax} v_n \sqrt{\frac{\log{n}}{n}}
\sum_{k=0}^{2^l-1} \1_{(l,k)\in \cT_Y^{\sim}} |\psi_{lk}(t_i)|.
\end{equation}
A credible band with radius $\tilde\sigma_n(Y)$ as in \eqref{radiusproxre} and center $\tilde f_T$ as in \eqref{bulkemp} is
\begin{equation} \label{credibleemp}
\cc_n^e=\left\{f:\ %\|f-\ix\|_{\cM(w)} \le R_n/\rn,\
 \|f-\tilde f_T\|_{\infty,n} \le \tilde\sigma_n(Y) \right\}.
\end{equation}
We obtain in the next statement the analogue  of Theorem \ref{csthm} for regression. The confidence band is in terms of the natural empirical norm $\|\cdot\|_{\infty,n}$. We slightly update the definition of self-similiar functions by restricting to $f$'s in $\cH_M^\al$ in Definition \ref{def-ssi}, instead of $\cH(\al,M)$ defined from wavelet coefficients. This is for technical convenience because $\cH_M^\al$ is more natural for controlling empirical wavelet coefficients. We denote by $\cH_{SS}'(\alpha,M,\veps)$ the corresponding class (note that the classes are nearly the same; we could also have opted for taking $\cH_M^\al$ in Definition \ref{def-ssi}, which would have avoided the distinction).
\begin{theorem} \label{csthmemp}
Let $0<\al_1\le \al_2 \le 1,\, M\ge 1$ and $\veps>0$. Let $\Pi_b$ be a prior as in the statement of Theorem \ref{thm_np_emp}. Let $\tilde\sigma_n$ be as in \eqref{radiusproxre} with $v_n$ such that $(\log{n})^{1/2}=o(v_n)$ and  let $\tilde f_T$ denote the median tree estimator \eqref{bulkemp}.  Then for $\cc_n^e$ defined in \eqref{credibleemp}, uniformly over $\alpha\in[\alpha_1,\alpha_2],$ as $n\to\infty$,
\[  \inf_{f_0\in \cH_{SS}'(\alpha,M,\veps)} P_{f_0}(f_0\in \cc_n^e)  \to 1.\]
For every $\al\in[\al_1,\al_2]$ and uniformly over $f_0\in \cH_{SS}(\al,M,\veps)$, the  diameter $|\cc_n^e|_{\infty,n}=\sup_{f,g\in \cc_n^e}\|f-g\|_{\infty,n}$ and the credibility of the band verify, as $n\to \infty$,
\begin{align} 
 |\cc_n^e|_{\infty,n} & = O_{P_{f_0}}((n/\log n)^{-\al/(2\al+1)} v_n),\label{cred_diam}\\
 \Pi[\cc_n^e\given Y] & = 1 + o_{P_{f_0}}(1). 
\end{align} 
\end{theorem}
The confidence band $\cc_n^e$ is slightly conservative in the sense that its coverage and credibility go to $1$ as $n\to\infty$. By intersecting this band with an appropriate ball, using a nonparametric BvM theorem, one can build a band with desired prescribed coverage $1-\ga$, $\ga>0$, as demonstrated in  Section  \ref{sec-bvm}. The latter `intersection'-band is actually the one we implement in simulations in the next section and illustrated in Figure \ref{fig:band}.

\begin{figure}[!] \label{fig:band}
\includegraphics[width=7cm]{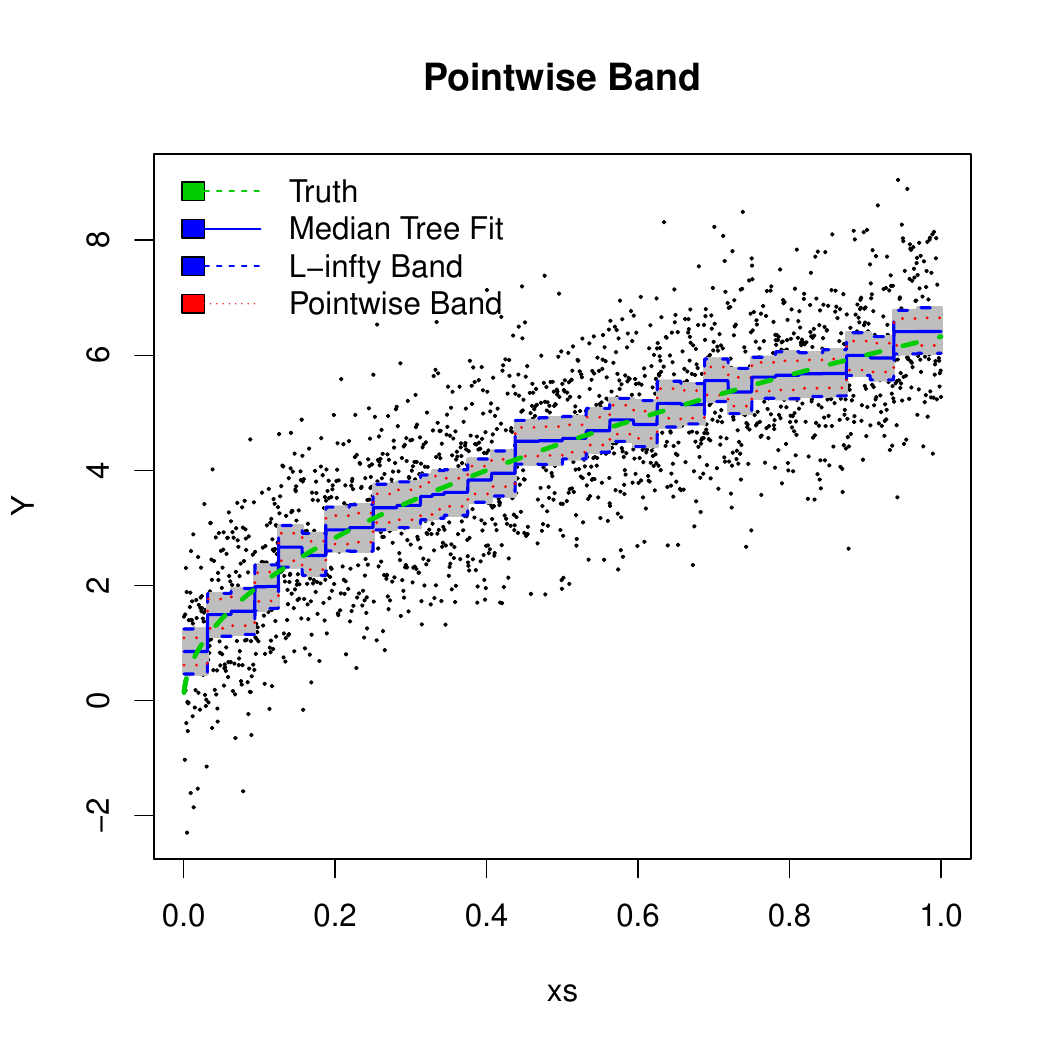}
\includegraphics[width=7cm]{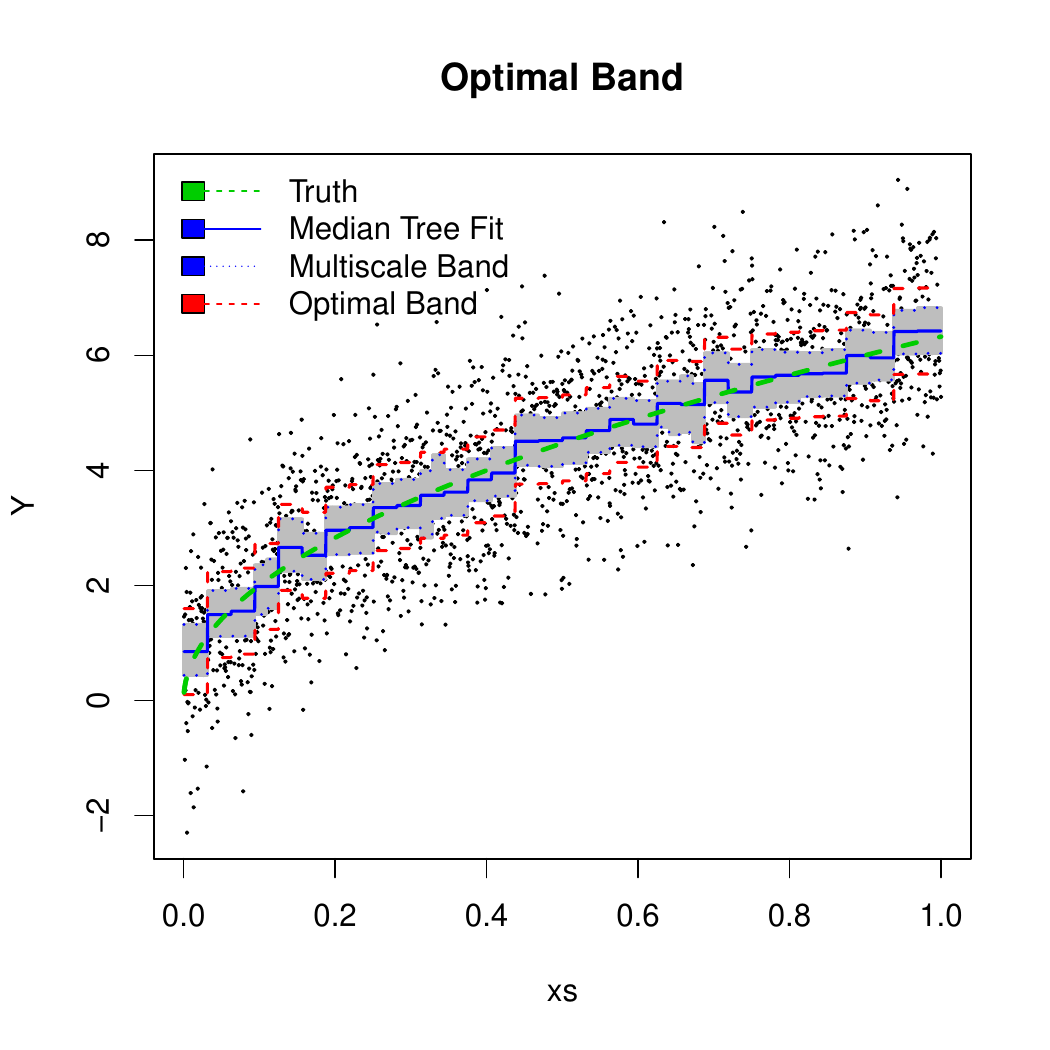}
\includegraphics[width=7cm]{Figures/piecewise_pointwise.pdf}
\includegraphics[width=7cm]{Figures/piecewise_optimal_v1.pdf}
%\begin{subfigure}{0.5\textwidth}\includegraphics{pointwise.pdf}\caption{Pointwise Bands}\label{fig:optimal}\end{subfigure}
\caption{\small (Left) Pointwise $95\%$ credible intervals and $95\%$-$L^\infty$ credible intervals (gray area). (Right) Non-intersected multiscale $95\%$ credible band  \eqref{multballemp} obtained with $w_l=l^{1/2+0.01}$ (gray area) and the `optimal' band \eqref{credible} obtained with $v_n=1$. (Upper) The true function is $f_0(x)=2\sqrt{10x}$. (Lower) The true function is $f_0(x)=(4x-1)\mathbb{I}(x\leq 1/2)+(-2x+2)\mathbb{I}(x>1/2)$.}\label{figure:uq}
\end{figure}

\subsection{Numerical examples}\label{sec:simul}
We highlight the practicality of the confidence bands presented in Section \ref{sec:uq} (and the following Section \ref{sec:multiscale_band}) on numerical examples. 
We implement Dyadic Bayesian CART as in \cite{cart1} using the standard Metropolis-Hastings algorithm with a proposal distribution consisting of two steps: grow (splitting a randomly chosen bottom node) and prune (collapsing two children bottom nodes into one). The implementation is fairly straightforward due to immediate access to the posterior tree probabilities through a variant of \eqref{eq:W} (these closed-form calculations can be easily updated for non-parametric regression). 

We illustrate uncertainty quantification for $f_0$ using the intersection band \eqref{credibleempmult} as well as the ``optimal" band \eqref{credible}. The intersection construction yields exact asymptotic coverage; it uses up more posterior information and in this sense can be viewed as more Bayesian in spirit.  Kolyan Ray \cite{ray}, Section 6, implemented the intersection-credible set in the case of spike-and-slab priors and the white noise model. Here we provide an implementation in the regression model for tree priors. For the computation of the empirical median tree estimator  in \eqref{bulkemp}, one can easily identify nodes $(l,k)$ that have occurred in at least $50\%$ of posterior samples.  Regarding the computation of $R_n$ in \eqref{multballemp}, the radius can be approximated using the $95\%$ quantile of the posterior samples of the  multiscale norm $\|\cdot\|_{\cM(w)}$--norm (acting on coefficients up to level $L_{max}$). Therefore, we see that while the confidence band in \eqref{credibleempmult} may at first look computationally cumbersome, it can be readily obtained from posterior samples.

We generate $n=2^{11}$ observations from \eqref{model} with $f_0(x)=2\sqrt{10x}$ (top row in Figure \ref{figure:uq}) and $f_0(x)=	(4x-1)\mathbb{I}(x\leq 1/2)+(-2x+2)\mathbb{I}(x>1/2)$ (bottom row in Figure \ref{figure:uq}). 
We choose $j_0(n)=4$ and $w_l=l^{1/2+\epsilon}$ (the multiscale weighting sequence) with $\epsilon=0.01$ and $\Gamma=1.01$ (the splitting probability parameter of the GW process prior). We run $2\,000$ iterations of the MH sampler with $500$ burnin samples. One tempting approach to uncertainty quantification is computing the  {\em pointwise}  $95\%$ credible intervals for each given $t_i$. These intervals are readily available from the posterior samples of the bottom node coefficients (transformed from the samples of the wavelet coefficients via the pinball formula \eqref{pinball}) and are portrayed in Figure \ref{figure:uq} on the left (red dotted lines). For both functions $f_0$ these intervals  are too narrow  to uniformly capture $f_0$  (depicted in a green dashed line). 
For comparisons, we also plot the $95\%$--$L^{\infty}$--credible bands (blue dashed lines), which have better coverage.  The $L^\infty$--band (gray area) is somewhat similar to the multiscale credible band but its coverage properties are not theoretically understood.

In comparison, the (not-intersected) multiscale $95\%$ credible band in \eqref{multballemp} (Figure \ref{figure:uq} on the right; gray area marked with blue dotted lines) is successful at containing the true function uniformly. These sets resemble the $L^\infty$--sets.
Figure \ref{figure:uq} plots the intersection band which has exact asymptotic coverage as well as the `optimal' set \eqref{credible} choosing $v_n=1$ (red dashed lines). Note, however, that this intersecting band is smaller than the band analyzed in Theorem \ref{csthmemp} which yields even better coverage.  The centering in Figure \ref{figure:uq} (right) is the median tree estimator. 
Similarly as in \cite{szabo_etal}, we have chosen the blow-up factor $v_n=1$ which yields a set \eqref{credible} which contains the multiscale band.
The intersection uses up substantial posterior information  and stabilizes the construction. Out of curiosity, we tried different values $v_n$ and found that the choice $v_n=0.5$ roughly corresponds to the multiscale band.

\iffalse
To illustrate the usefulness of Theorem 3, we plot the histogram of posterior samples (together with $2.5\%$ and $97.5\%$ quantiles) of the primitive functional $F(x)=\sum_{t_i\leq x} f(t_i)$ for $x=0.8$ for the piecewise linear model (Figure \ref{figure:uq2}). The true value is marked with a red solid line. We can see that the $95\%$-credible interval covers the true value.
We also construct the confidence band \eqref{bandF} (Figure \ref{figure:uq2} on the right).

\begin{figure}[!] \label{fig:band}
\includegraphics[width=6cm]{Figures/primitive_piecewise_04.pdf}
\includegraphics[width=6cm]{Figures/primitive_band_piecewise_04.pdf}
%\begin{subfigure}{0.5\textwidth}\includegraphics{pointwise.pdf}\caption{Pointwise Bands}\label{fig:optimal}\end{subfigure}
\caption{\small   $0.95\%$ credible interval for the primitive functional $F(x)$ with $x=0.8$ (left) and the confidence band \eqref{bandF} (right) obtained for  $f_0(x)=(4x-1)\mathbb{I}(x\leq 1/2)+(-2x+2)\mathbb{I}(x>1/2)$.}\label{figure:uq2}
\end{figure}
\fi

\subsection{Adaptive nonparametric BvM and applications}\label{sec-bvm}

\subsubsection{From $\|\cdot\|_\infty$ to BvM} 
Let us now formalize the notion of a  nonparametric BvM theorem in multiscale spaces following \cite{castillo_nickl1} (we refer also to  \cite{castillo_nickl2} for more background and discussion of the, different,  $L^2$--type setting). 
%\ma{[To update!] The formulation of the BvM theorem is not obvious in the white noise model with an infinite-dimensional parameter set,  in particular with respect to the centering of variables \cite{gine_nickl}. In a pathbreaking work, \cite{castillo_nickl1} formalized the non-parametric BvM phenomenon in the multi-scale spaces $\M_0(w)$ introduced in Section \ref{sec:intro} and showed that a (non-adaptive) BvM holds for sub-Gaussian product priors on the wavelet coefficients. Going further, \cite{ray}  obtained an {\em adaptive} BvM result for an independent product of spike-and-slab priors. In this section, we show that an {adaptive} BvM can be {\em also} obtained for the (more practical) CART priors which are {\em not} necessarily {independent products}.}
Such spaces are defined through the speed of decay of multiscale coefficients $\beta_{lk}=\langle f,\psi_{lk}\rangle$. % for, say, bounded functions $f$ on $[0,1]$.
For a monotone increasing weighting sequence $w=(w_l)_{l=0}^\infty$, with $w_l\geq 1$ and $w_l/\sqrt{l}\rightarrow\infty$ as $l\rightarrow\infty$ (such a  $w=(w_l)_{l=0}^\infty$ is called {\sl admissible}) we define the following {\sl multiscale sequence space} 
$$
\M(w)=\left\{x=(x_{lk}):\|x\|_{\M(w)}:=\sup_{l}\frac{\max_k|x_{lk}|}{w_l}<\infty\right\}.
$$
%The  space $\M(w)$ is non-separable and isomorphic to $\ell_\infty$ and it always contains $\ell_2$. Moreover,  for certain diverging sequences $\{w_l\}_{l=0}^\infty$, these spaces contain objects that are less regular than $\ell_2$. 
We consider a separable closed subspace  of $\M(w)$ defined as% obtained by considering  the weighted sequences in $\M(w)$ that vanish at infinity:
$$
\M_0(w)=\left\{x\in \M(w):\lim_{l\rightarrow\infty}\max_k\frac{|x_{lk}|}{w_l}=0\right\}.
$$
Defining random variables 
$
g_{lk}=\int_0^1\psi_{lk}(t)d W(t)\sim\mathcal{N}(0,1),
$
 according to Proposition 2  in \cite{castillo_nickl1}, the  Gaussian white noise $\mathbb W=(g_{lk})$ defines a tight Gaussian Borel measure in the space 
$\M_0(w)$ for admissible sequences $w$.
% when $w_l/\sqrt{l}\rightarrow\infty$ as $l\rightarrow\infty$ (i.e. when $w=\{w_l\}_{l=0}^\infty$ is {\sl admissible}). Throughout this paper, we will focus on the space $\M_0(w)$ for suitable admissible sequences $w$. 
%Any Borel probability measure on $\ell_2$ yields a right Borel probability measure on $\mathcal{M}_0(w)$. 
%Given the considerations above, ny prior on a sequence space $\ell_2$ can be viewed as the prior on the functional parameter $f\in L^2[0,1]$ (under wavelet isometry $L^2\approx\ell_2$). Consider a prior $\Pi$ on $\ell_2$ and let $\Pi_n=\Pi(\cdot\C\Y^{(n)})$ denote the posterior distribution based on an observed trajectory $\Y^{(n)}$ of \eqref{model2}.
The convergence in distribution of random variables in the multiscale  space $\mathcal{M}_0(w)$ is metrised  via the bounded Lipschitz metric
%The idea is to formulate convergence in distribution of the posterior density to a Gaussian process in a large enough space that enables convergence at the rate $\sqrt{n}$ (Castillo and Nickl (2015)). 
%To make such a statement rigorous, we first define the bounded Lipschitz metric 
$\beta_{\mathcal{M}_0(w)}$ defined below. For $\mu,\eta$ probability measures on a metric space $(S,d)$ define
$$
\beta_S(\mu,\eta)=\sup_{F:\|F\|_{BL}\leq 1}\left|\int_S F(x)(d\mu(x)-d\eta(x))\right|,
$$
$$
\|F\|_{BL}=\sup_{x\in S}|F(x)|+\sup\limits_{x\neq y,x,y\in S}\frac{|F(x)-F(y)|}{d(x,y)}.
$$
%We consider an admissible sequence $\{w_l\}$, prior $\Pi$ on $f$ and posterior $\Pi(\cdot\C\Y^{(n)})$ distribution on $\ell_2\subset\M_0(w)$ obtained from \eqref{model}. 
Denote with $\ix=\ix^{(n)}=(X_{lk}:l\in\N_0, 0\leq k<2^l),$ where $X_{lk}$ satisfy \eqref{eq:model2}. Let $\wt\Pi_n=\Pi_n\circ\tau^{-1}_{\ix}$ be the image measure of $\Pi(\cdot\C X)$ under
$
\tau_{\ix}:f\rightarrow \sqrt{n}(f-\ix). 
$ 
Namely, for any Borel set $B$ we have
\begin{equation}
\wt\Pi_n(B)=\Pi\left(\sqrt{n}(f-\ix)\in B\C X\right).\label{Pi_tilde}
\end{equation}
The following Theorem characterizes the {\em adaptive} nonparametric Bernstein-von Mises behavior of posteriors under the Bayesian Dyadic CART. In the result below, one assumes that trees sampled from $\Pi_\bT$ contain all nodes $(j,k)$ for all $j\le j_0(n)\to \infty$ slowly. Note that this constraint is easy to accommodate in the construction: for the GW process, one starts stopping splits only after depth $j_0(n)$, while for priors \eqref{prior:K}, it suffices to  constrain the indicator $\1_{\cT\in\bT}$ to trees that fill all first $j_0(n)$ layers.
%\ma{[Prior containing the first $j_0(n)$ layers to be specified precisely here]} 
%\sbl{Note: I prefer $\ix$ as in CN13 over the bold notation; to be changed in the appendix later}
%\ma{[Below, say explicitly that the result works for both Haar-wavelet prior, $g$-prior and smooth wavelet prior (which is used in the UQ section)]}

\begin{theorem}(Adaptive nonparametric BvM)\label{thm_bvm}
Let $\M_0=\M_0(w)$ for some admissible sequence $w=(w_l)$.
Assume the Bayesian CART priors $\Pi_\bT$ from Theorem \ref{thm-one} constrained to trees that fit $j_0(n)$ layers, i.e. $\gamma_{lk}=1$ for $l\leq j_0(n)$, for some strictly increasing sequence $j_0(n)\rightarrow\infty$
that satisfies $w_{j_0(n)}\geq c\,{\log n}$ for some $c>0$. 
Consider tree-shaped priors as in Theorem \ref{thm-one} (or  using an $S$--regular wavelet basis, $S\ge 1$). 
%, as in Theorem \ref{thm-two}. 
Then the posterior distribution satisfies the weak  Bernstein-von Mises phenomenon in $\M_0$ in the sense that
$$
E_{f_0}\beta_{\M_0}(\wt\Pi_n,\mathcal{N})\rightarrow 0\quad\text{as $n\rightarrow\infty$},
$$
where  $\mathcal{N}$ is the law of $\mathbb{W}$ in $\M_0$.
%$$
%E_{f_0}\Pi(\|f_{\mT,\b}-f_0\|_{\M(w)}\geq M_n n^{-1/2}\C X)\rightarrow0
%$$
%where $w=\{w_l\}$ is any admissible sequence satisfying $w_{j_0(n)}\geq c\,{\log n}$ for some $c>0$.
\end{theorem}
This result states an {\em adaptive} nonparametric BvM result, in the sense that the prior it considers also leads to an adaptive nonparametric convergence rate in $L^\infty$ (optimal up to log terms). It is only the second result of this kind after the one derived by Ray in \cite{ray}.  
This  statement, proved in Section \ref{proof_thm_bvm}, can be shown, for example, by verifying the conditions in Proposition 6 of \cite{castillo_nickl1} (appropriate `tightness' and convergence of finite dimensional distributions). The first tightness condition pertains to contraction in the $\mathcal{M}_0$--space, which can be obtained from our $\|\cdot\|_\infty$--results. In order to attain BvM, we need to modify the prior to always include a few coarsest dense layers in the tree (similarly as \cite{ray}). Such trees are semi-dense, where sparsity kicks in only deeper in the tree after $j_0(n)$ dense layers have already been fitted. This enables one to derive the convergence of finite dimensional distributions to suitable Gaussian distributions. For the independent wavelet prior, the last point follows easily from results in \cite{castillo_nickl2}. For the $g$--prior on trees corresponding to actual Bayesian CART, it requires a completely new argument based on the conditional posteriors given possible trees.

\subsubsection{Application 1: multiscale confidence sets} %While Theorem \ref{thm_bvm} will be of use in the next subsection, we now briefly mention its several implications, referring to \cite{castillo_nickl1} for details. 
%Theorem \ref{thm_bvm}  has important implications for uncertainty quantification.  Castillo and Nickl \cite{castillo_nickl1} propose the 
First, let us consider  multiscale credible balls for $f_0$, which we will use in the next subsection for refining the band construction used in the main paper. Such multiscale balls consist  of functions $f$ that simultaneously satisfy multi-scale linear constraints (see e.g. (5) in  \cite{castillo_nickl1}):
\begin{equation} \label{crediblem}
\cB_n=\left\{f:\ \|f-\ix\|_{\cM(w)} \le R_n/\rn\right\},
\end{equation}
where $R_n$ is chosen such that $\Pi[\cB_n\given X]=1-\ga$ (or the smallest radius such that $\Pi[\cB_n\given X]\ge1-\ga$), i.e. $\cB_n$ is a credible set of level $1-\ga$. \begin{proposition} \label{propmult}
Let $f_0\in \cH_M^\al$ for some $\al\in(0,1], M>0$. Then for $\cB_n$ as in \eqref{crediblem},
\[ P_{f_0}(f_0\in \cB_n)\rightarrow 1-\gamma,\qquad (\text{as }n\to\infty).\]
\end{proposition}
The proof of the Proposition \ref{propmult} is a consequence of the fact that the nonparametric BvM in the multiscale space holds: one combines Theorem \ref{thm_bvm} and Theorem 5 in \cite{castillo_nickl1}.

\subsubsection{Application 2: BvM for functionals}  As a second application of Theorem \ref{thm_bvm}, we derive confidence bands for $F(t):=\int_0^tf(x)dx, 0\leq t\leq 1$: those result from taking quantile credible bands  in the following limiting distribution result. This application is described in Theorem \ref{thm-fun} and the (short) proof is in Section \ref{sec:proof_thm-fun}.

\subsubsection{Application 3: multiscale confidence bands in the white noise model}

Let us first consider the case of the white noise model. For $R_n$ as in \eqref{crediblem}, $\sigma_n(X)$ as in \eqref{radiusprox} and $\wh f_T$ as in \eqref{bulkest}, let us set 
\begin{equation} \label{crediblemult}
\cc_n^\cM=\left\{\,f:\ \|f-\ix\|_{\cM(w)} \le R_n/\rn\ ,\ \|f-\wh f_T\|_\infty \le \sigma_n(X) \,\right\}.
\end{equation}
Let us recall the definition of the self-similarity class $\cH_{SS}(\al,M,\veps)$ from Definition \ref{def-ssi}.

\begin{corollary} \label{csthm_multi}
Let $0<\al_1\le \al_2 \,\le 1,\, M\ge 1,\, \ga\in(0,1)$ and $\veps>0$. Let the prior $\Pi$ and the sequence $(w_l)$ be as in the statement of Theorem \ref{thm_bvm}. Take $R_n$ as in \eqref{crediblem}, $\sigma_n$ as in \eqref{radiusprox} with $v_n$ such that $(\log{n})^{1/2}=o(v_n)$ and  let $\wh f_T$ denote the median tree estimator \eqref{bulkest}.  Then for $\cc_n^\cM$ defined in \eqref{crediblemult}, uniformly over $\alpha\in[\alpha_1,\alpha_2],$
\[  \sup_{f_0\in \cH_{SS}(\alpha,M,\veps)} |P_{f_0}(f_0\in \cc_n^\cM) - (1-\ga) | \to 0,\]
as $n\to\infty$. 
In addition, for every $\al\in[\al_1,\al_2]$ and uniformly over $f_0\in \cH_{SS}(\al,M,\veps)$, the  diameter $|\cc_n^\cM|_\infty=\sup_{f,g\in \cc_n^\cM}\|f-g\|_\infty$ and the credibility of the band  verify, as $n\to \infty$,
\begin{align*}
 |\cc_n^\cM|_\infty & = O_{P_{f_0}}((n/\log n)^{-\al/(2\al+1)} v_n),\\
 \Pi[\cc_n^\cM \given X] & = 1-\ga + o_{P_{f_0}}(1). 
\end{align*} 
\end{corollary}

This result states that by intersecting the confidence set $\cc_n$ in \eqref{credible} with the ball $\cB_n$ from \eqref{crediblem}, one obtains a set with confidence (and credibility) at the prescribed level $1-\ga$. It directly follows by applying Proposition \ref{propmult} (which guarantees that $\cB_n$ has confidence level $1-\ga$ and hence also $\cc_n^\cM$, since $\cc_n$ has confidence that goes to $1$) and the fact that by definition $\cB_n$ has credibility $1-\ga$ (and hence also $\cc_n^\cM$ asymptotically). 

\subsubsection{Application 4: multiscale confidence bands in regression} \label{sec:multiscale_band}

Now consider the regression setting \eqref{model}. Let us define a discrete analogue for the multiscale confidence ball \eqref{crediblem}. Recall that in this setting the observations are the $Y$'s and that $X$ here denotes the matrix of $\psi_{lk}(t_i)$'s as introduced in Section \ref{sec:npreg}.  Denote $\tilde{F}=(\tilde{f}_T(t_i))_{1\le i\le n}$ for $\tilde{f}_T$ the empirical median tree estimator. Set
\begin{equation}\label{multballemp}
\cB_n^b=\{b=(b_{lk})_{l\le L_{max},k}:\ \|b-X^{-1}\tilde{F}\|_{\cM(w)} \le R_n/\rn \},
\end{equation}
where $R_n$ is defined in such a way that $\Pi[ \cB_n^b \given Y]=1-\ga$ (or possibly $\ge 1-\ga$) and where in slight abuse of notation $\|\cdot\|_{\cM(w)}$ stands for the multiscale norm acting on coefficients up to level $L_{max}$ only (i.e. the supremum over $l$ in the definition of $\|\cdot\|_{\cM(w)}$ is replaced by the maximum over $l\le L_{max}$). For $\cc_n^e$ as in \eqref{credibleemp}, recalling the notation $F=(f(t_i))_{1\le i\le n}'$, define
\begin{equation} \label{credibleempmult}
\tilde\cc_n^\cM=\cc_n^e \cap \left\{f:\ %\|f-\ix\|_{\cM(w)} \le R_n/\rn,\
   X^{-1}F \in \cB_n^b \right\}.
\end{equation}

\begin{corollary} \label{csthm_multi_emp}
Consider a prior as in Theorem \ref{thm_bvm}, and let $\cc_n^e$ be constructed as in Theorem  \ref{csthmemp}. For $\ga>0$, let $\tilde\cc_n^\cM$ be defined as in \eqref{credibleempmult}.  The set $\tilde\cc_n^\cM$ verifies the properties stated in Theorem \ref{csthmemp} except that both confidence and credibility go to the nominal level $1-\ga$ as $n\to\infty$.
\end{corollary}
This result is obtained in a similar way as for Corollary \ref{csthm_multi}: first, one notes that a BvM result similar to Theorem \ref{thm_bvm} in white noise holds (details are omitted, the proof being similar). This implies that the ball $\cB_n^b$ in \eqref{multballemp} has confidence going to nominal level $1-\ga$. One concludes  by using Theorem \ref{csthmemp}, that ensures that $\cc_n^e$, and then in turn $\tilde\cc_n^M$, has the desired properties in terms of coverage, confidence and credibility.

\subsection{Multi-dimensional extensions} \label{sec:multid}

Our tree-shaped wavelet reconstruction generalizes to the multivariable case, where a fixed number $d\geq 1$ of covariate directions are available for split. 
We outline one such generalization using the tensor product of Haar basis functions $\psi_{lk}$ from \eqref{haar}  defined as  
$$
\Psi_{l\vk}(\x):=\psi_{lk_1}(x_{1})\cdots \psi_{lk_d}(x_{d})
$$
for $l\geq 0$ and $\vk=(k_1,\ldots,k_d)'$ with $0\le k_i\le 2^l-1$ for $i=1,\ldots,d$, where  $\Psi_{-1\bm 0}(\x)=\mathbb{I}_{(0,1]^d}(\x)$.
These wavelet tensor products can be associated with $d$-ary trees (as opposed to binary trees), where each internal node has $2^d$ children. 
The  nodes in a $d$-ary tree satisfy a hierarchical constraint: $(l,\bm k)\in \mT, l\geq 1\Rightarrow (l-1,\bm \lfloor\bm k/2\rfloor)\in\mT$, where the floor operation is applied element-wise.
This intuition can be gleaned from Figure \ref{tensor_haar} which organizes  tensor wavelets with $l=0,1$ and $d=2$ in a flat $4$-ary tree.
%\sbl{[Figure 'plot' to be added]}
\begin{figure}
\scalebox{0.3}{\includegraphics{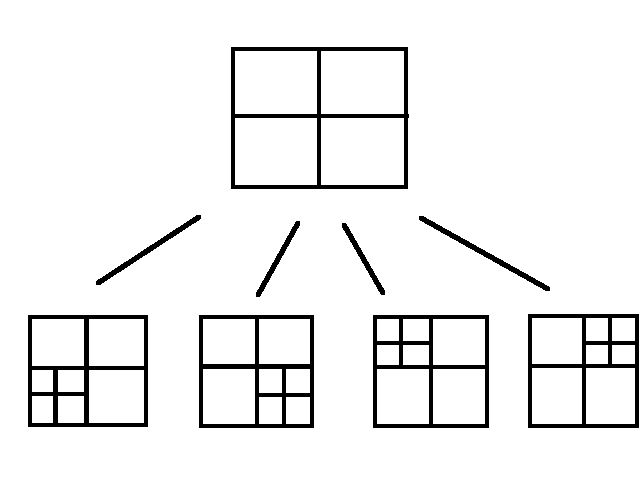}}
\caption{A plot of tensor Haar wavelets. The top figure plots $\Psi_{0\, (0,0)'}$  and  the bottom figures are $\Psi_{1\, (0,0)'}, \Psi_{0\, (1,0)'},\Psi_{0\, (0,1)'},\Psi_{0\, (1,1)'}$  (from left to right). }
\label{tensor_haar}
\end{figure}
We assume  that $f_0$ belongs to $\alpha$-H\"{o}lder  functions  on $[0,1]^d$ for $0<\alpha\leq 1$ defined as
\begin{equation}\label{eq:haar4}
 \mathcal{H}^{\alpha,d}_M:= \left\{f\in\mathcal{C}([0,1]^d): \|f\|_\infty +\sup_{\x\neq\y}\frac{|f(\x)-f(\y)|}{\|\x-\y\|^\alpha}\le M\right\}.
 \end{equation}
The multiscale coefficients   $\beta_{-1 \bm 0}=\psg f_0, \Psi_{-1\bm 0} \psd$ and  
\[ 
\beta_{l\vk} = \psg f_0 , \Psi_{l \bm k} \psd = \int_{[0,1]^d} f_0(\x)\Psi_{l \bm k}(\x)\, d \x. 
\]
can be verified to satisfy, for some universal constant $C>0$,
\begin{equation}\label{eq:decay_multi}
|\beta_{l\vk}|\le C  2^{-l(\frac{1}{2}+\al)d}. 
\end{equation}
Similarly as in Section \ref{tree_wavelet_prior}, denoting with $\mT_{int}'$ the collection of internal nodes $(l,\bm k)$ in a $d$-ary tree  (including the node $(-1,\bm 0)$),
one then obtains a wavelet reconstruction $f_{\mT,\b}(\x)=\sum_{(l,\bm k)\in\mT_{int}'}\beta_{l\bm k}\Psi_{l\bm k}(\x)$, where  coefficients $\beta_{l\bm k}$ can be assigned, for instance, a Gaussian independent product prior. {There are several options for defining the $d$--dimensional version of the prior $\Pi_\bt$. Restricting to Galton-Watson type priors, the most direct extension, for each node $(l,\bm k)$ to be potentially split, either does not split it with probability $1-\Gamma^{-l}$, or splits it into $2^d$ children, leading to a full $2^d$--ary tree. Another, more flexible option, is to split $(l,\bm k)$ into a random number of children inbetween $0$ and $2^d$, where a split in each specific direction occurs with probability $\Gamma^{-l}$, for $\Gamma$ a large enough constant.} 

Assuming that $d$ is fixed as $n\rightarrow\infty$, the general proving strategy of Theorem \ref{thm-one} can still be applied to conclude $\ell_\infty$--posterior convergence at the rate $\varepsilon_{n}=(\log n/n)^{\alpha/(2\alpha+d)}\log^\delta n$ for some $\delta>0$.
The proof requires  the threshold $\mL_c$ in  \eqref{cutoff} to be modified as satisfying $2^{\mL_c}\asymp (n/\log n)^{1/(2\alpha+d)}$.

The basis we consider here is a tensor product where, within each tree layer, splits occur along each direction simultaneously.  This is not necessarily what Bayesian CART does in practice. 
Multivariate Bayesian CART can be more transparently translated using  anisotropic Haar wavelet basis functions which  more closely resemble  recursive partitioning (as explained in \cite{donoho}).
Our approach extends more naturally to the tensor product basis, but it could be in principle applied to other basis functions such as this one.

\iffalse
\subsubsection{Proof of Proposition \ref{propmultid}}

\begin{proof}
First, the elements are normalized, i.e.
$$
\int_{(0,1]^d}(\Psi_{lk}^{B,S})^2(\x)d\x=\left(\prod_{j\neq s}\int_{(0,1]} \frac{\mathbb{I}_{I_{lk}^j}(x_j)}{|I_{lk}^j|} d x_j\right) \times \int_{(0,1]}(\psi_{lk}^{s}(x_s))^2d x_{s}=1.
$$
Second, the elements are pairwise orthogonal. This follows from the fact that if $(l,k)$ is {\em not} a descendant of a node $(l',k')$, functions $\Psi_{lk}(\x)$ and $\Psi_{l'k'}(\x)$  have disjoint supports. When  $(l,k)$  {\em is} a descendant of a node $(l',k')$, orthogonality follows from the fact that $\Psi_{l'k'}(\x)$ is constant in the nonzero domain of $\Psi_{lk}(\x)$, where the function
$\psi_{lk}^{s}(x)$ integrates to zero.
\end{proof}
\fi

\section{Basic Lemmata}

\subsection{Properties of the pinball matrix \eqref{onetoone}}\label{ap:sec_pinball}

While $A_\mT' A_\mT$ is not proportional to an identity matrix (for trees other than flat trees), it {\em does} have a  nested sparse structure which will be exploited  in  our analysis. 
\begin{proposition}\label{prop1}
Denote with $(l_1,k_1)$ the deepest rightmost internal node in the tree $\mT$, i.e. the node $(l,k)\in\mT_{int}$ with the highest index $2^l+k$. Let $\mT^-$ be a tree obtained from $\mT$ by turning $(l_1,k_1)$ into a terminal node.
Then 
\begin{equation}\label{eq:prop1}
 A_\mT' A_\mT=\left(
\begin{matrix}
 A_{\mT^-}' A_{\mT^-}+\bm{v}\bm{v}' & \bm 0\\
\bm 0'&2^{l_1+1}  
\end{matrix}
\right)
\end{equation}
for a vector of zeros $\bm 0\in\R^{|\mT_{ext}|-1}$ and a vector  $\bm{v}\in\R^{|\mT_{ext}|-1}$ obtained from $ A_\mT$ by first deleting its last column and then transposing the last row of this reduced matrix.
\end{proposition}
\begin{proof}
The index $(l_1,k_1)$, by definition, corresponds to the last entry in the vector $\b_\mT$. We note that $\mT^-_{int}=\mT_{int}\backslash \{(l_1,k_1)\}$ and $\mT^-_{ext}= \mT_{ext}\backslash\{(l_1+1,2k_1),(l_1+1,2k_1+1)\}\cup\{(l_1,k_1)\}$.
The matrix $ A_{\mT^-}$ can be obtained from $ A_\mT$ by deleting the  last column  of $ A_\mT$ and then deleting the last row,    further denoted with $\bm v'$. 
The desired statement \eqref{eq:prop1} is obtained by noting that the  last column of $ A_\mT$ (associated with $\beta_{l_1,k_1})$ is orthogonal to all the other columns. This is true because (a) this column has only two nonzero entries that correspond to the last two siblings $\{(l_1+1,2k_1),(l_1+1,2k_1+1)\}$,  (b)  the last two rows of $ A_\mT$ differ only in the sign of the last entry because $\{(l_1+1,2k_1),(l_1+1,2k_1+1)\}$ are siblings and  share the same ancestry  with the same weights up to the sign of their immediate parent. Finally, the entry $2^{l_1+1}$ follows from \eqref{pinball}.
\end{proof}

\begin{corollary}
Under the prior \eqref{eq33}, the coefficient $\beta_{lk}$ of {\sl any} internal node $(l,k)$ which has terminal descendants is independent of all the remaining internal coefficients. 
\end{corollary}
\begin{proof}
Follows directly from Prop. \ref{prop1} after reordering the nodes.
\end{proof}
The following proposition characterizes the eigenspectrum of $A_\mT'A_\mT$ which will be exploited in our proofs.
\begin{proposition}\label{prop:eigenspectrum}
The eigenspectrum of $ A_\mT' A_\mT$ consists of the diagonal entries of $\bm {D}=\mathrm{diag}(\wt d_{lk,lk})= A_\mT A_\mT'$ in \eqref{eq:lemma1}. Moreover,
the diagonal entries $\mathrm{diag}( A_\mT' A_\mT)=\{d_{lk,lk}\}_{lk\in\mT_{int}}$ satisfy $d_{-10,-10}=|\mT_{ext}|$ and 
%\begin{equation}\label{diag}
$
d_{lk,lk}=\sum_{j=l+1}^{d(\mT)}2^{j}\sum_{m=0}^{2^{j-1}}\mathbb{I}[\beta_{lk}\in   [(0,0)\leftrightarrow(j,m)]_{\mT}]$ with $ [(0,0)\leftrightarrow(l,k)]_{\mT}:=\{(0,0),(1, \lfloor k/2^{l-1}\rfloor),\dots,(l-1,\lfloor k/2\rfloor)\}.$
%\end{equation}
\end{proposition}
\begin{proof}
The first statement follows from \eqref{eq:lemma1} and the fact that $ A_\mT' A_\mT$ and $ A_\mT A_\mT'$ have the same spectrum, and the second statement   from \eqref{pinball}.
\end{proof}

\subsection{Other lemmata}
\begin{lemma}\label{lemma:varah}
Assume  that a square  matrix  $A$  is diagonally dominant by rows (i.e.,  $a_{kk}>\sum_{j\neq k}|a_{kj}|$). Then 
$$
\|A\|_\infty<\frac{1}{\min_k (|a_{kk}|-\sum_{j\neq k}|a_{kj}|)}.
$$
\end{lemma}
\begin{proof} 
Theorem 1 in Varah \cite{varah}. 
\end{proof}

\begin{lemma}\label{woodbury} 
For an invertible matrix $M\in\R^{p\times p}$ and $\bm v\in\R^p$ we have 
$$
(M^{-1}+\bm v\bm v'/g_n)^{-1}=M-\frac{M\bm v\bm v'M}{g_n+\bm v' M\bm v}\quad\text{for $g_n>0$}.
$$
\end{lemma}
\begin{proof} 
Follows immediately {by direct computation}.%from Sherman-Morrison  formula.
\end{proof}
\begin{lemma} \label{catalan}
Let $\mathbb{C}_{K}$ denote the number of full binary trees with $K+1$ leaves. Then 
\[ \mathbb{C}_{K}=\frac{(2K)!}{(K+1)!K!} \asymp 4^K/K^{3/2}. \]
\end{lemma}
\begin{proof}
The number $\mathbb{C}_{K}$ is the Catalan number (see e.g. \cite{stanley_book}), which verifies the identity. The second assertion follows from Stirling's formula.
\end{proof}

\begin{lemma}\label{lemma:exp_max}
Let $\bm Y\sim\mathcal{N}_K(\bm \mu,\Sigma)$ be a Gaussian random vector.
Denote with $\{\sigma_i\}_{i=1}^K=\mathrm{diag}(\Sigma)$, with $\bar{\mu}=\max\limits_{1\leq i\leq K}\mu_i$ and with
 $\bar{\sigma}^2=\max\limits_{1\leq i\leq K}\sigma^2_i$ the maximal mean and variance.
Then
\begin{equation}\label{lemma_max}
\E \left[\max_{1\leq i\leq K} |Y_i|\right]\leq \bar\mu+ \sqrt{2\,\bar{\sigma}^2\log K}+2\sqrt{2\pi\bar{\sigma}^2}.
\end{equation}
\end{lemma}
\begin{proof}
We start by noting that $|Y_i|\leq \bar\mu+|Y_i-\mu_i|$. Next, {one can use the formula, valid for any real $\mu_i$, $c>0$ and real random variables $Y_i$}, 
\begin{equation}\label{bound_ross}
\E [\max_{1\leq i\leq K} |Y_i-\mu_i|]\leq c+\sum_{i=1}^K\int_c^\infty\P(|Y_i-\mu_i| > x)dx.
\end{equation}
Assuming the Gaussian distribution, the integral is of order 
$
\int_{c}^\infty 2\e^{-x^2/2\sigma_{i}^2}dx$ \\ $\le \,\sqrt{2\pi\sigma_i^2}\,\e^{-c^2/2\sigma^2_i}$. 
Then \eqref{lemma_max} follows from \eqref{bound_ross} by choosing $c=\sqrt{2\bar{\sigma}^2\log K}$.
\end{proof}

\begin{lemma}[see, e.g., \cite{devroye_tv}] \label{lem-tv}
For a positive integer $d$, let $\mu, \mu_1, \mu_2\in \RR^d$ and let $\Sigma, \Sigma_1,\Sigma_2$ be positive definite $d\times d$ matrices. Then the exist universal constants $C_1,C_2>0$ such that, for TV the total variation distance,
\begin{align*}
TV\left(\cN(\mu,\Sigma_1), \cN(\mu,\Sigma_2)\right) & \le 
C_1\|\Sigma_1^{-1}\Sigma_2 -I_d \|_F \\
TV\left(\cN(\mu_1,\Sigma), \cN(\mu_2,\Sigma)\right) & \le 
C_2\frac{\|\mu_1-\mu_2\|^2}{\sqrt{(\mu_1-\mu_2)'\Sigma(\mu_1-\mu_2)}},
\end{align*}
where $\|\cdot\|_F$ denotes the Frobenius norm.
\end{lemma}
\begin{proof}
The first inequality follows from Theorem 1.1 in \cite{devroye_tv} and the second by Theorem 1.2 in \cite{devroye_tv} (by setting $\Sigma=\Sigma_1=\Sigma_2$ in their statement).
\end{proof}

%\subsection{Proofs of Auxiliary Lemmata and Propositions}

\section{Non-dyadic Bayesian CART: properties and examples} \label{sec-suppnd}

\subsection{Basic properties and examples}

%We characterized H\"{o}lder  functions though the speed at which the multiscale coefficients decay as a function of the resolution index $l$ (see  \eqref{eq:haar}). With UH wavelets, the decay can be expressed in terms the lengths of the right/left wavelet pieces $L_{lk}$ and $R_{lk}$. %\ma{The proof of the following Lemma is in Section \ref{ap:proof_lemma_haar} in the Appendix. }

\begin{lemma}\label{lemma:haar}
{For a set $A$ of admissible nodes $(l,k)$ as above}, let us define  $\beta_{lk}^B=\langle f,\psi_{lk}^B \rangle$, where $\psi_{lk}^B$ is  the unbalanced Haar wavelet in \eqref{eq:unbalanced_haar} and  where  $f\in\mathcal{H}^\alpha_M$ was defined in \eqref{eq:haar2}. Then  
\begin{equation}\label{eq:haar3}
|\beta_{lk}^B|\leq M 2^{-1/2}\max\{|L_{lk}|,|R_{lk}|\}^{\alpha+1/2}.
\end{equation}
\end{lemma}
%\proof Section \ref{sec:proof_lemma:haar}.

\begin{proof}
Let us denote by  $\bar C\equiv1/\sqrt{|L_{lk}|^{-1}+|R_{lk}|^{-1}}$. We have
\begin{align*}
|\beta_{lk}^B|&=\left|\bar  C\left\{\int_{L_{lk}} \frac{f(x)}{|L_{lk}|}d x- \int_{R_{lk}} \frac{f(x)}{|R_{lk}|}d x \right\}\right|\\
%&=\frac{\bar C}{|L_{lk}|}\int_{l_{lk}}^{b_{lk}}{f(x)} -  \frac{\bar C}{|R_{lk}|}\int_{l_{lk}}^{b_{lk}} f\left(b_{lk}+(y-l_{lk})\frac{|R_{lk}|}{|L_{lk}|}\right)\frac{|R_{lk}|}{|L_{lk}|}d y\\
%&=\frac{\bar C}{|L_{lk}|}\int_{l_{lk}}^{b_{lk}}\left[{f(x)} -  f\left(b_{lk}+(x-l_{lk})\frac{|R_{lk}|}{|L_{lk}|}\right)\right]d x\\
& \leq\frac{\bar C}{|L_{lk}|}\int_{0}^{|L_{lk}|}\left|{f(x+l_{lk})} -  f\left(b_{lk}+x\frac{|R_{lk}|}{|L_{lk}|}\right)\right|d x.
\end{align*}
Next,  from $\alpha$-H\"{o}lder continuity \eqref{eq:haar2}, we have
\[
\left|{f(x+l_{lk})} -  f\left(b_{lk}+x\frac{|R_{lk}|}{|L_{lk}|}\right)\right|\leq  M \left| |L_{lk}|+x\left(\frac{|R_{lk}|}{|L_{lk}|}-1\right)\right|^\alpha.
\]
Since $\bar C\leq 2^{-1/2}\max\{|L_{lk}|,|R_{lk}|\}^{1/2}$ we have
$$
|\beta_{lk}^B|\leq M\bar C\max\{|L_{lk}|,|R_{lk}|\}^\alpha\leq M 2^{-1/2}\max\{|L_{lk}|,|R_{lk}|\}^{\alpha+1/2}.
$$
This follows from the fact that for $x\in(0,|L_{lk}|)$ we have
$$
M\left| |L_{lk}|+x\left(\frac{|R_{lk}|}{|L_{lk}|}-1\right)\right|^\alpha\leq \max\{|L_{lk}|,|R_{lk}|\}^\alpha.\qedhere
$$
%Suppose that $|R_{lk}|>|L_{lk}|$, then for $x\in(0,|L_{lk}|)$ the above can be bounded with $|R_{lk}|^\alpha$. When the contrary is true, the bound is $ 2^\alpha|L_{lk}|^\alpha$.
%\begin{align}
%|\beta_{lk}^B|&\leq M\,2^{\alpha-1/2}\max\{|L_{lk}|,|R_{lk}|\}^{\alpha+1/2}\qedhere
%\end{align}
\end{proof}

For the classical Haar basis \eqref{haar}, one obtains \eqref{eq:haar} from \eqref{eq:haar3} by noting $\max\{|L_{lk}|,|R_{lk}|\}=2^{-(l+1)}$. \cite{piotr} points out that the computational complexity of the discrete UH transform could be unnecessarily large and imposes the balancing requirement 
$\max\left\{|L_{lk}|,|R_{lk}|\right\}\leq E (|L_{lk}|+|R_{lk}|)\ \forall (l,k)\in A
$, for some $1/2\leq E<1$.    %Under this assumption, it turns out that $ |\beta_{lk}^B|\leq M\,2^{\alpha-1/2} C^{l(\alpha+1/2)}$
%and one can show that the posterior probability of trees with more than $\mathcal{L}^\star_C$  layers goes to zero, where 
%$(1/C)^{\mathcal{\mathcal{L}_C^\star}}\doteq M^{1/(2\alpha+1)}\left(\frac{n}{\log n}\right)^{1/(2\alpha+1)}.$   
In order to control the combinatorial complexity of the basis system, we also require that the UH wavelets are not too imbalanced. To this end, we introduce  the notion of {\em weakly balanced} wavelets in Section \ref{sec:unbal}. A graphical depiction of the unbalanced Haar wavelets for certain non-dyadic choices of split points $b_{lk}$ are in Figure \ref{fig:diagram}.

%\sbl{[Somewhat of a detail, but I just noticed that $C$ is not very practical, as we cannot write anymore $C$ to refer to a large enough universal constant... We use this quite a bit in the proofs, so I suggest that we call rather `$E$' the '$C$' of the balancing condition]}

\begin{figure}
{\includegraphics[width=12cm,height=7cm]{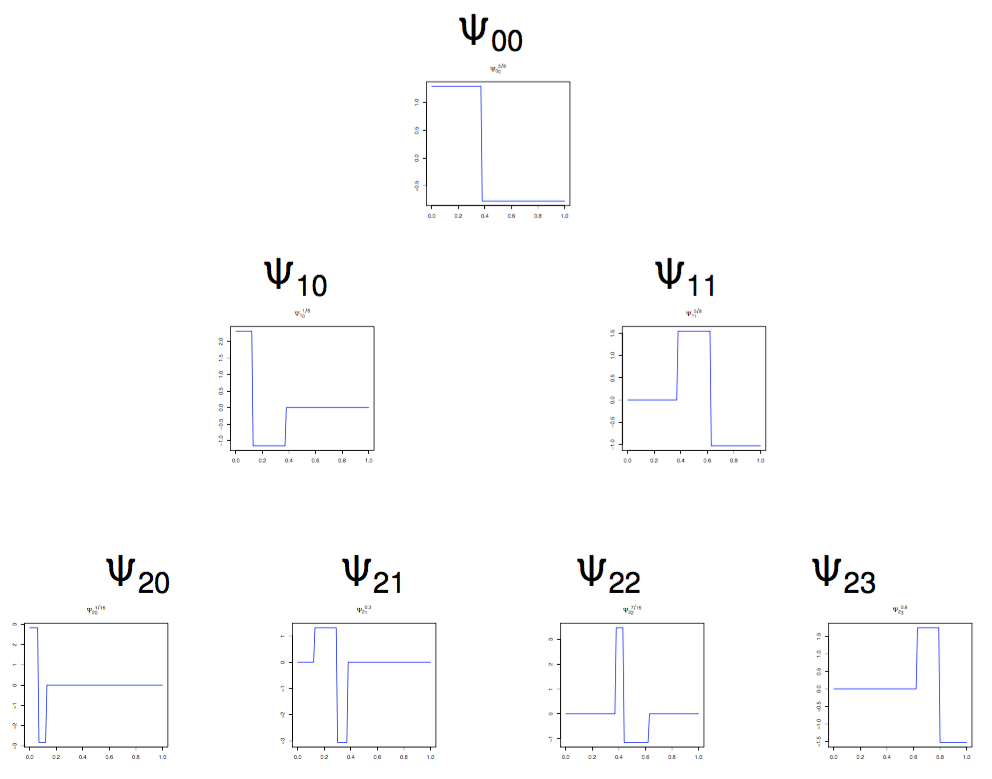}}
\caption{Diagram of  unbalanced (right) Haar wavelets for $l\leq 2$.}\label{fig:diagram}
\end{figure}

%Note that balanced Haar wavelets are also weakly balanced  for any $D\geq 1$ and $2^{D-1}\leq C$ since 
%$|L_{lk}|=|R_{lk}|=1/2^{l+1}$. 
%To illuminate the balancing condition \eqref{eq:balance2}, we consider the following ``counterexample".
\begin{example} \label{counterexample}
To glean insights into the balancing condition \eqref{eq:balance2},
we first consider an example of  UH  system which is {\em not} weakly balanced for some given $n,D$, say $n=2^4$ and $D=2$. If we choose $b_{00}=1/2$, $b_{10}=1/2-1/n$   and $b_{11}=3/4$, we have
\begin{align*}
L_{10}&=(0,1/2-1/n],& R_{10}&=(1/2-1/n,1/2],\\
L_{11}&=(1/2,3/4],   &R_{11}&=(3/4,1].
\end{align*}
While the node $(1,1)$ is  seen to  satisfy \eqref{eq:balance2} with $E=5$,  
 we note that
%Indeed,
%$\max \{|L_{11}|,|R_{11}|\}=\frac{3}{4}=\frac{M_{11}}{2^{1+D}}\quad\text{for}\quad M_{11}=6.$
%Next, we have
$
\max \{|L_{10}|,|R_{10}|\}=(n-2)/(2n)=7/16.$
However,  $7/16$ {\sl cannot} be written as $M_{10}/2^{D+1}=M_{10}/8$ for {\sl any} integer $M_{10}$. 
This is why the split points $b_{00}, b_{10}$ and $b_{11}$ {\em do not} belong to any weakly balanced UH wavelet system with  balancing constant $D=2$. 
Weakly balanced systems can be  built by choosing splits in such a way that the ``granularity" does not increase too rapidly throughout the branching process. With granularity $R(l, \Psi_A^B)$  of  the $l^{th}$  layer we mean  the smallest integer $R\ge 1$ such that $\min_{0\leq k<2^l}\min\{|L_{lk}|,|R_{lk}|\}=j/2^R$ for some  $j\in\{1,2,\dots, 2^{R-1}\}$. For instance, setting $D=2$ and $E=3$ one can build weakly balanced wavelets by first picking $b_{00}$ from values $\{\frac{1}{4},\frac{1}{2},\frac{3}{4}\}$. If, e.g. $b_{00}=3/4$ (i.e. $R(0,\Psi_{A}^B)=2$), the next split $b_{10}$ can be selected from  $\{\frac{1}{4},\frac{3}{8},\frac{1}{2}\}$, while $b_{11}$ has to be set as $7/8$. 
%The resolution at  the level $l=1$ is $4$, which is larger than $D+1=3$.
\end{example}
\iffalse
To glean further insights into the balancing condition \eqref{eq:balance2},  we define the {\sl resolution $R(l, \Psi_A^B)$}  of  the $l^{th}$ layer of the basis system $\Psi_{A}^B$ as the smallest integer $R\in\N$ such that
$$
\min\limits_{0\leq k<2^l; (l,k)\in A}\min\{|L_{lk}|,|R_{lk}|\}=j/2^R\quad\text{for some}\quad j\in\{1,2,\dots, 2^{R-1}\}.
$$
The resolution of the system $\Psi_A^B$ is important for controlling its ``complexity" (as will be seen below) and it should be somewhat compatible with the resolution $l+1$ of the classical Haar wavelets.  By definition, the resolution of weakly balanced systems $\Psi_{A}^B$ is no larger than $l+D$. 
In other words,  the splits are chosen in such a way that the resolution does not increase too rapidly throughout the branching process. 
\fi

Our   theoretical development relies in part on combinatorial properties of weakly balanced UH  systems and on the speed of decay of the multiscale functionals $\beta_{lk}^B=\langle f,\psi_{lk}^B\rangle$  as the layer index $l\in\N$ increases.
These two fundamental properties are encapsulated in Lemma \ref{lemma:complex} which is vital to the proof of   Theorem \ref{thm-three}. 

\subsection{Granularity lemma}
{
\begin{lemma}\label{lemma:grbis}
Denote with $\Psi_A^B$ {a} weakly balanced UH system according to Definition \ref{defbalance}. Then for any $(l,k){\in A}$,   
\[
\min\{ |L_{lk}|,|R_{lk}| \} = \frac{m_{lk}}{2^{l+D}}\quad\text{ for some}\quad m_{lk}\in\{1,\dots, C+l\}.
\] 
%Then, in weakly balanced systems with balancing constants $(C,D)$,  the granularity is no larger than $l+D$.           
\end{lemma}
\begin{proof}
We prove the statement by induction. First, from the definition of weak balancedness, we have $\min\{|L_{00}|,|R_{00}|\}=1-M_{00}/2^{D}=j/2^D$ (for $j=2^D-M_{00}$) and by definition this is less than $M_{00}/2^D\le C/2^D$, so $j \le C$.
  Assume now that the statement holds for $l-1\geq 0$ and consider a  node $(l,k){\in A}$ for some $0\leq k<2^l$. 
  %The wavelet $\psi_{l-1\,\lfloor k/2\rfloor}^B$ of the parent node $(l-1,\lfloor k/2\rfloor)$ is supported on an interval $I_{l-1\,\lfloor k/2\rfloor}^B$ and has left and right pieces $L_{l-1\,\lfloor k/2\rfloor}$ and $R_{l-1\,\lfloor k/2\rfloor}$, respectively.   
The union $L_{lk} \cup R_{lk}$ is either $L_{l-1\,\lfloor k/2\rfloor}$ or $R_{l-1\,\lfloor k/2\rfloor}$; without loss of generality, suppose it is $R_{l-1\,\lfloor k/2\rfloor}$.
%is the right  child of the node $(l-1,\lfloor k/2\rfloor)$ (i.e. $k$ is odd)
%where the wavelet $\psi_{lk}^B$ has a support $I_{lk}^B=R_{l-1\,\lfloor k/2\rfloor}$.
  Then, from weak balancedness, we find
\begin{equation}\label{boundmin}
\min\{|L_{lk}|,|R_{lk}|\}=|R_{l-1\, \lfloor k/2\rfloor}|-M_{lk}/2^{l+D}.
\end{equation}
 If $|R_{l-1\, \lfloor k/2\rfloor}|\leq |L_{l-1\, \lfloor k/2\rfloor}|$, we use induction to find   $|R_{l-1\, \lfloor k/2\rfloor}|=j_1/2^{l-1+D}$ for some $j_1\in\{1,\dots,C+l-1\}$ and thereby \eqref{boundmin} equals $j/2^{l+D}$ for $j=2j_1-M_{lk}$. As this is at most $M_{lk}/2^{l+D}=\max\{|L_{lk}|,|R_{lk}|\}$, one deduces $M_{lk}\ge j_1$ and then $j/2^{l+D}\le j_1/2^{l+D}$ with $j_1\le C+l-1\le C+l$. If $|R_{l-1\, \lfloor k/2\rfloor}|>|L_{l-1\, \lfloor k/2\rfloor}|$, we again use weak balancedness to write  \eqref{boundmin} as  $j/2^{l+D}$ with $j=2M_{l-1\,\lfloor k/2\rfloor}-M_{lk}\le M_{lk}$, 
 %M_{l-1\,\lfloor k/2\rfloor}$
  using again $M_{lk}/2^{l+D}=\max\{|L_{lk}|,|R_{lk}|\}$, so that $j$ is again less than $C+l$. The result follows by induction.
\end{proof}
}

\subsection{Complexity lemma}

\begin{lemma}\label{lemma:complex}
Consider  a  {\em weakly balanced} UH wavelet system $\Psi_{A}^B=\{\Psi_{lk}^B:(l,k)\in A\}$ according to \eqref{eq:balance2} and let $f\in\mathcal{H}^\alpha_M$.  Then  the following conditions hold  for $\delta=3$, with constants independent of $B$: for any $(l,k)\in A$
\begin{enumerate} 
\item[(B1)] the  basis function $\psi_{lk}^B$  can be expressed as a linear combination of  {\sl at most} $C_0l^\delta$ Haar functions  $\psi_{jk}$ for $j\le l+D$  and some $C_0>0$, and
\item[(B2)] there exists $C_1>0$ {(depending only on $E,D$ from \eqref{eq:balance2})} such that $ |\beta_{lk}^B|\le {C_1 M} l^{\delta/2} 2^{-l(\al+1/2)}$.
\end{enumerate}
\end{lemma}

\begin{proof}
First, the function $\psi_{lk}^B$ belongs to $\text{Vect}\{\1_{I_{(l+D)m}}: 0\le m< 2^{l+D}\}$ and the support of $\psi_{lk}^B$ spans at most $2(E+l)$ of the indicators $\1_{I_{(l+D)m}}$. These indicators  can be expressed in terms of at most $l+D$ of $\psi_{lk}$'s (one per level above $l+D$), which yields  an upper bound $2(E+l)(l+D)\asymp l^2$ and thereby (B1) with $\delta=2$. Second, the balancing condition \eqref{eq:balance2} gives $\max\{|L_{lk}|, |R_{lk}|\} \le (E+l)2^{-l-D}$ which, combined with Lemma \ref{lemma:haar} implies
\[ |\beta_{lk}^B| \le M2^{\alpha-1}(E+l)^{\al+1/2} 2^{-(l+D)(\al+1/2)}\leq C_1 M  l^{3/2} 2^{-l(\al+1/2)},\]
 by taking the worst case $\al=1$, which proves (B2) with $\delta=3$.
\end{proof}

\subsection{The quantile example}\label{ap:sec_quantiles}

\begin{lemma}\label{lem-quantiles}
The quantile system $\Psi_{A}^B$ from Example \ref{exquantiles} is weakly balanced in the sense of Definition \ref{defbalance} for balancing constants satisfying 
$E=2+3\,C_q2^{D-1}$, where $\|1/g\|_\infty\leq C_q$ and $\|g\|_\infty<2^{D-1}/(2E)$.
\end{lemma}
\begin{proof}
Let us start by writing explicitly the intervals $L_{lk}, R_{lk}$. Assuming without loss of generality that $k$ is odd, i.e. $(l,k)$ is the right child node,
\begin{align*}
|L_{lk}|&=b_{lk}-b_{(l-1)\lfloor k/2\rfloor}=G_{\lmax+D}^{-1}[(2k+1)/2^{l+1}]-G_{\lmax+D}^{-1}[(2\lfloor k/2\rfloor+1)/2^{l}],\\
|R_{lk}|&=b_{(l-2)\lfloor k/4\rfloor}-b_{lk}=G_{\lmax+D}^{-1}[(2\lfloor k/4\rfloor+1)/2^{l-1}]-G_{\lmax+D}^{-1}[(2k+1)/2^{l+1}].
\end{align*}
 We first show by contradiction that $\max\{|L_{lk}|,|R_{lk}|\}\geq E/2^{l+D}$ for $E\geq 1$.  Let us denote $y_1=G^{-1}[(2k+1)/2^{l+1}]$, 
$y_2=G^{-1}[(2\lfloor k/2\rfloor+1)/2^{l}]$ and $y_3=G^{-1}[(2\lfloor k/4\rfloor+1)/2^{l-1}]$. Assuming  $|L_{lk}|<E/2^{l+D}$, one obtains
$$
\lfloor 2^{\lmax+D} y_1\rfloor-\lfloor 2^{\lmax+D}y_2 \rfloor< E 2^{\lmax-l},
$$
and thereby
$y_1-y_2 < E2^{-l-D+1}.$
Next, using the fact that $k$ is odd, 
\begin{align*}
\frac{1}{2^{l+1}} & = |(2k+1)/2^{l+1}-(2\lfloor k/2\rfloor+1)/2^{l}|\\
& =|G(y_1)-G(y_2)| \leq \|g\|_\infty |y_1-y_2| \leq \|g\|_\infty 2E2^{-l-D},
\end{align*}
which yields a contradiction when $ \|g\|_\infty<2^{D-1}/(2E)$. Similarly, when  $|R_{lk}|<E/2^{l+D}$, one obtains
\begin{align*}
 1/2^{l+1} & < |(2k+1)/2^{l+1}-(2\lfloor k/4\rfloor+1)/2^{l-1}|\\
 & =|G(y_3)-G(y_1)|\leq \|g\|_\infty |y_3-y_1|< \|g\|_\infty 2E\,3\,2^{-l-D}.
\end{align*}   
 Next, we note that for $\|1/g\|_\infty\leq C_q$ and $E:= \left(2+3\,C_q2^{D-1}\right)$,
 \begin{align*}
 |R_{lk}| & =\frac{1}{2^{\lmax+D}}\left[\lfloor 2^{\lmax+D} y_1\rfloor-\lfloor 2^{\lmax+D}y_3 \rfloor\right] \\
 & \leq 
\frac{2}{2^{\lmax+D}}+ \left\|\frac{1}{g}\right\|_\infty\frac{3}{2^{l+1}}\leq\frac{E}{2^{l+D}}.
 \end{align*}
Similarly, one obtains $|L_{lk}|<E/2^{l+D}$, which concludes that the quantile system is weakly balanced.%$\qedhere$
\end{proof}

\section{Remaining proofs for Section \ref{sec:dyadic}}
%\subsection{Proofs of Theorems 2-7}
\subsection{Proof of Theorem \ref{thm-one}: remaining settings}\label{sec:thm-one}
In this section, we provide the proof of Theorem \ref{thm-one} under more general prior settings.
We consider different tree priors $\Pi_\bT(\mT)$ from Section \ref{sec:prior_trees} (i.e. the conditionally uniform prior in \eqref{prior:K}  and the exponential-type prior mentioned in Extension \ref{remark:general}). In addition, we consider  the prior
\begin{equation}\label{eq:prior_cov}
\pi(\b_\mT)\sim\mathcal{N}(0,\Sigma_\mT)
\end{equation}
with $\Sigma_\mT=g_n(A_\mT'A_\mT)^{-1}$ (the $g$-prior) as well as more general covariance structures $\Sigma_\mT$ that satisfy 
\begin{equation}\label{eq:eigenvalues}
\lambda_{min}(\Sigma_{\mT})\gtrsim 1/\sqrt{\log n}\quad\text{and}\quad \lambda_{max}(\Sigma_{\mT})\lesssim n^a\quad\text{for some $a>0$}.
\end{equation}
First, we show that Lemma \ref{lemma:dim} holds under this setting.
\subsubsection{Proof of Lemma \ref{lemma:dim} (general setting)}\label{sec:proof_lemma1_gprior}
Recall the ratio \eqref{eq:ratio_W} from Section \ref{sec:dim}
\begin{align}
\frac{W_X(\mT)}{W_X(\mT^-)}&=\frac{\Pi_{\bT}(\mT)}{\Pi_{\bT}(\mT^-)}\frac{\sqrt{| I+n\Sigma_{\mT^-}|}}{\sqrt{| I+n\Sigma_{\mT}|}}\frac{\e^{n^2\X_{\mT}'(n I+\Sigma_{\mT}^{-1})^{-1}\X_{\mT}/2}}{\e^{n^2\X_{\mT^-}'(n I+\Sigma_{\mT^-}^{-1})^{-1}\X_{\mT^-}/2}}.
\end{align}

{\sl The $g$-prior.}  We first focus on bounding this ratio assuming $\Sigma_{\mT}=g_n\,( A_\mT' A_\mT)^{-1}$. {Proposition \ref{prop1} implies 
\begin{equation} \label{tech1}
nI_K+\Sigma_{\mT}^{-1} =
\left(
\begin{matrix}
nI_{K-1}+\Sigma_{\mT^-}^{-1}+\frac{1}{g_n}\bm v\bm v'& \bm 0\\
\bm 0' & n+\frac{2^{l_1+1}}{g_n}\\
\end{matrix}
\right),
%=
%\left(\begin{matrix}
%M-\frac{M\bm v\bm v'M}{g_n+\bm v' M\bm v}& \bm 0\\
%\bm 0' & 1/(n+2^{l_1+1}/g_n)\\
%\end{matrix}
%\right),
\end{equation}
where the vector $\bm v\in\R^{|\mT_{ext}|-1}$ (defined in Proposition \ref{prop1})  is obtained from $A_{\mT}$ by first deleting its last column and then transposing the last row of this reduced matrix.
%where $M=(nI_{K-1}+\Sigma_{\mT^-}^{-1})^{-1}$
Using the determinant formula $|A+\bm u\bm u'|=|A|(1+\bm u'A^{-1}\bm u)$ for $A$ invertible} (see Lemma 1.1 in \cite{determinant_lemma}), and setting $M=(nI_{K-1}+\Sigma_{\mT^-}^{-1})^{-1}$, one gets
\begin{equation}\label{eq:ratio_det}
\frac{{| nI_{K-1}+\Sigma_{\mT^-}^{-1}|}}{{| nI_{K}+\Sigma_{\mT}^{-1}|}}=\frac{1}{(n+2^{l_1+1}/g_n)\left[1+\bm v'M\bm v\right/g_n]}<\frac{1}{(n+2^{l_1+1}/g_n)}.
\end{equation}
 Using again the determinant formula and Proposition \ref{prop1}, we can write
$$
\frac{{|\Sigma_{\mT^{-}}|}}{{|\Sigma_\mT|}}=\frac{|A_\mT'A_\mT|}{|A_{\mT^-}'A_{\mT^-}|}=
\frac{2^{l_1+1}}{g_n}\left(1+\bm v'( A_{\mT^-}' A_{\mT^-})^{-1}\bm v\right)\le \frac{2^{l_1+1}}{g_n}\left(1+2^{l_1}\right).
$$
The  inequality above uses two facts. First, we have  $\|\bm v\|_2^2=1+\sum_{l=0}^{l_1-1}2^l=2^{l_1}$ which follows from the definition of $\bm v$ which describes its entries as amplitudes of the Haar wavelets in the ancestry of the deepest rightmost internal node $(l_1,k_1)$ (using the notation from Lemma \ref{lemma:dim}). Second, from Proposition \ref{prop:eigenspectrum} we have $\lambda_{max}[( A_{\mT^-}' A_{\mT^-})^{-1}]=1/\lambda_{min}( A_{\mT^-}' A_{\mT^-})\leq 1$. Indeed,  Proposition \ref{prop:eigenspectrum} (combined with \eqref{eq:lemma1}) shows that the smallest eigenvalue of $ A_\mT' A_\mT$ equals $2^{d}$ where $d$ is the depth of the most shallow  leaf node in $\mT_{ext}$.
Let us now set 
\begin{equation}\label{eq:define_D}
D:=\X_{\mT}'(n I+\Sigma_{\mT}^{-1})^{-1}\X_{\mT}-
\X_{\mT^-}'(n I+\Sigma_{\mT^-}^{-1})^{-1}\X_{\mT^-}.
\end{equation} 
Combining with \eqref{tech1}, it follows from a variant of the Sherman--Morrison's matrix inversion formula  (Lemma \ref{woodbury}) that
\[ (nI_K+\Sigma_{\mT}^{-1})^{-1} 
=
\left(\begin{matrix}
M-\frac{M\bm v\bm v'M}{g_n+\bm v' M\bm v}& \bm 0\\
\bm 0' & 1/(n+2^{l_1+1}/g_n)\\
\end{matrix}
\right), \]
from which one deduces that
\begin{align} \label{techD}
D=&\frac{X_{l_1k_1}^2}{n+2^{l_1+1}/g_n} -\frac{\X_{\mT^-}'M\bm v\bm v'M\X_{\mT^-}}{g_n+\bm v'M\bm v} < \frac{X_{l_1k_1}^2}{n+2^{l_1+1}/g_n}.
\end{align}
Since for $l_1>\mathcal{L}_c$ we have $2X_{l_1k_1}^2\leq 5 \log n/n$, we can write
$$
\frac{W_X(\mT)}{W_X(\mT^-)}\frac{\Pi_{\bT}(\mT^-)}{\Pi_{\bT}(\mT)}{<}\sqrt{\frac{2^{l_1+1}\left(1+2^{l_1}\right)}{g_n(n+2^{l_1+1}/g_n)}}\e^{\frac{X_{l_1k_1}^2n^2}{2(n+2^{l_1+1}/g_n)}}<\sqrt{\frac{2^{2(l_1+1)}}{ng_n}}n^{5/4}.
$$ 
For $g_n=n$, and for $\mT\in\bT_d$, so that $2^{l_1}\leqa 2^{d}$, the last display is bounded by a constant times ${n^{-1/4} 2^d p_d}$,   %we obtain a bit tighter upper  bound than  in \eqref{W_ratio},
 and the argument can be completed in similar vein as before, with now $\Pi[d(\mT)>\cL_c\given X]=o_P(1)$ if ${\Gamma>2\e^{5}}$. % \re{[Please double-check. One should include the $2^d$ coming from non-injectivity.]}

\medskip

{\em General Covariance $\Sigma_\mT$}.  We now show how the proof can be modified by assuming
a general covariance matrix $\Sigma_{\mT}$ on the internal wavelet coefficients.  
Recall  again the ratio \eqref{eq:ratio_W} from Section \ref{sec:dim}. %Denote with $S_{\mT,\mT^-}$ the Schur complement of $(I+n\Sigma_{\mT^-)$ in $(I+n\Sigma_\mT)$, then 
We use the Cauchy's interlace theorem for eigenvalues of Hermitian matrices which states that the eigenvalues of a principal submatrix are interlaced within eigenvalues of the original matrix ({Theorem 8.1.7 of  \cite{golub})}. 
Since $\Sigma_{\mT^-}$ is a $(K-1)\times (K-1)$ dimensional  submatrix of a $(K\times K)$ dimensional matrix $\Sigma_{\mT}$, we have (denoting with $\lambda_l(\Sigma)$ the $l^{th}$ largest eigenvalue of $\Sigma$)
$$
\lambda_1(\Sigma_\mT)\geq \lambda_1(\Sigma_{\mT^-})\geq\lambda_2(\Sigma_\mT)\geq\cdots\geq \lambda_{K-1}(\Sigma_\mT)\geq 
 \lambda_{K-1}(\Sigma_{\mT^-})\geq  \lambda_{K}(\Sigma_{\mT})
$$
and thereby
$$
\frac{|I+n\Sigma_{\mT^-}|}{|I+n\Sigma_{\mT}|}=\frac{\prod_{l=1}^{K-1}[1+n\lambda_l(\Sigma_{\mT^-})]}{
\prod_{l=1}^{K}[1+n\lambda_l(\Sigma_{\mT})]}\leq \frac{1}{1+n\lambda_{min}(\Sigma_{\mT})}.
$$
Using the matrix inversion formula $(I+B)^{-1}=I-(I+B^{-1})^{-1}$ (a variant of Sherman-Morrison-Woodbury formula), we get
$$
(nI+\Sigma_{\mT}^{-1})^{-1}=\frac{1}{n}\left[ I-(I+n\Sigma_\mT)^{-1}\right]
$$
%$$
%(nI+\Sigma_{\mT}^{-1})^{-1}=\Sigma_{\mT}-n\Sigma_{\mT}(I+n\Sigma_{\mT})^{-1}\Sigma_{\mT}
%$$
and thereby
$$
\X_{\mT}'(n I+\Sigma_{\mT}^{-1})^{-1}\X_{\mT}=\frac{1}{n}\|\X_{\mT}\|_2^2- \frac{1}{n}\X_{\mT}'(I+n\Sigma_{\mT})^{-1}\X_{\mT}.
$$
Writing $\X_\mT=(\X_{\mT^-},X_{l_1 k_1})'$, where $(l_1,k_1)$ it the deepest rightmost internal node in $\mT$ (as in Section \ref{sec:dim}), and using the definition of $D$ in \eqref{eq:define_D}, we have
\begin{align*}
D&=\frac{X_{l_1k_1}^2}{n} -\frac{1}{n}\left[\X_{\mT}'(I+n\Sigma_{\mT})^{-1}\X_{\mT}-
\X_{\mT^-}'(I+n\Sigma_{\mT^-})^{-1}\X_{\mT^-}\right]\\
&<\frac{X_{l_1k_1}^2}{n}\left(1 -\frac{1}{1+n\lambda_{max}(\Sigma_{\mT})}\right)\\
&\qquad\qquad+\frac{\|\X_{\mT^-}\|_2^2}{n}\left(\frac{1}{1+n\lambda_{min}(\Sigma_{\mT^-})}-\frac{1}{1+n\lambda_{max}(\Sigma_{\mT})}\right).
\end{align*}
This inequality follows from the fact that 
$$
\X_{\mT}'(I+n\Sigma_{\mT})^{-1}\X_{\mT}\geq\|\X_{\mT}\|_2^2\lambda_{min}[(I+n\Sigma_{\mT})^{-1}]
=\frac{\|\X_{\mT^-}\|_2^2+X_{l_1k_1}^2}{1+n\lambda_{max}(\Sigma_{\mT})}.
$$
It follows from the proof of Lemma \ref{lemma:dim} that $X_{l_1k_1}^2\lesssim \log n/n$ and $\|\X_{\mT^-}\|_2^2\leq C_1$ (as was shown in \eqref{eq:signal_bound}).
Moreover, from our assumption \eqref{eq:prior_cov} we have $\lambda_{min}(\Sigma_{\mT^-})\geq 1/\sqrt{\log n}$ and thereby
\begin{align*}
\frac{W_X(\mT)}{W_X(\mT^-)}&<\sqrt{\frac{\log n}{n}}\frac{\Pi_{\bT}(\mT)}{\Pi_{\bT}(\mT^-)}
\exp\left\{\frac{nX_{l_1k_1}^2}{2}+\frac{n\|\X_{\mT^-}\|_2^2}{2(1+n\lambda_{min}(\Sigma_{\mT^-}))}\right\}\\
&<\frac{\Pi_{\bT}(\mT)}{\Pi_{\bT}(\mT^-)}\e^{C_2\log n}.
\end{align*}
Proceeding as in the proof of Lemma \ref{lemma:dim}, one can show \eqref{eq:dim} for a suitably large $\Gamma>0$.

{\sl Other Tree Priors $\Pi_\bT(\mT)$.} %The proof can be carried over to the other two tree priors in the statement of Theorem \ref{thm-one}. 
The only modification needed to carry over the proof to the other two priors is the bound for the ratio $\Pi_\bT(\cT)/\Pi_\bT(\cT^-)$.
% for too large trees and in the ratio $\Pi(\cT)/\Pi(\cT_+)$ for missing a significant node. 
Consider the conditionally uniform prior from Section \ref{sec:bc1} defined in \eqref{prior:K}. % with the Poisson distribution on the number of leaves (constrained to $\N\backslash\{0\}$).
Denoting $K=|\mT_{ext}|$ and $\mathbb{C}_K$ the number of full binary trees with $K+1$ leaves, we have
\[
\frac{\Pi_\bT(\cT)}{\Pi_\bT(\cT^-)} 
= \frac{\pi(K)}{\pi(K-1)}\frac{\mathbb{C}_{K-1}}{\mathbb{C}_{K-2}}, 
\]
%where $\mathbb{C}_{K}:=(2K)!/\{(K+1)!K!\}$ is the Catalan number (which equals the number of full binary trees with $K+1$ leaves). Stirling's formula yields $\mathbb{C}_K\asymp 4^K/K^{3/2}$ for large $K$.
and Lemma \ref{catalan} now implies, for a universal constant $C>0$,
\[\frac{\Pi_\bT(\cT)}{\Pi_\bT(\cT^-)} 
\lesssim \frac{\lambda}{K} \frac{4^{K-1} (K-2)^{3/2} }{4^{K-2} (\{K-1\}\vee 1)^{3/2}}\le C \lambda/K. \]
Choosing $\lambda=1/n^c$ for some {$c>7/4$}, it follows that  
$$
\Pi[ d(\cT) > \cL_c \given X]  \le 4\,\lambda\,n^{3/4}\sum_{d=\cL_c+1}^{L}{2^d} /K\le 4\,\lambda\,n^{3/4} {2^L}\rightarrow 0.
$$
Finally, for the exponential-type prior mentioned in Extension \ref{remark:general}),  one has
$\Pi_\bT(\cT)/\Pi_\bT(\cT^-)= 1/n^c$,  {so one can argue similarly}.

\subsubsection{Proof of Lemma \ref{lemma:sig} (general setting)}
We recall the ratio  
\begin{equation}\label{eq:ratio_W22}
\frac{W_X(\mT)}{W_X(\mT^+)}=\frac{\Pi_{\bT}(\mT)}{\Pi_{\bT}(\mT^+)}
\prod_{j=0}^s\frac{ N_X(\mT^{(j-1)}) }{ N_X(\mT^{(j)}) }.
%=\frac{\Pi_{\bT}(\mT)}{\Pi_{\bT}(\mT^+)}\frac{N_X(\mT^{cut}_{\wt l})}{N_X(\mT^+)}
%\prod_{l=\wt l}^{d-1} \frac{N_X(\mT^{cut}_{l+1})}{N_X(\mT^{cut}_{l})}
\end{equation}
in \eqref{eq:ratio_W2} and find an upper bound assuming different priors.

\medskip
{\sl The $g$-prior.} We now modify the proof of Lemma \ref{lemma:sig} for the  $g$-prior obtained with $\Sigma_\mT=g_n(A_\mT'A_\mT)^{-1}$. 
Denoting with $K_j=|\mT^{(j)}_{ext}|$ and because $\mT^{(j-1)}$ is obtained from $\mT^{(j)}$ by removing two children nodes (or, equivalently, an internal node $X_{l_jk_j}$), we can apply Proposition \ref{prop1}  to obtain the following upper bound for $N_X(\mT^{(j-1)})/N_X(\mT^{(j)})$. 
Namely, going back to \eqref{eq:define_D} and \eqref{techD}, we invoke again the matrix determinant lemma $|A+\bm u\bm u'|=|A|(1+\bm u'A^{-1}\bm u)$ and the matrix inversion lemma (Lemma \ref{woodbury}) as in \eqref{eq:ratio_det} to obtain
{\begin{align*}%\label{eq:ratio_N}  
\frac{N_X(\mT^{(j-1)})}{N_X(\mT^{(j)})}\leq& \sqrt{\frac{(g_nn+2^{l_{j}+1})(g_n+\bm v'M \bm v)}{g_n^2}}\\
&\times \exp\left\{-\frac{n^2\,X_{l_jk_j}^2}{2(n+2^{l_j+1}/g_n)}+
n^2\frac{\X_{\mT^{(j-1)}}'M\bm v\bm v' M\X_{\mT^{(j-1)}}}{2(g_n+\bm v'M\bm v)}\right\},
\end{align*}}
\noindent where $M=(nI_{K_{j-1}}+\Sigma^{-1}_{\mT^{(j-1)}})^{-1}$ for   $\bm v\in\R^{|\mT^{(j-1)}|}$ which depends on $\mT^{(j-1)}$ according to Proposition \ref{prop1}. 
Next, for $C>0$ a large enough constant and $l_{j-1}$ the depth of the deepest internal node in $\mT^{(j-1)}$,
%we can write for some constants $C_1,C_2,C_3,C_4>0$
\begin{align}\label{eq:signal_bound}
\|\X_{\mT^{(j-1)}}\|_2^2&\leq X_{-10}^2+\sum_{l=0}^{l_{j-1}}\sum_{k=0}^{2^l-1}X_{lk}^2\leq C\left[1+\sum_{l=0}^{l_{j-1}}\left(2^{-2\,l\,\alpha} + \frac{2^l}{n}\,\log n\right)\right]\leq C_4.
\end{align}
Moreover, using the fact that $\bm v' MM\bm v=\bm v' M^{1/2}M^{1/2}M\bm v\leq \lambda_{max}(M)\bm v' M\bm v$ we obtain
\begin{align*}
\X_{\mT^{(j-1)}}'M\bm v\bm v' M\X_{\mT^{(j-1)}}&\leq \|\X_{\mT}^{(j-1)}\|^2_2\lambda_{max}(M\bm v\bm v' M)
\leq  \|\X_{\mT}^{(j-1)}\|^2_2\mathrm{tr}(M\bm v\bm v' M)\\
&\leq  \|\X_{\mT}^{(j-1)}\|^2_2\bm v' MM\bm v \leq \|\X_{\mT}^{(j-1)}\|^2_2\lambda_{max}(M)\bm v' M\bm v,
\end{align*}
and 
$$
\lambda_{\max}(M)=\frac{1}{n+\lambda_{min}( A_{\mT^{(j-1)'}} A_{\mT^{(j-1)}})/g_n}<\frac{1}{n},
$$
 one can write
\begin{align*}
n^2\frac{\X_{\mT^{(j-1)}}'M\bm v\bm v' M\X_{\mT^{(j-1)}}}{2(g_n+\bm v'M\bm v)}&<\frac{n^2\|\X_{\mT^{(j-1)}}\|_2^2\lambda_{max}(M)}{2g_n/(\bm v'M\bm v)+2}\leq
\frac{n^2C_4\|\bm v\|_2^2\lambda_{max}^2(M)}{2g_n}\\
&\leq  C_4\frac{2^{l_j}}{2g_n},
\end{align*}
where we used the fact that $\|\bm v\|_2^2=2^{l_j}$ (as explained previously in Section \ref{sec:proof_lemma1_gprior}). Finally, because $X^2_{l_sk_s}\geq C_5A^2\log^2n/n$ for some $C_5>0$ we use the above bounds in the expression \eqref{eq:ratio_W22} to obtain 
\begin{align*}
\prod_{j=0}^s\frac{N_X(\mT^{(j-1)})}{N_X(\mT^{(j)})}&<\left(\frac{n(g_n+2)}{g_n}\right)^{s+1}\exp\left\{  C_4\sum_{j=0}^s\frac{2^{l_j-1}}{g_n}-\frac{nC_5A^2\log^2n}{2(n+2^{{l_j}+1}/g_n)}\right\}\\
&<\exp\left\{(s+1)\log(3n)+ C_4 \frac{(s+1)2^{l_s-1}}{g_n}-C_5A^2\log^2n/{2}\right\}.
\end{align*}
With $g_n=n$, the exponent is dominated by the last term. One then proceeds with \eqref{eq:ratio_W2} as  before in the proof of Lemma \ref{lemma:sig}.
\medskip

 {\sl General Covariance $\Sigma_\mT$.} 
We again deploy  the interlacing eigenvalue theorem in  the expression \eqref{eq:ratio_W2} to obtain
the following upper bound for $\frac{N_X(\mT^{(j-1})}{N_X(\mT^{(j)})}$ (using matrix determinant and inversion lemmata as before)
$$
\sqrt{1+n\lambda_{\max}(\Sigma_{\mT^{(j)}})}\exp\left\{-\frac{nX_{l_jk_j}^2}{2}\frac{n\lambda_{min}(\Sigma_{\mT^{(j)}})}{1+
n\lambda_{min}(\Sigma_{\mT^{(j)}})}+\frac{n\|\X_{\mT^{(j-1)}}\|_2^2}{2(1+n\lambda_{min}(\Sigma_{\mT^{(j)}}))}\right\}.
$$
Using the expression \eqref{eq:ratio_W2}  and assumptions $\lambda_{max}(\Sigma_{\mT})\lesssim n^a$ for some $a\geq 1$ and $\lambda_{min}(\Sigma_{\mT^{(j)}})\geq 1/\sqrt{\log n}$, we obtain for $C_2,C_3>0$
$$
\frac{W_X(\mT)}{W_X(\mT^+)}<\frac{\Pi_{\bT}(\mT)}{\Pi_{\bT}(\mT^+)}\exp\left\{C_2(s+1)\sqrt{\log n}-C_3A^2\log^2n\right\}.
$$
Using this bound, one can proceed as in the proof of Lemma \ref{lemma:sig} and show \eqref{eq:sig}.

\medskip

{\sl Other Tree Priors $\Pi_\bT(\mT)$.} As before, %The proof can be carried over to the other two tree priors in the statement of Theorem \ref{thm-one}. 
the only modification needed is the bound for  $\Pi_\bT(\cT)/\Pi_\bT(\cT^+)$.
% for too large trees and in the ratio $\Pi(\cT)/\Pi(\cT_+)$ for missing a significant node. 
Denote by  $K^+=|\mT^+_{ext}|$ and $K=|\mT_{ext}|$ and note that $K^+-K=l_s-l_0$.  For the conditionally uniform prior from Section \ref{sec:bc1} we  then have
\begin{align*}
\frac{\Pi_\bT(\mT)}{\Pi_\bT(\mT^+)} & =\lambda^{-(l_s-l_0)}\frac{K^+!\mathbb{C}_{K^+-1}}{K!\mathbb{C}_{K-1}}\lesssim \left(\frac{\lambda}{4}\right)^{-(l_s-l_0)}\frac{(K^+)!}{
K! }\\
%\leq  \left(\frac{\lambda}{4}\right)^{-(l_S-l_0)}\e^{(l_S-l_0)\log [K+(l_S-l_0)]}\\
& \leqa \left(\frac{\lambda}{4}\right)^{-(l_s-l_0)} {\e^{(l_s-l_0)\{1+\log [K+(l_s-l_0)]\}}}.
\end{align*}
With $\lambda=n^{-c}${ and $K\le 2^{\cL_c+1}$},  this is bounded from above by  $C\e^{C\,\log^2n}$ for some $C>0$ and the proof is completed as before. 
For the exponential-type prior, one similarly uses 
$\Pi_\bT(\cT)/\Pi_\bT(\cT^+) =\e^{c(l_s-l_0)\log n}\leq \e^{c\log^2 n}. $

\subsubsection{Proof of Lemma \ref{lemma:signal} (general setting)}\label{suppl:subsec:signal}
We now show the proof of Lemma \ref{lemma:signal} for the $g$-prior with $\Sigma_\mT=g_n(A_\mT'A_\mT)^{-1}$ and  a general covariance matrix under the assumptions \eqref{eq:eigenvalues}. Recall  expressions \eqref{eq:exp} and \eqref{tech2} in Section \ref{subsec:concentration} and the fact that
 $(I+B)^{-1}=I-(I+B^{-1})^{-1}$ which yields $n\wt\Sigma_{\mT}-I_{K}=-(I_{K}+n\Sigma_{\mT})^{-1}$  for $\wt\Sigma_\mT=(nI_K+\Sigma_\mT^{-1})^{-1}$ (recalling that $K=|\mT_{ext}|\leq n$). 
Focusing on the $g$-prior, we  have $\lambda_{max}(\wt\Sigma_{\mT})<1/n$ and  (using Proposition \ref{prop:eigenspectrum} which yields $\lambda_{max}( A_\mT' A_\mT)<2^{d}\leq n$ where $d$ is the depth  of the deepest node)
 $$
 \lambda_{max}(I_{K}+n\Sigma_{\mT})^{-1}<\lambda_{max}( A_\mT' A_\mT)/(ng_n)<1/g_n.
 $$    Assuming $g_n=n$   we can thus write  
\begin{equation}\label{eq:part1}
\|(n\wt\Sigma_{\mT}-I_{K})\b_\mT^0\|_\infty\leq \frac{\|\b_\mT^0\|_\infty\sqrt{K}}{1+n\lambda_{min}(\Sigma_{\mT})}
\leq \frac{C\sqrt{K} \lambda_{max}( A_\mT' A_\mT)}{n\,g_n}\leq C/\sqrt{n}.
\end{equation}
Next, we note that $\wt\Sigma_{\mT}^{-1}$ is strictly diagonally dominant. Indeed, with $g_n=n$ we have
$\wt\Sigma_{\mT}^{-1}=nI_K+\frac{1}{g_n}A_\mT'A_\mT$ and 
$$
\frac{1}{g_n}\|A_\mT'A_\mT\|_\infty\leq \frac{\sqrt{K}}{g_n}\lambda_{max}(A_\mT'A_\mT)<\sqrt{n},
$$
where $\|A\|_\infty=\max\limits_{1\leq i\leq m}\sum_{j=1}^n|a_{ij}|$ is defined as the maximum absolute row sum of an $(m\times n)$ matrix $A$.
Writing $A_\mT'A_\mT=(a_{ij})_{i,j}^{K,K}$, {it then follows from 
  Lemma \ref{lemma:varah}  (Theorem 1 in \cite{varah})  that
\begin{equation}\label{eq:part2}
\|\wt\Sigma_{\mT}\|_\infty\leq \frac{1}{n+\frac{1}{g_n}\min\limits_{1\leq k\leq K}\Delta_k},\quad\text{where}\quad \Delta_k=|a_{kk}|-\sum_{j\neq k}|a_{kj}|.
\end{equation}}
Since $\Delta_k/g_n>-\frac{1}{g_n}\|A_\mT'A_\mT\|_\infty>-\sqrt{n}$ and using the fact that   $\|\bm\varepsilon_\mT\|_\infty\lesssim \sqrt{\log n}$ on the event $\mathcal{A}$, we obtain
\begin{equation}\label{bound_eq}
\sqrt{n}\|\wt\Sigma_{\mT}\bm\veps_\mT\|_\infty\leq \sqrt{n}\|\wt\Sigma_{\mT}\|_\infty\|\bm\veps_\mT\|_\infty\lesssim \sqrt{\frac{\log n}{n}}.
\end{equation}
%\begin{align}
 % \|\bm{\mu}_{\mT}-\b^0_\mT\|_\infty&\leq \sqrt{n|\mT_{ext}|}\lambda_{max}(\wt\Sigma_{\mT})\|\bm\veps_\mT\|_\infty+\frac{\|\b_\mT^0\|_\infty\sqrt{K}}{(1+n\lambda_{min}(\Sigma_{\mT}))}\\
 %&\leq \sqrt{\frac{C\,|\mT_{ext}|\log n}{n}}+\frac{M\sqrt{|\mT_{ext}|} \lambda_{max}( A_\mT' A_\mT)}{n\,g_n}\lesssim \sqrt{\frac{|\mT_{ext}|\log n}{n}},
 %\end{align}
%where we used the fact that
%  $\|\bm\varepsilon_\mT\|_\infty\lesssim \sqrt{\log n}$ on the event $\mathcal{A}$. 
The sum of the remaining two terms in \eqref{eq:exp}  can be bounded by a multiple of $\sqrt{\log n/n}$ by noting that 
$
\bar\sigma^2\leq \|\wt\Sigma_{\mT}\|_\infty\lesssim 1/n.
%\leq \sqrt{|\mT_{ext}|}\lambda_{max}(\wt\Sigma_{\mT})\leq \sqrt{|\mT_{ext}|}/n.
$
The statement \eqref{post_signal} then follows from \eqref{eq:exp}.

\medskip
For the general covariance matrix $\Sigma_\mT$ we find (similarly as in \eqref{eq:part1}) that when $\lambda_{min}(\Sigma_{\mT})\geq 1/\sqrt{\log n}$ we have 
$
\|(n\wt\Sigma_{\mT}-I_K)\b^0_\mT\|_\infty\leq C\sqrt{\log n/n}.
$
Next, because 
$$
\|\Sigma^{-1}_\mT\|_\infty\leq \sqrt{K}\lambda_{max}(\Sigma_{\mT}^{-1})\leq \sqrt{K\log n}<\sqrt{n\log n}
$$ 
the matrix $\wt\Sigma_{\mT}=(nI_K+\Sigma_\mT^{-1})^{-1}$ is diagonally dominant and thereby (using Lemma \ref{lemma:varah})
$$
\|\wt\Sigma_\mT\|_\infty\leq \frac{1}{n+\min\limits_{1\leq k\leq K}\Delta_k}\quad\text{where}\quad \Delta_k=|\sigma_{kk}|-\sum_{j\neq k}|\sigma_{kj}|
$$
and where $\Sigma_\mT^{-1}=(\sigma_{jk})_{j,k=1}^{K,K}$. Since $\Delta_k>-\sqrt{n\log n}$ for all $k=1,\dots, K$,  the inequalities \eqref{bound_eq} and  \eqref{post_signal} hold. 

\subsubsection{End of proof of Theorem \ref{thm-one} (general setting)}
The rest of the proof can now be completed using similar arguments as in Section \ref{sec:proof_rate}. \qedhere
\iffalse

\begin{example}\label{example:general}
As an example of the general covariance matrix  $\Sigma_\mT$, consider an autoregressive histogram prior $\wt\b_\mT\sim\mathcal{N}(0,\wt\Sigma_\mT)$ with $\wt\Sigma_\mT=c_n\left(\rho^{|i-j|}\right)$ for some $0<\rho<1$ and $c_n>0$. This prior links jump sizes of neighboring histogram cells and implies $\Sigma_\mT=(A_\mT'A_\mT)^{-1}A_\mT'\wt\Sigma_\mT A_\mT(A_\mT'A_\mT)^{-1}$. 
 From Proposition \ref{prop:eigenspectrum} in the Supplement  and the Gershgorin circle theorem, the maximal eigenvalue satisfies  $\lambda_{max}(\Sigma_\mT)\leq \lambda_{max}(\wt\Sigma_{\mT}){/ \lambda_{min}(A_\mT'A_\mT) } \leq c_n(1+\frac{2}{1-\rho})$, where we have used the fact that the spectral matrix norm is sub-multiplicative, that {the non-zero eigenvalues of $AB$ and $BA$ are the same} and that $\wt\Sigma_\mT$ is symmetric and positive semi-definite. Moreover,  $\lambda_{min}(\Sigma_\mT)\geq \lambda_{min}(\wt\Sigma_\mT)/\lambda_{max}(A_\mT'A_\mT)\geq 1/\sqrt{\log n}$ for large enough $c_n>0$. 
\end{example}

\fi

\subsection{Smooth wavelets}\label{sec:thm-wav}
The strategy of the proof of Theorem \ref{thm-one} can be directly applied   for  an $S$--regular wavelet basis. 
%, also noting that the corresponding $\cH(\al,M)$ ball in \eqref{eq:haar} contains the usual H\"older--ball for $\al\le S$, using the properties of $S$--regular wavelets (see e.g. Section 2.2 in \cite{castillo_sup}).} 
Similar to \cite{castillo_nickl1}, Section 2, one updates the index set $l\ge 0, 0\le k< 2^l$ for the Haar basis in the definition of the ball \eqref{eq:haar} as follows: the index sets becomes $l\ge J_0-1, k=0,\ldots, 2^l-1$, for $J_0=J_0(S)$ large enough, and one denotes the usual ``scaling function" $\vphi$ as the first wavelet $\psi_{(J_0-1)0}$. The proof is then the same (up to a multicative constant depending on the chosen basis) as in the Haar basis case, up to replacing the `localisation' identity $\|\sum_{k} |\psi_{lk}|\|_\infty=2^{l/2}$ in the Haar basis case by the `localisation' inequality  $\|\sum_{k} |\psi_{lk}|\|_\infty\le C 2^{l/2}$, with $C$ depending on the chosen basis only.

\subsection{Proof of Theorem \ref{sharp_lb}: exact rate}\label{proof:sharp_lb}
\proof
Define a sequence 
\[ L^*=\left\lceil \log_2\left[M^{1/(\alpha+1/2)}\left(n/\log^2 n\right)^{1/(2\alpha+1)}\right]\right\rceil,\] 
so that $2^{L^*}\asymp (n/\log^2{n})^{1/(2\al+1)}$. 
Define the sequence of functions $f_0^n$ (below we write simply $f_0$ for simplicity) by its sequence of wavelet coefficients as follows: set all coefficients $\beta_{lk}^0$ to $0$ except for $\beta_{L^*0}^0=M2^{-L^*(1/2+\al)}$. By definition, $f_0$ belongs to $\cH(\al,M)$. 
Let us also note that if $(L^*,0)$ does not belong to the tree $\cT$, one can bound from below
\[ \ell_\infty(f_{\mT,\b},f_0) \ge 2^{L^*/2} |\beta_{L^*0}|= M2^{-L^*\al}\ge C'\veps_n. \]
So, to prove the result, it is enough to show that  $\Pi[(L^*,0)\notin \mT_{int}\C X]\rightarrow 1$, i.e. the node $(L^*,0)$ does not belong to a tree sampled from the posterior with probability going to $1$, or equivalently, if $\bT_{L^*0}$ denotes the set of all full binary trees (of depth at most $L_{max}$) that contain   $(L^*,0)$ as an internal node, that $\Pi[\bT_{L^*0}\given X]=o_P(1)$. %Indeed, %this is enough to imply the result as, 
To prove this, let us consider a given tree $\cT\in \bT_{L^*0}$. As it contains the node $(L^*,0)$, it must also contain all nodes  $(\la,0)$ with $0\leq\la\le L^*$, in particular $(L_1,0)$, where $L_1=\lceil L^*/2\rceil$, say. We note that $L^*\asymp L^*-L_1\asymp \log{n}$. Let  $\tau^*$ be the maximal subtree of $\cT$ that has 
$(L_1,0)$ as its root. Next, let $\cT^*_-$ denote the remainder tree built from $\cT$ by erasing all of $\tau^*$ except for the node $(L_1,0)$ (so that $\cT^*_-$  still has a  full-binary tree structure). So, 
$\cT^*_-$ and $\tau^*$ have only the node $(L_1,0)$ in common, and the union of their nodes gives the original tree $\cT$. 
%To bound $\Pi[\bT_{L^*0}\given X]$, it is enough to control 
Let us now write
\[ \Pi[\bT_{L^*0}\given X] =  \frac{ \sum_{\cT\in \bT_{L^*0}} W_X(\cT)}{ \sum_{\cT\in\bT} W_X(\cT)}
= \frac{ \sum_{\cT\in \bT_{L^*0}} \frac{W_X(\cT)}{W_X(\cT^*_-)} W_X(\cT^*_-)}{ \sum_{\cT\in\bT} W_X(\cT)}.
  \]
Let ${q=q(\ta^*)}$ denote the number of internal nodes $\tau^*_{int}$ of the subtree $\tau^*$. 
%Denoting by $\cL[\cT]$ and $\cP[\cT]$ respectively the sets of leave nodes and parent nodes in $\cT$, 
From the Galton-Watson prior, we obtain
{
\begin{align}\label{eq:lb_gw}
\frac{\Pi_\bT[\cT]}{\Pi_\bT[\cT^*_-]} = \prod_{(l,k)\in \tau^*_{int}}  p_l \prod_{(l,k)\in \tau^*_{ext}} (1- p_l) \frac{1}{1-p_{L_1}} \le 2 \prod_{(l,k)\in \tau^*_{int}}  \Gamma^{-l}.
\end{align}
}
Then, by definition of $\cT_-^*$ and $\tau^*$,
\begin{align}
\frac{W_X(\cT)}{W_X(\cT^*_-)} & = \frac{\Pi_\bT(\cT)}{\Pi_\bT(\cT^*_-)} \prod_{(l,k)\in \tau^*_{int}} \frac{\exp\{(n+1)X_{lk}^2/2\}}{\sqrt{n+1}}\notag \\
& \leq 2  (n+1)^{-q/2} \cdot \prod_{(l,k)\in \tau^*_{int}} \Ga^{-l}   \exp\{(n+1)X_{lk}^2/2\}.\label{lb_eq}
\end{align}
  We bound the data-dependent part in the previous line by using $(a+b)^2\le 2a^2+2b^2$. Furthermore, noting that the noise variables $|\veps_{lk}|$ are uniformly bounded for $l+1\le L_{max}, 0\le k<2^l,$ by $2\log{n}$ on an event of overwhelming probability, we can upper-bound $(n+1) \sum_{(l,k)\in \tau^*_{int}} X_{lk}^2/2$ by
\begin{align*}
& (n+1)  \sum_{(l,k)\in \tau^*_{int}} 
\left[(\beta_{L^*0}^0)^2\1_{(l,k)=(L^*,0)} + \frac1n \max_{l+1\le L_{max}, k} \veps_{lk}^2 \right] \\
& \le (n+1) (\beta_{L^*0}^0)^2 + 2\frac{n+1}{n}(\log{n})q  \le \frac{n+1}{n} \log^2{n}  + 2\frac{n+1}{n}(\log{n})q.
\end{align*}
Now, using that for $(l,k)\in \tau^*_{int}$, we have $l\ge L_1\ge \frac{1}{2\alpha+1}\log_2\left(\frac{M^2n}{\log n}\right)=: c(\alpha, M,n)$, one notes that 
\[ \sum_{(l,k)\in \tau^*_{int}} l \ge \max\Big(c(\alpha, M,n)q, \sum_{l=L_1}^{L^*} l  \Big)
\ge c(\alpha, M,n)\max\left(q, \frac{3}{2}c(\alpha, M,n) \right),
  \]
  which is bounded from below by $c(\alpha, M,n)q$, 
%  \sbl{[here: invoke that the cardinality of $\cP[\cT^*]$ is about half of $\cT^*$]}
 where we have used that $q\geq L^*-L_1+1\geq L_1\geq c(\alpha, M,n)$ and $L_1+L^*>3c(\alpha, M,n)$.
 %Note that $q_{\mT^*}\geq L^*-L_1+1$  
 %For $0<\alpha\leq 1$ and large enough $n$, we note that $c(\alpha, M,n)>1/3\log(n/2)$. 
 One then deduces that the product of terms $\Gamma^{-l}$ dominates \eqref{lb_eq}, as long as $\log(\Gamma)$ is large enough (noting also that $1/(2\al+1)\ge 1/(2S+1)$), in the control of $W_X(\cT)/W_X(\cT_-^*)$. That is, for some constant $C>0$,
\[ \frac{W_X(\cT)}{W_X(\cT_-^*)}  \le \exp\{ -C(\log{n})q\}=: b_{q}, \]
where the last bound only depends on the number of internal nodes $q$ of $\tau^*$. By coming back to the above bound on the posterior $\Pi[\bT_{L^*0}\given X]$, let us  split the sum on the numerator as follows. Let $\bT_{L^*0}^{q}$ denote the set of trees $\cT=\cT_-^*\cup\ta^*$
 in $\bT_{L^*0}$ such that $|\tau_{int}^*|=q$. Then  $\Pi[\bT_{L^*0}\given X]$ is bounded by
{\begin{align*}
%\lefteqn{  
\sum_{q\ge 1} \frac{ \sum_{\cT\in \bT_{L^*0}^{q}}  b_{q} W_X(\cT_-^*)}{\sum_{\cT\in\bT} W_X(\cT)}  \le 
%\e^{-C\, c(\alpha, M,n)\log n}
\sum_{q\ge 1}
\frac{ \sum_{\cT_1\in \bT_{-}^{*q}} a_q b_q W_X(\cT_1)}{ \sum_{\cT_1\in \bT_{-}^{*q}} W_X(\cT_1) }, 
%} 
%\\
%&& \le \e^{-C\, c(\alpha, M,n)\log n} \sum_{q}4^{q} \frac{ \sum_{\cT\in \bT_{-}^{*q}}  W_X(\cT)}{ \sum_{\cT\in \bT_{-}^{*q}} W_X(\cT) }\leq
% \e^{-C'\, \log^2 n} \sum_{q} 4^{q},
\end{align*}
}
%\re{[Proof to be updated]} \sbl{The lonely factor $\sum_{q} 4^q$ is problematic; in the first draft, I was down-weighting the $4^q$ with the $b_q$ (sum $\sum a_q b_q$); to do so, one needs to define $\bT_{L^*0}^q$ I guess as the set of trees that  contain $(L^*,0)$ and such that $\ta^*$ has $q$ nodes. I'll look into that.}
where {$\bT_{-}^{*q}$} denotes the set of all possible $\mT^*_-$ that can be obtained from $\mT\in \bT_{L^*0}^q$ and 
where $a_q$ denotes the number of different possible trees $\tau^*$ such that $|\tau_{int}^*|=q$. {To obtain the last bound, we also used that each $\cT\in \bT_{L^*0}$ is uniquely caracterised by a pair $(\cT_-^*,\ta^*)$, so that the sum over $\cT$ can be rewritten as a double sum over $\cT_-^*$ and $\ta^*$. } {One deduces that, as $q$ cannot be larger than $2^L$,
\[ \Pi[\bT_{L^*0}\given X] \le \sum_{q=1}^{2^L} a_q b_q.\]}
{As $a_q$ is less (because of the restriction $|\cT|\le L$) or equal to the number of full binary trees with $q$ internal nodes, i.e. with $2q+1$ nodes in total, we have $a_q\le  \mathbb{C}_{2q}$, 
%\re{[Double-check that argument compatible with new normalisation]\ma{It looks fine to me}}
which is bounded from above by $4^{2q}$ by Lemma \ref{catalan}.} We conclude that $\Pi[\bT_{L^*0}\given X]$ is bounded  above by $\exp(-  C\log^2{n})$ for some $C>0$, on an event of overwhelming probability,  which concludes the proof for the Galton-Watson prior.
Similarly, for the exponential prior, we replace \eqref{eq:lb_gw} directly with 
$$
\frac{\Pi_\bT[\cT]}{\Pi_\bT[\cT_-^*]}\propto \e^{-c (|\mT_{ext}|-|\mT^*_{- ext}|)\log n}=\e^{-c\,  q\log n},
$$
where we used the fact that $|\mT_{ext}|-|\mT^*_{- ext}|=|\tau^*_{ext}|-1=|\tau^*_{int}|=q$.
For the conditionally uniform prior, we have for $\lambda=1/n^c$ for some $c>0$
$$
\frac{\Pi_\bT[\cT]}{\Pi_\bT[\cT_-^*]}=\frac{\pi(|\mT_{ext}|)}{\pi(|\mT^*_{-ext}|)}\frac{\mathbb{C}_{|\mT^*_{-int}|}}{
\mathbb{C}_{|\mT_{int}|}}\lesssim \lambda^{q}\e^{-q\log |\mT_{ext}|}\lesssim \e^{-c\,q\log n},
$$
and the end of the proof is then the same as for the GW prior. \qed

\subsection{Proof of Theorem \ref{csthm}: confidence bands}\label{proof:csthm}

In the proof, we repeatedly use the properties of the median tree $\cT^*_X$ established in Lemma \ref{lembulk2}. We denote by $\cE$ the event from  Lemma \ref{lembulk2}.
We first show  the diameter statement \eqref{cred_diam}.  The depth of the median tree estimator $\wh f_T$ verifies condition {(i)} of Lemma \ref{lembulk2} on the event $\mE$. For any $f,g\in \cc_n$, by definition of $\cc_n$, we have, for $C(\psi)$ a constant depending on the wavelet basis only,  on the event $\cE$,
\begin{align*}
\|f-g\|_\infty & \le \|f-\wh f_T\|_\infty + \|\wh f_T-g\|_\infty \\
& \le 2 \sup_{x\in[0,1]} \sum_{l=0}^{L_{max}} v_n \sqrt{\frac{\log{n}}{n}}
\sum_{k=0}^{2^l-1} \1_{(l,k)\in \cT_X^*} |\psi_{lk}(x)| \\
& \le 2  v_n C(\psi)\sqrt{\frac{\log{n}}{n}} \sum_{l: 2^l \le C_1 2^{\cL_c}} 2^{l/2}\le C'v_n \sqrt{\frac{\log{n}}{n}2^{\cL_c}}.
\end{align*}
We now turn to the confidence statement. First, one shows that the median estimator \eqref{bulkest} is (nearly) rate optimal. 
Denote with $\wh f_{T,lk}=\langle \wh f_T,\psi_{lk}\rangle$ and let $\cS=\{ (l,k):\ |\beta_{lk}^0|\ge A\log{n}/\rn \}$.  
Let us consider the event
\begin{equation} \label{eventb}
B_n = \{ \wh f_{T,lk}\neq 0,\ \forall\,(l,k)\in\cS \} \cap \{  \wh f_{T,lk} = 0,\ \forall\, (l,k):\, 2^l \ge C_1\,2^{\cL_c}\} \cap \cA,
\end{equation}
where the noise-event $\cA$ is defined in \eqref{event}. Lemma \ref{lembulk2} together with $P_{f_0}(\cA)=1+o(1)$ imply that $P_{f_0}(B_n)=1+o(1)$.  On the event $B_n$, we have
\begin{align*}
 \| \wh f_T - f_0 \|_\infty  \le & \sum_{l:\, 2^l\le C_12^{\cL_c}} 2^{l/2}
{ \max  \left(\max_{0\le k<2^l:\, (l,k)\in\cS} |X_{lk} - \beta_{lk}^0|,\max_{0\le k<2^l:\, (l,k)\notin\cS} \{|\beta_{lk}^0| \}\right)}
 \\
& + \sum_{l:\, 2^l > C_12^{\cL_c}} 2^{l/2}\max_{0\le k<2^l} |\beta_{lk}^0|\\
\lesssim & \ 2^{\cL_c/2} \sqrt{\frac{\log{n}}{n}} +  \sum_{l:\, 2^l \le C_12^{\cL_c}}2^{l/2} \min\left(\max_{0\le k<2^l}|\beta_{lk}^0|, A\frac{\log{n}}{\rn}\right) + 2^{-\al \cL_c},
\end{align*}
where we have used the definition of $\cS$, that $\wh{f}_T$ equals $0$ or $X_{lk}$,  that $f_0$ belongs to $\cH(\al,M)$ and $\max(a,b)\le a+b$ (note also that the term with the minimum in the last display is an upper bound of the maximum over $(l,k)\notin \cS$ on the first line of the display). This shows that the median tree estimator is rate-optimal up to a logarithmic factor, in probability under $P_{f_0}$. In particular,  on $B_n$, we have for some $C>0$
\begin{equation} \label{bmedrate}
\|\wh f_T - f_0\|_\infty  \le C(\log^2{n}/n)^{\al/(2\al+1)},
\end{equation}
where we used the inequality in \eqref{Lstar_ineq} {in the case of smooth wavelets}.
 Second, we now show that $\sigma_n$ is appropriately large. By  the proof of Proposition 3 of \cite{hoffmann_nickl}, we have for $f_0\in\cH_{SS}(\al,M,\veps)$, for $l_n\geq j_0$ suitable sequence chosen later 
 %large enough \ora{[check from which $j$ and the constant $b$] {\color{red} Is it $j\geq j_0$?}}
\[ \sup_{(l,k):\ l\ge l_n} |\beta_{lk}^0| \ge C(M,\psi,\alpha,{\veps}) 2^{-l_n(\al+1/2)}, \]
for some constant $C(M,\psi,\alpha,{\veps})$ depending on $\al, M$, the wavelet basis and ${\veps}$ (as in (2.12) of \cite{hoffmann_nickl}). Let $\La_n(\al)$ be defined by, for $\eta>0$ to be chosen below,
\begin{equation} \label{defLa}
\eta(n/\log^2{n})^{1/(2\al+1)} \le 2^{\La_n(\al)} 
\le 2\eta(n/\log^2{n})^{1/(2\al+1)}
\end{equation}
Combining the previous two displays leads to, for $f_0\in\cH_{SS}(\al,M,\veps)$,
\[  \sup_{(l,k):\ l\ge \La_n(\al)} |\beta_{lk}^0| \ge C(M,\psi,\alpha,{\veps})\eta^{-\al-1/2}\frac{\log{n}}{\rn}. \]
By taking $\eta$ small enough, one obtains that there exists $(\la,\kappa)$ with $\la\ge \La_n(\al)$ verifying $|\beta_{\la\kappa}^0|\ge A\log{n}/\rn$ and thus, in turn, $\wh f_{T,\la\kappa}\neq 0$, by the second part of Lemma \ref{lembulk2}. One deduces that the term $(l,k)=(\la,\kappa)$ in the sum defining $\sigma_n$ is nonzero on the event $B_n$, so that
\[ \sigma_n \ge v_nc(\psi)\sqrt{\frac{\log{n}}{n}} \|\psi_{\la\kappa}\|_\infty
\ge v_nc(\psi)\sqrt{\frac{\log{n}}{n}} 2^{\La_n(\al)/2}.
 \]
%by chosing $x$ such as $|\psi_{\la\kappa}(x)|=\|\psi_{\la\kappa}\|_\infty$. 
This leads, on the event $B_n$, to
\begin{equation}\label{lbound_sigma}
 \sigma_n/ \left(\frac{\log^2{n}}{n}\right)^{\al/(2\al+1)}
 \ge c'v_n(\log{n})^{-1/2}. 
 \end{equation}
The ratio in the last display goes to infinity for $v_n$ of larger order than $\log^{1/2}{n}$. 
Now, on the event $B_n$, one can thus write $\|\wh f_T-f_0\|_\infty\leq \sigma_n/2$ for large enough $n$, uniformly over $f_0\in \cH_{SS}(\al,M,\veps)$, implying that $B_n\subset\{\|\wh f_T-f_0\|_\infty\leq \sigma_n\}$.
This implies the desired coverage property, since $P_{f_0}( f_0\in \mathcal{C}_n)
\ge P_{f_0}(B_n)=1+o(1)$. 
%=P_{f_0}\left(\{\|f_0-\ix\|_{\mathcal{M}(w)}\leq R_n/\sqrt{n}\}\cap B_n\right)+o_{P_{f_0}}(1)=1-\gamma+o_{P_{f_0}}(1)$, where we used  Theorem 5 of \cite{castillo_nickl1}.

For the credibility statement, %we only need to show that the second constraint in $\cc_n$ is satisfied with posterior probability tending to one,  since $\Pi[\|f-\ix\|_{\mathcal{M}(w)}\leq R_n/\sqrt{n}\C X]{=} 1-\gamma$ by definition of $R_n$. In addition, 
we  note that the posterior distribution (from 
%Theorem \ref{thm-one}
{Theorem \ref{thm-one}) and the median estimator $\wh f_T$ (from \eqref{bmedrate})  converge at a rate strictly faster than $\sigma_n$  on the event $B_n$, using again the lower bound on $\sigma_n$ in \eqref{lbound_sigma}. In particular, because $B_n\subset\{\|\wh f_T-f_0\|_\infty\leq \sigma_n/2\}$, one can write
{\[ 
\Pi[\|f-\wh f_T\|_\infty \leq \sigma_n \given X] \ge 
\Pi[\|f-f_0\|_\infty \leq \sigma_n/2 \given X]\mathbb{I}_{B_n}
+o_{P_{f_0}}(1).
\]}
The right side converges to $1$ in  $P_{f_0}$-probability, which concludes the proof of the theorem.

\begin{lemma} \label{lembulk}
The set of nodes $\cT^*_X$ in \eqref{bulktree} $P_{f_0}$-almost surely defines a binary tree.
\end{lemma}
\begin{proof}
Let us recall that $\bT$ is the set of all admissible trees that can be obtained by sampling from the prior $\Pi_\bT$ and with depth at most $L_{max}$. For any given node $(l_1,k_1)$ with $0\le l_1\le L_{max}$, one can write 
\begin{align*}
 \Pi[(l_1,k_1)\in \cT \given X] & = \sum_{\cT_1\in \bT} \Pi_\bT[\cT_1\given X]\times\Pi[(l_1,k_1)\in \cT_1 \given X, \cT=\cT_1] \\
 & =  \sum_{\cT_1\in \bT:\ (l_1,k_1)\in \cT_1} \Pi_\bT[\cT_1\given X].
 \end{align*}
Let $(l_1-1,k_1^-)$ denote the parent node of $(l_1,k_1)$ in $\cT_1$, where $k_1^-= \lfloor k_1/2\rfloor$. Any (full-)binary tree that contains $(l_1,k_1)$ must also contain $(l_1-1,k_1^-)$, so that, using the formula in the last display, $\Pi[(l_1,k_1)\in \cT \given X] \le \Pi[(l_1-1,k_1^-)\in \cT \given X]$. This implies,  by definition of $\cT^*_X$, that if a given node $(l_1,k_1)$ belongs to $\cT^*_X$, so does the node $(l_1-1,k_1^-)$. Therefore $\cT^*_X$ is a tree.
\end{proof}

\begin{lemma} \label{lembulk2}
Consider a prior distribution $\Pi$ as in Theorem \ref{thm-one}. 
There exists an event $\cE$ such that $P_{f_0}[\cE]=1+o(1)$ on which the 
tree $\cT^*_X$ defined in \eqref{bulktree} has the following properties: {there exists a constant $C_1>0$ such that}
%\sbl{improve writing wrt the class parameters $\be, L$ and the constants $c,C$ to be specified, unify notation $\cL_c$ etc.} 
\begin{enumerate}
\item[(i)] the depth of the tree satisfies $2^{d(\cT^*_X)} \leq C_1 2^{\mathcal{L}_c}\asymp (n/\log{n})^{\al/(2\al+1)}$, where $\mathcal{L}_c$ is as in \eqref{cutoff},
\item[(ii)] the tree  contains {as interior nodes} all  nodes $(l,k)$ that satisfy 
$|\beta_{lk}^0|\ge A\log{n}/\sqrt{n}$, for some $A>0$.
\end{enumerate}
\end{lemma}
\begin{proof}
We focus on the GW-prior, the proof for the other two  priors $\Pi_\bT$ being similar. 
Let  $\bT^{(1)}$, respectively $\bT^{(2)}$, denote the set of binary trees that satisfy condition (i), respectively (ii), in the statement of the lemma. By the proof of Theorem \ref{thm-one}, $\Pi[\bT^{(1)}\given X]$ and $\Pi[\bT^{(2)}\given X]$ both tend to $1$ in probability under $P_{f_0}$, hence also $\Pi[\bT^{(1)}\cap\bT^{(2)}\given X]$. In fact, it also follows from the proof of Theorem \ref{thm-one} that, for $\bT^{(1)}$, we also have the stronger estimate $\Pi[d(\mT)>d\given X] \le  2^{-c_1 d \log{\Gamma}}$ for some $c_1>0$  under the  GW process prior, uniformly over $\cL_c<d\le L_{max}$, on an event $\mathcal{A}$ of $P_{f_0}$-probability going to $1$. The latter probability is  $o(2^{-d})$ provided $\Gamma$ is chosen {large enough}, which will be used below. 
Defining $\cE=\{\Pi[\bT^{(1)}\cap\bT^{(2)}\given X]\ge 3/4\}$, we have $P_{f_0}[\cE]\to 1$ as $n\to\infty$. For any node $(l_2,k_2)$ such that $|\beta_{l_2k_2}^0|\ge C\log{n}/\sqrt{n}$, we have 
\[  \Pi[(l_2,k_2)\in \cT_{int} \given X] 
 =  
 \sum_{\cT_2\in \bT:\ (l_2,k_2)\in \cT_{2\, int}} \Pi[\cT_2\given X] \ge \Pi[\bT^{(2)}\given X],
 \]
where we used that any tree in $\bT^{(2)}$ must, by definition, contain $(l_2,k_2)$. As $\Pi[\bT^{(2)}\given X]\ge 3/4>1/2$, we deduce that $(l_2,k_2)$ belongs to $\cT_X^*$ on the event $\cE$.
In other words, $\cT_X^*$ verifies the second property (ii) of the lemma on $\cE$. To conclude the proof of the lemma, one observes that on $\cE$, for a given node $(l_3,k_3)$ with $2^{l_3}>C_12^{\mathcal{L}_c}$, 
{\[ \Pi[(l_3,k_3)\in \cT_{int} \given X]  \le \Pi[d(\mT)>l_3\given X],\]
%where $\bT_d$ denotes the set of binary trees of depth exactly $d$.
Recall that  $\Pi[d(\mT)>l_3\given X]\le C2^{-c_1 l_3\log{\Gamma}}$  on $\mathcal{A}$ (which holds uniformly over $l_3\in[\cL_c,L_{max}]$).  %and which holds also for the conditionally uniform and exponential priors). 
Then, on the event $\mathcal{A}$, we can write
\begin{align*}
P_{f_0}[\{\cT_X^*\notin \bT^{(1)}\}\cap \cA] & \le P_{f_0}[\{\exists\, (l_3,k_3):\ 2^{l_3}>C_12^{\mathcal{L}_c},\ (l_3,k_3)\in\cT_X^*\}\cap \cA]  \\
& \le \sum_{l_3:\ 2^{l_3}>C_12^{\mathcal{L}_c}}^{L_{max}} \sum_{k_3=0}^{2^{l_3}-1} P_{f_0}\left[
\{\Pi[(l_3,k_3)\in\cT_{int}\given X]\ge 1/2\}\cap \cA\right] \\
& \le  \sum_{l_3:\ 2^{l_3}>C_12^{\mathcal{L}_c}}^{L_{max}} 2^{l_3+1}  E_{f_0}\left[
\Pi[d(\mT)>l_3\given X] \1_{\cA}\right] = o(1).
\end{align*}
Using that $P_{f_0}[\cA]$ goes to $1$, one obtains $P_{f_0}[\cT_X^*\notin \bT^{(1)}\}]=o(1)$, which concludes the proof.}
\end{proof}

\subsection{Proof of Theorem \ref{thm-fun}: smooth functionals} \label{sec:proof_thm-fun}

For the second point, it suffices to combine Theorem \ref{thm_bvm} with the choice $w_l=l^2$ (which satisfies the condition $w_{j_0(n)}\geqa \log{n}$ given the assumption $j_0(n)\asymp \sqrt{\log{n}}$) and  Theorem 4 in \cite{castillo_nickl1} (the condition $\sum_l w_l 2^{-l/2}$ holds), from which one indeed obtains 
\[ \beta_{C([0,1])}\left(\mathcal{L}(\sqrt{n}(F(\cdot)-\int_0^\cdot dX^{(n)}\C X),\mathcal{L}(G)\right)\rightarrow0\] in $P_{f_0}$-probability, where $(G(t):t\in[0,1])$ is a Brownian motion. 

For the first point, we argue as in the paragraph preceding the statement of Theorem 4 in \cite{castillo_nickl1}: proceeding as in the proof of that Theorem, one notes that for a BvM to hold for the functional $\psi_b(f)=\int_0^1 f(u)b(u)du$, it suffices that, for some sequence $(c_l)$ of positive numbers, both $\sum c_lw_l<\infty$ and $\sum_k |\psg b,\psi_{lk}\psd|\le c_l$ hold. With the choice $w_l=l^2$, this gives the condition of the present theorem and concludes the proof.

We note that other choices of $j_0(n)$ are possible, providing $w_l$ is chosen appropriately: indeed $w_{j_0(n)}\geqa \log{n}$ is enough for Theorem \ref{thm_bvm} to hold. Then the BvM result for the linear functional $\psi_b$ holds whenever $\sum c_lw_l<\infty$ and $\sum_k |\psg b,\psi_{lk}\psd|\le c_l$ hold, and for the functional BvM for $F(\cdot)$, it suffices that $\sum 2^{-l/2}w_l<\infty$.
\qed

\subsection{Proof of Theorem \ref{thm_lb}: lower bound for flat trees}\label{sec:proof_thm_lb}
\begin{proof}
Denote with $\cT^F_D$  the flat tree of depth $D+1$ (i.e. all $\beta_{lk}$'s for $l\leq D$ are active). %Moreover, we assume that $\beta_{lk}$'s for $l\leq D$ are standard Gaussian.
The formula \eqref{eq:W} gives %the posterior probability
\begin{align*}
\Pi[\cT^F_D\given X] & \propto W_X(\cT^F_D) = \Pi_\bT(\cT^F_D)\prod_{(l,k)\in \cT^{F\, '}_{D\, int}}
\frac{\e^{\frac{n^2}{2(n+1)}X_{lk}^2}}{\sqrt{n+1}} \\
%& \propto  \Pi_\bT(\cT^F_D)\exp\left\{ -\frac12\log(n+1)2^D
%+ \sum_{l\le D, k} \frac{n^2X_{lk}^2}{2(n+1)} \right\} \\
& \propto \exp\left\{-\log{\Pi_\bT(\cT^F_D)} -2^\cD\log(n+1) + \frac{n^2}{2(n+1)}\|\X^{(D)}\|_2^2
\right\},
\end{align*}
where $\|\X^{(D)}\|_2^2= \sum_{l\le D, k} X_{lk}^2$ is the squared $L^2$--norm of the signal, truncated at the level $D$. Next, we have
\begin{align*}
\|\X^{(D)}\|_2^2  %\sum_{l\le n, k} X_{lk}^2 - \sum_{ \cD<l\le L_{max}, k} X_{lk}^2 
 & =  \sum_{l\le D,k } (\beta_{lk}^0)^2 +
 \sum_{l\le D,k }\frac{2}{\sqrt{n}}\veps_{lk}\beta_{lk}^0 + \sum_{l\le D,k} \frac1n\veps_{lk}^2, \\
 & = C_n - \sum_{D<l\le L_{max},k }(\beta_{lk}^0)^2 -
 \sum_{D<l\le L_{max},k }\frac{2}{\sqrt{n}}\veps_{lk}\beta_{lk}^0 + 
 \sum_{l\le D,k} \frac1n\veps_{lk}^2,
\end{align*} 
where $C_n=C(n,\{\veps_{lk}\},f_0)$ does not depend on $D$. We can also write
\begin{equation}\label{eq:decompose} 
\|\X^{(D)}\|^2 = 
C_n-  \sum_{D<l\le L_{max},k }(\beta_{lk}^0)^2 + \frac{2^{D+1}}{n} 
-\frac{2}{\rn} Z(D) + \frac{1}{n}Q(D),
\end{equation}
where we have used $\sum_{l\le D, k} 1=2^{D+1}$ and have set
\begin{align*}
Z(D)  =  \sum_{D<l\le L_{max},k }\beta_{lk}^0 \veps_{lk},\qquad
Q(D)  = \sum_{ l\leq D,k} (\veps_{lk}^2-1).
\end{align*}
Let $D^*$ be an integer defined as, for $f_0$ to be chosen below, 
\[ D^* = \underset{{0\le D\le n}}{\text{argmin}}\ \left[ 2^D\log(n+1) + \frac{n}{2}\sum_{l=D+1}^{L_{max}}\sum_{k} (\beta_{lk}^0)^2 \right]. \]
Consider the following true signal $f_0=f_0^*$ which belongs to $\mathcal{H}(\al,M)$ with $M=1$ (which we can assume without loss of generality) and which is characterized by the following wavelet coefficients
\begin{equation}
\beta_{lk}^{0*}=
\begin{cases}
2^{-l(\frac12+\al)}& \quad \text{if } k=0,\\
0 & \quad \text{otherwise}.
\end{cases}
\end{equation}
For such a signal, $D^*$ above has the following behavior
\begin{equation} \label{cutoffd}
 2^{D^*} \asymp \left( \frac{n}{\log{n}} \right)^{\frac{1}{2\al+2}}. 
\end{equation} 
With the maximum-type norm $\ell_\infty$ defined in \eqref{sup_bound}, we use the decomposition $\ell_\infty(f,f_0)=\ell_\infty(f,f_0^D)+\ell_\infty(f_0^D,f_0)$, where $f_0^D$ is the $L^2$--projection of $f_0$ onto the first $D$ levels of wavelet coefficients. Moreover,  using $\ell_\infty(f_0^D,f_0)
=c\,2^{-D\al}$ for some $c>0$, we can write, for $\rho_n=(\log{n}/n)^{\al/(2\al+2)}$
\begin{align*}
\Pi[\ell_\infty(f,f_0)<\mu\rho_n \given X]
& \le \Pi[\ell_\infty(f_0,f_0^D)<\mu\rho_n\given X] \\
& {=} \Pi[ c2^{-D\al} <{\mu}\rho_n{\given X} ] = 
\Pi[ 2^D> (c\mu^{-1}\rho_n^{-1})^{1/\al}\given X ] \\
& \le \Pi[ 2^D> (c\mu^{-1})^{1/\al} 2^{D^*}\given X ].
\end{align*}
To conclude, it is enough to show that for $B=\{2^D> (c\mu^{-1})^{1/\al} 2^{D^*}\}$, where $\mu>0$ is a small enough constant, we have $\Pi[B\given X]=o(1)$ or, equivalently, $\Pi[B\given X]=o(\Pi[B^c\given X])$
(possibly on an event of vanishing probability). 
Rewriting $B=\{D:\ D>cD^*\}$ for $c=c(\mu)\ge 1$ (up to taking $\mu$ small enough), and using the above expression of $\Pi[\cT_D^F\given X]$, one obtains
\begin{align*}
\frac{\Pi[B\given X]}{\Pi[B^c\given X]} & 
= \frac{\sum_{D>cD^*} \exp\left\{-\log\Pi(\cT_D^F) - 2^D\log(n+1) +\frac{n^2}{2(n+1)}\|\X^{(D)}\|_2^2\right\}}{
\sum_{D\le cD^*} \exp\left\{-\log\Pi(\cT_D^F) - 2^D\log(n+1) +\frac{n^2}{2(n+1)}\|\X^{(D)}\|_2^2\right\}} \\
& \le   \frac{\sum_{D>cD^*} \exp\left\{-\log\Pi(\cT_D^F) - 2^D\log(n+1) +\frac{n^2}{2(n+1)}\|\X^{(D)}\|_2^2\right\}}{
 \exp\left\{-\log\Pi(\cT_{D^*}^F) - 2^{D^*}\log(n+1) +\frac{n^2}{2(n+1)}\|\X^{(D^*)}\|_2^2\right\}}.
\end{align*}
Since $c\ge 1$ we have  $D\ge D^*+1$  for any $D>cD^*$ and  from the monotonicity assumption on the prior we obtain $\log\Pi(\cT_{D^*}^F)-\log\Pi(\cT_{D})\le 0$ on $B$. In addition, note that $2^{D^*}-2^{D} \le -2^{D}/2$ on $B$, which implies
\[ (2^{D^*}-2^{D})\log(n+1) \le -\frac12 2^{D}\log(n+1). \]
Going further, using the decomposition of $\|\X^{(D)}\|_2^2$ in \eqref{eq:decompose} we have for  $Z=\|\X^{(D)}\|_2^2 -\|\X^{(D^*)}\|_2^2$ the following
\begin{align*}
Z& = 
\sum_{D^*<l\le D,k } (\beta_{lk}^0)^2 + \frac{1}{n}(2^{D+1} -2^{D^*+1})
-\frac{2}{\rn} (Z(D)-Z(D^*)) + \frac{1}{n}(Q(D)-Q(D^*)) \\
& \le \sum_{D^*<l\le L_{max},k } (\beta_{lk}^0)^2+\frac{2^{D+1}}{n}+\frac{2}{\rn}(|Z(D)|+|Z(D^*)|) + \frac{1}{n}(|Q(D)|+|Q(D^*)|).
\end{align*}
%noting that the constant $C_n$ vanishes when considering the difference. 
We now provide bounds for the stochastic terms $Z$ and $Q$. First, for any $D>D^*$, denoting 
$\sigma_D^2:= \sum_{D<l\le L_{max},k } (\beta_{lk}^0)^2$, we have
\[ |Z(D)|\le \sigma_D \max_{D^*\le D\le L_{max}} \sigma_D^{-1} |Z(D)|. \]
The variables $Z(D)/\sigma_D$ are standard normal, which implies that, on some event $A_1$ such that $P_{f_0}(A_1^c)=o(1)$, we have, uniformly in $D\in B$,
\[ |Z(D)|\le \sigma_D\sqrt{2\log{L_{max}}}. \]
To bound the term $Q(D)$, one can use the following standard concentration bound for chi-square distributions. Namely, for $\xi_q$ standard normal variables and any $t>0$, we can write
\[ \P\left[ \sum_{q=1}^Q (\xi_q^2-1) \ge t \right]
\le \exp\left\{ -\frac{t^2}{4(Q+t)} \right\}.
\]
Applying this bound for the noise variables $\veps_{lk}$ and choosing $t=t_D:=(D2^D)^{1/2}$ leads to 
\[ \P\left[ \sum_{l\le D, k} (\veps_{lk}^2-1) > t_D \right]
\le \exp\left\{ -\frac{D2^D}{4(2^{D+1}+t_D)} \right\}.  \]
For $D\ge D^*$, one has $t_D\le 2^{D+1}$ so the last display is bounded from above by $\exp\{-C_1D\}$. %\le \exp\{-C_1D^*\}=o(1)$. 
%Denoting by $A_2$ the event $\{\sum_{l\le D, k} (\veps_{lk}^2-1) \le t_D\}$, we have just showed that $P_{f_0}(A_2^c)=o(1)$. 
Let us consider the event, with $t_D$ as above,
\[ A_2 = \bigcap_{D=D^*}^{L_{max}}\, \left\{ \sum_{l\le D} \sum_{k=0}^{2^l-1} (\veps_{lk}^2-1)\le t_D\right\}. \]
A union bound gives $P_{f_0}[A_2^c]\le C \exp(-c_1D^*)$, which is a $o(1)$ using the previous bound. 
%Now coming back to the definition of $D^*$ we have, for a universal constant $C_2>0$,
%\[ n\sum_{D^*<l\le L_{max},k }  (\beta_{lk}^0)^2 \le C_2 2^{D^*}\log(n+1). \]
Now let us choose $\mu$ small enough in such a way that $C_22^{D^*}\le 2^{D}/2$ for any $D$ in the set $B$ defined above (this is possible by definition of $B$) and thereby
\[ \frac{n}{2}\sum_{D^*<l\le L_{max},k }  (\beta_{lk}^0)^2 \le \frac{2^{D}}{4}\log(n+1)\]
for any $D$ in $B$. This in particular implies that $\sigma_D \le (2^D\log(n+1)/n)^{1/2}$. Now, on the event $A_1\cap A_2$, we have
\begin{align*}
&\frac{\Pi[B\given X]}{\Pi[B^c\given X]} 
 \le  \sum_{D>cD^*} \exp\Big\{ - \frac12 2^D\log(n+1) 
+ \frac{2^D}{4}\log(n+1)+ 2^{D} \\
&\qquad \qquad \qquad\ \ + 2\rn\sigma_D\sqrt{2\log{L_{max}}}
+ (D2^D)^{1/2} \Big\}\\
&\le \sum_{D>cD^*} \exp\Big\{ - \frac18 2^D\log(n+1) \\
&\qquad \ \ + \left[ 2^{D} 
 + 2\sqrt{2^D\log(n+1) 2\log{L_{max}}} 
+ (D2^D)^{1/2}  - \frac18 2^D\log(n+1)\right] \Big\}\\
&\le \sum_{D>cD^*} \exp\Big\{ - \frac18 2^D\log(n+1)\Big\}
\le \exp\Big\{ - C 2^{D^*}\log(n+1)\Big\},
\end{align*}
where we have used that the term under brackets in the second inequality is negative for large enough $n$, as $2^D\geqa 2^{D^*}$ goes to infinity. This shows that the last display goes to $0$, which concludes the proof.

\end{proof}

\section{Proof of Theorem \ref{thm-three}: non-dyadic Bayesian CART} \label{sec:proof_thm-three}

\iffalse
We show the theorem under more general assumptions
\begin{enumerate} 
\item[(B1)] There exists universals $C_0, G>0$ such that for any $l\le L$ and $k$, the function $\psi_{lk}^B$ can be expressed as a linear combination of  at most $C_0l^\delta$ functions  $\psi_{jk}$ for $j\le l+G$.\\
\item[(B2)] If $f$ is $\alpha$-H\"{o}lder continuous according to \eqref{eq:haar2}, there exists $C_1=C_1(\al,M)$ such that, for any $l\ge 0$,
\[ \max_{0\le k<2^l-1} |f_{lk}^B|\le C_1 l^{\delta/2} 2^{-l(\al+1/2)}.  \]
\end{enumerate}
which are satisfied by weakly balanced wavelet systems $\Psi_{A}^B$ with $\delta=3$.
\fi
{As the breakpoints verify the balancing condition \eqref{eq:balance2}, they verify the properties (B1)--(B2) in the complexity Lemma \ref{lemma:complex} for $\delta=3$.  }
%We consider breakpoints $B$ that satisfy the balancing requirements $(B1)$-$(B2)$ for some constant $\delta>0$ and a Haar system $\Psi_{A}^B=\{\psi_{-10}^{B},\psi_{lk}^B:(l,k)\in A)\}$. 
%\ma{According to Lemma \ref{lemma:complex}, weakly balanced systems satisfy these two assumptions for $\delta=3$.}
The Gaussian white noise model projects onto the Haar system $\Psi_{A}^B=\{\psi_{-10}^{B},\psi_{lk}^B:(l,k)\in A)\}$ as follows:
\begin{equation}\label{eq:wn2} 
X_{lk}^B = \beta_{lk}^{0B} + \frac{1}{\rn}\veps_{lk}^B, 
\end{equation}
where $X_{lk}^B=\langle X,\psi_{lk}^B\rangle$, $\beta_{lk}^{0\,B}=\langle f_0,\psi_{lk}^B\rangle$ and  $\veps_{lk}^B=\langle W, \psi_{lk}^B\rangle$. As the functions $\psi_{lk}^
B$ form an orthonormal system, the variables $\veps_{lk}^B$ are iid standard Gaussian given $B$. 
The observations here are viewed as the collection of $X_{lk}^B$ variables which depend on $B$.  We regard the breakpoints $B$ as one extra ``variable" in the model. Given the breakpoints $B$, we use the same notation $\X_\mT^B$ and 
$\bm\varepsilon_\mT^B$ for the ordered  responses and noise variables (as in the proof of Theorem \ref{thm-one}). Similarly, $\b_\mT=\b_\mT^B$ are the ordered internal wavelet coefficients.

As  the priors on breakpoints $B$ and trees $\cT$ are {\em independent}, the tree posterior remains relatively tractable where the amount of signal at each location $(l,k)$  now depends on $B$, which requires a separate  ``{\em uniform} in $B$" treatment.

\medskip

{\em The Multiscale Posterior Distribution.}
To determine the posterior distribution on $f$, it is enough to consider   the posterior on wavelet coefficients $(\beta_{lk})$, which then induces a posterior on $f$ via
\begin{equation}\label{tree_expand2}
f_{\mT,\wt\b}^{B}(x)= \sum_{(l,k)\in\mT_{ext}}\wt\beta_{lk}^BI_{lk}^B(x)= %\beta_{-10}^B\psi^B_{-10}(x)+
\sum_{(l,k)\in\mT_{int}'}\beta_{lk}^B\psi_{lk}^B(x).
\end{equation}
Again, the {\sl internal unbalanced} Haar wavelet coefficients  $\b_\mT=(\beta^B_{lk}:(l,k)\in\mT_{int}')$ are linked to  the {\sl external}  histogram coefficients $\wt\b_\mT=(\wt\beta^B_{lk}:(l,k)\in\mT_{ext})$
through 
%\begin{equation}\label{onetoone2}
$\wt\b_{\mT}=A_\mT \b_\mT$
%\end{equation} 
for some sparse matrix $A_\mT\in \R^{|\mT_{ext}|\times|\mT_{ext}|}$ (a generalization of \eqref{onetoone}).
 This section describes the posterior distribution over coefficients $(\beta_{lk})$ driven by  the prior distribution
\begin{align} \label{priorbal}
\begin{split}
(B,\cT) \, & \, \sim \, \Pi_\mb\otimes \Pi_\bt\\
%\cT & \sim \Pi_\bt\\
(\beta_{lk})_{l\le L,k} \given B,\cT\, &\, \sim\, \pi(\b_\mT) \, \otimes \, 
\bigotimes_{(l,k)\notin\cT_{int}'} \delta_0(\beta_{lk}),
\end{split}
\end{align}
where $L=L_{max}=\lfloor \log_2 n\rfloor$. From the white noise model, we have, given $B$,
$$
\X_\mT^B=\b_\mT+\frac{1}{\sqrt{n}}\bm\varepsilon_{\mT}^B,\quad\text{ with} \quad\bm\varepsilon_{\mT}^B\sim\mathcal{N}(0,I_{|\mT_{ext}|}).
$$
The joint density of $(B,\cT,(\beta_{lk})_{l\le L, k},X)$ arising from the above distributions equals
\begin{align*}
& %\sum_{B\in\mb, \cT\in \bT} 1_{ S(f^L)=\cT, \cB=B} 
\Pi_{\bt}(\cT) \Pi_{\mb}(B)\pi(\b_\mT) \left[\prod_{(l,k)\in\cT_{int}'} \phi_{\frac{1}{\sqrt{n}}}(X_{lk}^B-\beta_{lk})\right]
\left[\prod_{(l,k)\notin\cT_{int}'} \phi_{\frac{1}{\sqrt{n}}}(X_{lk}^B-\beta_{lk}) \1_{0}(\beta_{lk})\right]\\
& %= \sum_{B\in\mb, \cT\in \bT} 1_{ S(f^L)=\cT, \cB=B} 
= \Pi_{\bt}(\cT) \Pi_{\mb}(B)
\left[ \prod_{l\le L,k} \phi_{\frac{1}{\sqrt{n}}}(X_{lk}^B)\right]
\left[\prod_{(l,k)\notin\cT_{int}'} \1_{0}(\beta_{lk})\right]\e^{-\frac{n}{2}\|\b_\mT\|_2^2+n\X_\mT^{B'}\b_\mT}\pi(\b_\mT). 
%& = \left[\Pi_{\mb}(B)\prod_{l\le L,k} \phi_{\frac{1}{\sqrt{n}}}(X_{lk}^B)\right]
%\left[ \Pi_{\bt}(\cT) \prod_{(l,k)\in\cT} \frac{e^{\frac{n+1}{2}{Y^B_{lk}}^2}}{\sqrt{n+1}} \right]  \prod_{(l,k)\in\cT} \phi_{\frac{1}{\sqrt{n+1}}}(f_{lk}-X_{lk}^B)
%\prod_{(l,k)\notin\cT} \1_{f_{lk}=0},
\end{align*}
%where we have set $X_{lk}^B=n X_{lk}^B/(n+1)$. 
Integrating out $(\beta_{lk})$, one obtains the marginal density of $(B,\cT,X)$ as
\begin{equation}\label{eq:joint} 
\left[\Pi_{\mb}(B)\prod_{l\le L,k} \phi_{\frac{1}{\sqrt{n}}}(X_{lk}^B)\right]
 \Pi_{\bt}(\cT) N_X^B(\mT),
 \end{equation}
where
$$
N_X^B(\mT)=\int\prod_{(l,k)\in\mT_{int}'}\e^{nX_{lk}^B\beta_{lk}-n\beta_{lk}^2/2}d\pi(\b_\mT).
$$
The first bracket in \eqref{eq:joint} only depends on $B$ and $X$, from which one deduces that the posterior distribution of $\cT$, given $B$ and $X$, satisfies
$$ \Pi[\cT\given B,X]=\frac{W_X^B(\mT)}{\sum_{\mT\in\mathbb{T}_L} W_X^B(\mT)},\quad\text{with}\quad 
 W_X^B(\mT)=\Pi_\bT(\mT)N_X^B(\mT). $$
 Next, the posterior distribution  on $B$, given $X$, is given by 
\begin{align*}
\Pi[B \given X] & \propto  \Pi_\mb(B)\prod_{l\le L,k} \phi_{\frac{1}{\sqrt{n}}}(X_{lk}^B) 
\left\{ \sum_{T\in\bT} W_X^B(\mT)\right\}.
\end{align*}
{Also, we have %the distribution of $\be_\mT$ given $\cT, B, X$
%Finally, the posterior of $\{\beta_{lk}\}$ given $B,X$ writes as
\begin{align}\label{eq:posterior2}
(\beta_{lk})_{l\le L, k}\C (X_{lk})_{l\le L, k},\cT,B\
%&\propto \sum_{\cT\in \bT} 
%\Pi_{\bt}(\cT) \left[\prod_{(l,k)\in\cT} \phi_{\frac{1}{\sqrt{n}}}(X_{lk}-\beta_{lk})
% \pi(\b_\mT) \right]
%\left[\prod_{(l,k)\notin\cT} \phi_{\frac{1}{\sqrt{n}}}(X_{lk}-\beta_{lk})\delta_0(\beta_{lk})\right]\\
%& = \sum_{\cT\in \bT}
%\Pi_{\bt}(\cT)
%\e^{-\frac{n}{2}\|\b_\mT\|^2_2+n\X_\mT'\b_\mT}\pi(\b_\mT)
%\left[ \prod_{l\le L,k} \phi_{\frac{1}{\sqrt{n}}}(X_{lk})\right]\left[\prod_{(l,k)\notin\cT} \delta_0(\beta_{lk})\right]\\
 %= \sum_{\cT\in \bT} \Pi(\mT\C X)\pi(\b_\mT\C\X_\mT^B)\prod_{(l,k)\notin\cT} \delta_0(\beta_{lk}),
 \sim\ \pi(\b_\mT\C\X_\mT^B)\otimes \bigotimes_{(l,k)\notin\cT_{int}'} \delta_0(\beta_{lk}),
 \end{align}
 where the posterior density on the selected coefficients on $\cT$ is (in slight abuse of notation writing in the same way the distribution and its density) 
% \vspace{-0.5cm}
  \begin{align}
 \pi(\b_\mT\C\X_\mT^B)&= \frac{\e^{-\frac{n}{2}\|\b_\mT\|^2_2+n\X_\mT^{B\prime}\b_\mT} 
\pi(\b_\mT)}{N_X^B(\cT)}\label{eq:posterior_beta2}.
\end{align}
}

\iffalse
From these calculations one deduces that the posterior distribution $\Pi[\cdot\given X]$ satisfies
\begin{align*}
B \given X & \sim \Pi_\mb[\cdot\given X] \\
\cT \given B, X& \sim \Pi[\cdot\given B,X] \\
(f_{lk})_{l\le L,k} \given \cT,B,X & \sim \bigotimes_{(l,k)\in\cT} 
\cN\left(Y^B_{lk},\frac{1}{n+1}\right) \, \otimes \, 
\bigotimes_{(l,k)\notin\cT} \delta_0
\end{align*}

\fi
\medskip

{\em Controlling the Noise.}
Similarly as in the proof of Theorem \ref{thm-one}, we will condition on a set of large probability, where the noise level is relatively small.
Denote with $\mb$ the set of {\em all} breakpoints $B$ that can be obtained by performing steps (a) and (b) in Section \ref{sec:unbal} and that yield a  system $\Psi_{B}^A$ satisfying conditions (B1)--(B2) from Lemma \ref{lemma:complex}.
Recall $L=L_{max}=\lfloor\log_2 n\rfloor$ and $\veps_{lk}^B=\int_0^1 \psi^B_{lk}(u)dW(u)$, and let $\delta$ be as in $(B2)$. We define
\begin{equation} \label{eventb}
 \cA_{\mb}=\left\{ \max_{B\in\mb} \max_{l\in[0,L], k\in[0,2^{l}-1]} (\veps_{lk}^B)^2
\le D_1\log^{1+\delta}{n} 
\right\}
\end{equation}
for some $D_1>0$.
Using assumption $(B1)$, % that the resolution does not increase too much when working with $\psi_{lk}^B$ instead of $\psi_{lk}$, 
one can express every single $\psi_{lk}^B$ for $l\le L$ in terms of a number $C_0l^\delta$ of $\psi_{jm}$'s for $j\le l+D$, where $\psi_{jm}$ are the regular Haar wavelet functions from \eqref{haar}. 
That is, with $\cC^B_{lk}$ the set of such pairs $(j,m)$ and $\text{Card}(\cC^B_{lk})\le C_0 l^\delta$, we have
\[ \psi_{lk}^B = \sum_{(j,m)\in \cC^B_{lk}} p_{jm}^B \psi_{jm}, \]
for some real numbers $p_{jm}^B$ that satisfy 
$ \sum_{(j,m)\in \cC^B_{lk}} (p_{jm}^B)^2 =1$
(since $\psi_{lk}^B$ has a unit $L^2$--norm and  $\psi_{lk}$'s are orthonormal in $L^2[0,1]$). Next, we have
\begin{align*}
\veps_{lk}^B & = \sum_{(j,m)\in \cC_{lk}^B} 
p_{jm}^B \veps_{jm}.
\end{align*}
This itself implies the following, by the Cauchy-Schwarz inequality,
\[ |\veps_{lk}^B| \le \max_{l\le L,k} \left\{\text{Card}(\cC^B_{lk}) \max_{l\le L+D,k}\veps_{lk}^2 \right\}^{1/2} \le C_0^{1/2}L^{\delta/2} 
 \max_{l\le L+D,k}|\veps_{lk}|.  \]
Using $L\leq \log_2{n}$ and denoting  
$$
\cA:=\left\{\max\limits_{l\in[0,L+D],k\in[0,2^l-1]}\varepsilon_{lk}^2\leq 2\log(2^{L+D+1})\right\},
$$ one obtains the inclusion
$\cA \subset \cA_{\mb}$, provided that $D_1$ is chosen larger than a universal constant (in particular it is independent of $B$). This implies $P_{f_0}(\cA_{\mb}^c)\le P_{f_0}(\cA^c)\leq c_0/\sqrt{\log (2^{L+D+1})}$. 
Next, we follow the structure of the proof of Theorem \ref{thm-one}.
%{\em Control of $\beta_{lk}^0^B$'s.} 
\medskip

{\em Posterior Probability of Too Deep Trees.} 
For a given tree $\cT$, we again denote with  $\cT^-$ the pruned subtree obtained by turning the deepest rightmost internal node $(l_1,k_1)\in\mT_{int}$ into a leaf. 
Given $B\in\mb$, we proceed as in the proof of Lemma \ref{lemma:dim} and evaluate the ratio $W_X^B(\cT)/W_X^B(\cT^-)$.
When $l_1>\cL_c$, with $\cL_c$ as in \eqref{cutoff},   Lemma \ref{lemma:complex} leads to $(B2)$, that is $|\beta_{l_1k_1}^B|\lesssim(\log{n})^{\delta/2} \sqrt{\log{n}/n}$ for large enough $n$.
% Note that the dependence in M in the log term vanishes for $n$ large enough 
 Similarly as in \eqref{eq:Xlk}, 
  we can write for $l_1>\mL_c$ and {some $C_2>0$ (depending on $E,D$ only)}, on the set $\cA_\mb$ from \eqref{eventb},  
\[ (X_{l_1k_1}^B)^2\le \frac{C_2}{n}({\log{n}})^{1+\delta}. \] 
Under the Galton-Watson process prior from Section \ref{sec:bc1} with  $p_l \le 1/2$ and the independent prior  with $\Sigma_\mT=I_{|\mT_{ext}|}$ this gives {\em for all $B\in\mb$} and
 $d\ge \cL_c$,
\[ \frac{W_X^B(\cT)}{W_X^B(\cT^-)} \le  2\,p_{d-1} \e^{(C_2/2)(\log{n})^{1+\delta}},\]
from which one deduces that 
\begin{equation} \label{toodeep}
\Pi[ d(\cT) > \cL_c \given B,X] \le 4\, \e^{(C_2/2)(\log{n})^{1+\delta}}
\sum_{d=\cL_c+1}^{L_{max}} 2^{d-1}p_{d-1}.
\end{equation}
The right side goes to $0$ {at rate $\e^{-C(\log{n})^{1+\delta}}$} if, e.g., $p_d$ is of the order $(1/\Gamma)^{d^{1+\delta}}$ for some  large enough $\Gamma>0$. This also holds for a variant of the tree prior $\pi(\mT)\propto \e^{-c|\mT_{ext}|\log^{1+\delta}n}$ and the conditionally uniform prior from Remark \ref{remark:general} using 
$\pi(K)\propto \e^{-c K\,\log^{1+\delta}n}$ where $K=|\mT_{ext}|$. A statement similar to \eqref{toodeep} can  also be obtained for the general prior $\pi(\b_\mT)\sim\mathcal{N}(0,\Sigma_\mT)$ where $\lambda_{min}(\Sigma_{\mT})>\sqrt{1/(\log n)^{1+\delta}}$.

\medskip

{\em Posterior Probability of Missing a Significant Node.}
We show a variant of Lemma \ref{lemma:signal} assuming instead that a   signal node $(l_S,k_S)$ satisfies
\begin{equation} \label{condlog2}
l_S\le \cL_c,\qquad |\beta^{0B}_{l_Sk_S}|\ge \frac{A(\log{n})^{1+\frac{\delta}{2}}}{\sqrt{n}},
\end{equation}
for some $A>0$ to be chosen below. As before, for a tree $\mT\in\bt_{\setminus (l_S,k_S)}$ 
that does not have a cut at $(l_S,k_S)$, we denote with $\cT^+$  the smallest full binary tree (in terms of number of nodes) that contains $\cT$ and cuts at $(l_S,k_S)$.  
Using similar arguments as in the proof of Lemma \ref{lemma:signal}, we use the fact 
$(X^B_{l_Sk_S})^2\ge ({\beta^{0B}_{l_Sk_S}})^2/2-(\veps^{B}_{l_Sk_S})^2/n$ to find that on the event $\cA_{\mb}$ and for $A>0$ large enough,
\begin{equation} \label{eq:ratio_tplus2}
\frac{W_X^B(\cT)}{W_X^B(\cT^+)} 
\le  \Gamma^{l_S^{1+\delta}(l_S+1)} \e^{\frac{3}{2}D_1(l_S+1)\log^{1+\delta}n-\frac{A^2}{8}\log^{2+\delta}{n}}\leq \e^{-\frac{A^2}{16}\log^{2+\delta}{n}}
\end{equation}
under the independent Gaussian prior on $\b_\mT$ and  the Galton-Watson process prior from Section \ref{sec:bc1} with $p_l \asymp (1/\Gamma)^{l^{1+\delta}}$.
Following the steps in the proof of Lemma \ref{lemma:sig}, one can show similarly that $\Pi\left[(l_S,k_S)\notin\mT_{int}\C X,B\right]\rightarrow 0$ for each {$B\in\mb$ sufficiently quickly. More precisely, if 
%, for $A>0$ to be chosen suitably large,
\begin{equation}
S^B(f_0;A) = \left\{ (l,k):\  |\beta_{lk}^{0B}| \ge A\frac{(\log{n})^{1+\frac{\delta}{2}}}{\sqrt{n}} \right\},
\end{equation}
where $\mL_c$ is defined in \eqref{cutoff}, we have, on the 
 event $\cA_{\mb}$ and for $A$ large enough,
\begin{equation}\label{eq:sig}
\Pi\left[\left\{\cT:\, S^B(f_0;A)\nsubseteq \cT\right\} \C X\right]\le  \e^{-C(\log{n})^{1+\delta}}.
\end{equation}
uniformly in $B\in\mb$.} % \sbl{[see if we replace $C$ by $C_3$ or so]}
%(although $o(1)$ is sufficient for the bounds below, it can more precisely be shown that the probability at stake is bounded by $$ uniformly in $B\in\mb$).
 This statement can be obtained also for the general prior $\pi(\b_\mT)\sim\mathcal{N}(0,\Sigma_{\mT})$ with $\lambda_{max}(\Sigma_{\mT})\lesssim n^a$ for some $a\geq 1$ and for other tree priors
from Section \ref{sec:thm-one}.
 
\medskip
{\em Putting Pieces Together.} 
Let us also set
\begin{equation}\label{eq:tree_eventsB}
\mathsf{T}^B=\{\mT: d(\mT)\leq \mathcal{L}_c, \,S^B(f_0;A)\subset \mT\},
\quad\quad  \mE^B=\{f_{\mT,\b}:\mT\in \mathsf{T}^B\}.
\end{equation}
%, recalling the notation $S^B(f)=\{(l,k):\ f_{lk}^B\neq 0\}$, and setting $S_{\cL_c}:=\{(l,k): l\le \cL_c\}$ with $\cL_c$ as in \eqref{cutoff},
%\begin{equation}
%\cE^B = \{ f:\ S^B(f_0;A)\subset S^B(f)\subset S_{\cL_c} \}.
%\end{equation}
%Noting that if $f$ is sampled from the prior distribution, its support $S^B(f)$ coincides with the tree $\cT$, one may write $\cE^B$  equivalently  in terms of a collection of trees $\cT_{\cE^B}=\{\cT: S^B(f_0;A)\subset\cT,\ d(\cT)\le \cL_c\}$, that is $\cE^B=\{f:\ S^B(f)\in \cT_{\cE^B}\}$. 
From the two previous  subsections one obtains that for some constant $C>0$  
\[ \Pi[\mT\notin\mathsf{T}^B \given X,B] \le \e^{-C(\log{n})^{1+\delta}},\]
 {\em for any} possible set of breakpoints $B\in\mathbb B$ (that satisfy  the balancing conditions). The uniformity in $B$ is essential in the next bounds. 

Using the definition of the event $\cA_\mb$ from \eqref{eventb}, one can bound
\begin{align*}
E_{f_0}\Pi[ \|f_{\mT,\b}-f_0\|_\infty > \veps_n \given X]  & \le P_{f_0}[\cA_\mb^c] 
 + E_{f_0}\left\{\Pi[ \|f_{\mT,\b}-f_0\|_\infty > \veps_n \given X]\1_{\cA_\mb}\right\}.
 \end{align*}
By decomposing the posterior along $B$ and $\cT$ and using Markov's inequality one obtains, on the event $\cA_\mb$,
\begin{align*}  
\lefteqn{\Pi[\|f_{\mT,\b}-f_0\|_\infty > \veps_n \given X]  = \sum_{B} \Pi[B\given X] \,\sum_{\cT} \Pi[\cT\given X,B]
\,\Pi[ \|f_{\mT,\b}-f_0\|_\infty >\veps_n \given X,\cT,B]} && \\
 & \le  \sum_{B} \Pi[B\given X]  \Pi[\mT\notin\mathsf{T}^B \given X,B]
+ \sum_{B} \Pi[B\given X]   \sum_{\mT\in\mathsf{T}^B} \Pi[\cT\given B,X]
 \Pi[ \|f_{\mT,\b}-f_0\|_\infty >\veps_n \given X,\cT,B] \\
&  \le  \e^{-C(\log{n})^{1+\delta}} + \sum_{B} \Pi[B\given X]   \sum_{\cT\in \mathsf{T}^B} \Pi[\cT\given B,X]
\veps_n^{-1} \int \|f_{\mT,\b}-f_0\|_\infty d\Pi[f_{\mT,\b}\given X,\cT,B].
\end{align*}    
%The last term in the last display can be further bounded by 
% using the  previous bounds on the supremum norm. 
Let us now turn to bounding $\|f_{\mT,\b}-f_0\|_\infty$. 
First, note that unlike the traditional Haar basis,  the UH basis system is never built up until
$L=\infty$ because, by construction, we stop splitting when there are no $x_i$ are available (i.e. we do not split nodes that are not {\em admissible}). 
In result,  the very high frequencies are not covered by the system, which might  induce some unwanted bias.
This is, however, {\em not an issue} with our {\em weakly balanced} UH wavelets. 
The following Lemma shows that in weakly balanced UH systems, all nodes at levels $l\leq \La:=\lfloor\log_2(n/\log^cn)\rfloor$ for any $c>0$ are admissible.

\begin{lemma}\label{lemma:admissible}
Consider a weakly balanced UH wavelet system $\Psi_{A}^B$, where $A$ is the set of admissible nodes $(l,k)$ in the sense that $\mX\cap (l_{lk},r_{lk}]\neq \emptyset$ with 
$\mX=\{x_i:x_i=1/n,1\leq i\leq n\}$. Let $c>0$, then for $\La=\La(c)=\lfloor\log_2(n/\log^cn)\rfloor$,  we have
$$
A\supset \{(l,k): l\leq \La\}.
$$
\end{lemma}
\begin{proof}
The proof follows from the fact that the granularity of weakly balanced UH systems is very close to $l$.
In Example \ref{counterexample} we defined the granularity $R(l, \Psi_A^B)$  of  the $l^{th}$  layer as the smallest integer $R\ge 1$ such that 
$\min_{0\leq k<2^l}\min\{|L_{lk}|,|R_{lk}|\}=j/2^R$ for some  $j\in\{1,2,\dots, 2^{R-1}\}$.
From Lemma \ref{lemma:grbis}, the granularity of weakly balanced systems $\Psi_{A}^B$ is no larger than $l + D$. This means that for $l< \La$, $0\leq k<2^l$, any $c>0$ and $n$ large enough
$$
\min\{|L_{lk}|,|R_{lk}|\}\geq 1/2^{l+D}>\frac{\log^cn}{2^D n}>1/n.
$$
This implies that $\mX\cap (l_{lk},r_{lk}]\neq \emptyset$ for any $(l,k)$ with $l\leq \La$, where we used the fact that  $(l_{lk},r_{lk}]$ is either  $R_{l-1\,\lfloor k/2\rfloor}$ (for when $(l,k)$ is the right child) or 
$L_{l-1\,\lfloor k/2\rfloor}$ (for when $(l,k)$ is the left child).
\end{proof}

Next, we show that the weakly balanced UH systems are indeed rich enough to approximate $f_0$ well.

\begin{lemma}\label{lemma:bias}
Consider the weakly balanced UH system $\Psi_{A}^B$.  Let $f_0^\La$ denote the $L^2$--projection of $f_0\in \mathcal{H}^\alpha_M$ onto $\text{Vect}\{\psi_{lk}^B: l\leq \La\}$ for $\La=\lfloor\log_2(n/\log^cn)\rfloor$ with some $c>0$. Then
 $$
 \|f_0-f_0^\La\|_\infty \lesssim |\La 2^{-\La}|^\al \leqa 
 (\log^{c+1}{n}/n)^\al.
 $$
\end{lemma}
\begin{proof}
The $L^2$--projection is  a step function 
$f_0^\La=\sum_{m}\mathbb{I}_{\Omega_m}\wt\beta_m$ supported on the pieces $\Omega_m\in\{L_{\La k},R_{\La k}:\, 0\leq k<2^\La \}$ where the jump sizes equal $\wt\beta_m=|\Omega_m|^{-1}\int_{\Omega_m}f_0(x)dx$. From the H\"{o}lder continuity in \eqref{eq:haar2} we have 
$|f_0(x)-f_0^\La(x)|\leq M|\Omega_m|^\alpha$  for $x\in\Omega_m$.  From the definition of weakly balanced UH systems, we have $\max_m|\Omega_m|\leq \frac{C+\La}{2^{\La+D}}$. The rest follows from the definition of $\La$.
\end{proof}

We can now write the following decomposition.
For $f_{\mT,\b}$ in $ \mE^B$, we have
\begin{align} 
\lefteqn{ \|f_{\mT,\b}-f_0\|_\infty \lesssim 
\sum_{l\le \cL_c} 2^{l/2} \max_{0\le k<2^l} |\beta_{lk}^B-\beta_{lk}^{0B}|}
\nonumber\\
&& \qquad\  +\sum_{\cL_c\le l\le \La} 2^{l/2} \max_{0\le k<2^l}|\beta_{lk}^{0B}| + \|f_0-f_0^\La\|_\infty. \label{decompose}
\end{align}
In the last display, we  have used the fact that for weakly balanced UH systems  one has
$$
\max_{0\leq k<2^l}\|\psi_{lk}^B\|_\infty<\max_{0\leq k<2^l}\left[\left(\frac{1}{|L_{lk}|}\vee\frac{1}{|R_{lk}|}\right)\frac{1}{\sqrt{|L_{lk}|^{-1}+|R_{lk}|^{-1}}}\right]< {2^{(l+D)/2}}.
$$
The second term in \eqref{decompose} can be upper-bounded by $(\log{n})^{1+\delta/2}(\log{n}/n)^{\al/(2\al+1)}$ by using $(B2)$ and the definition of $\cL_c$. 
Using Lemma \ref{lemma:bias}, the term $\|f_0-f_0^\La\|_\infty$ is 
 {always of smaller order than the previous one (as the bound decreases as $n^{-\al}$ up to a logarithmic factor)}.   
%smaller than $(\log{n})^{1+\delta/2}(\log{n}/n)^{\al/(2\al+1)}$  for, e.g., $c=1/\alpha$.

{Regarding the first term, one obtains for $\cT\in\mathsf{T}^B$
\begin{align*} 
\lefteqn{\int \max_{0\le k<2^l} |\beta_{lk}^B-\beta^{0B}_{lk}|d\Pi[f_{\mT,\b}\given X,\cT,B] } && \\
 & = \int\max\left(\max_{0\le k<2^l,\, (l,k)\notin\cT_{int}} |\beta^{0B}_{lk}| , \max_{0\le k<2^l,\, (l,k)\in\cT_{int}}  |\beta_{lk}^B-\beta^{0B}_{lk}| \right)d\Pi[f_{\mT,\b}\given X,\cT,B] \\
 & \le A\frac{(\log{n})^{1+\frac{\delta}{2}}}{\sqrt{n}}
 +  \int \max_{0\le k<2^l,\, (l,k)\in\cT_{int}}  |\beta_{lk}^B-\beta^{0B}_{lk}| d\Pi[f_{\mT,\b}\given X,\cT,B],
\end{align*}
where we have used that on the set $\cE^B$, selected trees cannot miss any true signal larger than $A(\log{n})^{1+\delta/2}/\sqrt{n}$.
This means that any node $(l,k)$ that is not in a selected tree must satisfy $|\beta^{0B}_{lk}|\le A(\log{n})^{1+\delta/2}/\sqrt{n}$.

We now focus on the independent prior $\b_\cT\sim\mathcal{N}(0,I_{|\mT_{ext}|})$.
We have seen above that, given $X$, $B$ (so for fixed $\veps_{lk}^B$) and $\cT$,  if $(l,k)$ belongs to $\cT_{int}$, the difference $\beta_{lk}^B-\beta^{0B}_{lk}$ has a Gaussian distribution $Q_{lk}$ given by 
\[ 
Q_{lk}\overset{\cL}{=} 
X_{lk}-\beta^{0B}_{lk} + \cN\left(0,\frac1{n+1}\right)=-\frac{\beta^{0B}_{lk}}{n+1}+ \frac{\sqrt{n}\veps^B_{lk}}{n+1} +  \cN\left(0,\frac1{n+1}\right).   \]
If $Z_{lk}$ are arbitrary random variables distributed according to $Q_{lk}$, and $\cZ_{lk}$ arbitrary $\cN(0,1)$ random variables,
\[ \E\left[ \max_{0\le k<2^l} |Z_{lk}| \right] \le 
\max_{0\le k<2^l} \frac{|\beta^{0B}_{lk}|}{n} 
+ \max_{0\le k<2^l}  \frac{|\veps^B_{lk}|}{\sqrt{n}} 
+ \frac{1}{\sqrt{n}}\E\left[ \max_{0\le k<2^l} |\cZ_{lk}| \right].  \]
On the event $\cA_\mb$ from \eqref{eventb}, the sum of the first two terms on the last display is bounded by $M/n+C\sqrt{(\log{n})^{1+\delta}/n}$ while the last expectation is at most $C\sqrt{l/(n+1)}$ by Lemma \ref{lemma:exp_max}.
This implies 
\[ \int\max_{0\le k<2^l,\, (l,k)\in\cT_{int}}  |\beta_{lk}^B-\beta^{0B}_{lk}| d\Pi[f_{\mT,\b}\given X,\cT,B]
\le C'\sqrt{\frac{(\log{n})^{1+\delta}}{n}}
\]
{\em uniformly} over  $B$ and $\mathsf{T}^B$, where we have used $l\le C\log{n}$. 
Putting the various pieces together and using the fact that $P_{f_0}[\cA_\mb^c] = o(1)$, we obtain
\begin{align*}
\lefteqn{E_{f_0} \Pi[ \|f_{\mT,\b}-f_0\|_\infty > \veps_n \given X] }&&\\
 & \le  o(1) + \veps_n^{-1}\left\{
\sum_{l\le \cL_c} 2^{l/2} \left[ A\frac{(\log{n})^{1+\delta/2}}{\sqrt{n}} + C'\sqrt{\frac{\log^{1+\delta}{n}}{n}} \right] + 2(\log{n})^{1+\frac{\delta}{2}}
 \left(\frac{\log{n}}{n}\right)^{\frac{\al}{2\al+1}}  \right\}
 \\
& \le o(1) + \veps_n^{-1}\left\{\left[A(\log{n})^{\frac{1+\delta}{2}} +C'(\log n)^{\frac{\delta}{2}}\right] 2\sqrt{ \frac{2^{\cL_c}\log{n}}{n} } + 
2(\log{n})^{1+\frac{\delta}{2}}
 \left(\frac{\log{n}}{n}\right)^{\frac{\al}{2\al+1}}
\right\}\\
& \le o(1) + \veps_n^{-1}2C'\left[A(\log{n})^{\frac{1+\delta}{2}} +4(\log n)^{1+\frac{\delta}{2}}\right] \left(\frac{\log{n}}{n}\right)^{\frac{\al}{2\al+1}}.
\end{align*}
This means that one can set 
\[ \veps_n ={(\log{n})^{1+\delta/2}} \left(\frac{\log{n}}{n}\right)^{\frac{\al}{2\al+1}} \]
and this is the obtained posterior rate in terms of the supremum norm.  A similar conclusion can be obtained for the general prior $\pi(\b_\mT)\sim\mathcal{N}(0,\Sigma_\mT)$ using Lemma \ref{lemma:signal}
under the assumption $\lambda_{min}(\Sigma_\mT)\gtrsim\sqrt{1/\log^{1+\delta}n}$.

\section{Proofs for additional results}

\subsection{Proof of Theorem \ref{thm_np_emp}: rate in non-parametric regression } \label{sec:proofsreg}

 \begin{proof}
 For the first statement of Theorem \ref{thm_np_emp}, it suffices to note that the same exact proof of Theorem \ref{thm-one} can be used. Indeed, first, by  Lemma \ref{lem:empnorm}, the maximum norm $\|f-f_0\|_{\infty,n}$ is bounded in a similar way as $\|f-f_0\|_\infty$ in terms of empirical coefficients $\bb-\bb^0$ (instead of original wavelet coefficients). Second, the H\"older regularity of $f_0$ induces a decrease of the order $2^{-l(1/2+\al)}$ in terms of empirical wavelet coefficients $\bb^0$ (Lemma \ref{lem:empdec}). Third, under $P_{f_0}$, the distribution of the observed $Z_{lk}$ in \eqref{mod:wne} is $\cN(b^0_{lk},1/n)$, identical to that of $X_{lk}\sim \cN(\beta^0_{lk},1/n)$, up to replacing $\beta^0_{lk}$ by $b^0_{lk}$.  It is then enough to prove that $\sum_{l\le L_{max}} 2^{l/2} \max_k |b_{lk}-b_{lk}^0| \leqa \veps_n$ under the posterior distribution, which follows by the same proof as Theorem \ref{thm-one}.
 
 For the second statement of Theorem \ref{thm_np_emp}, for a given $t\in[0,1]$, let $i_t/n$ denote the closest (leftmost) design point $t_i=i/n$. For a given function $f$ on $[0,1]$, let us denote $\bar{f}=\cI((f(t_i))_{1\le i\le n})$ as a shorthand. Then for any $t\in[0,1)$ (assuming $i_t\neq n$; if $i_t=n$ one adapts the argument) and $f_0\in \cH_{\al}^M$,
 \begin{align*}
 |\bar{f}(t)-f_0(t)| & \le |f_0(t)-f_0(i_t/n)|+|f_0(i_t/n)-\bar{f}(i_t/n)|
 +|\bar{f}(i_t/n)-\bar{f}(t)| \\
 & \le M n^{-\al} + \|\bar{f}-f_0\|_{\infty,n}+|\bar{f}(i_t/n)-\bar{f}((i_{t}+1)/n)|,
 \end{align*}
where the first line uses the triangle inequality and the second that $f_0\in \cH_{\al}^M$ as well as the fact that when linearly interpolating, the difference of two function values inbetween breakpoints is always the largest when taken at the two different breakpoints. Using the triangle inequality again,
\[ |\bar{f}(i_t/n)-\bar{f}((i_{t}+1)/n)|\le 2\|\bar{f}-f_0\|_{\infty,n} 
+ |f_0((i_{t}+1)/n)-f_0(i_t/n)|. \]
Using again the H\"older property of $f_0$, and combining with the previous bounds one gets, noting that $\|\bar{f}-f_0\|_{\infty,n}=\|f-f_0\|_{\infty,n}$,
\[ \|\bar{f}-f_0\|_\infty \le 2Mn^{-\al}+3\|f-f_0\|_{\infty,n}.  \] 
This shows that if $\|f-f_0\|_{\infty,n}\leq \zeta_n$ and if $n^{-\al}=o(\zeta_n)$, then $\|\bar{f}-f_0\|_\infty \leqa \zeta_n$. As $n^{-\al}=o(\veps_n)$, this shows that a posterior rate on $f$ in the $\|\cdot\|_{\infty,n}$ norm translates into the same rate for the distribution $\bar{\Pi}_Y$ in terms of the supremum norm. 
 \end{proof}

%\subsection{Lemmata for regression (Approach 1)}

\begin{lemma} \label{lem:empnorm}
Given two functions $f,f_0$ defined on $[0,1]$ with empirical wavelet coefficients 
$\bb=(b_{lk}), \bb^0=(b_{lk}^0)$ given as in \eqref{empwave}, 
\[ 
\|f-f_0\|_{\infty,n} \le \sum_{l\le L_{max}} 2^{l/2} 
\max_{k} \left|b_{lk}-b_{lk}^0\right|. 
\]
\end{lemma}
\begin{proof}
Using that $F=X\bb$ and next that $\sum_k |\psi_{lk}(t_i)|=2^{l/2}$,
\begin{align*}
\max_i |f(t_i)-f_0(t_i)| & =\max_i \left|
\sum_{l\le L_{max},\, k}  (b_{lk}-b_{lk}^0)\psi_{lk}(t_i) \right| \\
& \le \sum_{l\le L_{max}} \max_{k} \left| b_{lk}-b_{lk}^0 \right| \max_i \sum_k |\psi_{lk}(t_i)| \\
& \le \sum_{l\le L_{max}} 2^{l/2} \max_{k} \left|b_{lk}-b_{lk}^0\right|. \qquad \qedhere
\end{align*} 
\end{proof}

\begin{lemma} \label{lem:empdec}
Suppose $f_0$ is in $\cH_M^\al[0,1]$ as in \eqref{eq:haar2} with $\al\in(0,1], M>0$, with empirical Haar wavelet coefficients 
$\bb^0=(b^0_{lk})$ as in \eqref{empwave}. Then for any $l\le L_{max}$ and $k$,
\[ |b_{lk}^0| \le C2^{-l(1/2+\al)}. \]
\end{lemma}
\begin{proof}
By splitting the support $I_{lk}$ of $\psi_{lk}$ in two halves $I_{lk}^+, I_{lk}^-$,
\begin{align*}
\lefteqn{n^{-1}|\sum_{i=1}^n f_0(t_i)\psi_{lk}(t_i)| =\frac{2^{l/2}}{n}\left|\sum_{i:\, t_i\in I_{lk}^+} f_0(t_i)-\sum_{i:\, t_i'\in I_{lk}^-} f_0(t_i')\right|}&&\\
& \le \frac{2^{l/2}}{n} \text{Card}(I_{lk}^+) \max_{i}C|t_i-t_i'|^\al \le 
\frac{2^{l/2}}{n} \frac12 \frac{n}{2^l}C2^{-l\al}. \qquad \qedhere
\end{align*}
\end{proof}
\begin{lemma} \label{lem:empcomp}
In the setting of Lemma \ref{lem:empdec}, let  $(\be_{lk}^0=\psg f_0 , \psi_{lk}\psd)$ denote the Haar wavelet coefficients of $f_0$. Then for any $l\le L_{max}$ and $k$,
\[ |b_{lk}^0-\be_{lk}^0| \le C2^{-l/2} n^{-\al}. \]
Further assume that $f_0\in \cH_{SS}'(\al,M,\veps)$, i.e. belongs to $\cH^{\al}_M$ and is self-similar as in Definition \ref{def-ssi}. Then for any diverging sequence $l_n\to\infty$ with $2^{l_n}\le n$,
\[ \sup_{(l,k):\ l\ge l_n} |b_{lk}^0| \ge C2^{-l_n(1/2+\al)}. \] 
\end{lemma} 
\begin{proof}
For the first part, let $f_0^D$ be the piecewise constant function that equals $f_0(t_i)$ on $[t_i,t_{i+1})$. Then observing that
\[ b_{lk}^0= 2^{l/2}(\int_{I_{lk}^+} f_0^D -\int_{I_{lk}^-} f_0^D), \]
it is enough to show that $|\int_{I_{lk}^+} (f_0^D-f_0)|\leqa 2^{-l}n^{-\al}$, since then a symmetric bound similarly holds on $I_{lk}^-$. This follows from $\|f_0^D-f_0\|_\infty \leqa n^{-\al}$, using the H\"older property of $f_0$. 

To prove the second part of the lemma, the proof of Proposition 3 in 
\cite{hoffmann_nickl} gives, for $f_0\in \cH_{SS}(\al,M,\veps)$, for any diverging sequence $l_n$,
\[ \sup_{(l,k):\ l\ge l_n} |\be_{lk}^0| \ge C2^{-l_n(1/2+\al)}. \]
The result now follow by combining the triangle inequality, the first part of the lemma, and $n^{-\al}\le 2^{-l_n \al}$ (using $2^{l_n}\le n$ by assumption).
\end{proof}
\subsection{Proof of Theorem \ref{csthmemp}: band in non-parametric regression} \label{proof_csthmemp}
The proof follows the lines of that of Theorem \ref{csthm}, here we only highlight the few differences. First, one shows that the empirical median tree $\cT_X^{\sim}$ verifies the properties stated in Lemma \ref{lembulk2} for the median tree $\cT_X$ in white noise: the depth of the tree satisfies $2^{d(\cT_X^{\sim})}\leqa
 2^{\cL_c}$, and $\cT_X^{\sim}$ contains all nodes $(l,k)$ such that $|b^0_{lk}|\ge A\log{n}/\rn$, for some $A>0$. The proof is the same as that of Lemma \ref{lembulk2}. From this one deduces that the diameter of the credible set $\cc_n^e$ is as announced, by the same argument as in the proof of Theorem \ref{csthm}. 

Second, one shows that the empirical median tree estimator $\tilde{f}_T$ converges at rate $(\log^2/n)^{\al/(2\al+1)}$ in terms of the $\|\cdot\|_{\infty,n}$--norm: one adapts the proof in white noise. The event $B_n$ and set $\cS$ are defined similarly, with $\hat{f}_T$ replaced by $\tilde{f}_T$, $\b^0$ by $\bb_0$ and the noise sequence $\veps_{lk}$ in the event $\cA$ by the noise sequence $\zeta_{lk}$ from \eqref{mod:wne}. 
On the corresponding $B_n$, we have
\begin{align*}
 \| \tilde{f}_T - f_0 \|_{\infty,n}  \le & \max_i \Big|
\sum_{l\le L_{max},\, k}  (Z_{lk}\1_{(l,k)\in\cT_X^\sim}-b_{lk}^0)\psi_{lk}(t_i) \Big| \\
\le & \sum_{l:\, 2^l\le C_12^{\cL_c}} 2^{l/2}
{ \max  \left(\max_{0\le k<2^l:\, (l,k)\in\cS} |Z_{lk} - b_{lk}^0|,\max_{0\le k<2^l:\, (l,k)\notin\cS} \{|b_{lk}^0| \}\right)}
 \\
& + \sum_{l:\, n \ge 2^l > C_12^{\cL_c}} 2^{l/2}\max_{0\le k<2^l} |b_{lk}^0|\\
\lesssim & \ 2^{\cL_c/2} \sqrt{\frac{\log{n}}{n}} +  \sum_{l:\, 2^l \le C_12^{\cL_c}}2^{l/2} \min\left(\max_{0\le k<2^l}|b_{lk}^0|, A\frac{\log{n}}{\rn}\right) + 2^{-\al \cL_c},\\
& \leqa   (\log^2{n}/n)^{\al/(2\al+1)},
\end{align*}
where we proceed as in the proof of Theorem \ref{csthm} but this time using the bound on empirical coefficients from Lemma \ref{lem:empdec}.

Third, one shows that $\tilde\sigma_n$ as in \eqref{radiusprox} is appropriately large. To do so, using that $f_0\in \cH_{SS}(\al,M,\veps)'$ and arguing as in the proof of Theorem \ref{csthm}, for $\La_n(\al)$ as in \eqref{defLa},
\[  \sup_{(l,k):\ l\ge \La_n(\al)} |\beta_{lk}^0| \ge C(M,\psi,\alpha,{\veps})\eta^{-\al-1/2}\frac{\log{n}}{\rn}. \]
Now the same lower bound up to a different constant is obtained for $b_{lk}^0$, using Lemma \ref{lem:empcomp}. From there on the proof is the same as for Theorem \ref{csthm}, which concludes the proof.

\subsection{Proof of Theorem \ref{thm_bvm}: BvM}\label{proof_thm_bvm}
This BvM statement can be shown, for example, by  verifying the conditions in Proposition 6 of \cite{castillo_nickl1} or by proceeding as in the proof of Theorem 3.5 of \cite{ray}.  The first requirement is the ``tightness condition" (Proposition 6 of \cite{castillo_nickl1}) summarized by the following lemma.
\begin{lemma}
Under the assumptions of Theorem \ref{thm_bvm}, we have
$$
E_{f_0}\Pi(\|f_{\mT,\b}-f_0\|_{\M(w)}\geq M_n n^{-1/2}\C X)\rightarrow0.
$$
\end{lemma}
\begin{proof}
%For $j\in\N$, denote with  $V_j$ is the subspace of $\M(w)$ consisting of only the scales $l\leq j$ and denote with  $\pi_j$ the projection operator  onto $V_j$.
%Similarly, $\pi_{>j}$ is the projection operator onto the complement of $V_j$.
Similarly as in Section \ref{sec:proof_rate}, for $j\in\N$ and $f\in L^{2}[0,1]$ we denote with $f^j$ the $L^2$ projection onto the first $j$ layers of wavelet coefficients and write $f=f^j+f^{\backslash j}$.
Similarly as in the proof of Theorem \ref{thm-one}, we denote with $\cA$ the event \eqref{event} and with  $S(f_0;A)$ the set \eqref{signalset}.
Recall also the notation 
$$
\mathsf{T}=\{\mT: d(\mT)\leq \mathcal{L}_c, S(f_0;A)\subset \mT\}
\quad\text{ and}\quad  \mE=\{f_{\mT,\b}:\mT\in \mathsf{T}\}
$$ 
from \eqref{eq:tree_events}, where $\mE$ is the subset of tree-based functions $f_{\mT,\b}$ with up to $\mL_c$ leaves that do not miss any signal.
Similarly as in the proof of Theorem \ref{thm-one},  we will condition on the event $\cA$ and focus on the set $\mE$ (as in \eqref{eq:conditioning}). 
For simplicity, we will write $j_0=j_0(n)$. Following \cite{ray}, one can write for some suitably chosen $D=D(\eta)>0$, where $\eta>0$ {is a fixed small constant.}

% \ma{[Explain what is $\eta$ here; at the moment we still have a bit of mix of an asymptotic argument with terms in $o(1)$ and bounds by $\eta$ $\to$ unify this by first fixing $\eta$ and make terms smaller than e.g. $\eta/4$]}
\begin{align}
&E_{f_0}\left\{\Pi( f_{\mT,\b}:\| f_{\mT,\b}-f_0\|_{\mathcal{M}_{(w)}}\geq M_n n^{-1/2}\C X)\right\}\leq o(1)\notag\\
&\quad\quad+ E_{f_0}\left\{\Pi(f_{\mT,\b}\in\mE:\|f_{\mT,\b}^{j_0}-f_0^{j_0}\|_{\mathcal{M}_{(w)}}\geq D\, n^{-1/2}\C X)\1_\cA\right\}\label{eq1}\\
&\quad\quad+E_{f_0}\left\{\Pi(f_{\mT,\b}\in\mE:\|f_{\mT,\b}^{\backslash j_0}-f_0^{\backslash j_0}\|_{\mathcal{M}_{(w)}}\geq \widetilde{M}_n\, n^{-1/2} \C X)\1_\cA\right\}\label{eq2},
\end{align}
where $\widetilde{M}_n=M_n-D\rightarrow\infty$ as $n\rightarrow\infty$.  We have for $f_{\mT,\b}\in\mE$
 \begin{equation}\label{eq:sum1}
 \|f_{\mT,\b}^{\backslash j_0}-f_0^{\backslash j_0}\|_{\mathcal{M}_{(w)}}\leq\sup_{j_0<l\leq \mL_c}\frac{\max_k|\beta_{lk}-\beta_{lk}^0|}{w_l}+\|f_0^{\backslash \mL_c}\|_{\mathcal{M}(w)}.
 \end{equation}
% Focusing on the term \eqref{eq2} when $j_0> \mL_c$, we  use the fact that $w_l/\sqrt{ l}\rightarrow\infty$ and obtain,
 From the H\"{o}lder property \eqref{eq:haar} and the definition of $\mL_c$ in \eqref{cutoff} we have
\begin{equation}\label{eq:remainder}
\|f_0^{\backslash \mL_c}\|_{\mathcal{M}(w)}=\max_{l> \mL_c}\frac{\max_k|\beta_{lk}^0|}{w_l}
 \lesssim \frac{2^{-\mL_c(\alpha+0.5)}}{\sqrt{ \mL_c}}\leq C/\sqrt{n},
\end{equation}
where we used the fact that $\{w_l\}$ is {\em admissible} in the sense that $w_l/\sqrt{l}\rightarrow\infty$ as $l\rightarrow\infty$.
Using Markov's inequality and bounds \eqref{eq:sum1} and \eqref{eq:remainder}, the term \eqref{eq2} can be bounded with
 \begin{equation}\label{eq:markov}
E_{f_0}\left\{ \frac{\sqrt{n}}{\wt M_n\,w_{j_0}}\int_{\mE}\max_{j_0<l\leq \mL_c}\max_{0\leq k\le 2^l}|\beta_{lk}-\beta_{lk}^0|d\Pi[f_{\mT,\b}\C X]\1_\cA\right\}+ C/\wt M_n.
 \end{equation}
Using similar arguments as in Section \ref{sec:proof_rate} and using Lemma \ref{lemma:signal},  we can upper bound the integral above on the event $\cA$ by 
\begin{align*}
\lefteqn{\sum_{\mT\in\mathsf{T}}\pi[\mT\C X]\int \max_{j_0<l\leq \mL_c}\max_{0\leq k\le 2^l}|\beta_{lk}-\beta_{lk}^0|d\Pi[\b_\mT\C X]}\\
&&\leq  \left(A\frac{\log n}{\sqrt{n}}+C'\sqrt{\frac{\log n}{n}}\right)
\lesssim \frac{\log n}{\sqrt{n}}.
\end{align*}
For $w_{j_0}\geq c\log n$ for some $c>0$, the term \eqref{eq:markov} goes to zero.
Now we focus on the first  term \eqref{eq1}. By Markov's inequality and using the notation $\mathbb{W}=(g_{lk})$ for the white noise from Section \ref{sec:intro} and $\ix=(X_{lk})$ for the observation sequence, we find the following upper bound 
\begin{align}
&\frac{\sqrt{n}}{D}\left\{E_{f_0}\int_\mE 
\|f_{\mT,\b}^{j_0}-f_0^{j_0}\|_{\mathcal{M}(w)}d\Pi[f_{\mT,\b}\C X]\1_{\cA}\right\}
\leq\frac{E_{f_0}\left\{\|\mathbb{W}^{j_0}\|_{\mathcal{M}(w)}\1_{\cA}\right\}}{D}\label{eq:bound1}\\
&\quad\qquad\qquad+\frac{\sqrt{n}}{D}
E_{f_0}\left\{\int_\mE 
\|\ix^{j_0}-f_{\mT,\b}^{j_0}\|_{\mathcal{M}(w)}d\Pi[f_{\mT,\b}\C X]\1_{\cA}\right\}.
\end{align}
We can write the second term as
\begin{equation}\label{eq:term}
\sum_{\mT\in\mathsf{T}}\pi[\mT\C X]E_{f_0}\int_{\mathcal{E}}
\left(\sup_{l\leq j_0}l^{-1/2}\max_{0\leq k<2^l}\frac{\sqrt{n}}{D}|X_{lk}-\beta_{lk}|\right)d\Pi[\b_\mT\C \X_\mT].
\end{equation}
Note that all trees $\mT\in\mathsf{T}$ fit $j_0$ layers and under both the $g$-prior and the independent prior, the coefficients $\beta_{lk}$ for $0\leq l\leq j_0$ are a-priori (and a-posteriori) independent given $\cT$. Similarly as in the proof of Theorem 2 in \cite{castillo_nickl1}, we can show that the term \eqref{eq:term} is bounded by a constant by  first showing that for each $l\leq j_0$ and $0\leq k<2^l$
\begin{equation}\label{eq:mgf_0}
E_{f_0}\left\{\int \e^{t\sqrt{n}(\beta_{lk}-X_{lk})}d\Pi[\b_\mT\C \X_\mT]\1_\cA\right\}\leq C\e^{t^2/2}.
\end{equation}
This follows  from \cite{castillo_nickl1}. The second term $E_{f_0}\left\{\|\mathbb{W}^{j_0}\|_{\mathcal{M}(w)}\1_\cA\right\}$ is also bounded by $C^\star/D$ for some $C^\star>0$. Choosing $D=D(\eta)>0$
large enough, the term on the left side of \eqref{eq:bound1} can be made smaller than $\eta/2$. \qedhere

\end{proof}

The second step {in the proof of Theorem \ref{thm_bvm}} is showing convergence of finite-dimensional distributions (as in Proposition 6 of \cite{castillo_nickl1}). 
 Similarly as in the proof of Theorem 2 of \cite{castillo_nickl1}, convergence of the finite-dimensional distributions can be established by showing  BvM for the projected  posterior distribution onto $V_{j}=\mathrm{Vect}\{\psi_{lk},l\leq j\}$   
for any fixed $j\in\N$. Denote with $\b_j=(\beta_{-10},\beta_{00},\dots,\beta_{j\,2^{j}-1})'$ the Haar wavelet coefficients up to the level $j$. The prior on $\b_j$ consists of  $\b_j\sim\mathcal{N}(0,\Sigma_{j})$, where $\Sigma_j$ is the submatrix of $\Sigma$ that corresponds to coefficients up to level $j$. 

Let us first consider the case of the independent prior $\Sigma_\cT=I_K$. Because $j_0(n)\rightarrow\infty$, for large enough $n$ we have an {\sl independent product prior} on $\b_j$ when  $\Sigma_j=I$. Then one is exactly in the setting of Theorem 7 of \cite{castillo_nickl2} which derives finite-dimensional BvM for product priors (see the paragraph below the statement of Theorem 7 in \cite{castillo_nickl2} for two different arguments). 

The case of the $g$--prior $\Sigma_\cT=g_n(A_\cT'A_\cT)^{-1}$ is more involved, as  the induced prior distribution on the first coordinates is not of product form. Nevertheless, one can express the posterior distribution on coefficients as a mixture over trees (all containing the first $j_0(n)$ layers) of certain $\cT$--dependent Gaussian distributions (complemented by zeroes for the coefficients outside the tree $\cT$), and study each individual mixture component separately.  %Let $V_0$ denote the linear space $\text{Vect}\{\psi_{lk},\, l\le j_0(n)\}$. 
Let  $P_{V_j}$ be the $n\times n$ projection matrix onto $V_j$ and $P_{V_j}^\cT$ the $|\cT_{ext}|\times|\cT_{ext}|$ projection matrix onto $V_j$, projecting the coordinates corresponding to nodes in $\cT$ only (recalling that by definition of the prior, all nodes of $V_j$ are in trees $\cT$ sampled from the prior). We also denote by $I_{V_j}$ the identity matrix on $V_j$. It is enough for our needs to show, if TV$(P,Q)$ denotes the total variation distance between the probability distributions $P$ and $Q$, that
\begin{equation}\label{goalfd}
\text{TV}\left(\Pi[\cdot\given X]\circ P_{V_j}^{-1}\, , R^X_j\right)=o_P(1),
\end{equation}
where $R^X_j:=\cN(P_{V_j}X,I_{V_j}/n)$. 
From the expression of the posterior \eqref{eq:posterior_beta}, %it follows that given $\cT$, 
\begin{equation} \label{postgit} 
 \be_\cT \given \cT, X \,\sim\, \cN(\mu_\cT(X),\wt\Sigma_\cT)=:Q_\cT^X,
\end{equation} 
where $\mu_\cT(X):=n\wt\Sigma_\cT X_{\cT}$ and $\wt\Sigma_\cT=(nI_K+\Sigma_\cT)^{-1}$. 
Further, the coefficients $\be_{lk}$ for $(l,k)\notin \cT_{int}'$ are zero, which together gives a prior on $\b_{L-1}\in \RR^{2^L}=\RR^{n}$. By definition of the prior distribution, all trees $\cT$ sampled from the prior contain the nodes $(l,k), l\le j_0(n)$, in particular all nodes corresponding to $V_j$, so (identifying in slight abuse of notation a matrix with its corresponding linear map) the projected posterior $\Pi[\cdot\given X,\cT]\circ P_{V_j}^{-1}$ coincides with $\cN(\mu_\cT(X),\wt\Sigma_\cT)\circ {P_{V_j}^{\cT}}^{\, -1}=:Q_{\cT,j}^X$. %Let $\|P-Q\|_1$ denote the $L^1$--distance between $P,Q$ and recall $\|P-Q\|_1=2\text{TV}(P,Q)$. 
Then
\begin{align*}
\lefteqn{\text{TV}\left(\Pi[\cdot\given X]\circ P_{V_j}^{-1}\, , R^X_j\right)=\text{TV}\left(\sum_{\cT} \Pi[\cT\given X] Q_{\cT,j}^X \,,\, R^X_j\right)}\\
& =% \frac12 \left\|  \sum_{\cT} \Pi[\cT\given X] \left(Q_{\cT,j}^X-R_j^X \right)\right\|_1 
 %\le\frac12\sum_{\cT}   \Pi[\cT\given X] \left\|Q_{\cT,j}^X-R_j^X \right\|_1\\
\text{TV}\left(\sum_{\cT} \Pi[\cT\given X] Q_{\cT,j}^X \,,\, 
\sum_{\cT} \Pi[\cT\given X]  R^X_j\right) 
  %\left\|  \sum_{\cT} \Pi[\cT\given X] \left(Q_{\cT,j}^X-R_j^X \right)\right\|_1 
 \le\sum_{\cT}   \Pi[\cT\given X] \text{TV}\left(Q_{\cT,j}^X\,,\,R_j^X \right)\\
& \le \max_{\cT}\, \text{TV}\left(\cN(\mu_\cT(X),\wt\Sigma_\cT)\circ {P_{V_j}^{\cT}}^{\, -1},\cN(X_\cT,I_K/n)\circ {P_{V_j}^{\cT}}^{\, -1}\right)\\
& \le \max_{\cT}\, \text{TV}\left(\cN(\mu_\cT(X),\wt\Sigma_\cT),\cN(X_\cT,I_K/n)\right),
\end{align*}
where sums and maxima in the last display span over trees that fill in the first $j_0(n)$ layers of nodes, and where the last line uses that the total variation distance can only decrease after projecting onto $V_j$ (one restricts to marginal probabilities in the definition of the t.v. distance). In order to obtain \eqref{goalfd}, one now needs to bound  individual distances given the tree $\cT$, in a uniform way with respect to $\cT$. By the triangle inequality, for any $\cT$ as above,
\begin{align*}
 \lefteqn{\text{TV}\left(\cN(\mu_\cT(X),\wt\Sigma_\cT),\cN(X_\cT,I_K/n)\right) }\\
&\le 
\text{TV}\left(\cN(\mu_\cT(X),\wt\Sigma_\cT),\cN(\mu_\cT(X),\frac{I_K}{n})\right) 
+ \text{TV}\left(\cN(\mu_\cT(X),\frac{I_K}{n}),\cN(X_\cT,\frac{I_K}{n})\right).
\end{align*}
Both terms on the right hand side of the last display can be bounded using Lemma \ref{lem-tv}, where one sets $d=K=|\cT_{ext}|$, $\mu=\mu_\cT(X)=\mu_1$, $\mu_2=X_\cT$, and $\Sigma=I_K/n=\Sigma_1$, $\Sigma_2=\wt\Sigma_\cT$. Then $\Sigma_1^{-1}\Sigma_2-I_K=n(nI_K+\Sigma_\cT^{-1})^{-1}-I_K=-(I_K+n\Sigma_\cT)^{-1}$,  using the formula $(I+B)^{-1}=I-(I+B^{-1})^{-1}$ for $B$ invertible. Setting $M_\cT:=(I_K+n\Sigma_\cT)^{-1}$, the first and second inequalities of Lemma \ref{lem-tv} lead to 
\begin{align*} 
\text{TV}\left(\cN(\mu_\cT(X),\wt\Sigma_\cT),\cN(\mu_\cT(X),\frac{I_K}{n})\right) 
& \leqa \|M_\cT\|_F \\
 \text{TV}\left(\cN(\mu_\cT(X),\frac{I_K}{n}),\cN(X_\cT,\frac{I_K}{n})\right)
& \leqa \frac{\|M_\cT X_\cT\|^2}{\sqrt{\frac{1}{n}
 \|M_\cT X_\cT\|^2}}=\sqrt{n}\|M_\cT X_\cT\|.
\end{align*}
One now notes  $\|M_\cT\|_F\le \sqrt{K}\la_{max}(M_\cT)=\sqrt{K}/\la_{min}(I_K+n\Sigma_\cT)$. By Proposition \ref{prop:eigenspectrum}, we have $\la_{min}((A_\cT'A_\cT)^{-1})$ is at least $2^{-L}\ge 1/n$, and one deduces that $\la_{min}(I_K+n\Sigma_\cT)\geqa 1+ng_n/n\ge 1+g_n$, so that $\|M_\cT\|_F\leqa \sqrt{K}/g_n\leqa 2^{L/2}/g_n\leqa 1/\sqrt{n}=o(1)$, uniformly over $\cT$. On the other hand, we have, as $\la_{max}(M_\cT)\le 1/(1+g_n)$ as seen above and $X_\cT=\be^0_\cT+\veps_\cT/\sqrt{n}$, so that, working on the event $\cA$ from \eqref{event},
\begin{align*}
\|M_\cT X_\cT\|^2 & \le \la_{max}(M_\cT)^2 \|X_\cT\|^2 \leqa g_n^{-2}(\|\be^0_\cT\|^2
+ \|\veps_\cT\|^2/n)\\
&  \leqa g_n^{-2}(1
+ n(\log{n})/n) \leqa (\log{n})/g_n^2,
\end{align*}
where we have used that $\|\be^0\|^2=\|f^0\|^2$ is bounded and $\|\veps_\cT\|^2\leqa  n\log{n}$ on the event $\cA$. The previous bounds together imply that the total variation distance between $\cN(\mu_\cT(X),\wt\Sigma_\cT)$ and $\cN(X_\cT,I_K/n)$ goes to $0$ uniformly in $\cT$ on the event $\cA$. As $P[\cA^c]$ vanishes, this proves \eqref{goalfd}.
% \re{The $g$-prior induces correlation structure. \color{blue} Indeed :) Thank you for noticing. 
%Can we still show this step using, e.g., Theorem 7 of \cite{castillo_nickl2}? If not, we should just rephrase the result for the independent prior.}
%According to Lemma 6 of \cite{castillo_nickl1}, this concludes the proof of Theorem  \ref{thm_bvm}.

\pagebreak

\bibliographystyle{imsart-number}
%\bibliography{references}

%\section{Proofs of Auxiliary Lemmata}
%\subsection{Proof of Lemma \ref{lemma:haar}}\label{sec:proof_lemma:haar}

\end{document}